\documentclass[reqno]{amsart}

\usepackage[all]{xy}
\usepackage{graphicx}

\usepackage{amssymb}
\usepackage{amsmath}
\usepackage{mathrsfs}
\usepackage{epsfig}
\usepackage{amscd}
\usepackage{graphicx, color}

\def\E{\ifmmode{\mathbb E}\else{$\mathbb E$}\fi} %natural numbers
\def\N{\ifmmode{\mathbb N}\else{$\mathbb N$}\fi} %natural numbers%
\def\R{\ifmmode{\mathbb R}\else{$\mathbb R$}\fi} %real numbers
\def\Q{\ifmmode{\mathbb Q}\else{$\mathbb Q$}\fi} %rational numbers
\def\C{\ifmmode{\mathbb C}\else{$\mathbb C$}\fi} %complex numbers
\def\H{\ifmmode{\mathbb H}\else{$\mathbb H$}\fi} %complex numbers
\def\Z{\ifmmode{\mathbb Z}\else{$\mathbb Z$}\fi} %integers
\def\P{\ifmmode{\mathbb P}\else{$\mathbb P$}\fi} %real numbers
\def\T{\ifmmode{\mathbb T}\else{$\mathbb T$}\fi} %real numbers
\def\SS{\ifmmode{\mathbb S}\else{$\mathbb S$}\fi} %real numbers
\def\DD{\ifmmode{\mathbb D}\else{$\mathbb D$}\fi} %real numbers

\renewcommand{\a}{\alpha}
\renewcommand{\b}{\beta}
\renewcommand{\d}{\delta}

\newcommand{\e}{\varepsilon}
\newcommand{\g}{\gamma}
\newcommand{\G}{\Gamma}

\renewcommand{\t}{\tau}

\renewcommand{\o}{\omega}

\newcommand{\del}{\partial}

\newcommand{\ben}{\begin{enumerate}}
\newcommand{\een}{\end{enumerate}}
\newcommand{\be}{\begin{equation}}
\newcommand{\ee}{\end{equation}}
\newcommand{\bea}{\begin{eqnarray}}
\newcommand{\eea}{\end{eqnarray}}
\newcommand{\beastar}{\begin{eqnarray*}}
\newcommand{\eeastar}{\end{eqnarray*}}
\newcommand{\bc}{\begin{center}}
\newcommand{\ec}{\end{center}}

\theoremstyle{theorem}
\newtheorem{thm}{Theorem}[section]
\newtheorem{cor}[thm]{Corollary}
\newtheorem{lem}[thm]{Lemma}
\newtheorem{prop}[thm]{Proposition}

\theoremstyle{definition}
\newtheorem{defn}[thm]{Definition}
\newtheorem{rem}[thm]{Remark}

\newtheorem{exm}[thm]{Example}

\newtheorem*{thm*}{Theorem}

\numberwithin{equation}{section}

\def\R{{\mathbb R}}

\def\Crit{{\hbox{Crit}}}

\def\E{{\mathbb E}}
\def\Z{{\mathbb Z}}
\def\C{{\mathbb C}}
\def\R{{\mathbb R}}
\def\P{{\mathbb P}}

\def\N{{\mathbb N}}

\def\11{{\mathbb I}}

\def\delbar{{\overline \partial}}
\def\dudtau{{\frac{\del u}{\del \tau}}}
\def\dudt{{\frac{\del u}{\del t}}}

\def\C{\mathbb{C}}
\def\Z{\mathbb{Z}}

\def\T{\mathbb{T}}

\def\Q{\mathbb{Q}}

\def\E{\ifmmode{\mathbb E}\else{$\mathbb E$}\fi} %natural numbers
\def\N{\ifmmode{\mathbb N}\else{$\mathbb N$}\fi} %natural numbers
\def\R{\ifmmode{\mathbb R}\else{$\mathbb R$}\fi} %real numbers
\def\Q{\ifmmode{\mathbb Q}\else{$\mathbb Q$}\fi} %rational numbers
\def\C{\ifmmode{\mathbb C}\else{$\mathbb C$}\fi} %complex numbers
\def\H{\ifmmode{\mathbb H}\else{$\mathbb H$}\fi} %complex numbers
\def\Z{\ifmmode{\mathbb Z}\else{$\mathbb Z$}\fi} %integers
\def\P{\ifmmode{\mathbb P}\else{$\mathbb P$}\fi} %real numbers
\def\SS{\ifmmode{\mathbb S}\else{$\mathbb S$}\fi} %real numbers
\def\DD{\ifmmode{\mathbb D}\else{$\mathbb D$}\fi} %real numbers

\def\R{{\mathbb R}}

\def\Crit{{\hbox{Crit}}}
\def\E{{\mathbb E}}
\def\Z{{\mathbb Z}}
\def\C{{\mathbb C}}
\def\R{{\mathbb R}}

\def\N{{\mathbb N}}
\def\LL{{\mathcal L}}
\def\MM{{\mathcal M}}

\def\UU{{\mathcal U}}
\def\VV{{\mathcal V}}

\def\delbar{{\overline \partial}}

\def\a{\alpha}
\def\b{\beta}

\def\d{\delta}  
\def\e{\varepsilon} 
  
\def\g{\gamma}  \def\G{\Gamma}

\def\l{\lambda}  
\def\m{\mu}

\def\r{\rho}

\def\o{\omega}  

  \def\S{\Sigma}
  
\def\t{\tau}

\def\x{\xi}

\def\CA{{\mathcal A}}
\def\CB{{\mathcal B}}
\def\CC{{\mathcal C}}
\def\CD{{\mathcal D}}

\def\CF{{\mathcal F}}

\def\CH{{\mathcal H}}

\def\CJ{{\mathcal J}}
\def\CK{{\mathcal K}}
\def\CL{{\mathcal L}}
\def\CM{{\mathcal M}}

\def\CR{{\mathcal R}}

\def\CU{{\mathcal U}}
\def\CV{{\mathcal V}}

\def\darr#1{\raise1.5ex\hbox{$\leftrightarrow$}
\mkern-16.5mu #1}

\def\roughly#1{\raise.3ex\hbox{$#1$\kern-.75em
\lower1ex\hbox{$\sim$}}}

\def\opname#1{\mathop{\kern0pt{\rm #1}}\nolimits}

\def\End{\opname{End}}
\def\dim{\opname{dim}}

\def\supp{\operatorname{supp}}

\def\Per{\operatorname{Per}}
\def\End{\operatorname{End}}
\def\Reeb{\operatorname{Reeb}}
\def\Aut{\operatorname{Aut}}

\begin{document}
\quad \vskip1.375truein

\def\mq{\mathfrak{q}}
\def\mp{\mathfrak{p}}
\def\mH{\mathfrak{H}}
\def\mh{\mathfrak{h}}
\def\ma{\mathfrak{a}}
\def\ms{\mathfrak{s}}
\def\mm{\mathfrak{m}}
\def\mn{\mathfrak{n}}
\def\mz{\mathfrak{z}}
\def\mw{\mathfrak{w}}
\def\Hoch{{\tt Hoch}}
\def\mt{\mathfrak{t}}
\def\ml{\mathfrak{l}}
\def\mT{\mathfrak{T}}
\def\mL{\mathfrak{L}}
\def\mg{\mathfrak{g}}
\def\md{\mathfrak{d}}
\def\mr{\mathfrak{r}}

\title[Scale-dependent gluing]
{Floer trajectories with immersed nodes and scale-dependent gluing }
\author{Yong-Geun Oh}
\address{Department of Mathematics, University of Wisconsin-Madison, Madison, WI, 53706
\& Korea Institute for Advanced Study, Seoul, Korea}
\email{oh@math.wisc.edu}
\thanks{Y.-G. Oh is partially supported by the NSF grant \# DMS 0503954}

\author{Ke Zhu}
\address{Department of Mathematics, University of Wisconsin-Madison, Madison, WI, 53706}
\email{kzhu@math.wisc.edu}

\date{Revision, January 6, 2010}

\begin{abstract}
Development of pseudo-holomorphic curves and Floer homology in
symplectic topology has led to moduli spaces of pseudo-holomorphic curves
consisting of both
``smooth elements" and ``spiked elements", where the latter are
combinations of $J$-holomorphic curves (or Floer trajectories) and
gradient flow line segments. In many cases the ``spiked elements"
naturally arise under adiabatic degeneration of ``smooth elements"
which gradually go through \emph{thick-thin decomposition}. The
reversed process, the recovering problem of the ``smooth elements"
from ``spiked elements" is recently of much interest.

In this paper, we define an enhanced compactification of the moduli space of Floer
trajectories under Morse background using the adiabatic
degeneration and the scale-dependent gluing techniques. The
compactification reflects the 1-jet datum of the smooth Floer
trajectories nearby the limiting nodal Floer trajectories arising
from adiabatic degeneration of the background Morse function. This
paper studies the gluing problem when the limiting gradient
trajectories has length zero through a renomalization process. The
case with limiting gradient trajectories of non-zero length will
be treated elsewhere.

An immediate application of our result is a complete proof of the
isomorphism property of the PSS map: A
proof of this isomorphism property was outlined by
Piunikhin-Salamon-Schwarz \cite{PSS} in a way somewhat different
from the current proof in its details. This kind of scale-dependent
gluing techniques was initiated in [FOOO07] in relation to the
metamorphosis of holomorphic polygons under Lagrangian surgery and
is expected to appear in other gluing and compactification problem
of pseudo-holomorphic curves that involves `adiabatic' parameters or
rescaling of the targets.
\end{abstract}

\maketitle

\tableofcontents

\section{Introduction}

Development of pseudo-holomorphic curves and Floer
homology in symplectic geometry has led to moduli spaces consisting
of both ``smooth elements" and ``spiked elements", where the latter
are combinations of $J$-holomorphic curves (or Floer trajectories)
and gradient flow segments. For example, they appear in the
generalized holomorphic building in symplectic field theory
\cite{BEHWZ}, in the cluster complex \cite{cor-lal}, and even
earlier in the works \cite{fukaya:homotopy}, \cite{oh:newton},
\cite{PSS}, \cite{mschwarz0} and \cite{mundet-tian}. In many cases
the ``spiked elements" naturally arise from adiabatic degeneration
of ``smooth elements" which gradually decompose into ``thick parts"
and ``thin parts".

The adiabatic degeneration and its reversed process of the type
studied in this paper has appeared in \cite{foh:ajm}, \cite{Ek} and
\cite{Rwd}. The paper \cite{SW} studied another type of adiabatic
degeneration in a different context. All these papers are, however,
restricted to the case \emph{without quantum contribution}, i.e,
without bubbling phenomenon. The papers \cite{oh:adiabatic,oh:dmj}
and \cite{mundet-tian} studied adiabatic degeneration with quantum
contribution close to the one studied in this paper. However the
recovering problem was only mentioned and left as a future work in
\cite{oh:newton}, \cite{mundet-tian}.

Part of the difficulty for the recovering problem lies in finding
good local models near the junction points where the $J$-holomorphic
curve and gradient flow intersect. It turns out the derivative
information of the $J$-holomorphic curve and the gradient flow at
the junction point is needed to determine the local model. Besides a
good local model, appropriate Fredholm theory and implicit function
theorem are needed in order to glue the ``spiked elements" to
``smooth elements" in a controlled way to reflect the adiabatic
degeneration. It turns out that the scale-dependent gluing technique
carried out in chapter 10 of \cite{fooo07} in relation to
metamorphosis of $J$-holomorphic polygon under the Lagrangian
surgery, which treats a small region near the junction point as
about the same size as the original target manifold, is needed to
retain the geometric features of the local model under the
perturbation via implicit function theorem. Large part of the
analysis used in this paper is motivated by those in \cite{fooo07}.

\subsection{Adiabatic degeneration of Floer trajectories}

In this paper, we study the adiabatic degeneration of maps
$u:\R\times S^1\to M$ satisfying the following 1-parameter
($0<\e<\e_0$) family of Floer equations:
 \be\label{eq:KJE} (du + P_{K_{\e}}(u))^{(0,1)}_{J_\e} = 0
\quad\mbox{ or equivalently }\, \delbar_{J_\e}(u) +
(P_{K_\e})_{J_\e}^{(0,1)}(u) = 0, \ee
We refer to Section 3 for detailed exposition of \eqref{eq:KJE}, the
invariant form of the Floer equation. The expression of the
degenerating Hamiltonian $K_\e: \R\times S^1\times M\to \R$ is given
by
\be\label{eq:KR} K_{\e}(\tau,t,x) =
\begin{cases}\kappa_\e^+(\t)\cdot H(t,x) \quad &\mbox{for }\,
 \tau  \geq R(\e) \\
\r_\e(\t)\cdot \e f(x) \quad &\mbox{for }\,  |\t| \leq R(\e)  \\
\kappa_\e^-(\t)\cdot H(t,x) \quad &\mbox{for }\, \tau  \leq - R(\e)
\end{cases}
\ee
where $\kappa_e^\pm$ and $\rho_\e$ are suitable cut-off functions
(See \eqref{eq:KR} for the precise definition.) This type of
equations, for example, appears in the study of
isomorphism property of the PSS map introduced in \cite{PSS}.

Roughly speaking, the adiabatic degeneration
occurs because $K_\e$ restricts to Morse function $\e f$ on longer
and longer cylinder $[-R(\e),R(\e)]\times S^1$ in $\R\times S^1$. A basic assumption that
we put on this paper is that $R(\e)$ satisfies
\be\label{eq:length=0}
\lim_{\e\to 0} \e R(\e)=0.
\ee
The general case of $\lim_{\e\to 0} \e R(\e)= \ell$ for $\ell > 0$ will be
studied in a forthcoming paper \cite{oh-zhu3}. Under this assumption,
it is proved in \cite{oh:adiabatic,oh:dmj} and \cite{mundet-tian} that
as $\e \to 0$, a degenerating sequence of Floer trajectories converges to
a \emph{nodal Floer trajectory} denoted by $(u_-,u_+)$.
Descriptions of nodal Floer trajectories and immersed nodal points are now in order.

Let $\dot{\S}_+$  be the Riemann sphere with one marked point $o_+$
and one positive puncture $e_+$. Choose analytical charts
at $o_+$ and at $e_+$ on some neighborhoods $O_+$ and $E_+$ respectively, so that conformally
$O_+\backslash o_+ \cong (-\infty,0]\times S^1$, and $E_+\backslash
e_+\cong [0,+\infty)\times S^1$. We use $t$ for the $S^1$ coordinate
and $\t$ for the $\R$ coordinate. Then $\{-\infty\}\times S^1$ and
$\{+\infty\}\times S^1$ correspond to $o_+$ and $e_+$ respectively.

We consider a vector-valued 1-form $K_+$ on $\dot \Sigma$ with
its values in the set $ham(M,\omega)$ of Hamiltonian vector fields on $(M,\omega)$,
and $\dot \Sigma = \C P^1 \setminus \{e_+, o_+\} \cong \C \setminus \{0\} \cong \R \times S^1$.
We denote by $(\tau,t)$ the standard coordinate on $\R \times S^1$. With
respect to this coordinates, we require $K_+$ satisfy
\be \left\{
\begin{array}{rcll} K_+ &=& 0 &\text{near}\; o_+\\
K_+ & = & H_+(t,x)\,dt & \text{near}\; e_+
\end{array} \right.
\ee
$H_+:S^1\times M \to \R$  is a Hamiltonian function independent
of the variable $\tau$.

Let $z_+:S^1 \to M$ be a nondegenerate  periodic orbit of
$H_+$ and consider a finite energy solution $u_+: \dot \Sigma \to M$
of the Floer equation \eqref{eq:KJE} associated to $K_+$. By the finite
energy condition and since $K_+ \equiv 0$ near $o_+$, $u_+$ extends
smoothly across $o_+$ and can be regarded as a smooth map
defined on $\C$ that is holomorphic near the origin $0 \in \C$.

Similarly we consider 1-form $K_-$ on $\dot \Sigma$
and $\dot \Sigma = \C P^1 \setminus \{o_-, e_-\} \cong \C \setminus \{0\} \cong \R \times S^1$.
We denote by $(\tau,t)$ the standard coordinate on $\R \times S^1$
so that $+\infty$ corresponds to $o_-$ and $ -\infty$ to $e_-$.
With respect to this coordinates, we require $K_+$ satisfy
\be \left\{
\begin{array}{rcll} K_- &=& 0 &\text{near}\; o_-\\
K_- & = & H_-(t,x)\,dt & \text{near}\; e_-
\end{array} \right.
\ee
$H_\pm:S^1\times M \to \R$ are a pair of Hamiltonian functions independent
of the variable $\tau$. Let $z_\pm:S^1 \to M$ be a nondegenerate periodic orbit of
$H_\pm$ and its lifting $[z_\pm,w_\pm]$ of $z_\pm$.

A nodal Floer trajectory is, by definition, the gluing $u_- \# u_+$ at $u_-(o_-) = u_+(o_+)$
where $u_\pm$ are the solutions of the Floer equation associated to $K_\pm$
respectively. We say that \emph{a nodal point of $u_-\# u_+$ is immersed}
if $u_\pm$ are immersed at $o_\pm$ respectively.

One of the main results of the paper is the following enhanced
compactification theorem in Section 12 (Theorem 12.10) for gluing
and its surjectivity:

\begin{thm}\label{1-jetconvergence} Suppose that $u_-, \,
u_+$ are immersed at the node
$$
p = u_-(o_-) = u_+(o_+).
$$
Let $Glue(u_-,u_+)$ be the nodal Floer trajectory formed by $u_-$
and $u_+$ with nodal points $p = u_-(o-) = u_+(o_+)$. Suppose that
$u_n$ converges to $Glue(u_-,u_+)$ in level 0. Then there exists a
subsequence $u_{n_i}$ and a sequence $\e_i \to 0$ such that
$u_{n_i}$ converges to $(u_-,u_+, u_0)$ in the $\{\e_i\}$-controlled
way.
\end{thm}

A more detailed description of the local model $u_0$ above is in order
now.

The convergence in level 0 is the usual Gromov convergence
(Definition 12.6). The convergence in $\e$-controlled way is in
Definition 12.9. Roughly speaking, we magnify suitable small
neighborhood of the center of the neck of $u_{n_i}$ to keep track of
the degeneration in microscopic level, and in the limit we get a
proper holomorphic curve in ${\mathbb{C}}^n$ with asymptotic
convergence to simple Reeb orbits $\gamma_-$ and $\gamma_+$ of
the standard contact form $\lambda$ on $S^{2n-1}$
in the cylindrical end $\R \times (S^{2n-1},\lambda)$ of ${\mathbb{C}}^n$.
Such a holomorphic curve can be identified as a degree 2 rational curve
in ${\mathbb{C}} P^n$ intersecting the hyperplane at infinity at two points
$x_0$ and $x_{\infty}$. We have the following classification result of such rational
curves (Proposition \ref{degree2curves}):

\begin{prop}
\label{degree2curves} Fix a hyperplane $H$ in ${\mathbb{C}}P^{n}$ and two
points $x_{0},\,x_{\infty }\in H$. Then there exists a unique,
modulo the action of $Aut({\mathbb{C}}P^{n};H)$, rational curve
passing through $x_{0},\,x_{1}$ of degree 2 which is the group of automorphisms of ${\mathbb{C}}%
P^{n}$ fixing $H$.
\end{prop}
We can give an explicit formula for such rational curves. (See Remark \ref{rem:explicit}
for such formula.) We refer to Theorem \ref{uniqueness}
for the explanation how this proposition can be used to provide the local model
$u_0$. (Strictly speaking the microscopic adiabatic limit has some remnant from the background small Morse function
$\e f$ put in the middle of $K_\e$ \eqref{eq:KR} above and is a proper holomorphic curve perturbed by
a linear vector $\tau \nabla f(p)$, i.e., has the form $u_0 + \tau \nabla f(p)$
See Proposition \ref{surj+}.)

We call $(u_-,u_+,u_0)$ an \emph{enhanced nodal Floer trajectory}. For a
given enhanced nodal Floer trajectory, we glue a 1-parameter family of
smooth Floer trajectories and show that they are all possible nearby
smooth Floer trajectories according to the topology defined by the
above enhanced convergence. (See Theorem \ref{embedding} for the precise statement.)

The equation \eqref{eq:KJE} is nothing but a coordinate free
expression of the equation arising in the framework of the PSS map
described in \cite{PSS}. A key step during the PSS scheme of proof of the
isomorphism property is to resolve the nodal Floer trajectory to a
1-parameter family of smooth Floer trajectories. Unlike the
smoothing trajectories obtained via the more conventional gluing
outlined in \cite{PSS}, \cite{mcd-sal04} (see also \cite{LuG}), our
resolved Floer trajectories is more closely tied to the limiting
configurations arising through adiabatic degeneration in that they
are aligned in the gradient flow direction near the node and is
related to the disk-flow-disk elements.
%For the moduli spaces consisting of both ``smooth elements" and ``spiked
%elements", a suitable deformation theory is needed, especially during
%the transition from the stratum of ``smooth elements" to that of ``spiked
%elements". Our resolved Floer trajectories are closely related to
%the disk-flow-disk elements.
%so we expect that this gluing, when
%combined with the study of the case $\lim_{\e \to 0} \e R(\e) = l > 0$
%\cite{oh-zhu3}, can be used to put a smooth structure nearby a nodal Floer trajectory in
%the parameterized moduli space.

\begin{rem}
As far as we know, the detail of a key gluing result needed
in the proof of isomorphism property of the PSS map announced in \cite{PSS} has not
been given yet in any previous literature and our paper is the first one that
provides a full detail of the proof of isomorphism property of the PSS map.
Our proof uses a somewhat possibly ``overkilled'' gluing result obtained in
Theorem \ref{1-jetconvergence} and \ref{embedding}: Although we did not check it,
it is conceivable that one might be able
to write down a proof, \emph{without rescaling target manifolds}, following the
more standard approach of Floer's gluing \cite{floer:intersect}.
However we strongly believe that to materialize such a proof
one will still \emph{need to assume} that \emph{nodes are immersed} as we do in the present paper.
Such a requirement has not been addressed in the proposed PSS scheme
in \cite{PSS} or in any other existing literature related to it.
We refer to the next subsection for more discussion on the non-triviality
of this gluing theorem involved in the PSS scheme. See also Remark \ref{rem:Keef}.
\end{rem}

In this paper, we take the PSS framework as a test case to apply our scale-dependent gluing
scheme thereto because the PSS picture appears as the simplest case
for an adiabatic degeneration yet manifests the general technique.
Our gluing scheme can also be applied to other context such as in
the story told in \cite{oh:newton} where the adiabatic degeneration
of holomorphic polygons under the total collapse of $k$ Lagrangian
graphs $\operatorname{Graph} df_i$, $i =1, \cdots, k$ in a Darboux
neighborhood of a given Lagrangian submanifold $L \subset M$ was
outlined which involves configurations of holomorphic curves joined
by gradient trajectories of $k$ different Morse functions in a
general symplectic manifold $(M,\omega)$. In this general case there
are non-constant holomorphic spheres or discs around unlike the case
of cotangent bundle studied in \cite{foh:ajm}. This is a subject of
future study \cite{oh-zhu3}.

\subsection{Nodal Floer trajectories with immersed nodes}

Temporarily we denote by $\CM_\e$ the general moduli space parameterized by
$\e$ for $ -\e_0 \leq \e \leq \e_0$ with some phase change at $\e = 0$.
We will focus on the one that appears in the above mentioned PSS scheme
but the same story can be applied to more general setting.  In relation
to the scheme of proof of the isomorphism property of the PSS map, for example, one
would like to prove a certain parametrized moduli space
$$
\CM^{para}: = \bigcup_{-\e_0 \leq \e \leq \e_0} \CM_\e \to [-\e_0,\e_0]
$$
defines a \emph{piecewise smooth} compact cobordism between $\CM_{-\e_0}$ and $\CM_{\e_0}$:
there occurs a `phase change' at $\e=0$. Due to the `phase change' at $\e=0$,
one needs to prove a \emph{bi-collar theorem}
of $\CM_0 \subset \CM^{para}$ to materialize the PSS-scheme.
From $-\e_0$ to $0$, one can construct the left one-sided collar by
finite dimensional differential topology (See Section 9). On the other hand,
for the right one-sided collar over $[0,\e_0]$, \cite{PSS} attempts to produce
the collar by a `standard gluing method' of `some' perturbed Cauchy Riemann
equation. More specifically, \cite{PSS} attempts to produce a diffeomorphism
$$
\bigcup_{0 \leq \e \leq \e_0} \CM_\e \cong \CM_0 \times [0,\e_0]
$$
for a sufficiently small $\e_0> 0$. However the details of this gluing theorem are
given neither in \cite{PSS}, \cite{LuG} nor in the recent book \cite{mcd-sal04}.
As far as the authors understand, construction of this diffeomorphism is not
as standard as \cite{PSS}, \cite{mcd-sal04} indicated.

The main result of the present
paper is to construct this one-sided collar, \emph{at $\e=0$} (not at $\infty$),
by producing a one-parameter family of Floer trajectories out of the nodal
Floer trajectories (out of $\CM_0$) by the adiabatic degeneration \cite{oh:adiabatic,oh:dmj},
\cite{mundet-tian} and a \emph{scale-dependent gluing method.}
We would like to emphasize that
due to the phase change at $\e = 0$ the standard gluing theorem
of parameterized moduli space over $\e$ \emph{cannot} be applied either here.

If we only consider the usual stable map convergence a $\e \to 0$, we only see the
standard nodal Floer trajectories as a degenerate limit when we ignore bubbling-off-spheres.
But to recover the nearby resolved Floer trajectories for $\e > 0$ and
construct the above mentioned one-sided collar, we need extra 1-jet data
that is lost into the node during the standard stable map convergence.
For this purpose, it is essential to assume that nodal points are \emph{immersed}.
For the purpose of completing the proof of isomorphism property of the PSS map,
consideration of such nodal Floer trajectories will be sufficient.

Let $[z_\pm,w_\pm]$ be periodic orbits with caps of $H_\pm$ respectively.
We denote by
$$
\CM_{stand}^{nodal}([z_-,w_-],[z_+,w_+]; (K,J))
$$
the set of nodal Floer trajecotories in class $B \in \pi_2(z_-,z_+)$ that satisfies
$$
[w_-] \# B \# [w_+] = 0.
$$
Here $\pi_2(z_-,z_+)$ is the set of homotopy class of maps $w: [0,1] \times S^1 \to M$
satisfying $w(0,t) = z_-(t)$, $w(1,t) = z_+(t)$. Note that the gluing $u_-\
\# u_+$ canonically assigns a class in $\pi_2(z_-,z_+)$. A general index theorem
\cite{sal-zehn} says that the virtual dimension of the moduli space
$\CM_{stand}^{nodal}([z_-,w_-],[z_+,w_+]; (K,J))$ is given by
$$
\mu_{H_-}([z_-,w_-]) - \mu_{H_+}([z_+,w_+])
$$
where $\mu_H([z,w])$ is the Conley-Zehnder index \cite{conley-zehn} of the
periodic orbit $z$ with cap $w$ associated to the Hamiltonian $H$. The sign conventions
of \cite{conley-zehn}, \cite{sal-zehn} are different from those used
in the present paper one way or the other. We refer to Appendix of \cite{oh:montreal}
for a discussion of the index formula in the convention used in the present paper.

The following theorem enable us to consider only the nodal Floer trajectories
with immersed node for the purpose of proving isomorphism property of the PSS map.

\begin{thm}\label{intro-immersed} Let $(K_\pm,J_\pm)$ be a Floer datum with the
asymptotic Hamiltonian $H_\pm$. Suppose that
$$
\mu_{H_-}([z_-,w_-]) - \mu_{H_+}([z_+,w_+]) < 2n-1.
$$
Then there exists a dense subset of $\CJ_\omega$ consisting of $J$'s
such that for any quintuple
$$
(u_-,u_+, r_-,r_+) \in  \CM_{stand}^{nodal}([z_-,w_-], [z_+,w_+]; (K,J))
$$
with $u_-(r_-) = u_+(r_+)$, $r_-$ and $r_+$ are immersed points of
$u_-$ and $u_+$ respectively, and
$$
[du_-(r_-)] \neq [du_+(r_+)] \quad \mbox{ in } \, \P(T_xM)
$$
with $x = u_-(r_-) = u_+(r_+)$. The same holds for a one
parameter family of such $(K_\pm,J_\pm)$.

In particular, these hold when $\mu([z_-,w_-]) - \mu([z_+,w_+]) =0$, or $ -1$.
\end{thm}

\subsection{Related works and organization of the content}

Our gluing theorem involves two moduli spaces in different scales.
This kind of gluing theorem first appeared in
\cite{foh:ajm} in symplectic geometry, in which Fukaya and Oh glued
holomorphic discs with boundary punctures at the intersections of
several gradient trajectories of different Morse functions after
they are shrunk \emph{with a prescribed scale} depending on
degeneration parameter $\e$. Another scale-dependent gluing theorem
has been also used in \cite{fooo07} in relation to the Lagrangian
surgery and metamorphosis of holomorphic polygons. Furthermore the
kind of analysis that has been used for the analysis of proper
pseudo-holomorphic curves in symplectic manifolds with cylindrical
ends \cite{hofer93}, \cite{HWZ96I,HWZ96II,HWZ:smallenergy} also
plays a crucial role in our analysis. This analysis is further
complicated by the fact that we have to work out the relevant
estimates in the setting of \emph{asymptotically cylindrical ends}
on \emph{incomplete} manifolds, especially in the proof of
surjectivity of the gluing.

Finally it would be worthwhile to mention that the analysis given
in the present paper is a first step towards a full understanding
of the conjectural picture described in \cite{oh:newton} which
would require this type of scale-dependent gluing theorem of
pseudo-holomorphic curves under the background Morse function, or
twisted by the Hamiltonian flow of a Morse function. Based on the
argument of adiabatic degeneration, the senior author indicated
that `homology' of the quantum chain complex will be isomorphic to
that of the Floer complex, if they defined. Study of some related
collapsing degenerations has been carried out by the senior author
in \cite{oh:adiabatic,oh:dmj} and by Mundet i Rierra and Tian \cite{mundet-tian}.

We would like to mention one potential application of our gluing scheme.
In \cite{fooo07}, scale-dependent gluing was used to compare the moduli space of
$J$-holomorphic triangles ending on 3 Lagrangian submanifolds $(L_0,
L_1, L_2)$ and the moduli space of $J$-holomorphic 2-gons ending on
two Lagrangian submanifolds $(L_0, L_1\#_{\lambda} L_2)$, where
$L_1\#_{\lambda}L_2$ is obtained by Lagrangian surgery from $L_1$
and $L_2$. Similar to the Lagrangian surgery to smooth the singular
Lagrangian submanifolds $L_1 \cup L_2$ to $L_1\#_{\lambda} L_2 $, we
expect our scale-dependent gluing can be used to understand
$J$-holomorphic curves in singular target spaces, or its change when
the target manifold undergoes some surgery.

A brief summary of each part of the paper is in order. In Part I,
we set-up a new geometric framework which addresses an enhancement
of the description of standard nodal Floer trajectories.
In this enhancement, it is essential to assume that the nodes of
nodal Floer trajectories are \emph{immersed} and to insert suitable
local models at the nodes in 1-jet level, so we prove Theorem
\ref{intro-immersed} (Theorem 5.4).

In Part II, we carry out a \emph{scale-dependent gluing} analysis to
glue two outer pseudo-holomorphic curves and the local model in
different scale which is somewhat reminiscent of the ones in
\cite{foh:ajm}, \cite{fooo07}. In this scale-dependent analysis,
\emph{the immersion property of nodal points and a proper choice of
scales} of neck-stretching relative to the adiabatic parameter is
essential.

In Part III, we combine these with the standard
deformation-cobordism argument to explain how our gluing theorem
can be used to give a proof of the isomorphism property of the PSS
map.

\medskip

Y.-G. Oh would like to thank K. Fukaya, H. Ohta and K. Ono for the
collaboration of the book \cite{fooo07}. A large part of analysis
carried out in Part II of the current work is much influenced by the
scale-dependent gluing analysis given in Chapter 10 of the book. He
also thanks Bumsig Kim for the help in the argument used in the
proof of Proposition \ref{degree2curves}, and National Institute for
Mathematical Sciences(NIMS) in Korea for providing its financial
support and office space during the fall of 2009 where the final
version of the present paper is finished. K. Zhu would like to thank
the math department of University of Wisconsin-Madison and Korea
Institute of Advanced Study for the nice environment, where most of
his research was carried out.

Finally, but not the least, we would like to express
our deep gratitude to the anonymous referee for pointing out many
inaccuracies in our presentation and providing many suggestions to improve
the presentation of the paper.

\section{Review of the classical Floer's equation}
\label{sec:classic}

Throughout this paper, $(M,\omega)$ is a compact symplectic
manifold. We will always identify $S^1$ with $\R/\Z$ which in
particular has the canonical marking $0 (\mod 1) \in S^1$. Denote
by $S = \R \times S^1$ the infinite cylinder with the unique
complex structure, denoted by $j$. We denote by $(\tau,t)$  be the
associated cylindrical coordinates such that
$$
\tau + it, \quad \tau \in \R, \, t \in S^1=\R/\Z
$$
provides the standard complex coordinates on $S$ identified with
the quotient space $S = \C/i \Z$ which lifts to the standard
coordinates $z = e^{2\pi(\tau + it)}$ on $\C$.

Let $J = J(\tau,t)$ be a 2-parameter family of almost complex
structures compatible with $\omega$ for $(\tau,t)\in R \times S^1$
satisfying the asymptotic condition
\be\label{eq:asympJ}
J(\tau,t) \equiv J(\pm \infty,t) \quad \mbox{for $\tau \geq R_+$
and $\tau \leq -R_-$}
\ee
for some $R_\pm \geq 0$. Denote the set of all such
$J$ by $\mathcal{J} = \CJ_\omega$, and by $\CJ^{cyl}_\omega$
the set of such $J$'s independent of $\tau$.

Next we consider two parameter family of smooth functions on $M$
parameterized by $(\tau,t) \in \R \times S^1$
$$
H = H(\tau,t,x)
$$
such that $H(\tau,t,x) \equiv H_\pm(t,x)$ for $\tau \geq R_+$ or
$\tau \leq R^-$. We call $H$ \emph{cylindrical} if $H$ is independent
of $\tau$. For each given cylindrical $H$, we consider the Hamilton
equation
$$
\dot x = X_H(t,x), \, t\in S^1,
$$
where $X_H$  is the Hamiltonian vector field of $H$, and denote by $\Per H$ the set of one-periodic solutions
$z(t)$, i.e., those satisfying $z(0) = z(1)$. We note that $z(t)$ can be written
as $z(t) = \phi_H^t(x)$ for some $x \in M$, where $\phi_H^t$ is the Hamiltonian flow for $H$ at time $t$. Then $z(t)$ is periodic if and
only if $x$ is a fixed point of the time-one map $\phi_H^1$ of
$X_H$.

For each given periodic orbits $z_\pm(t)$  at $\t=\pm \infty$ of
$H_\pm$ respectively, the Floer's perturbed Cauchy Riemann equation
associated to the pair $(H,J)$ has the form
\be\label{eq:HJCR}
\begin{cases}\dudtau + J\Big(\dudt - X_H(u)\Big) = 0\\
u(-\infty,t) = z_-(t), \, u(\infty,t) = z_+(t)
\end{cases}
\ee
for a map $u: \R \times S^1 \to M$. We call this equation
{\it Floer's perturbed Cauchy-Riemann equation} or simply as the perturbed
Cauchy-Riemann equation (associated to the pair $(H,J)$). This equation
may be regarded as the negative gradient flow equation of an action
functional defined on the Novikov covering space.
The Floer theory largely relies on the study of  the moduli spaces of
{\it finite energy} solutions $u: \R \times S^1 \to M$ of the kind
(\ref{eq:HJCR}). The relevant energy function is given by

\begin{defn} For a given smooth map $u: \R \times S^1 \to M$,
we define the energy, denoted by $E_{(H,J)}(u)$, of $u$ by
$$
E_{(H,J)}(u) = \frac{1}{2} \int \Big(\Big|\dudtau\Big|^2_{J_t} + \Big|
\dudt - X_H(u)\Big|_{J_t}^2 \Big)\, dt\, d\tau.
$$
\end{defn}

The equation (\ref{eq:HJCR}) has translational symmetry for the cylindrical
pair $(H,J)$ and counting the isolated trajectories of such pair
defines the Floer boundary map, and counting isolated trajectories of
generic (non-cylindrical) pair defines the Floer chain map.
This finishes the summary of Floer's original set-up of the
Floer homology.

When one considers the product structure on the Floer homology, one needs
to consider general Riemann surfaces, $\dot \Sigma$ of genus zero with
punctures. We denote by $\Sigma$ a closed Riemann surface,
possibly with non-empty boundary $\del \Sigma$, and $\dot \Sigma$
the corresponding punctured Riemann surface with a finite
number of marked points in $\operatorname{Int}\Sigma$.

\section{Invariant set-up of the Floer equation}
\label{sec:invariant}

In this section, we will formulate the set-up for
the general Floer's perturbed Cauchy-Riemann equation on compact Riemann
surface with a finite number of punctures.
This requires a coordinate-free framework of the equation.

\subsection{Punctures with analytic coordinates}
\label{subsec:punctures}

We start with the description of positive and negative
\emph{punctures}. Let $\Sigma$ be a compact Riemann surface with
a marked point $p \in \Sigma$. Consider the corresponding
punctured Riemann surface $\dot \Sigma$  with an analytic
coordinates $z: D\setminus \{p\} \to \C$  on a neighborhood $D
\setminus \{p\} \subset \dot \Sigma$. By composing $z$ with a
linear translation of $\C$, we may assume $z(p) = 0$.

We know that $D \setminus \{p\}$ is conformally isomorphic to both
$[0,\infty) \times S^1$ and $(-\infty,0] \times S^1$.

\begin{enumerate}
\item We say that the pair $(p;(D,z))$ has a \emph{incoming
cylindrical end} (with analytic chart) if we have
$$
D = z^{-1}(D^2(1))
$$
and are given by the biholomorphism
$$
(\tau, t) \in S^1 \times (-\infty,0] \mapsto e^{2\pi(\tau + it)}
\in D^2(1) \setminus \{0\} \mapsto z^{-1} \in D\setminus \{p\}.
$$
We call the corresponding puncture $p \in \Sigma$ a \emph{positive
puncture}.

\item We say that the pair $(p;(D,z))$ has a \emph{outgoing
cylindrical end} (with analytic chart) if we have
$$
D = z^{-1}(D^2(1))
$$
and are given by the biholomorphism
$$
(\tau,t) \in S^1 \times [0,\infty) \mapsto
e^{-2\pi(\tau + it)} \in D^2(1) \setminus \{0\} \mapsto z^{-1}
\in D \setminus \{p\}.
$$
In this case, we call the corresponding puncture $(p;(D,z))$ a
\emph{negative puncture} (with analytic chart).
\end{enumerate}

\subsection{Hamiltonian perturbations}
\label{eq:perturb}

Now we describe the Hamiltonian perturbations in a
coordinate free fashion. Such a description was given, for example,
by Seidel in \cite{seidel03,seidelbook,mcd-sal04}.

Let $\Sigma$ be a compact Riemann surface and $\dot \Sigma$ denote
$\Sigma$ with a finite number of punctures and analytic coordinates.
We denote by $\CJ_{0,\omega}$ the set of almost complex structures
that are cylindrical near the puncture with respect to the
given analytic charts $z = e^{\pm(2\pi(\tau + it)}$.
Define $\CJ_\Sigma$ or $\CJ_{\dot\Sigma}$ to be the set of maps
$J: \Sigma, \, \dot \Sigma \to \CJ_{0,\omega}$ respectively.

We recall that the standard $\delbar$-operator
$$
\delbar_J : u \mapsto \delbar_Ju:= \frac{du + J\circ du \circ
j}{2}
$$
defines a section of the vector bundle
$$
\Omega^{(0,1)}_J(\Sigma,M) \to C^\infty(\Sigma,M)
$$
where the fiber thereof at $u$ is given by the vector space
$$
\Omega^{(0,1)}_J(u^*TM): = C^\infty(\Lambda_J^{(0,1)}(u^*TM))
$$
where $\Lambda_J^{(0,1)}(u^*TM)$ is the set of anti-$J$-linear
maps from $(T\Sigma,j) \to (TM,J)$ lifting $u$, or in other words,
$u^*TM$-valued $(0,1)$-forms on $\Sigma$. Recall we have the decomposition
$$
\Omega^1(u^*TM) = \Omega^{(1,0)}_J(u^*TM)\oplus \Omega^{(0,1)}_J(u^*TM).
$$
In the cylindrical coordinates $(\tau,t)$,  the map
$$
\frac{\del}{\del \tau} \rfloor (\cdot) : \Omega^{(0,1)}_J(u^*TM)
\to \Omega^0(u^*TM) = C^\infty(u^*TM)
$$
defines a local isomorphism and the expression
$\dudtau + J \dudt$ in the Floer equation is nothing but
$$
2 \delbar_Ju\left(\frac{\del}{\del \tau}\right).
$$

We want to regard the perturbation term $-JX_H(u)$ in a similar
way. It will be the value of the $(0,1)$-part of
some one-form $P_{\dot\Sigma}(u)\in \Omega^1(u^*TM)$.
 Furthermore the term
involves a Hamiltonian vector field, not a general vector field.
We recall the exact sequence
$$
0 \to \R \to C^\infty(M) \to ham(M,\omega) \to 0
$$
where $ham(M,\omega)$ is the set of Hamiltonian vector fields on
$(M,\omega)$ and we assume that $M$ is compact and connected. This
sequence canonically splits : we have the integration map
$$
\int_M :C^\infty(M) \to \R \, ;\, h \mapsto \int_M h \, d\mu.
$$
Therefore this induces a natural exact sequence
$$
0 \to \Omega^1(\Sigma,\R) \to \Omega^1(\Sigma,C^\infty(M)) \to
\Omega^1(\Sigma,ham(M,\omega)) \to 0.
$$
If we restrict the Hamiltonians to the mean-normalized ones, i.e.,
those in the kernel of the above integral map, we have the
isomorphism
$$
\Omega^1(\Sigma,C^\infty_m(M)) \cong \Omega^1(\Sigma,ham(M,\omega)).
$$
We denote $C^\infty_m(M) = \ker \int_M$.

Now let $K \in \Omega^1(\Sigma,C^\infty(M))$ and denote by $P_K$
the corresponding one-form of $\Omega^1(\Sigma,ham(M,\omega))$.
Then for each choice of $\xi \in C^\infty(T\Sigma)$, ${K(\xi)}$ gives
a function on $M$ and so a Hamiltonian vector field $P_K(\xi) =
X_{K(\xi)}$ on $M$. In cylindrical coordinate $(\tau,t)$, we want
$K$ to satisfy
$$
- 2(P_K)^{(0,1)}(u)\left(\frac{\del}{\del\tau}\right) = - J X_H(u).
$$
It is easy to check that one such choice of $K$ will be
\be
\label{eq:K-cylin} K(\tau,t) =  H(t)\, dt
\ee
on the cylindrical ends for an arbitrary choice of $H$.

\begin{defn} We call $K \in \Omega^1(\Sigma,C^\infty(M))$
\emph{cylindrical} at the puncture $p \in \Sigma$ with analytic
chart $(D,z)$, if it has the form
$$
K(\tau,t) = H(t)\, dt
$$
in $D \setminus \{p\}$. We denote by $\CK_{\dot\Sigma}$ the set of
such $K$'s.
\end{defn}

One important quantity associated to the one-form $K$ is a two-form,
denoted by $R_K$, and defined by
\be\label{eq:R_K}
R_K\left(\xi_1,\xi_2\right)
= \xi_1[K(\xi_2)] - \xi_2[K(\xi_1)]
- \left\{K(\xi_2), K(\xi_1)\right\}
\ee
for two vector fields $\xi_1, \, \xi_2$, where $\xi_1[(K(\xi_2)]$
denotes directional derivative of the function $K(\xi_2)(z,x)$ with
respect to the vector field $\xi_1$ as a function on $\Sigma$,
holding the variable $x \in M$ fixed.
It follows from the expression that
$R_K$ is tensorial on $\Sigma$.

\begin{rem} This quantity has the interpretation as the \emph{curvature} of
a symplectic vector bundle over $\Sigma$ in the following way
\cite{banyaga}, \cite{seidel97}.
We regard the product $E = \Sigma \times (M,\omega)$ as a bundle of
symplectic manifold whose structure group is $Symp_0(M,\omega)$, the
identity component of $Symp(M,\omega)$. Each one-form $K$ defines
a horizontal subspace of $T_{(p,x)}E$ given by the subspace
$$
\CD_K(p,x) : = \{ (\xi, X_{K(\xi)}(x)) \mid \xi \in T_p\Sigma, \, x \in M\}
$$
and so can be regarded as an Ehresmann connection of $TE \to \Sigma$.
Then $R_K$ is the corresponding curvature of this connection $K$.
Note that the distribution $\CD_K \subset TE$ is integrable if
and only if $R_K = 0$ and also equivalent to saying that locally
$P_K$ can be integrated as the two-parameter family of
Hamiltonian isotopies
$$
\Lambda: (s,t) \mapsto \phi(s,t) \in Ham(M,\omega),
$$
where  $Ham(M,\omega)$ is the Hamiltonian diffeomorphism group on $M$.
This last statement was essentially proved by Banyaga \cite{banyaga}.
Motivated by this observation, we will call $R_K$ as the
curvature of the connection $K$.
\end{rem}

\subsection{Floer moduli spaces}
\label{subsec:floer}

Now we are ready to give the definition of the moduli space of
perturbed Cauchy-Riemann equation in a coordinate-free form.
The Hamiltonian-perturbed Cauchy-Riemann equation has the form
\be\label{eq:KJ}
(du + P_K(u))^{(0,1)}_J = 0 \quad\mbox{ or equivalently }\,
\delbar_J(u) + (P_K)_J^{(0,1)}(u) = 0
\ee
on $\Sigma$ in general. Following Seidel \cite{seidelbook}, we call a pair
$(K,J) \in \CK_{\dot\Sigma} \times \CJ_{\dot\Sigma}$ a \emph{Floer datum}.

For each given such pair $(K,J)$, it defines a perturbed Cauchy-Riemann operator
by
$$
\delbar_{(K,J)} u := \delbar_Ju + P_K(u)^{(0,1)}_J =
(du +P_K(u))^{(0,1)}_J.
$$
Let $(\frak p, \frak q)$  be a given set of
positive punctures $\frak p = \{p_1, \cdots, p_k\}$
and with negative punctures $\frak q = \{ q_1, \cdots, q_\ell\}$
on $\Sigma$. For each given Floer datum $(K,J)$ and a collection
$\vec z = \{z_*\}_{* \in \frak p \cup \frak q}$
of asymptotic periodic orbits
$z_*$ attached to the punctures $* = p_i$ or $* = q_j$,
we consider the perturbed Cauchy-Riemann equation
\be\label{eq:KJ-asymp}
\begin{cases}
\delbar_{(K,J)}(u) = 0 \\
u(\infty_*,t) = z_*(t).
\end{cases}
\ee
One more ingredient we need to give the definition of the
Hamiltonian-perturbed moduli space is the choice of an appropriate
energy of the map $u$. For this purpose, we fix a metric
$h_\Sigma$ which is compatible with the structure of the Riemann surface
and which has the cylindrical ends with respect to the given cylindrical
coordinates near the punctures, i.e., $h_\Sigma$ has the form
\be\label{eq:gSigma}
h_\Sigma = d\tau^2 + dt^2
\ee
on $D_* \setminus \{*\}$. We denote by $dA_\Sigma$ the
corresponding area element on $\Sigma$.

Here is the relevant energy function
\begin{defn}[Energy]
For a given asymptotically cylindrical pair $(K,J)$, we define
$$
E_{(K,J)}(u) =
\frac{1}{2}\int_\Sigma |du - P_K(u)|_J^2\, dA_\Sigma
$$
where $|\cdot|_{J(\sigma,u(\sigma))}$ is the norm of
$\Lambda^{(0,1)}(u^*TM) \to \Sigma$ induced by the metrics
$h_\Sigma$ and $g_J: = \omega(\cdot, J \cdot)$.
\end{defn}
Note that this energy depends only on the conformal
class of $h_\Sigma$, i.e., depends only on the complex structure
$j$ of $\Sigma$ and restricts to the standard energy for the usual
Floer trajectory moduli space given by
$$
E_{(H,J)} = \frac{1}{2}\int_{C_*}
\left(\left|\dudtau\right|_J^2 + \left|\dudt - X_H(u)\right|_J^2\right) \,
dt\, d\tau
$$
in the cylindrical coordinates $(\tau,t)$ on the cylinder $C_*$ corresponding
to the puncture $*$. $E_{(K,J)}(u)$ can be bounded by a more topological
quantity depending only on the asymptotic orbits, or more
precisely their liftings to the \emph{universal covering space}
of $\CL_0(M)$, where the latter is the contractible loop space of $M$.
As usual, we denote such a lifting of a periodic orbit $z$ by $[z,w]$
where $w:D^2 \to M$ is a disc bounding the loop $z$.

We recall the definition of the standard action functional
$\CA_H: \widetilde \CL_0(M) \to \R$ on the \emph{Novikov covering space}
\cite{hofer-sal} given by
$$
\CA_H([\gamma,w]) = -\int w^*\omega - \int_0^1 H(t,\gamma(t)) \, dt
$$
The following lemma can be derived by a straightforward computation.
See \cite{mschwarz}, \cite{oh:dmj}, \cite{seidel03}  for related calculations.

\begin{lem} Assume that $(K,J)$ is asymptotically cylindrical.
Let $\{[z_*,w_*]\}_{* \in \frak p\cup \frak q}$ be
a given collection of asymptotic periodic orbits and let $u$ have
finite energy. Then we have the identity
\be\label{eq:energyid}
E_{(K,J)}(u) = \sum_{i=1}^k \CA_{H_{p_i}}([z_i^+,w_i^+]) -
\sum_{j=1}^\ell \CA_{H_{q_j}}([z_j^-,w_j^-])
+ \int_\Sigma R_K(u)
\ee
where $R_K \in \Omega^2(\Sigma,C^\infty(M))$ is the curvature two-form
of the one-form $K$.
\end{lem}

Here we remark that the last curvature integral converges as $R_K(u)$
will have compact support by the hypothesis that $K$ is cylindrical
near the ends of $\dot \Sigma$.

We also consider the \emph{real blow-up} of $\dot \Sigma \subset \Sigma$ at
the punctures and denote it by $\overline\Sigma$ which is a compact Riemann surface
with boundary
$$
\del \overline \Sigma = \coprod_{* \in \mp \cup \mq} S^1_*
$$
where $S^1_*$ is the exceptional circle over the point $*$.
We note that since there is given a preferred coordinates near
the point $*$, each circle $S^1_*$ has the canonical identification
$$
\theta_*: S^1_* \to \R/\Z = [0,1] \mod 1.
$$

We note that for a given asymptotic orbits $\vec z$, one can define the
space of maps
$
u : \dot \Sigma \to M
$
which can be extended to $\overline \Sigma$ such that $u \circ \theta_* =
z_*(t)$ for $* \in \frak p \cup \frak q$.
Each such map defines a natural homotopy class $B$ relative to
the boundary. We denote the corresponding set of homotopy classes by $\pi_2(\vec z)$.
When we are given the additional data of bounding discs
$w_*$ for each $z_*$, then we can form a natural homology (in fact a homotopy
class), denoted by $B \# \left(\coprod_{* \in \mp \cup \mq} [w_*]\right) \in H_2(M)$,
by `capping-off' the boundary components of $B$ using the discs $w_*$
respectively.
\begin{defn} Let $\{[z_*,w_*]\}_{* \in \frak p \cup \frak q}$ be given.
We say $B \in \pi(\vec z)$ is \emph{admissible} if it satisfies
\be\label{eq:Bsharpws}
B \# \left(\coprod_{* \in \mp \cup \mq} [w_*]\right) = 0 \quad \mbox{in }\, H_2(M,\Z)
\ee
\end{defn}
where
$$
\# : \pi_2(\vec z) \times \prod_{* \in \frak{p} \cup \frak{q}} \pi_2(z_*) \to H_2(M,\Z)
$$
is the natural gluing operation of the homotopy class from $\pi_2(\vec z)$
and those from $\pi_2(z_*)$ for $* \in \frak{p} \cup \frak{q}$.
Now we are ready to give the definition of the Floer moduli spaces.

\begin{defn} Let $(K,J)$ be a Floer datum over $\Sigma$ with
punctures $\frak p, \, \frak q$, and let
$\{[z_*,w_*]\}_{* \in \frak p\cup \frak q}$ be the given
asymptotic orbits. Let $B \in \pi_2(\vec z)$ be a homotopy class admissible to
$\{[z_*,w_*]\}_{* \in \frak p\cup \frak q}$.  We define the moduli space
\be\label{eq:Mz*w*}
\CM(K,J;\{[z_*,w_*]\}_*) = \{u: \dot \Sigma \to M \mid
u \, \mbox{ satisfies (\ref{eq:KJ-asymp}) and
$[u]\# (\coprod_{* \in \mp \cup \mq} [w_*]) =0$ }\}.
\ee
\end{defn}

We note that the moduli space $\CM(K,J;\{[z_*,w_*]\}_*)$ is a
finite union of the moduli spaces
$$
\CM(K,J;\vec z; B); \quad B \# (\coprod_{* \in \mp \cup \mq} [w_*]) = 0:
$$
It follows from the energy estimate \eqref{eq:energyid} and Gromov compactness that
there are only finitely many elements $B \in \pi_2(\vec z)$ admissible to the given collection
$\{[z_*,w_*]\}_{* \in \frak p\cup \frak q}$.

\section{Formulation of the PSS maps}
\label{sec:pss-map}

In this section, we will give a precise formulation of the so called
PSS-map introduced in \cite{PSS}.

Let $f: M \to \R$ be a back-ground Morse function on
$M$ and $H = H(t,x)$ and $J = J(t,x)$. The goal of the PSS-map is to establish
an isomorphism between the Morse homology of $f$ and
the Floer homology of $(H,J)$.

One of the moduli spaces entering in the construction of the PSS-map
is the space of solutions of (\ref{eq:KJ-asymp}) with one
puncture, which can be either positive or negative, and with
one marked point playing the role of the origin of $\dot \Sigma$.

\subsection{The smooth moduli space $\CM_{(s_0;s_+,s_-)}(K,J;B)$}
\label{subsec:CMrspm}
\index{$\CM_{(s_0;s_+,s_-)}(K,J;B)$}
%From now on we assume the Riemann surface $\S$ is $S^2$.
We consider the triple
$$
\mp = \{p_1,\cdots, p_{s_+}\}, \, \mq = \{q_1, \cdots, q_{s_-}\}, \,
\mr = \{r_1, \cdots, r_{s_0}\}
$$
of positive and negative punctures, with analytic charts
assigned, and marked points respectively. We assume they
are all distinct points. We denote by
$$
\widetilde \CM_{(s_0;s_+,s_-)}
$$
the set of all such triples and by $\CM_{(s_0;s_+,s_-)}$ the quotient
space by the action of automorphisms of the punctured Riemann surface with marked points.
We call a triple $(\mr; \mp, \mq)$
\emph{stable} if it has a finite automorphism group. The space
$\CM_{(r;s_+,s_-)}$ is non-empty as long as $s_0 + s_+ + s_- \geq 3$.

Next we define $\CM_{(s_0;s_+,s_-)}(K,J;\vec z;B)$ in an obvious way,
\be\label{eq:CMKJB}
\CM_{(s_0;s_+,s_-)}(K,J;\vec z;B)=\{(u;\mr;\mp,\mq) \mid
u \, \mbox{ satisfies (\ref{eq:KJ-asymp}) }\, [u] = B\}
\ee
where $B$  is a given homotopy class of maps $u$ satisfying the
asymptotic conditions at the punctures. Here $\vec z = \{ z_* \}_{* \in \frak p \cup \frak q}$
is a given set of  asymptotic periodic orbits.
To avoid having continuous automorphisms, we will always assume that the asymptotic
Hamiltonian $H$ at the puncture is \emph{not} time-independent
when we consider the moduli space corresponding to
$$
(r;s_+,s_-) = (1;1,0) \, \mbox{ or }\, (1;0,1).
$$
We will not need to consider the case where $r=0$, $s_+ + s_- = 1$.
This assumption rules out the possibility of a circle symmetry for
the asymptotic solutions at infinity.

In addition, we will also assume that $K$ and $J$ satisfy
\bea\label{eq:0atr}
K & \equiv & 0 \\
J & \equiv & J_0 \quad \mbox{ near the marked point }\,  r\in \dot\Sigma
\eea
respectively where $J_0$  is a (time-independent) compatible almost complex structure
of $(M,\omega)$. We assume $J_0$ is generic.

This assumption together with the condition on the asymptotic
Hamiltonian being nondegenerate makes such $K$ do not carry
continuous symmetry and so a genuinely two-dimensional family
over $\dot \Sigma$. In particular any solution in these moduli space
has automatically a finite automorphism group at most. The following
can be derived by a standard argument.

\begin{prop} Let $J_0$ be a given compatible almost complex
structure on $(M,\omega)$. Suppose that all the asymptotic pairs $(H_*,J_*)$ are
Floer-regular in that $H_*$ are non-degenerate in the sense of
Lefshetz fixed point theory, and in that
the corresponding Floer moduli space is transverse.
Then there exists a generic choice of such $(K,J) \in \CK_{\dot\Sigma}
\times \CJ_{\dot\Sigma}$ such that the moduli space
$\CM_{(s_0;s_+,s_-)}((K,J);\vec z;B)$ become transverse. Furthermore
the dimension of the moduli space is given by
\beastar
\dim \CM_{(s_0;s_+,s_-)}(K,J;\vec z;B) & = &
\sum \mu_{H_*^+}([z_*^+,w_*^+]) - \sum \mu_{H_*^-}([z_*^-,w_*^-])\\
&{}& \quad + 2s_0 + n(s_+ - s_-)
\eeastar
where $[z_*,w_*]$ are the liftings of the asymptotic orbits with $B$ satisfying
\eqref{eq:Bsharpws}.
\end{prop}

\emph{From now on in the rest of the paper, we will exclusively concern
$\Sigma$ of genus zero.}

\subsection{The PSS maps $\Phi$ and $\Psi$}
\label{subsec:pss-map}
In this subsection, we recall the definitions
of the two PSS maps $\Phi$ and $\Psi$ from \cite{PSS} except that we
follow different grading conventions using the ones from
\cite{oh:alan} for the various grading issues. And we also use Morse
cycles of $-f$, instead of $f$, to represent the homology  of $M$:
In particular, the grading of Morse cycles is given by
$$
\operatorname{Index}_{(-f)}(p) = 2n - \operatorname{Index}_{f}(p).
$$
The issue of grading is \emph{not} essential for the proof and so
can be largely ignored. We just put this here for the consistency
with the papers by the senior author
\cite{oh:alan}-\cite{oh:montreal}.

Let $\dot{\S}_+$  be the Riemann sphere with one marked point $o_+$
and one positive puncture $e_+$. We
identify $\dot{\S}_+ \setminus \{o_+\} \cong \R \times S^1$ and
denote by $(\tau,t)$ the corresponding coordinates so that
$\{+\infty\}\times S^1$ correspond to $e_+$. We note that the
coordinates $(\tau,t)$ is defined modulo the the action of $\R \times S^1 \cong \C^*$
$$
(\tau,t) \mapsto (\tau + a, t + b).
$$
We consider the one form $K_+ \in \Omega^1(\dot \Sigma,
ham(M,\omega))$  such that in the above mentioned
coordinates $K_+$ as a 1-form in $\Omega^1(\dot \Sigma,ham(M,\omega))$ can be written as
\be \left\{
\begin{array}{rcll} K_+&=& 0 &\text{near}\; o_+\\
K_+ & = & H_+(t,x)\,dt & \text{near}\; e_+
\end{array} \right.
\ee
where $H_+:S^1\times M \to \R$ is a $t$-dependent Hamiltonian function.
By the remark on the coordinate $(\tau,t)$ made above, the phrases ``near $e_+$''
or ``near $o_+$'' put on the above definition of $K_+$ does not depend on
the choice of coordinates and has well-defined meaning.

Let $\LL_0(M)$ be the contractible free loop space of $M$
and $\LL_0(M)$ the Novikov covering space $\widetilde \LL_0(M)$.
Let $z_+ = z_+(t)\;(t\in S^1)$ be a nondegenerate periodic orbit of $H_+(t,x)$
We denote by $[z_+,w_+]$ a lifting of $z_+$ to $\widetilde \LL_0(M)$
and denote
$$
\widetilde{Per}(H_+) = \{[z_+,w_+] \mid \dot z_+ = X_{H_+}(z_+) \}.
$$
We note that $\widetilde{Per}(H_+)$ is precisely the set of critical points
of the action functional $\CA_H: \widetilde \LL_0(M) \to \R$.
Using the bounding disc $w_+:D^2\to M$, we
trivialize the symplectic bundle $z_+^*(TM)$ and get a loop in
$Sp(2n)$, which gives rise to the Conley-Zehnder index
$\m_{H_+}([z_+,w_+])\in \Z$.

Now we consider the moduli space \index{$\CM(K_+,J_+;[z_+,w_+];A_+)$}

\beastar \CM(K_+,J_+;[z_+,w_+];A_+) & = &
\Big \{u:\dot{\S}\to M \mid \delbar_{(K_+,J_+)} u = 0, \\
&{}& \quad  u(+\infty, t)=z_+(t), [u\#w_+] = A_+ \Big\} \eeastar
where $A_+\in H_2(M,\Z)$ is in the image of
the Hurwitz map $\pi_2(M)\to H_2(M,\Z)$.

For generic $J_+$ or $K_+$, the moduli space is regular and its
dimension is equal to
$$
\text{Index}D_{u}\delbar_{(K_+,J_+)} = n - \m_{H_+}([z_+,w_+]) + 2c_1(A_+).
$$
Here we follow the conventions from \cite{oh:alan} (See section 6.2
\cite{oh:alan}). Similarly for $u$ in the moduli space \index{$\CM(K_-,J_-;[z_-,w_-];A_-)$}
\beastar
 \CM(K_-,J_-;[z_-,w_-];A_-) & = &
\Big \{u:\dot{\S}\to M \mid \delbar_{(K_-,J_-)} u = 0, \\
&{}&  \quad  u(-\infty, t)=z_-(t), [\overline w_-\#u] = A_- \Big\},
\eeastar
where $A_-$  is similar to $A_+$.
$$
\text{Index}D_{u}\delbar_{(K_-,J_-)} = n + \m_{H_-}([z_-,w_-]) + 2c_1(A_-).
$$
Recall that the quantum homology
$QH_*(M)=H_*(M)\otimes\Lambda_{\omega}$, where $\Lambda_{\omega}$ is
the Novikov ring defined as
\begin{multline}
\Lambda_{\omega}= \Big\{ \sum_{A\in \Gamma} r_A
q^{-A} \mid r_A\in \Q, \text{ such that for all } \l\in \R,\\
\# \{A\in\G \mid r_A\neq 0, \omega(A)>\l \} <\infty \Big\}
\end{multline}
Here $\G\subset H_2(M)$ is the image of $\pi_2(M)$ under the
Hurewicz homomorphism, and $q$ is a formal variable.
If we use the Morse homology of $-f$ to represent $H_*(M)$, then we
can represent $QH_*(M)$ as the homology of $C_*(-f)\otimes \Lambda_{\omega} $,
where $C_*(-f)$ is the chain complex of the Morse homology of $-f$
generated by the critical points of $f$. The grading of $[p]q^{-A}$
is $\m_{(-f)}(p)-2c_1(A)$, where $[p]\in C_*(-f) $, and $\m_{(-f)}(p)$ is the
Morse index of $f$ at $p$.

We are going to define the PSS map
$$
\Phi_*: QH_k(M)\to FH_{n-k}(M).
$$
Following [PSS], we first define the chain level map
$\Phi:   C_*(-f)\otimes \Lambda_{\omega} \to CF_*(M)$
by defining it on the generators $[p]$ of $C_*(-f)$ and then linearly
extending over ring $\Lambda_\omega$ as
$$
\Phi: [p] \to \sum_{[z_+,w_+]\in \widetilde{Per}(H_+)}
\quad\#(\CM(p,[z_+,w_+];A_+)[z_+,w_+] q^{-A_+}.
$$
Here, roughly speaking, the moduli space $\CM(p,[z_+,w_+];A_+)$
consists of ``spike discs" emerging from the critical point $p$ and
ending on the periodic orbit $z_+$ in class $[u\# w_+] = A_+$ in
$\Gamma$. More precisely, we have the definition \index{$\CM(p,[z_+,w_+];A_+)$}
\beastar
\CM(p,[z_+,w_+];A_+)&=&\{(\chi_+,u_+)\mid u_+:\dot{\S}_+\to
M,[u_+\#w_+]=A_+,\\
& & u(+\infty,t)=z_+(t), \delbar_{(K_+,J_+)}u_+=0,\\
& & \dot{\chi}_+=\nabla f(\chi_+), \chi_+(-\infty)=p,
\chi_+(0)=u_+(o_+) \}. \eeastar Here we put index condition such
that $\CM(p,[z_+,w_+];A_+)$ is a $0$-dimensional oriented manifold
so we can do algebraic count ``$\#$". The index condition is
$$ (n-\m([z_+,w_+]+2c_1(A_+))+(2n-\m(p))-2n=0$$
$$i.e. \qquad \m([z_+,w_+])=n-(\m(p)-2c_1(A_+)).$$
Standard gluing argument shows that $\phi$ is a chain map (similar
to the continuation map that proves Morse homology is independent on
the Morse function), so it passes to homology and we get the PSS map
$\Phi_*: QH_k(M)\to HF_{n-k}(M) $.

\begin{figure}[ht]\centering
\includegraphics{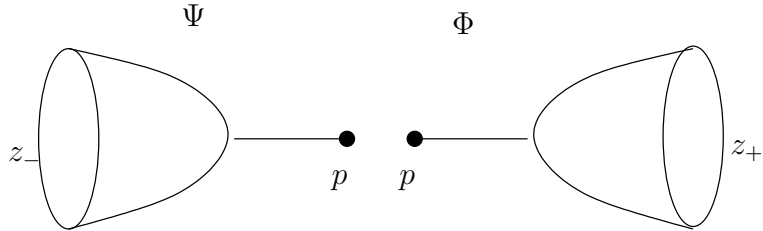} \caption{The PSS maps $\Psi$ and $\Phi$}
\end{figure}

Next we define the inverse of $\Phi$
$$
\Psi_*: HF_*(M)\to QH_*(M).
$$
For any $[z_-,w_-]\in \widetilde{Per}(K_-)$, define $\psi:
CF_*(M)\to C_*(-f)\otimes \Lambda_{\omega}$,
$$
\Psi: [z_-,w_-]\to \sum_{p\in \text{Crit}(-f); A_- \in \pi_2(M)} \#
\CM([z_-,w_-],p;A_-)p \otimes q^{-A_-},
$$
where
$\CM([z_-,w_-],p;A_-)$ consists of ``spiked-discs" emerging from the
periodic orbit $z_-$ and ending on the critical point $p$, namely \index{$\CM([z_-,w_-],p;A_-)$}
\beastar \CM([z_-,w_-],p;A_-)&=&\{(u_-,\chi_-) \mid u_-:\dot{\S}_-\to
M,[\overline w_-\#u_-]=A_-,\\
& & \quad u(-\infty,t)=z_-(t), \, \delbar_{(K_-,J_-)}u_-=0,\\
& & \quad \dot{\chi}_- =\nabla f(\chi_-), \, \chi_-(+\infty)=p,
\chi_-(0)=u_-(o_-) \}.
\eeastar
Here we also put the index condition
$$  \m_{H_-}([z_-,w_-])= n-(\m_{(-f)}(p)-2c_1(A_-))$$
so $\CM([z_-,w_-],p;A_-)$ becomes a $0$-dimensional (orientable)
manifold. The same continuation map argument shows $\Psi$ is a chain
map so it induce the homomorphism $\Psi_*: HF_{n-k}(M)\to QH_{k}(M)$.

\subsection{The PSS-scheme of proof of the isomorphism property $\Phi_*$}
\label{PSS:scheme}

In this section, we sketch the argument of Piunikhin-Salamon-Schwarz
towards a proof of isomorphism property of the PSS-maps based on
some picture which describes a deformation leading to the chain
isomorphism between the composition
$$
\Psi \circ \Phi, \quad id : CF^*(M) \to CF^*(M)
$$
and the identity map. The deformation involves the moduli spaces of three different
types in the course of deformations (see Figure 2):

\begin{figure}[ht] \centering
\includegraphics{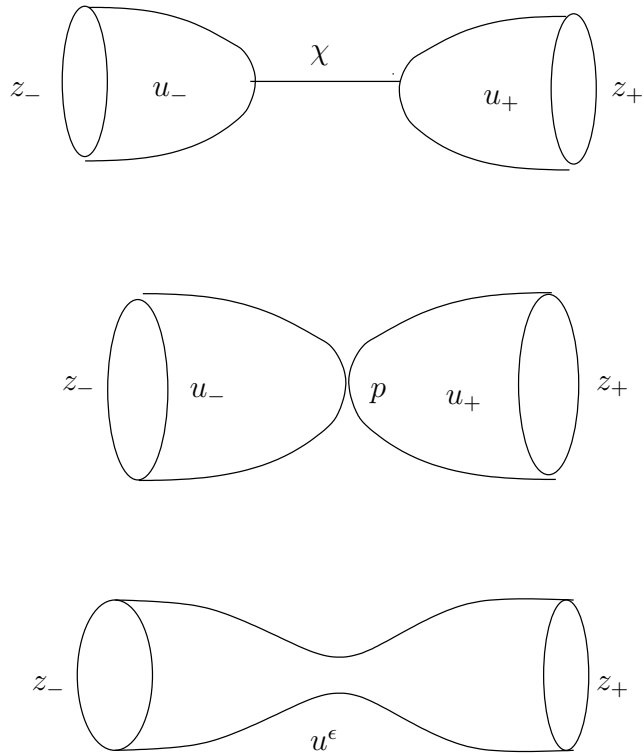} \caption{The PSS scheme}\end{figure}

\begin{enumerate}
\item Disk-flow-disk

\item Nodal Floer trajectories

\item Chain map Floer trajectories
\end{enumerate}

For the sake of following discussion, we denote the deformation
parameter by $\lambda \in [-1,1]$ so that the nodal configuration
occurs at $\lambda = 0$. As long as $\lambda > 0$ or $\lambda < 0$,
the deformation involves the same type of moduli spaces and so can
be applied the standard argument to construct a corbodism over
$[-1,-\e_0]$ or $[\e_0,1]$ for $\e > 0$.
To complete the cobordism over the whole interval $[-1,1]$, one needs to connect
the two cobordisms to one over $[-\e_0,\e_0]$. However there occurs
`phase change' in the moduli spaces over the interval $[-\e_0,\e_0]$
at $\lambda = 0$. Due to the `phase change' at $\lambda=0$, one can a priori
expect only a \emph{piecewise smooth} corbodism and needs to
prove a \emph{bi-collar theorem} of $\CM_0 \subset \CM^{para}$ to materialize the PSS-scheme.
From $-\e_0$ to $0$, one can construct the left one-sided collar by
finite dimensional differential topology (See Section \ref{sec:smoothing-sls}).
For the right one-sided collar over $[0,\e_0]$, we will construct
the collar by the method of adiabatic degeneration \cite{oh:adiabatic,oh:dmj},
\cite{mundet-tian} and scale-dependent gluing of immersed nodal
Floer trajectories developed in the present paper

This then implies the following isomorphism property as stated
in \cite{PSS}.
We refer to Part III in the present paper for the details of proof of
this isomorphism property based on this bi-collar neighborhood theorem
and the adiabatic degeneration result.
This final step largely reproduces the argument used in \cite{PSS}.

\begin{thm} Let $(f;g)$ be a generic Morse-Smale pair of
a Morse function $f$ and a metric $g$ on $M$ and $H^{Morse}(f;g)$
the Morse homology of $(f;g)$ and $(H,J)$ be a generic time-periodic
Hamiltonian function $H$ and a family of compatible almost complex
structure $J=\{J_t\}$ on $M$. Let $\Psi, \, \Phi$ be the PSS maps
given in \cite{PSS}. Then there exists a homomorphism
$$
h_{pss}: CF^*(H,J) \to CF^{* +1}(H,J)
$$
that satisfy \be\label{eq:chainhomotopy} \Psi\circ\Phi - id =
\del_{(H,J)} \circ h_{pss} - h_{pss} \circ \del_{(f,g)}. \ee In
particular, we have $\Psi_*\circ\Phi_*
 = id$ in homology.
\end{thm}

This shows  $\Psi _*\circ \Phi_* = id$. The other identity $\Phi_*\circ
\Psi_* = id$ is much easier to prove. Details of the proof are given
in section \ref{sec:fmf}.

\begin{rem} The adiabatic degeneration of the moduli space of solutions of
the Floer trajectory equation with a Morse function $\e f$ in the middle
does not produces just nodal Floer trajectories as used in the PSS-scheme
but produces the nodal Floer trajectories with some 1-jet datum
which reflects the back ground Morse function $f$.
This datum enters in our scale dependent gluing which
is the correct reversal process of the the adiabatic degeneration of
the moduli space as $\e \to 0$.
\end{rem}

\part{Geometry : Floer trajectories with immersed nodes}

\section{Definition of the deformation-cobordism moduli space}
\label{sec:cobordism}

In this section, we will provide the precise mathematical
formulation of the moduli spaces appearing in each stage of the
deformation-cobordism described in subsection \ref{PSS:scheme} which
was proposed by Piunikhin-Salamon-Schwarz \cite{PSS},
\cite{mcd-sal04}.

\subsection{Moduli space of `disk-flow-disk' configurations}
\label{subsec:sphere-line-sphere} This subsection is the first stage
of the deformation of the parameterized moduli space entering in the
construction of the chain homotopy map between $\Psi\circ \Phi$ and
the identity on $HF(H,J)$.

A ``disk-flow-disk" configuration consists of two perturbed
$J$-holomorphic discs joined by a gradient flow line between their
marked points. In this section we will define the moduli space of
such configurations.

For notation brevity, we just denote
$$
\CM^{\e}(K^{\pm},J^{\pm};[z_\pm, w_\pm],f;A_{\pm}) =
\CM^{\e}([z_\pm, w_\pm],f;A_{\pm})
$$
respectively omitting the Floer datum $(K^{\pm},J^{\pm})$, as long as it does not
cause confusion.

Given the two moduli spaces $\CM([z_-,w_-];A_-)$ and
$\CM([z_+,w_+];A_+)$
and the Morse function $f$, let the moduli space of
``disk-flow-disk" configurations $(u_-,\chi,u_+)$  of \emph{flow time}
$\e$ to be \index{$\CM^{\e}([z_-,w_-];f;[z_+,w_+];A_{\pm})$}
\begin{multline}
\CM^{\e}([z_-,w_-];f;[z_+,w_+];A_{\pm}):=\{(u_-,\chi,u_+)\mid
u_{\pm}\in \CM(K^{\pm},J^{\pm};\vec z_{\pm};A_{\pm}),\\
\chi: [0,\e]\to M, \dot{\chi}-\nabla f(\chi)=0, \;
u_-(o_-)=\chi(0),\; u_+(o_+)=\chi(\e) \}
\end{multline}
Then the moduli space of ``disk-flow-disk" configurations is defined
to be \index{$\CM^{para}([z_-,w_-];f;[z_+,w_+];A_{\pm})$}
\be
\CM^{para}([z_-,w_-];f;[z_+,w_+];A_{\pm}):=\bigcup_{\e\ge 0}
\CM^{\e}([z_-,w_-];f;[z_+,w_+];A_{\pm})\ee
Note we have included the $\e=0$ case, which corresponds to the
nodal Floer trajectory moduli space.

We now provide the off-shell formulation of the ``disk-flow-disk"
moduli spaces. We first provide the Banach manifold hosting
$\CM^{\e}([z_-,w_-];f;[z_+,w_+];A_{\pm})$. We define
\begin{multline}\label{Bmfd}
\CB^{res}_{\e}(z_-,z_+):=\{(u_-,\chi,u_+) \mid u_{\pm}\in
W^{1,p}(\dot{\S},M;z_{\pm}),\\
\chi\in W^{1,p}([0,\e],M), u_-(o_-)=\chi(0),\; u_+(o_+)=\chi(\e) \}
\end{multline}
for $p>2$. Then for each $u=(u_-,\chi,u_+)\in
\CB^{res}_{\e}(z_-,z_+)$, we define
$$
L^p_u(z_-,z_+)=L^p(\Lambda^{0,1}u^*TM)
$$
and form the Banach bundle
$$
\CL^{dfd}_{\e}=\bigcup_{u\in \CB^{dfd}_{\e}(z_-,z_+)} L^p_u(z_-,z_+)
$$
over $\CB^{dfd}_{\e}(z_-,z_+)$. Here the superscript `dfd' stands for
`disk-flow-disk'. We refer to \cite{floer:witten} for a more
detailed description of the asymptotic behavior of the elements in
$\CB^{dfd}_{\e}(z_-,z_+)$ in the context of Floer moduli spaces.

For $u=(u_-,\chi,u_+)\in \CB^{dfd}_{\e}(z_-,z_+)$, its tangent space
$T_u\CB^{dfd}_{\e}$ consists of $\xi=(\xi_-,a,\xi_+)$, where
$\xi_{\pm}\in W^{1,p}(u_{\pm}^*TM)$, $a\in W^{1,p}(\chi^*TM)$,
with the matching condition \be
 \label{match} \xi_-(o_-)=a(0),\quad
\xi_+(o_+)=a(\e) \ee We denote the set of such $\xi$ as \index{$W^{1,p}_u(z_-,z_+)$}
$W^{1,p}_u(z_-,z_+)$.

We let \index{$\CB^{dfd}(z_-,z_+)$}
$$
 \CB^{dfd}(z_-,z_+)= \bigcup_{\e\in (0,\e_0)} \CB^{dfd}_{\e}(z_-,z_+)
\quad \text{and} \quad \CL^{dfd}(z_-,z_+)=\bigcup_{\e\in (0,\e_0)}
\CL^{dfd}_{\e}(z_-,z_+)
$$
\begin{rem} If we regard $u$ in $\CB^{dfd}$ instead of
$\CB^{dfd}_{\e}$, then its tangent space consists of
$\xi=(\xi_-,a,\xi_+,\m)$, where $\xi_{\pm}\in
W^{1,p}(u_{\pm}^*TM)$, $a\in W^{1,p}(\chi^*TM), \m\in
T_{\e}\R\cong \R$, with the matching condition \be \label{p-match}
  \xi_-(o_-)=a(0),\quad
\xi_+(o_+)=a(\e)+\m \nabla f(\chi(\e)) \ee Here the $\m$ comes from
the variation of the length $\e$ of the domain of gradient flows.
\end{rem}

Now we define a natural section
\be e:
\CB^{dfd}_{\e}(z_-,z_+) \to \CL^{dfd}_{\e} (z_-,z_+)
\ee
such that $e(u)\in L^p_u(z_-,z_+)$ is given by
$$
e(u) = (\delbar_{J,H_{\pm}}u_-, \dot{\chi}-\nabla f(\chi),\delbar_{J,H_{\pm}}u_+)
$$
where the $u_{\pm}$ and $\chi$
satisfy the matching condition in \eqref{Bmfd}. The linearization of
$e$ at $u \in
e^{-1}(0)=\CM^{\e}([z_-,w_-];f;[z_+,w_+];A_{\pm})$ induces a
linear operator
\be E(u):=D_ue: W^{1,p}_u (z_-,z_+) \to L^p_u(z_-,z_+)
\ee
where the value $D_ue(\xi) = : \eta$ has the expression
$$
\eta= \left(\eta_-,b,\eta_+)=(D_{u_-}\delbar_{(K_-,J_-)}(\xi_-),\,
\frac{Da}{d\t}- \nabla_a\operatorname{grad}(f),D_{u_+}\delbar_{(K_+,J_+)}(\xi_+)\,\right)
$$
for $\xi=(\xi_-,a,\xi_+)$.\\

Now we show $E(u)$ is Fredholm and compute its index:

\begin{prop} \label{prop:dfdindex} If the deformed evaluation map
\be \label{deformev} \phi_f^{\e} ev_- \times
ev_+:\CM_1([z_-,w_-];A_-) \times \CM_1([z_+,w_+];A_+)\to M
\times M \ee is transversal to $\Delta_M$ in $M\times M$, then for any
$$
u=(u_-,\chi,u_+)\in \CM^{\e}([z_-,w_-];f;[z_+,w_+];A_{\pm}),
$$
the operator
$E(u)$ is Fredholm with \be
\operatorname{Index}E(u)=\m_{H_-}([z_-,w_-])-\m_{H_+}([z_+,w_+])+2c_1(A_-)+2c_1(A_+)
\ee and \be
 \operatorname{coker}E(u)\cong
\operatorname{coker}D_{u_{+}}\delbar_{(K_+,J_+)}\times
\operatorname{coker}D_{u_{-}}\delbar_{(K^-,J^-)} \label{cok} \ee
\end{prop}
\begin{proof} We compute the kernel and cokernel of
$$E(u):W^{1,p}_u(z_-,z_+)\to L^p_u(z_-,z_+).$$
By the matching condition \eqref{p-match} it is clear that
\begin{multline}
\operatorname{ker} E(u)=\Big\{(\xi_-,\xi_+,a)\mid \xi_{\pm}\in
\operatorname{ker}D_{u_{\pm}}\delbar_{(K^{\pm},J^{\pm})},\\
\frac{Da}{\partial\tau}-\nabla_a\operatorname{grad}f(\chi)=0,
\xi_-(o_-)=a(0),\, \xi_+(o_+)=a(\e)\Big\},
\end{multline}
By the diagonal transversal condition \eqref{deformev} in this
proposition, the map
$$
d\phi^{\e}_f\times
\operatorname{id}:\operatorname{ker}D_{u_+}\delbar_{(K^-,J^-)}
 \times \operatorname{ker}D_{u_+}\delbar_{(K_+,J_+)} \to T_{u_+(o_+)}M\times T_{u_+(o_+)}M
$$
is transversal to the $\Delta\subset T_{u_+(o_+)}M\times
T_{u_+(o_+)}M$. It is easy to see
$$
\operatorname{ker} E(u)=(d\phi^{\e}_f\times
\operatorname{id})^{-1}(\Delta),
$$
noticing that $ a(\e)=d\phi^{\e}_f a(0)$. Therefore \be
\label{kerdim:dfd}
\dim \ker E(u)=\dim \ker D_{u_-}\delbar_{(K^-,J^-)}
+\dim \ker D_{u_+}\delbar_{(K_+,J_+)}-2n. \ee

Next we compute the cokernel of $E(u)$. Let $E(u)^\dagger$ be the
$L^2$ adjoint operator of $E(u)$, such that
$$
E(u)^\dagger: L^q_u(z_-,z_+)\to W^{-1,q}(z_-,z_+),$$ where
$\frac{1}{p}+\frac{1}{q}=1$, and $L^q_u(z_-,z_+)$ and
$W^{-1,q}(z_-,z_+)$ are defined similarly to $L^p_u(z_-,z_+)$ and
$W^{1,p}(z_-,z_+)$. Then for any given $\eta:=(\eta_-,b,\eta_+)\in
\operatorname{coker}E(u)$,
\beastar
0 &=&\int_0^\e\left\langle
\frac{Da}{\partial\t}-\nabla\operatorname{grad}f(\chi)a, b \right\rangle\\
&{}&\quad + \int_{\dot{\Sigma}_-} \left\langle
D_{u_-}\delbar_{K^-,J^-}\xi_-,\eta_- \right\rangle +
\int_{\dot{\Sigma}_+} \left\langle D_{u_+}\delbar_{K_+,J_+}\xi_+,\eta_+
\right\rangle.
\eeastar
for all $(\xi_+,a,\xi_-)\in W_u^{1,p}$.
Integrating by parts, we have
\beastar
0&=&\langle a(\e),b(\e)\rangle-\langle a(0), b(0)\rangle+
\int_0^\e \left\langle -\frac{D}{\partial \t}b -\nabla f(\chi) b,
a \right\rangle \\
&{}&\quad -\int_{\dot{\Sigma}_-} \langle
(D_{u_{-}}\delbar_{(K^-,J^-)})^t\eta_-,\xi_- \rangle-
\int_{\dot{\Sigma}_+} \langle (D_{u_{+}}\delbar_{(K_+,J_+)})^t
\eta_+,\xi_+ \rangle\\
&{}&\quad + \int_{\partial\dot{\Sigma}_-}\langle \xi_-,\eta_-\rangle
+\int_{\partial\dot{\Sigma}_+}\langle\xi_+,\eta_+\rangle.
\eeastar
Noting that $\xi_{\pm}|_{\partial \Sigma_{\pm}}=0$ due to the fixed
boundary condition, we have
\bea
-\frac{D}{\partial \t}b -\nabla f(\chi) b&=&0   \label{g}\\
(D_{u_{\pm}}\delbar_{(K^{\pm},J^{\pm})})^t \eta_{\pm}&=&0
\label{cdbar} \eea
with matching condition \be \label{mg} \langle
a(\e),b(\e)\rangle-\langle a(0), b(0)\rangle=0. \ee From \eqref{g}
$b(\e)$ linearly depends on $b(0)$ since this is a initial value
problem of a linear ODE. But $a(0)$ and $a(\e)$ are arbitrary, so
\eqref{mg} forces $b(0)=0$, $b\equiv 0$. From \eqref{cdbar}, we see
$\eta_{\pm} \in \operatorname{ker}
(D_{u_{\pm}}\delbar_{(K_{\pm},J_{\pm})})^t
=\operatorname{coker}D_{u_{\pm}}\delbar_{(K_{\pm},J_{\pm})}$. So we have
$$
\operatorname{coker}E(u)\cong
\operatorname{coker}D_{u_{+}}\delbar_{(K_+,J_+)}\times
\operatorname{coker}D_{u_{-}}\delbar_{(K_-,J_-)}.
$$
In particular
$\operatorname{coker}E(u)$ is a finite dimensional subspace in
$L^p_u(z_-,z_+)$. Combining the above dimension counting for
$\operatorname{ker}E(u)$, we conclude that $E(u)$ is Fredholm.

We calculate the index of $E(u)$:
\beastar
\operatorname{Index}E(u)&=&\dim \ker D_{u_{+}}\delbar_{(K^{+},J^{+})}
+\dim \ker D_{u_{-}}\delbar_{(K^{-},J^{-})}-2n\\
&{}&-\dim \ker(D_{u_{+}}\delbar_{(K^{+},J^{+})})^\dagger -
\dim \ker(D_{u_{-}}\delbar_{(K^{-},J^{-})})^\dagger\\
&=&\operatorname{Index}D_{u_{+}}\delbar_{(K^{+},J^{+})}+
\operatorname{Index}D_{u_{-}}\delbar_{(K^{-},J^{-})}-2n\\
&=&(n+\mu_{H_-}([z_-,w_-])+2c_1(A_-)) +(n-\mu_{H_+}([z_+,w_+])+2c_1(A_+))\\
&=& \mu_{H_-}([z_-,w_-])-\mu_{H_+}([z_+,w_+])+c_1(A_-)+c_1(A_+)
\eeastar
where we have used
$$
\operatorname{Index}D_{u_{\pm}}\delbar_{(K_{\pm},J_{\pm})}
=(n-\pm\mu_{H_\pm}([z_{\pm},w_{\pm}])+2c_1(A_{\pm}))
$$
for the third identity.
\end{proof}

In section \ref{sec:smoothing-sls}, we will show that for given
generic $J^{\pm}, f$, there exists some $\e_0>0$, such that for
$\e\in (0,\e_0]$, every ``disk-flow-disk" curves
$u\in\CM^{\e}([z_-,w_-];f;[z_+,w_+];A_{\pm})$ is regular, in
the sense that $E(u)$ is surjective. So
$\CM^{\e}([z_-,w_-];f;[z_+,w_+];A_{\pm})$ is a smooth
manifold with  dimension equal to the index of $E(u)$:
$$\operatorname{dim}
\CM^{\e}([z_-,w_-];f;[z_+,w_+];A_{\pm})
=\mu_{H_-}([z_-,w_-])-\mu_{H_+}([z_+,w_+])+c_1(A_-)+c_1(A_+)$$ for generic
$J^{\pm},f$ and small $\e$.

\subsection{Nodal Floer trajectories of PSS deformation at $\lambda = 0$}
\label{subsec:node}

This is the middle stage of the construction of the above mentioned
piecewise smooth corbodism at which the `phase transition' of the
moduli spaces occurs as $\lambda$ pass through $\lambda = 0$.
In next subsection, we will give an enhanced version of the corresponding
moduli space entering in our construction of the cobordism.
The definition of the enhanced moduli space will involve a picture of recently developed
symplectic field theory \cite{EGH}, \cite{BEHWZ} in the Morse-Bott setting.
(See \cite{fooo07} also.)

Let $\dot\Sigma_\pm$ be two compact surfaces each with one
positive puncture (resp. one negative puncture) with analytic
coordinates. Let $o_\pm \in \dot \Sigma_\pm$ a marked point
and denote by $(\tau,t)$ with $\pm \tau \geq 0$ the
cylindrical chart of $\dot\Sigma_\pm \setminus \{o_\pm\}$
such that $z = e^{\pm 2\pi(\tau + it)}$.
We fix periodic orbits $z_\pm$ of $H_\pm = H_\pm(t)$
and their liftings $[z_\pm,w_\pm]$ respectively.
We denote \index{$\CM_1((K_\pm,J_\pm);[z_\pm,w_\pm])$}
$$
\CM_1((K_\pm,J_\pm);[z_\pm,w_\pm]) = \{(u_\pm,o_\pm) \mid
u_\pm \in \CM((K_\pm,J_\pm);[z_\pm,w_\pm]), \quad o_\pm \in \dot \Sigma_\pm\},
$$
respectively.  We have the natural evaluation maps
$$
ev_\pm:\CM_1((K_\pm,J_\pm);[z_\pm,w_\pm]) \to M; \quad ev_\pm(u_\pm) = u_\pm(o_\pm).
$$
The {\em standard nodal Floer trajectories}  will be the elements in the fiber product
\beastar &{}& \CM_1((K_+,J_+);[z_+,w_+])
{}_{ev_+}\times_{ev_-}
\CM_1((K_-,J_-);[z_-,w_-])\\
&=& \{(u_+,u_-) \mid u_\pm \in \CM((K_\pm,J_\pm);[z_\pm,w_\pm]),
\, u_+(o_+) = u_-(o_-)\}.
\eeastar
This is the space that appears
in the middle of the `chain homotopy' between $\Psi\circ \Phi$ and
the identity map on $HF(H,J)$ proposed by
Piunikhin-Salamon-Schwarz in \cite{PSS}. To differentiate this moduli space
from the later enhanced version of nodal Floer trajectories that we introduce
\emph{when the nodal points are immersed}, we denote this moduli
space by \index{$\CM^{nodal}_{stand}([z_-,w_-],[z_+,w_+];(K,J))$}
$$
\CM^{nodal}_{stand}([z_-,w_-],[z_+,w_+];(K,J)).
$$

On $U\pm$, using the given analytic coordinates $z =
e^{2\pi(\tau+it)}$, we fix a function
\be\label{eq:beta}
\kappa^+(\t) = \begin{cases} 0 \quad \mbox{if }\, |\tau| \leq 1 \\
1 \quad \mbox{if } \, |\tau| \geq 2
\end{cases}
\ee and let $\kappa^-(\t)=\kappa^+(-\t)$. We set
$\kappa^+_\e(\t)=\kappa^+(\t-R(\e)+1)$ and
$\kappa^-_\e(\t)=\kappa^+_\e(-\t)$. It is easy to see
\be\label{eq:betaR} \kappa_\e^+(\tau) =
\begin{cases} 1 \, \quad & \mbox{for }\, \tau \geq R(\e)+1\\
0 \quad & \mbox{for }\, \tau \leq R(\e)
\end{cases}, \quad
\kappa_\e^- = \begin{cases} 1 \, \quad & \mbox{for }\, \tau \leq -R(\e)-1\\
0 \quad & \mbox{for }\,  \tau \geq -R(\e)
\end{cases}
\ee
We then extend these outside the charts $U\pm$ by zero.

We define $(K_\e,J_\e)$  to be the obvious pairs
\bea\label{eq:KeJe}
K_\e(\tau,t,x) & = & \begin{cases}
\kappa_\e^+(\tau) \cdot K^+(\tau,t,x) \quad & (\tau,t) \in U_+ \\
\kappa_\e^-(\tau) \cdot K^-(\tau,t,x) \quad & (\tau,t) \in U_- \\
0 \quad & z \in \Sigma_\e \setminus U_+ \cup U_-
\end{cases}
\nonumber\\
J_\e^\pm(\tau,t,x) & = & \begin{cases}
J^{\kappa_\e^+(\tau)}(t,x) \quad & (\tau,t) \in U_+\\
J^{\kappa_\e^-(\tau)}(t,x) \quad & (\tau,t) \in U_-\\
J_0(x) \quad & z\in \Sigma_\e \setminus U_+ \cup U_-
\end{cases}
\eea
associated to $\kappa_\e^\pm$ respectively. Here we denote
a gluing of $\Sigma_+$ and $\Sigma_-$ by $\Sigma_+ \#_\e \Sigma_-$
(See Definition \ref{defn:deform}, Example \ref{exm:fukaya-ono} for details).
We then extend these to a constant family outside the charts $U\pm$. Thanks to the cut-off functions
$\beta_\pm$, this extension defines a smooth family on $\dot
\Sigma$.
We will vary $R=R(\e)$ depending
on $\e$ so that we are given a one-parameter family
$$
\dot \Sigma_\e, \, (K_{\e}, J_{\e}).
$$
Here we would like to emphasize that \emph{$K_\e \equiv 0$ in the
neck regions of $\Sigma_\e$}.

\subsection{Moduli space of enhanced nodal Floer trajectories}
\label{subsec:enhanced} If we attempt to construct a smooth
coordinate chart for the parameterized moduli space of dimension 1
near $\lambda = 0$, the resolved Floer trajectories should be
related to the ``disk-flow-disk" elements. One way is to break the
local conformal symmetry of the equation near the node by inserting
a small Morse function $\e f$ with $\e \to 0$ at the node. This
forces one to study \emph{adiabatic degeneration} as studied in
\cite{oh:adiabatic,oh:dmj}, \cite{mundet-tian} and the relevant gluing problem.
What distinguishes this gluing problem from the gluing problem in
the standard Gromov-Witten or in the Floer theory is that it glues
two configurations in different scales: nodal Floer trajectory $(u_-,u_+)$
in macroscopic level and local model $u_0$ in microscopic level. To find the correct local
model, we need to analyze the fine structure of the node in the
nodal trajectories. Description of this structure is in order.

First of all, we will need to require that the nodal points are
\emph{immersed} points for both $u_\pm$. We will prove that this
requirement holds for a generic choice of $J$. For the
moment, we will assume that the nodal points are immersed for both
$u_\pm$, and continue with our discussion.

Secondly, we need to enhance the moduli space of
standard nodal Floer trajectories by some local models which are to be
implanted at the intersection point $u_+(o_+) = u_-(o_-)$ of $u_{+}$ and $u_{-}$.

We first describe the space of local models.
Let $H$ be a hyperplane of $\C P^n$. We identify $(\C P^n,H)$ with
$$
\C P^n = \P(1\oplus \C^n), \quad H = \P(0\oplus \C^n)
$$
and $\operatorname{Aut}(\C P^n,H)$ is the set of homothety and translations
given by
$$
v \mapsto c v + a \, ; \C^n \to \C^n
$$
with $c \in \C^*$ and $a \in \C^n$. We define \index{$\CM_{(0;2,0)}(\C P^n,H;2)$}
$$
\begin{aligned}
\widetilde{\CM}_{(0;2,0)}(\C P^n,H;2) := &\{(u;p_1,p_2) \mid u: \Sigma \to \C P^n, u(p_i) \in H,
p_1 \neq p_2, \\
& \qquad  \text{\rm deg}u = 2, \, \text{\rm Im} u \not \subset H \}\\
\CM_{(0;2,0)}(\C P^n,H;2) = & \widetilde{\CM}_{(0;2,0)}(\C P^n,H;2)/PSL_2(\C)
\end{aligned}
$$
In terms of the cylindrical coordinates $(s,\Theta)$ of $\C^n \setminus \{0\} \cong
\R \times S^{2n-1}$, the above moduli space can be identified with
\index{$\widetilde{\CM}_{(0;2,0)}^{SFT}(\C^n;(1,1))$}
\be
\begin{aligned}\label{eq:hlmd}
& \widetilde{\CM}_{(0;2,0)}^{SFT}(\C^n;(1,1))
 =: \{u : \dot \Sigma \to \C^n \mid
u: \dot \Sigma \cong S^1\times \mathbb{R} \to
\C^n \, \mbox{is proper},  \\
& \hskip0.5in
 \bar\partial_{J_0}u=0,
\, \lim_{\tau\to\pm\infty}\Theta\circ u(\tau/2\pi,t )= \gamma_{\pm}(t),
\lim_{\tau \to \pm \infty} \int u_\tau^*\lambda = 2\pi\}:
\end{aligned}
\ee
We then define
$$
\CM_{(0;2,0)}^{SFT}(\C^n;(1,1)) = \widetilde{\CM}_{(0;2,0)}^{SFT}(\C^n;(1,1))/PSL_2(\C).
$$
Here $\lambda$ is the standard contact form on $S^{2n-1}(1) \subset \C^n$,
$\gamma_\pm$ closed Reeb orbits of $\lambda$ and $(1,1)$ stands for the multiplicity of the closed Reeb orbits
$\gamma_\pm$ respectively and $u_\tau: S^1 \to \C^n$
is the loop defined by $u_\tau(t) = u(\tau,t)$.
We recall that $T: = \int \gamma^*\lambda$ for a closed Reeb orbit $\gamma$
is the same as its period. We denote by
$$
\widetilde{\CR}_1(\lambda) = \Reeb^{min}(S^{2n-1},\lambda)
$$
the set of minimal Reeb orbits of period $2\pi$. Then the diagonal circle action on
$\C^n$ induces a free $S^1$-action on $\widetilde{\CR}_1(\lambda)$
which makes and $\widetilde{\CR}_1(\lambda) \to \C P^{n-1}$ a principal $S^1$-bundle.
We then denote by $\CR_1(\lambda)$ that of unparameterized ones, i.e.,
$$
\CR_1(\lambda) = \widetilde{\CR}_1(\lambda)/S^1.
$$

Now for the purpose of defining the correct nodal Floer trajectories
we need to consider 1-jet evaluation maps on both
$\widetilde{\CM}_{(0;2,0)}^{SFT}(\C^n;(1,1))$ and $\widetilde
\CM_1(K_\e^-,J_\e^- ;[z_-,w_-])$ (or $\widetilde
\CM_1(K_\e^+,J_\e^+ ;[z_+,w_+])$). We now explain these evaluation maps.

We now fix a local chart $I = \{I_p: \CU_p \to \CV_p\}$ be a given Darboux family.
(See subsection \ref{subsec:Darboux} for the definition.)
Using $I$ we can identify $(T_pM,J_p,\omega_p)$ with $\C^n$ equipped with the
standard K\"ahler structure. Then we can
define the evaluation maps
\be\label{eq:evsharp}
\widetilde{ev}^\#_{I,\pm} : \widetilde{\CM}_1([K_\e^\pm,J_\e^\pm ;[z_\pm,w_\pm]) \to
\widetilde{\CR}_1(\lambda):
\ee
By the immersed property of $u_\pm$ at $o_\pm$, the limit
$$
\gamma_\pm(t) :=  \lim_{\tau \pm \infty} u_\tau(\tau ,t )
$$
defines an element $\widetilde \CR_1(S^1(T_pM),\lambda_p) \cong \CR_1(\lambda)$. Then
the map $\widetilde{ev}^\#_{I,\pm}$ descends to the map
$$
ev^\#_{I,\pm}: \CM_{(0;2,0)}^{SFT}(\C^n;(1,1)) \to \CR_1(\lambda);
\quad ev^\#_{I,\pm}(u) = [\gamma_\pm]
$$
for $u_\pm \in \widetilde \CM(K_\e^\pm,J_\e^\pm ;[z_\pm,w_\pm])$.
Since we will not change the Darboux family $I$, we often omit the sub-index
$I$ as long as there occurs no confusion.

%Next for the moduli space $\CM_{(0;2,0)}^{SFT}(\C^n;(1,1))$
%of local models, we add a marked line $\ell\subset \dot \Sigma$ to
%$\dot \Sigma$ which connects the two punctures, and identify
%$\dot \Sigma \cong \R \times \R/\Z$, and put the marked point
%on this line, say the line $\{t = 0\}$. Then we consider
%the subgroup of $Aut(\C^n)$
%$$
%\widetilde{Aut}(\C^n) = \{v \mapsto c v + a \mid
%c \, \mbox{ real }, \, a \in \C^n \}
%$$
%consisting of real homothety and general translations of $\C^n$. Note that
%$$
%\dim\widetilde{Aut}(\C^n) = \dim \widetilde{Aut}(\C^n)-1 = 2n+1.
%$$
%
To get rid of the domain automorphism, we consider the moduli space
\be
\widetilde \CM_{(1;2,0)}^{SFT}(\C^n;(1,1)) = \{(u,(e_\pm,r))) \mid
u: \dot \Sigma \to \C^n \}
\ee
and
$$
\CM_{(1;2,0)}^{SFT}(\C^n;(1,1)) = \widetilde \CM_{(1;2,0)}^{SFT}(\C^n;(1,1)) /\sim
$$
where $e_\pm$ are punctures on $\dot\Sigma$ and $r$ is a marked point in the interior,
$\sim$ is the equivalence relation under the action of
$PSL(2,\C)$. After modding out by
$PSL(2,\C)$, we can identify $\CM_{(1;2,0)}^{SFT}(\C^n;(1,1))$
with the more concrete space
$$
\{u \mid  u: \R \times S^1 \to \C^n, \, \delbar_{J_0} u = 0, \,
[\lim_{\tau \to \pm \infty}u_\tau(\cdot )] \in
\CR_1(\lambda)\}
$$
via the unique conformal identification
$$
\varphi: (\dot \Sigma, r) \cong \R \times S^1
$$
such that $\varphi(e_\pm) = \pm \infty$ and $\varphi(r) = (0,0)$.

We denote the corresponding moduli space also by
$$
\CM^{SFT}(\R \times S^1,\C^n;\CR_1(\lambda)) \cong \CM_{(1;2,0)}^{SFT}(\C^n;(1,1)).
$$
We have the evaluation maps
$$
ev_{\#,-}^{SFT}: \CM_{(1;2,0)}^{SFT}(\C^n;(1,1)) \to \CR_1(\lambda)
$$
by
$$
ev_{\#,-}^{SFT}(u) = \left[\lim_{\tau \to -\infty} u_\tau\right] \in \CR_1(\lambda).
$$
Similarly, we define $ev_{\#,+}^{SFT}: \CM_{(1;2,0)}^{SFT}(\C^n;(1,1)) \to \CR_1(\lambda)$.

\begin{lem}\label{gauss} The above definition of $ev_{\#,\pm}^{SFT}$ pushes down to
the quotient moduli space
$$
\CM_{(1;2,0)}^{SFT}(\C^n;(1,1)) /\operatorname{Aut}(\C^n).
$$
\end{lem}
\begin{proof}
We need to prove
$$
ev_\#^{SFT}(u) = ev_\#^{SFT}(g\circ u)
$$
for all $(g : v \mapsto c v + a) \in \operatorname{Aut}(\C^n)$. For the
identity for $ev_\#^{SFT}$, we note that the action induced by
the elements of $\operatorname{Aut}(\C^n)$ does not change the asymptotic limit
$\lim_{\tau \to \pm \infty} u_\tau$ as an \emph{unparameterized} Reeb orbit.
This finishes the proof.
\end{proof}

%We also have the asymptotic evaluation map
%$$
%(ev_+^{SFT},ev_-^{SFT}) : \widetilde \CM^{SFT}(\R\times S^1;\C^n;\CR_1(\lambda))
%\to \CR_1(\lambda) \times \CR_1(\lambda)
%$$
%defined by \index{$ev_\pm^{SFT}$}
%$$
%ev_\pm^{SFT}(u) = \lim_{\tau \to \pm\infty}\Theta\circ u(\tau,\cdot)
%\in \CR_1(\lambda).
%$$
Now we are ready to give the definition of `enhanced' moduli space of
nodal Floer trajectories appearing in the PSS picture.

Let $I = \{I_p : \UU_p\to \VV_p\}$ be the given Darboux family.
(See subsection \ref{subsec:Darboux} for the definition.) This family provides
an isomorphism between $T_pM$ and $\C^n$ at any $p \in M$.
We now consider the evaluation maps
\beastar
ev^\#_{+,I} & : & \CM_1(K_\e^+,J_\e^+ ;[z_+,w_+]) \to \CR_1(\lambda) \\
ev^\#_{-,I} & : & \CM_1(K_\e^-,J_\e^- ;[z_+,w_+]) \to \CR_1(\lambda)
\eeastar
as before for $u \in \CM_1(K_\e^*,J_\e^* ;[z_*,w_*])$. We further note that
$\CM_1(K_\e^\pm,J_\e^\pm;[z_\pm,w_\pm])$ have decomposition
\beastar
\CM_1(K_\e^+,J_\e^+;[z_+,w_+]) & = & \bigcup_{A \in
\pi_2(M)}
\CM_1(K_\e^+,J_\e^+;[z_+,w_+];A) \\
\CM_1(K_\e^-,J_\e^-;[z_-,w_-]) & = & \bigcup_{A \in \pi_2(M)}
\CM_1(K_\e^-,J_\e^-;[z_-,w_-];A) \eeastar where
$\CM_1(K_\e^\pm,J_\e^\pm;[z_\pm,w_\pm];A)$ are the sets
\beastar
&{}& \CM_1(K_\e^+,J_\e^+;[z_+,w_+];A_+) \\
& = & \{ u_+ \in
\CM_1(K_\e^+,J_\e^+;[z_+,w_+]) \mid [u \# w_+] = A_+ \in \pi_2(M)\} \\
&{}& \CM_1(K_\e^-,J_\e^-;[z_-,w_-];A_-) \\
& = & \{ u_- \in
\CM_1(K_\e^-,J_\e^-;[z_-,w_-]) \mid [\overline w_-\# u] = A_- \in
\pi_2(M)\}.
\eeastar
We can define the fiber product
\beastar
&{}&\CM_1(K_\e^+,J_\e^+;[z_+,w_+];A_+)
{}_{ev^\#_{+,I}}\times_{ev^{SFT}_{\#,-}} \CM_{(0;2,0)}^{SFT}(\C^n;(1,1)) \\
&{}& \hskip0.7in {}_{ev^{SFT}_{\#, +}}
\times_{ev^\#_{-,I}}\CM_1(K_\e^-,J_\e^-;[z_-,w_-];A_-)
\eeastar
for each given pair $A_\pm \in \pi_2(M)$.
Then we form the union \index{$\CM^{nodal}_1([z_-,w_-],[z_+,w_+];(K,J))$}
\bea\label{eq:correctnodal}
&{}&\CM^{nodal}([z_-,w_-],[z_+,w_+];(K,J)) \nonumber\\
& : = & \bigcup_{(A_-, A_+); A_- + A_+ = 0}
\CM_1(K_\e^+,J_\e^+;[z_+,w_+];A_+)
{}_{ev^\#_{+,I}}\times_{ev^{SFT}_{\#,-}} \CM_{(0;2,0)}^{SFT}(\C^n;(1,1)) \nonumber \\
&{}& \hskip1in {}_{ev^{SFT}_{\#, +}}
\times_{ev^\#_{-,I}}\CM_1(K_\e^-,J_\e^-;[z_-,w_-];A_-)
\eea under
the assumption that $du_\pm(o_\pm) \neq 0$. We call elements $(u_+,u_0,u_-)$ therein
{\em enhanced nodal Floer trajectories} in vacuum. The following
theorem justifies the hypothesis that the nodal points are
immersed, and so that the above fiber product is well-defined. We
will postpone its proof to the next subsection.

\begin{figure} [ht]
\centering
\includegraphics{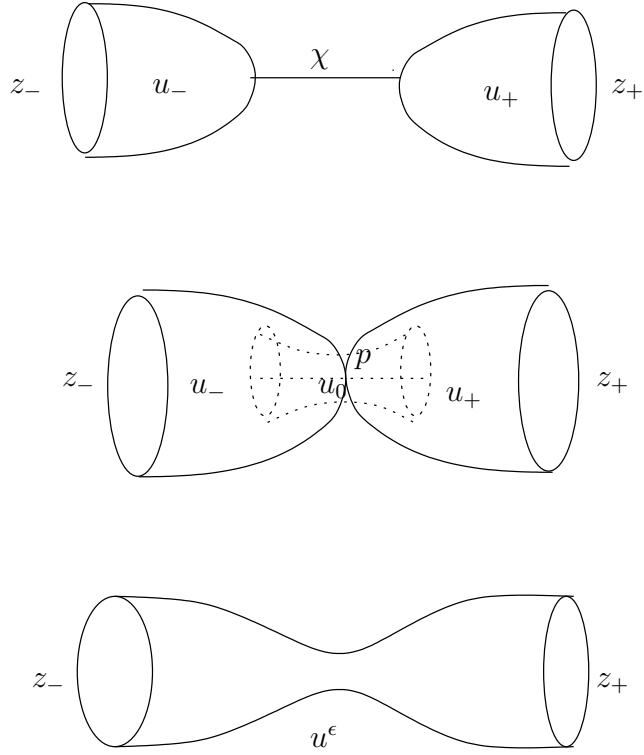} \caption{The enhanced PSS
scheme}
\end{figure}

\begin{thm}\label{immersed} Let $(K,J;\e)$ be a Floer datum with the
asymptotic Hamiltonian $H$ such that it satisfies \eqref{eq:betaR}, \eqref{eq:KeJe}. Suppose that
$$
\mu([z_-,w_-]) - \mu([z_+,w_+]) < 2n-1.
$$
Then there exists a dense subset of $\CJ_\omega$ consisting of $J$'s such that
for any quintuple
$$
(u_-,u_+, r_-,r_+;\e) \in \cup_{0 < \e \leq \e_0}
\CM_{stand}^{nodal}([z_-,w_-], [z_+,w_+]; (K_\e,J_\e))
$$
with $u_-(r_-) = u_+(r_+)$, $r_-$ and $r_+$ are immersed points of
$u_-$ and $u_+$ respectively, and
$$
[du_-(r_-)] \neq [du_+(r_+)] \quad \mbox{ in } \, \P(T_pM)
$$
with $p = u_-(r_-) = u_+(r_+)$.

In particular, these hold when
$\mu([z_-,w_-]) - \mu([z_+,w_+]) =0$.
\end{thm}

%From the classification Proposition \ref{degree2curves} , given $u_-$ and
%$u_+$, the {\em homogeneous} local model $u_0\in
%\CM_{(1;2,0),line}^{SFT}(\C^n;(1,1))$ that matches the tangent
%evaluations $ev_{\pm}^\#$ of $u_{\pm}$ exists and is unique.
%However, for the adiabatic limit purpose, we need to use the {\em
%inhomogeneous local model} $w_0$ with the same asymptotes as $u_0$
%but satisfies $\frac{\partial w_0}{\partial \t}+ J \frac{\partial
%w_0}{\partial t}=\nabla f(p)$, where $f: M \to \R$ is a Morse
%function such that $p$ is a regular point. It is not hard to see
%that $w_0=u_0+\nabla f(p)\t$ (for details see Section 7). For
%notation convenience, from now on we still call $w_0$ as $u_0$.

\medskip

\subsection{Nodal points are immersed}
\label{subsec:immersed}

In this subsection, we will give the proof of Theorem
\ref{immersed}. The proof is a variation of the dimension counting
argument and partially inspired by Hutchings and Taubes's proof of
Theorem 4.1 \cite{hutch-taubes} in which they studied immersion
properties of pseudo-holomorphic curves in the symplectization of
contact 3-manifolds.

Consider the parameterized family $(K,J) = (\{K_\e\},\{J_\e\})$
such that
$$
J \equiv J_0
$$
in a neighborhood of the marked point $r \in \dot \Sigma$.
We consider a pair of them denoted by $(K_\pm,J_\pm)$.

We consider $(J,(u_-,r_-),(u_+,r_+),\e)$ and the map
\beastar
\Upsilon & : &(J,(u_-,r_-),(u_+,r_+),\e) \\
& \mapsto &
(\delbar_{(J,K)^-;\e}(u_-), \delbar_{(J,K)^+;\e}(u_+); \del_J u(r_-),\del_J u(r_+))
\eeastar
where we denote
\beastar
\delbar_{(K,J)^-;\e}(u_-) & = & (du + P_{K_\e^-})^{(0,1)}_{J_\e^-}(u)\\
\delbar_{(K,J)^+;\e}(u_+) & = & (du + P_{K_\e^+})^{(0,1)}_{J_\e^+}(u).
\eeastar
We consider the bundles over $\Sigma \times M$
\beastar
H^{(0,1)}_{J_0}(\Sigma \times M) &: = & \cup_{(z,x)} Hom''_{J_0}(T_z\Sigma,T_x M) \\
H^{(1,0)}_{J_0}(\Sigma \times M) &: = & \cup_{(z,x)} Hom'_{J_0}(T_z\Sigma,T_x M)
\eeastar
whose fibers are $J$-anti-linear and $J$-linear parts of $Hom(T_z\Sigma,T_xM)$
respectively.
The union of standard nodal Floer trajectories
$\CM^{nodal}_{stand}([z_-,w_-],[z_+,w_+];(K,J);R(\e))$
over $J \in \CJ_\omega$ is nothing but
\beastar
&{}& \Upsilon^{-1}(\{0\}\times \{0\} \times H^{(1,0)}_{J_0}(\Sigma \times M)
\times_\Delta H^{(1,0)}_{J_0}(\Sigma \times M))\\
&: =& \CM^{nodal}_{stand}([z_-,w_-],[z_+,w_+];K;R(\e)).
\eeastar
We recall that $K_\pm \equiv 0$ near the marked points $o_\pm$.
Therefore we have
$$
\delbar_{(J,K)^\pm;\e}u_\pm(o_\pm) = \delbar_{J_0}(u_\pm)(o_\pm)
$$
which implies that for any $u_\pm$ with $\delbar_{(J,K)^\pm;\e}(u_\pm)(o_\pm) =0$,
we have
$$
du(o_\pm) = 0 \quad \mbox{if and only if } \, \del_Ju(o_\pm) = 0.
$$
Postponing the precise functional analytic details until section
\ref{subsec:off-nodal}, we introduce the necessary framework for the
Fredholm theory needed to prove Theorem \ref{immersed}.
We denote by
$$
\CF^- = \CF^-(\dot \Sigma, M; [z_-,w_-]), \quad
\CF^+ = \CF^+(\dot \Sigma, M; [z_+,w_+])
$$
the off-shell function space hosting the operator $\delbar_{(K,J)^\pm}$ and
the corresponding Floer moduli spaces respectively.
And we introduce the standard bundle
\beastar
\CH'' & = & \cup_{(u,J)} \CH''_{(u,J)}, \quad \CH''_{(u,J)} = \Omega_J^{(0,1)}(u^*TM)\\
\CH' & = & \cup_{(u,J)} \CH'_{(u,J)}, \quad \CH'_{(u,J)} = \Omega_J^{(1,0)}(u^*TM).
\eeastar
We have the natural evaluation map
$$
ev_{\CF^*}: \CF^* \to H^{(1,0)}_{J_0}(\Sigma\times M) \, ; \quad ev_{\CF^*}(u) =
(o_*,u(o_*); \del_J u(o_*))
$$
for $* = \pm$ respectively.
Then the above map $\Upsilon$ defines a section of the Banach bundle
$$
 \CJ_\omega \times \CF^- \times \CF^+
\to \CH''_- \times \CH''_+ \times H^{(1,0)}_{J_0} \times H^{(1,0)}_{J_0}
$$
We now prove the following proposition by a standard argument via the linearization
of $\Upsilon$. We use the convention that $o_{L}$ is the zero section for a bundle $L$.

\begin{prop}\label{trans-Upsilon} The map $\Upsilon$ is transverse to the
(stratified) submanifold
\beastar
o_{\CH''_-} \times o_{\CH''_+} & \times & \left(o_{H^{(1,0)}_{J_0}} \times_\Delta H^{(1,0)}_{J_0})
\bigcup H^{(1,0)}_{J_0} \times_\Delta o_{H^{(1,0)}_{J_0}} \right)\\
&{}& \subset \CH''_- \times \CH''_+ \times H^{(1,0)}_{J_0} \times H^{(1,0)}_{J_0}.
\eeastar
In particular the set
$$
\Upsilon^{-1}\left(o_{\CH''_-} \times o_{\CH''_+} \times \left(o_{H^{(1,0)}_{J_0}} \times_\Delta H^{(1,0)}_{J_0})
\bigcup H^{(1,0)}_{J_0} \times_\Delta o_{H^{(1,0)}_{J_0}} \right)\right)
$$
is a (stratified) submanifold of
$$
\CM^{nodal}_{stand}([z_-,w_-],[z_+,w_+];K;\e)
$$
of codimension $2n$.
\end{prop}
\begin{proof} It is easy to check the statement on the codimension and so
we will focus on proving the submanifold property.
We note that the subset
\be\label{eq:oCH''}
o_{\CH''_-} \times o_{\CH''_+} \times \left(o_{H^{(1,0)}_{J_0}} \times_\Delta H^{(1,0)}_{J_0})
\bigcup H^{(1,0)}_{J_0} \times_\Delta o_{H^{(1,0)}_{J_0}} \right)
\ee
consists of two strata : one is the open stratum given by
\beastar
&{}&o_{\CH''_-} \times o_{\CH''_+}  \times  \left(o_{H^{(1,0)}_{J_0}} \times_\Delta H^{(1,0)}_{J_0})
\bigcup H^{(1,0)}_{J_0} \times_\Delta o_{H^{(1,0)}_{J_0}} \right) \setminus \\
&{}& \quad o_{\CH''_-} \times o_{\CH''_+}  \times
(o_{H^{(1,0)}_{J_0}} \times_\Delta o_{H^{(1,0)}_{J_0}}) \eeastar
and the other is given by the lower order stratum
$$
o_{\CH''_-} \times o_{\CH''_+} \times \left(o_{H^{(1,0)}_{J_0}} \times_\Delta o_{H^{(1,0)}_{J_0}}
\right).
$$
We note that the lower dimensional stratum has codimension $2n$ insider the
set (\ref{eq:oCH''}).

The linearization of $\Upsilon$ is given by
\beastar
&{}& (B, (\xi_-,v_-), (\xi_+,v_+), h) \\
& \mapsto &
(D\delbar_{(K_-)}(u_-)(B,\xi_-), D\delbar_{(K_+)}(u_+)(B,\xi_+) ; \\
&{}& \qquad (\xi_-(o_-), (D\del_{J_-,u_-} (\xi_-)(o_-)), (\xi_+(o_+),
(D\del_{J_+,u_+} (\xi_+)(o_+))).
\eeastar
for $B\in T_{J_0}\CJ_\o$, $\xi_\pm \in T_{u_\pm}\CF^\pm$, and $v_\pm \in T_{o_\pm}\S_\pm$.
It is well-known that $D\delbar_{(K_-)}(u_-)(B,\xi_-)$, $D\delbar_{(K_+)}(u_+)(B,\xi_+)$
are surjective and so we will focus on the problem of finite dimensional
transversality of the linear map
$$
(\xi_-, \xi_+) \mapsto ((\xi_-(o_-), (D \del_{J_-,u_-} (\xi_-)(o_-)), (\xi_+(o_+),
(D \del_{J_+,u_+}(\xi_+)(o_+)))
$$
to the submanifold
$$
o_{H^{(1,0)}_{J_0}} \times_\Delta H^{(1,0)}_{J_0}
\bigcup H^{(1,0)}_{J_0} \times_\Delta o_{H^{(1,0)}_{J_0}}
$$
in $H^{(1,0)}_{J_0}(\Sigma \times M) \times H^{(1,0)}_{J_0}(\Sigma \times M)$

Since transversality of the map $(u_-,u_+) \mapsto (u_-(o_-), u_+(o_+))$
to $\Delta \subset M \times M$
is obvious, we will focus on the other factor on the tangential data.
We first consider the top dimensional stratum, i.e., for the
pair $(u_-,u_+)$ such that
$
u_-(o_-) = u_+(o_+)
$
and
$$
du_-(o_-) =0, \quad du_+(o_+) \neq 0.
$$
We need to prove that the equation
\be\label{eq:DJueta-}
D \del_{J,u}(\xi_-)(o_-) = \eta_-
\ee
has a solution $\xi_-$ for each given $\eta_- \in \Lambda^{(1,0)}(TM)$.
Similar consideration applies to the case of switching $+$ and $-$.

In general, a well-known computation shows
$$
D_u \del_{J}(\xi_-) = (\nabla_u \xi_-)^{(1,0)}_J + T^{(1,0)}_J(du_-,
\xi)
$$
with the torsion term $T$. However if $du_-(o_-) = 0$, we have
$T^{(1,0)}_J(du_-(o_-), \xi(o_-))=0$ for any $\xi$.

We now introduce the linear operator $q_{J,x_0}$ defined by
$$
q_{J,x_0}(x) = (J_{x_0} + J(x))^{-1}(J_{x_0} - J(x))
$$
for $x$ such that $d(x,x_0) < \delta$ for $\delta > 0$ depending only
on $(M,\omega,J)$ but independent of $x_0$. $q_{J,x_0}$ satisfies
$q_{J,x_0}(x_0) = 0$. (See \cite{sikorav}.)
Then if we identify $(T_{x_0} M,J_{x_0}) \cong \C^n$, we can write
the operator
$$
(\nabla_u \xi_-)^{(1,0)}_J = \del \xi_- - q_{J,r}(u)\delbar \xi_-
+ D\cdot \xi_-
$$
for some zero order operator $D$ with $D(o_-) = 0$.

Therefore if $u$ satisfies $du_-(o_-) = 0$, we can write
$$
D\del_{J,u}(\xi) = \del \xi_- - A \cdot \delbar \xi_- + C \cdot \xi_-
$$
in a neighborhood of $o_-$ where $A, \, C$ are smooth pointwise (matrix) multiplication
operators with
\be\label{eq:ABr-}
A(o_-) = C(o_-) = 0.
\ee
To finish the proof, we need to prove the existence of local solutions
of the equation
$$
\del \xi_- - A \cdot \delbar \xi_- + C \cdot \xi_- = \eta_-
$$
near the given point $r_-$. This equation can be transformed into
\be\label{eq:finaleq}
(Id - A \cdot \overline T) \del \xi_- + C \cdot \xi_- = \eta_-
\ee
where $T$ is the operator
\beastar
\overline Tg(z) & = & p.v.\left(\frac{1}{2\pi i}\int\int_D \frac{g(\zeta)}
{(\overline\zeta - z)^2}d\zeta \wedge d \overline \zeta\right)\\
& = &- \lim_{\d\to 0}\int\int_{\zeta \mid |\overline\zeta - z| \geq \d,
|\zeta| \leq 1} \frac{g(\zeta)}{(\overline \zeta - z)^2}
d\zeta\wedge d\overline\zeta.
\eeastar
The operator $T$ satisfies the a priori estimate
\be\label{eq:Tgkp}
\|Tg\|_{W^{k,p}} \leq A_{k,p}\|g\|_{k,p}.
\ee
(See \cite{vekua}, p. 166-167 \cite{sikorav}).

Now after multiplying a cut-off function to $\eta_-$ with its support
contained in a sufficiently small neighborhood of $r_-$ and using the
a priori estimate (for $k=1$), we can solve
(\ref{eq:finaleq}) by the contraction mapping theorem
in a neighborhood of $o_-$. This finishes the existence of a solution
to (\ref{eq:DJueta-}) for the top stratum.

For the lower dimensional stratum, i.e., for those pairs $(u_-,u_+)$ with
$$
u_-(o_-) = u_+(o_+), \quad du_-(o_-) = du_+(o_+) = 0
$$
we can prove the existence by the same argument. The only thing to make sure
is that the surjectivity proof of $D\delbar_{(K,J)^\pm}$ \emph{using the same
variation $B$}
still holds. But it is easy to check that the nodal condition $u_-(o_-) = u_+(o_+)$
ensures this, whose checking is left to the readers. (See \cite{oh-zhu1} for
a complete discussion on this matter.)

This finishes the proof of the proposition.
\end{proof}

We have the natural projection
$$
\pi_\Upsilon : \Upsilon^{-1}\left(o_{\CH''_-} \times o_{\CH''_+}
\times \left(o_{H^{(1,0)}_{J_0}} \times_\Delta H^{(1,0)}_{J_0})
\bigcup H^{(1,0)}_{J_0} \times_\Delta o_{H^{(1,0)}_{J_0}} \right)\right)
\to \CJ_\omega
$$
which is the restriction of the projection map
$$
\CM^{nodal}_{stand}([z_-,w_-],[z_+,w_+];K) \to \CJ_\omega
$$
where we denote
$$
\CM^{nodal}_{stand}([z_-,w_-],[z_+,w_+];K) = \bigcup_{J \in \CJ_\omega}
\CM^{nodal}_{stand}([z_-,w_-],[z_+,w_+];K,J).
$$
Since the latter projection has the index $\mu([z_-,w_-]) - \mu([z_+,w_+]) + 1$
(for the parameterized problem over $0 < \e \leq \e_0$),
the Fredholm index of $\pi_\Upsilon$ is given by $\mu([z_-,w_-]) - \mu([z_+,w_+]) +1 - 2n$.

Therefore for any regular value $J$ of $\pi_\Upsilon$, the preimage will be
empty whenever
$$
\mu([z_-,w_-]) - \mu([z_+,w_+]) < 2n - 1
$$
and in particular when $\mu([z_-,w_-]) - \mu([z_+,w_+])=0$ or $-1$.
This finishes the proof of Theorem \ref{immersed} except the requirement
$[du_-(o_-)] \neq [du_+(o_+)]$. But this itself can be proved by refining
the above genericity argument with an addition of another stratum
$$
\Delta_{H_J^{(1,0)}} \subset H_J^{(1,0)} \times H_J^{(1,0)}
$$
whose details we leave to the readers.

\subsection{Resolved nodal Floer trajectories in Morse back ground}
\label{subsec:resolved}

This subsection is the third stage of the deformation of the parameterized
moduli space of the PSS cobordism.

In this subsection, we consider the Riemann surface with
one positive and one negative punctures
$$
(\Sigma;p,q)
$$
with analytic charts. Modulo the action of $PSL(2,\C)$, we may
identify this with the standard cylinder
$$
(\R \times S^1; \{-\infty\}, \{+ \infty\})
$$
with a global conformal coordinates $(\tau,t)$. The coordinate is uniquely defined
modulo the linear translations
$$
(\tau,t) \mapsto (\tau + \tau_0, t + t_0).
$$
We provide the analytic charts at the punctures $p,\, q$ so that they are
compatible with this identification.
Using this coordinates, we write
$$
K = F(\tau,t)\, d\tau + H(\tau,t)\, dt
$$
and require the condition of cylindrical ends
$$
F\equiv 0, \, H \equiv H(t) \quad \mbox{at $\pm\infty$}
$$
for $K$. This condition does not affect under the coordinate change
$(\tau,t) \mapsto (\tau + \tau_0, t + t_0)$ and so is well-defined.

Similarly we also fix a homotopy from $J_0$ to $J(t)$
$\{J^s\}_{0 \leq s \leq 1}$ so that $J^0 = J_0, \, J^1 = J$.

We will consider a one-parameter family of such pairs $(K_\e,J_\e)$
with their cylindrical ends given by
$$
\End_\pm (K_\e,J_\e) = (H,J), \quad R_0 \leq R(\e) < \infty
$$
for a given Floer-regular pair $(H,J)$. For this purpose, we use the
family of function $\kappa^{\pm}_\e$ constructed in (\ref{eq:betaR})
for $R \in \R_+ = [0,\infty)$. We also define a function  $\r_\e: \R
\to [0,1]$ so that \bea\label{eq:chi}
\r_\e(\tau) & = & \begin{cases}1 \quad & \mbox{ for $|\tau| \leq R(\e)-1$} \\
0 \quad & \mbox{ for $|\tau| \geq R(\e)$} \\
\end{cases}\\
|\r_\e'(\tau)| & \leq & 2 \quad \mbox{ for $ R(\e)-1 \leq |\tau|
\leq R(\e)$}. \eea

For the main purpose of the present paper, we will later choose $R =
R(\e)$ so that \be\label{eq:eR} \e R(\e) \to 0 \quad \mbox{as } \,
\e \to 0. \ee
We remark that the choice of $R(\e)$ made in (\ref{eq:eR}) will be
needed for some normalization procedure which will be explained
later in the course of the adiabatic degeneration argument.

We define $J_\e$  by
\be\label{eq:JR} J_\e(\tau,t,x)
=\begin{cases}J^{\kappa_\e^+(\t)}(t,x) \quad &\mbox{for }\,
 \tau  \geq R(\e) \\
J_0(x) \quad &\mbox{for }\,  |\t| \leq R(\e)-1  \\
J^{\kappa_\e^-(\t)}(t,x) \quad &\mbox{for }\, \tau  \leq - R(\e).
\end{cases}
\ee
 Thanks to the cut-off functions $\kappa_\pm$, this defines a smooth
 $\R\times S^1 $ family of almost complex structures $J$ on M.

Similarly we define the family $K_{\e}:\R\times S^1\times M\to \R$
by \be\label{eq:KR} K_{\e}(\tau,t,x) =
\begin{cases}\kappa_\e^+(\t)\cdot H(t,x) \quad &\mbox{for }\,
 \tau  \geq R(\e) \\
\r_\e(\t)\cdot \e f(x) \quad &\mbox{for }\,  |\t| \leq R(\e)  \\
\kappa_\e^-(\t)\cdot H(t,x) \quad &\mbox{for }\, \tau  \leq - R(\e).
\end{cases}
\ee
\begin{rem}\label{rem:Keef} We would like to compare our choice of $K_\e$ above
with that of \cite{PSS,mcd-sal04}: the latter uses a family of
$K_\e$ with $K_\e \equiv 0$ in the neck region of $\Sigma_\e$, while we use
the one by putting a small Morse function $\e f$ in the neck and take
the adiabatic limit as $\e \to 0$.
With the choice $K_\e \equiv 0$ in the neck region,
this process of degenerating Floer trajectories to nodal ones
by letting $\e \to 0$ is not at all obvious to the authors. However
this process is not properly explained in \cite{PSS,mcd-sal04}.
\end{rem}

 Now using this particular one-parameter family $(K_R,J_R)$ for a given
cut-off function $\kappa = \{\kappa_+,\kappa_-\}$,
we consider the corresponding parameterized moduli space
\beastar
&{}&\CM^{para}([z_-,w_-],[z_+,w_+]);\{(K,J;\kappa)\}\\
& = &\bigcup_{0 < \e \leq \e_0}\CM([z_-,w_-],[z_+,w_+];K_\e,J_\e).
\eeastar
For the simplicity of notations, we will also write
$\CM^{para} = \cup_{0<\e \leq \e_0}\CM^\e$ whenever
there is no danger of confusion. To study the map $\Psi\circ \Phi$
in homology, we need to analyze compactification of
$\CM^{para}$. In the next several sections of the paper,
we prove the following theorem.

\begin{thm}\label{compactify}
The parameterized moduli space
$\CM^{para}$ as $0 < \e \leq \e_0$
can be compactified to a one-dimensional smooth manifold
with boundary whose collar is diffeomorphic to
$$
[0,\e_0) \times \CM^{nodal}_{(0;1,1)}([z_-,w_-],[z_+,w_+];(H,J),(f,J_0))
$$
for a sufficiently small $\e_0 > 0$.
\end{thm}

This theorem provides the right-hand one-sided collar neighborhood
of $\CM_0$ in $\CM^{para} = \cup_{-\e_0 \leq \e \leq \e_0} \CM_\e$
mentioned in section \ref{PSS:scheme}.

\section{Local models near nodes in vacuum}
\label{sec:models}

We study proper holomorphic curves in $\C^n$ with
cylindrical end $\R \times S^{2n-1}$ with a cylindrical metric on
at the ends thereof : We provide a metric
$g$ conformal to the standard metric on $\C^n$ and has the form
$$
g = ds^2 + g_{S^{2n-1}} = \frac{1}{r^2}g_{\C^n}
$$
at infinity where $(r,\Theta)$ is the standard polar coordinates of $\C^n \setminus \{0\}$
and $r = e^{s}$ for the cylindrical coordinates
$(s,\Theta) \in \R \times S^{2n-1}(1)$.

The standard complex structure on $\C^n$ provides an
almost complex structure on the cylinder that is translational
invariant, and the symplectic form written in the coordinates as
$$
\omega_0 = d(r^2 \Theta) = d(e^{2s}\Theta^*\lambda)
= e^{2s}(2ds \wedge \Theta^*\lambda + d\Theta^*\lambda)
$$
where $\lambda$ is the standard contact form on $S^{2n-1} =
S^{2n-1}(1)$. The set of Reeb orbits of $\lambda$ on $S^{2n-1}$
consists of the Hopf circles with constant speed which forms a
smooth family parameterized by $\C P^{n-1}$ and gives a Morse-Bott
type degenerate asymptotic condition at infinity for the relevant
pseudo-holomorphic curves on $\C^n$. A relevant Fredholm theory has
been given in \cite{HWZ:fredholm} in three dimension. And in a
general Morse-Bott setting the Fredholm theory  has been laid out in
\cite{bourgeois} and \cite{fooo07}.

We modify the exposition given in \cite{fooo07} in our context.
The book \cite{fooo07} dealt with the more non-trivial case with
Lagrangian boundary conditions. Because we cannot directly borrow
the results therefrom, we provide detailed explanation in our
current context.

\subsection{Classification of local models}
\label{subsec:classify}

We note that the unit sphere $S^{2n-1}$ has the
standard contact form given by
$$
\lambda = \frac{1}{2}\sum_{i = 1}^n (x_i dy_i - y_i dx_i)
$$
and the associated Reeb vector field by
$$
X_{\lambda} = \sum_{i = 1} \left(x_i \frac{\del}{\del y_i} - y_i
\frac{\del}{\del x_i}\right).
$$
It follows from the expression of the
Reeb vector field $X_{\lambda}$ that the {\it minimal} Reeb orbits
of $(S^{2n-1}, \lambda)$ are given by the
curves $\gamma: [0,2\pi] \to \C^n$ which parameterize a
Hopf circle in $S^{2n-1}$
$$
S^{2n-1} \cap L \subset S^{2n-1}
$$
where $L \subset \C^n$ is one-dimensional complex subspace. We note
that all the Reeb orbits have the same periods $2\pi$ and are
nondegenerate in the Morse-Bott sense.

We recall that an element $u \in \CM_{(0;2,0)}^{SFT}(\C^n;(1,1))$
is assumed to satisfy the convergence
\bea\label{asymp}
\lim_{\tau \to \pm\infty} s\circ u(\tau,t) & = & \infty \nonumber\\
\lim_{\tau \to \pm\infty}\Theta\circ u(\tau,t) & = & \gamma_\pm(t)
\eea
respectively for some Reeb orbit $\gamma_\pm \in \widetilde{\CR}_1(\lambda)$. The following
uniqueness result will be important later in the gluing problem.

\begin{prop}\label{degree2curves} Fix a hyperplane $H$ in $\C P^n$ and two points
$x_0, \, x_\infty \in H$. Consider a rational curve that is not contained
in $H$. Then there exists a unique rational curve
passing through $x_0, \, x_1$ of degree 2 modulo the action of
$Aut(\C P^n;H)$ which is the group of automorphisms of $\C P^n$ fixing
$H$.
\end{prop}
\begin{proof} It is easy to construct a degree two map
$u: S^2 \to \C P^n$ through any two given points in $\C P^n$ and hence
there always exists a map $u: S^2 \to \C P^n$ a holomorphic map of degree 2
satisfying
$$
u(0) = x_0, \quad u(\infty)
= x_\infty;  \quad x_0, \, x_\infty \in H \subset \C P^n.
$$

We now prove the uniqueness
modulo the action of $Aut(\C P^n;H)$.
Let $u'$ be another such curve with the same
asymptotic condition. Then the extension of $u'$ to $\C P^n$
has the condition $u'(0) = u(0)$ and $u'(\infty) = u(\infty)$
and $u(z) \in \C P^n \setminus H \cong \C$. Now we choose a point
$x \in C$ with $C = \operatorname{Image} u$. Composing $u'$ with an element
$g \in Aut(\C P^n;H)$ and replacing $u'$ by $g\circ u'$, we may assume that
$u$ and $u'$ pass through the three distinct points $\{x_0,x_\infty,x\}$.
We note that as long as $n \geq 3$, we can find a hyperplane $H' \subset
\C P^n$ that contains the three points. Then
both $C, \, C'$, which have degree 2, must be contained in the
hyperplane $H' \subset \C P^n$
containing the three points. Repeating this argument inductively whenever
$n \geq 3$, we can reduce the proof to the case $n=2$. i.e., to $\C P^2$.
From now on, we assume $n = 2$.

We choose a point $x \in C \setminus H$, which exists since $C$
is not a line. Let $L$ be a line that is tangent to $C$ at $x$.
We note that any irreducible degree two curve
is immersed (in fact embedded) and the action of $Aut(\C P^n,H) =
\Aut(\C^n)$ preserves the projective tangent line, i.e., the induced map
$$
g_*: \P(T_x \C^n) \to \P(T_{g(x)} \C^n)\, ;\, [\ell] \mapsto
[d_xg(\ell)]
$$
becomes the identity map under the canonical identification of
$\P(T_x \C^n) = \P(\C^n) = \P(T_{g(x)} \C^n)$.

Therefore there is a well-defined Gauss map
$$
\C P^1 \to \C P^1 = \P(\C^2) \, ;\, p \mapsto [du(p)]
$$
where $[du(p)]$ is the tangent line at $u(p)$. Since this map is holomorphic which
is not constant for a degree 2 curve $u$, it must be surjective.
Let $x' \in C'$ be a point such that $[T_xC] = [T_{x'}C']$ in $\P(\C^2)$. We apply an element
$g$ with $g(x) = x'$ to the map $u$. Then the map $g\circ u$ passes
through the three points $x_0, \, x_\infty, \, x'$ and becomes tangent
to $C'$ at $x'$. Finally if $g(C) = C'$ already, we are finished.
Otherwise $(C \setminus C')\setminus H \neq \emptyset$.
We choose a point $y \in (C \setminus C')\setminus H$ and consider
the line $L'$ through $x'$ and $y$. This line cannot coincide with the tangent
line $[T_{x'}C']$ of $C'$and so it must intersect with another point $y' \in C'$
because $C'$ has degree 2.
Now we apply a scaling at the center $x'$
$$
g_\lambda : y \mapsto x' + \lambda(y-x') \, ; \C^2 \to \C^2
$$
which satisfies
$g_\lambda(y) = y'$ (and $g_\lambda(x') = x'$). We consider the map
$g_\lambda\circ g\circ u$ and $u'$. They share 4 points and a common tangent
at the point $x'$. But any two such degree two curves, i.e., conics must
coincide up to reparameterization (see e.g., Remark 4.2.1 in Chapter V
\cite{hartshorne}).
This finishes the proof.
\end{proof}

We now derive the following uniqueness result from the above
proposition.

\begin{thm}\label{uniqueness}
For each given $\gamma_\pm \in \CR_1(\lambda)$,
there exists a unique proper holomorphic map $u: \R \times S^1 \to \C^n$
modulo the action of
$$
\Aut_{lmd} = \Aut(\C^n)
$$
with the given asymptotic condition
$$
ev_{\#,\pm}^{SFT}(u) = \gamma_\pm.
$$
\end{thm}
\begin{proof} The case where $\gamma_+ = \gamma_-$ can be regarded as the
limiting case of unique intersection of the curve and $H$ with multiplicity
two and can be treated similarly. Therefore we will assume
$\gamma_+  \neq \gamma_-$ as an unparameterized curve.

Recall that there is a one-one correspondence between
$\CM^{SFT}_{0;2,0}(\C^n;(1,1))$ and
$\CM_{(0,2,0)}(\C P^n,H;2)$. The theorem immediately follows from
Proposition \ref{degree2curves}.
\end{proof}

\begin{rem}\label{rem:explicit} We can be more explicit by giving the equations of
rational curves in $\C P^2$ with asymptote $x_+$ and $x_-$ on the
hyperplane $H_{\infty}$. In affine charts $\C=\C
P^1-{\infty}\subset \C P^1$ and $\C^2=\C P^2-H_{\infty}\subset \C
P^2$ the quadratic curve satisfies equation
$Ax^2+Bxy+Cy^2+Dx+Ey+F=0$, where the coefficients $A,...,F$ are in
$\C$. For generic coefficients the equation can be factorized as
$$\Big(k_+(x-e)+l_+(y-f)\Big)\Big(k_-(x-e)+l_-(y-f)\Big)=1$$ for
suitable $k_{\pm},l_{\pm},e$ and $f \in \C$. We rewrite the above
equation in a parameterized form \be
\begin{cases}
k_+(x-e)+l_+(y-f)=z \\
k_-(x-e)+l_-(y-f)=\frac{1}{z}
\end{cases}
\ee and solve
\be \label{C2}
\begin{cases} x(z)&=az+ b/z +e \\ y(z)&=cz+ d/z+f,
\end{cases}
\ee where $a,...,f$ are in $\C$. Since $z=e^{2\pi(\tau+it)}$ for
$(\tau,t)\in \R\times S^1$, \be\label{asymp}
\begin{cases}
\lim_{\tau\to +\infty}[x(z),y(z),1]=[a,c,0]=x_+\in H_{\infty}\\
\lim_{\tau\to -\infty}[x(z),y(z),1]=[b,d,0]=x_-\in H_{\infty}.
\end{cases}
\ee From \eqref{asymp} we can determine the coefficients in
\eqref{C2}, up to the ambiguity of $e,f$ arising from $Aut(\C P^2,
H_{\infty})$ and the ambiguity of the ratio $a/b$ arising from
$Aut(\R\times S^1)=\C^*$. Since any two $\C P^2$ in $\C P^n$ are
related by a projective linear transform, which restricted on the
affine chart $\C P^2\backslash H_{\infty}$ is a linear transform,
from \eqref{C2} we get the equation for the degree $2$ curves in the
above theorem: \be \label{eq:uz} u(z)= Az+B/z+C,  \quad
z=e^{2\pi(\tau+it)}\in \C_*, A,B,C \in \C^n, A\neq 0, B \neq 0 \ee
for $z \in \C^* \cong \R \times S^1$.
\end{rem}

\begin{rem} (Local models with Morse background) From the expression
\eqref{eq:uz} we can calculate that the center of mass of the loop
$u(\{\tau_0\} \times S^1)$ in $\C^n $ (with respect to the standard
metric) for any $\tau_0$ is always at the \emph{fixed} vector $C\in
\C^n$ (by mean value theorem of harmonic functions), which indicates
that the needed local model is not quite those lying in
$\CM^{SFT}_{0,2,0}(\C^n;(1,1))$ because they do not reflect the
background gradient flow of the given Morse function $f$. This
motivates us to look for some models in section \ref{sec:ihmodels} for which
the center of mass of the loop $u(\{\tau\} \times S^1)$ resembles
the straight line $a\tau$ in $\C^n \cong T_pM$ where $a = \nabla
f(p)$ the gradient vector of $f$ at $p$.
\end{rem}

\subsection{Fredholm theory of local models}
\label{subsec:fredholm}

We can improve the asymptotic property of such constructed
holomorphic curves in the following way. For each $\gamma \in
\CR_1(\lambda)$, we have the following exponential convergence
statement, which is the analog to the similar results from
\cite{BEHWZ} and can be proved in the same way as other cases of
Morse-Bott-Floer theory. (See also section 62 \cite{fooo07} for the
relevant discussion.)

\begin{prop}\label{exponential} Fix any $p>2$. Let $\e > 0$ be sufficiently small.
Consider a holomorphic map $u:\dot \Sigma \to \C^n$ satisfying
(\ref{asymp}) and write $u(\tau,t) = (s(\tau,t),\Theta(\tau,t))$
near $\tau =\infty$. Then there exist constants $\tau_0$
and $C_k, \,c_k > 0$ depending only on $k,p$ such  that
\bea\label{eq:nablaThetak}
\left|\nabla^k(s(\tau,t) -  2\pi(\tau - \tau_0))
\right| &\leq & C_k e^{\frac{-2\pi c_k|\tau|}{p}} \label{eq:nablask} \\
\left|\nabla^k(\Theta(\tau,t) - \gamma(2\pi t))\right| &\leq & C_k
e^{\frac{-2\pi c_{k}|\tau|}{p}} \eea
\end{prop}

Proposition \ref{exponential} dictates the adequate function space for the
proper Fredholm theory of the pseudo-holomorphic curves in our
problem, which we now explain. Let $\delta  < 2\pi$ be a positive number.

Our metric $h$ is conformal to the
standard Euclidean metric $|dz|^2$ on $\C$ such that
$$
h = \lambda(z) |dz|^2.
$$
where $\lambda: \C \to \R$ is a positive radial function such that
$$
\lambda(z)= \frac{1}{|z|^2}
$$
when $|z|$ sufficiently large.  We also fix a metric $g_{cyl}$ on
$\C^n$ conformal to the standard metric $|dw|^2$ so that
\be\label{eq:g} g_{cyl} = \mu(w) |dw|^2 \ee for a radial function
$\mu$ and it becomes the cylindrical metric on the end of $\C^n$,
i.e.,
$$
\mu(w) = \frac{1}{|w|^2}
$$
when $r = \sum_{i=1}^n |w_i|^2$ is sufficiently large.

With respect to these metrics on the domain and the target, we now
define the space \index{$W_{\delta,(0;2,0)}^{1,p}(\dot \Sigma, \C^n; \gamma_+, \gamma_-, \tau_+,\tau_-)$}
$$
W_{\delta,(0;2,0)}^{1,p}(\dot \Sigma, \C^n; \gamma_+, \gamma_-, \tau_+,\tau_-)
$$
for each fixed $\gamma_\pm$ in $\CR_1(\lambda)$ and $\tau_\pm \in \R$ as follows.

\begin{defn}
$
W_{\delta,(0;2,0)}^{1,p}(\dot \Sigma, \C^n; \gamma_+, \gamma_-, \tau_+,\tau_-)
$
is the set  of all $u = (s\circ u,\Theta\circ u)$ such that
\begin{enumerate}
\item $u \in W^{1,p}_{loc}$
\item
Using coordinate $(\tau,t)$ at the ends of $\dot\Sigma$,  $u$ satisfies
$$
e^{\frac{2\pi \delta |\tau|}{p}}|\Theta_\pm(\tau,t) - \gamma_\pm(t)|
\in W^{1,p}([0,\infty) \times S^1, \R)
$$
where $\Theta_\pm = \Theta\circ u|_{D_\pm}$
\item and
$$
e^{\frac{2\pi\delta |\tau|}{p}}|s_\pm(\tau,t) - 2\pi(\tau -
\tau_\pm)| \in W^{1,p}([0,\infty) \times S^1,\R)
$$
where $s_\pm = s\circ u|_{D_\pm}$.
\end{enumerate}
Here we denote by $D_\pm$ the given coordinate disks associated to
the analytic charts at $p_\pm$, and use the cylindrical metrics $h
= \lambda(|z|)|dz|^2$ to define $W^{1,p}$ and the metric $g_{cyl}$
(\ref{eq:g}) of $\C^n$ to define $|\quad |$. We call the tuple
$\left((\gamma_-,\tau_-),(\gamma_+,\tau_+)\right)$ the
\emph{asymptotic datum}  of $u$ relative to the cylindrical ends
associated to the given analytic charts.
\end{defn}

Proposition \ref{exponential} implies \index{$\MM_{(0;2,0)}(\dot \Sigma,\C^n)$}
$$
\MM_{(0;2,0)}(\dot \Sigma,\C^n) \subset \bigcup_{\gamma_\pm \in
\CR_1(\lambda)}\bigcup_{\tau_\pm \in
\R} W_{\delta,(0;2,0)}^{1,p}(\dot \Sigma,\C^n;\gamma_\pm,\tau_\pm).
$$
We define
$$
W_{\delta,(0;2,0)}^{1,p}(\dot \Sigma,\C^n) = \bigcup_{\gamma_\pm\in
\CR_1(\lambda) }\bigcup_{\tau_0 \in \R}
W_{\delta,(0;2,0)}^{1,p}(\dot \Sigma,\C^n;\gamma_\pm, \tau_\pm)).
$$
We recall that we have the natural projection
$$
\pi: \CR_1(\lambda) \to \C P^{n-1}
$$
forms a principal $S^1$-bundle with the $S^1$-action being the
Hopf action, which can also be realized by with the rotations of
the domain circle. We denote by
$$
ev_\pm: W_{\delta,(0;2,0)}^{1,p}(\dot \Sigma,\C^n) \to
\CR_1(\lambda)
$$
the evaluation map defined $ev_\pm(u) = u(\pm\infty,\cdot)$.

The following can be proved by a standard argument. We refer to
\cite{fooo07} for the relevant proof in the more complicated
context of proper holomorphic curves with Lagrangian boundary
conditions. Since we will need to use the description of the
tangent space thereof, we give an outline of the proof of this
lemma.

\begin{lem}
$W_{\delta,(0;2,0)}^{1,p}(\dot \Sigma,\C^n)$ has the structure of
Banach manifold such that the obvious projection \be\label{eq:proj}
((\pi \circ ev_+,ev_\R^+), (\pi\circ ev_-, ev_\R^-)) :
W_{\delta,(0;2,0)}^{1,p}(\dot \Sigma,\C^n) \to (\C P^{n-1} \times
\R)\times (\C P^{n-1} \times \R) \ee defines a locally trivial fiber
bundle.
\end{lem}
\begin{proof}
The tangent space of this Banach manifold can be constructed as in
Lemma 29.5 or Lemma 60.10 \cite{fooo07}. See \cite{bourgeois} for
a relevant discussion.

We take a smooth function
$\chi_+ :[0,\infty)\to [0,1]$ such that
$\chi_+(\tau) = 1$ for $\tau>2$, $\chi_+(\tau) = 0$ for $\tau <
1$, and $|\chi_+'(\t)|<2$. Symmetrically, we let
$\chi_-(\t)=\chi_+(-\t)$. Let $u \in W_{\delta,(0;2,0)}^{1,p}(\dot
\Sigma,\C^n; \gamma_\pm, \tau_\pm) \subset
W_{\delta,(0;2,0)}^{1,p}(\dot \Sigma,\C^n)$. We consider the set
of all quintuples $(U,V_{\CR_1(\lambda)}^\pm, V_{\R}^\pm)$
satisfying
\begin{enumerate}
\item $V_{\CR_1(\lambda)}^\pm(t) \in T_{\gamma_\pm}\CR_1(\lambda)$,
$V_{\R}^\pm \in \R \cong T_{\tau_\pm}\R$ respectively.
\item  $U \in W^{1,p}_{loc}(\dot \Sigma;u^*T\C^n)$.
\item Denote
\beastar %
\widetilde U(\tau,t) & = & U(\tau,t) - \chi_-(\t) Pal_{u(\t,t)}
U(-\infty,t) -\chi_+(\t) Pal_{u(\t,t)} U(+\infty,t),
% \widetilde U(\tau,t) & = &
%U(\tau,t) - \chi_-(\t)\left(-2\pi V_{\R}^- u(z) + e^{-2\pi
%((\tau-\tau_-)
%+ \sqrt{-1}t)}V_{\CR_1(\lambda)}^+(0)\right) \\
%&{}& \quad +\; \chi_+(\t)\left(2\pi V_{\R}^+ u(z) + e^{2\pi
%((\tau-\tau_+) + \sqrt{-1}t)}V_{\CR_1(\lambda)}^-(0)\right).
\eeastar %
where $U(+\infty,t)=(V^+_\R,V^+_{\CR_1(\l)}(t))$ in the cylindrical
end of $(\C^n,g_{cyl})$ , and $Pal_{u(\t,t)} U(+\infty,t)$ is the
parallel transport of $U(+\infty,t)$ from $u(+\infty,t)$ to
$u(\t,t)$ along the minimal geodesic of $(\C^n,g_{cyl})$.
Similarly we define $Pal_{u(\t,t)} U(+\infty,t)$. Then we have
$$
e^{\frac{2\pi\delta |\tau|}{p}}|\widetilde U(\tau,t)| \in
W^{1,p}_{D_\pm}(\dot\Sigma, \R)
$$
Here we regard $V^\pm_{\CR_1(\lambda)}$ to be a vector field on
$\gamma_\pm$ in $\C^{n}$. Since every parameterized simple Reeb
orbit $\g$ in $S^{2n-1}\subset \C^n$ satisfies $\g(t)=e^{\pm 2\pi
i}\g(0)$, we have $V^\pm_{\CR_1(\lambda)}(t)=e^{\pm 2\pi i}
V^\pm_{\CR_1(\lambda)}(0)$.
\end{enumerate}

Let $C^0(u)$ be the set of all such quintuples. It becomes a Banach
space with norm $\|\cdot\|_{1,p,\d}$ such that

\bea \Vert(U,V_{\CR_1(\lambda)}^\pm, V_{\R}^\pm)\Vert^p_{1,p,\d} & =
& \left\Vert e^{\frac{2\pi\delta |\tau|}{p}}\widetilde U(\tau,t)
\right\Vert^p_{W^{1,p}} \nonumber \\
&+&   |V_{\CR_1(\lambda)}^-(0)|^p + |V_{\R}^-|^p +
|V_{\CR_1(\lambda)}^+(0)|^p + |V_{\R}^+|^p \eea

We remark that $V_{\CR_1(\lambda)}^\pm,V_{\R}^\pm$ are determined
from $U$ in case the norm
$\Vert(U,V_{\CR_1(\lambda)}^\pm,V_{\R}^\pm)\Vert_{1,p,\d}$ is finite.

It is standard to check that $W_{\delta,(0;2,0)}^{1,p}(\dot \Sigma,\C^n)$
is a Banach manifold and
$$
C^0(u) = T_uW_{\delta,(0;2,0)}^{1,p}(\dot \Sigma,\C^n).
$$
To show that (\ref{eq:proj}) is a locally trivial fiber bundle,
we use the $U(n)$ action as a
biholomorphic isometry on $\C^n$ which preserves the contact form
$(S^{2n-1},\lambda)$.
It induces an $U(n)$-action on $W_{\delta,(0;2,0)}^{1,p}(\dot\Sigma,\C^n)$.
Then the map (\ref{eq:proj}) is $U(n)$-equivariant.

On the other hand the group $\C^* \cong \operatorname{Aut}(\C P^1;\{0,\infty\})$
acts on $W_{\delta, (0;1,0)}^{1,p}(\dot\Sigma,\C^n)$
as the automorphism of the domain and on $\C P^{n-1} \times \R$ trivially
on $\C P^{n-1}$ and by an addition by $\frac{1}{2\pi}\ln |z|$ on $\R$.

Then (\ref{eq:proj}) is $\C^*$-equivariant.
The local triviality (\ref{eq:proj}) follows from this equivariance.
\end{proof}

We next put
$$
C^1(u) = L^p_\delta(\dot \Sigma,\Lambda^{(0,1)}(u^*T\C^n)).
$$
Then there exists an infinite dimensional vector bundle over
$W_{\delta,(0;2,0)}^{1,p}(\dot \Sigma,\C^n)$ whose fiber at $u$ is
$C^1(u)$.

The formal linearization of the Cauchy-Riemann operator $\overline
\del$ defines an operator \be\label{eq:Dud} D_u \overline \del :
C^0(u) \to C^1(u). \ee We apply $D_u \overline \del$ only the
first component $U$ of the triple $(U,V_{\CR_1(\lambda)},V_\R)$.

To see $D_u \overline \del$ indeed maps to $C^1(u)$, consider the
function
$$\tilde{u}(\t,t)=e^{2\pi(\t-\t_+)}\g_+(t)=e^{2\pi(\t-\t_++\sqrt{-1}t)}\g_+(0)$$
in $\C^n$, then $\tilde{u}$ is holomorphic, and has the same
asymptote as $u$ when $\t\to +\infty$. It is easy to verify that
\be
D_{\tilde{u}}\delbar(Pal_{u(\t,t)}U(+\infty,t))=0.\label{vanish}
\ee Using \eqref{vanish} the fact that $u-\tilde{u}\in
W^{1,p}_{\d}([0,+\infty]\times S^1)$, we see
$$D_{u}\delbar(Pal_{\tilde{u}(\t,t)}U(+\infty,t))\in L^p_{\d}([0,+\infty)\times
S^1), $$ similarly
$$D_{u}\delbar(Pal_{\tilde{u}(\t,t)}U(-\infty,t))\in L^p_{\d}(-\infty,0]\times
S^1). $$ Hence $(D_u \overline \del)(U)$ is contained in $C^1(u)$.

%$$(D_u \overline \del)\left(2\pi V_{\R}^\pm u(\tau,t)  -
%e^{2\pi((\tau-\tau_0) + \sqrt{-1}t)}V_{\CR_1(\lambda)}^\pm\right)
%$$
%goes to zero in the exponential order as $\tau \to \infty$, we can show that
%$(D_u \overline \del)(U)$ is contained in $C^1(u)$.

\begin{prop}
The operator (\ref{eq:Dud}) is Fredholm with index given by
\be\label{eq:index}
\operatorname{Index }D_u\delbar = 4(n+1) + 4 - 6 +1  = 4n+3
\ee
\end{prop}
\begin{proof}
The Fredholm property can be proved in the same way as Lemma 60.14
\cite{fooo07} and so its proof is omitted referring readers
thereto. The index formula can be derived from the general theory
from \cite{bourgeois}, \cite{EGH} but we will give its proof as a
consequence of the classification, Proposition \ref{degree2curves},
and the transversality of the local models which we will prove in
the next subsection.
\end{proof}

\subsection{Transversality of local models}
\label{subsec:surjectivity}

We recall
that each element $u \in \MM_{(0;2,0)}(\dot \Sigma,\C^n)$
has the convergence
$$
\lim_{\tau \to \pm \infty}s\circ u(\tau,t)  = \infty, \qquad
\lim_{\tau \to \pm\infty}\Theta\circ u(\tau,t)  = \gamma_\pm(t)
$$
for some $\gamma_\pm \in \CR_1(\lambda)$.

By Proposition \ref{exponential}, we have
$$
\CM_{(0;2,0)}(\dot \Sigma,\C^n) \subset
W^{1,p}_{\delta,(0;2,0)}(\dot\Sigma,\C^n).
$$
We first prove the following surjectivity.

\begin{prop}\label{dbaronto} Let $u \in \MM_{(0;2,0)}(\dot \Sigma,\C^n)$
and
$$
C^0(u) = T_uW_{\delta,(0;2,0)}^{1,p}(\dot \Sigma,\C^n),\quad
C^1(u) = L_{\delta}^p(\dot \Sigma,\Lambda^{(0,1)}(u^*T\C^n)).
$$
Then the linearization operator
$$
D_u\delbar : C^0(u) \to C^1(u)
$$
is surjective.
\end{prop}
\begin{proof} We note that since the almost complex structure on $\C^n$ is
integrable, we have
$$
D_u\delbar = \text{the standard Dolbeault operator}.
$$
Here we recall from Theorem \ref{immersed}
that the nodal points are immersed and the tangent planes
of the two components at the nodes are different. This is translated as
the local model $u$ corresponds to a degree two rational
curve in $\P^n$ that intersect the hyperplane $\H$ at two
distinct points transversely. Furthermore since we use the cylindrical metrics on
both on the domain $\dot\Sigma$ and $\C^n$, the operator
$$
D_u\delbar : C^0(u) \to C^1(u)
$$
can be shown to be conjugate to the standard Dolbeault operator
$$
\delbar : W^{1,p}(u^*T\P^n) \to L^p(\Lambda^{(0,1)}_J(u^*T\P^n))
$$
if we regard $u$ as a map from $\P^1$ to $\P^n$.
Once we have these, the surjectivity
immediately follows from the well-known fact $H^1(u^*T\P^n) = \{0\}$
for any rational curve $ u: \P^1 \to \P^n$.
\end{proof}

Due to the Morse-Bott character of our gluing problem, this
surjectivity will not be enough for our purpose. We need to augment
this by the asymptotic evaluation datum at infinity. We recall that
$ev : W_{\delta,(0;2,0)}^{1,p}(\dot \Sigma,\C^n) \to (\C P^{n-1}
\times \R) \times (\C P^{n-1} \times \R)$ is the assignment of the
asymptotic datum followed by the projection $\CR_1(\lambda) \to \C
P^{n-1}$.

The following is the main result of this subsection.

\begin{thm}\label{surjectivity} Let $u$ be a holomorphic curve
in $\CM_{(0;2,0)}(\dot \Sigma,\C^n)$ with the asymptotic datum
$(\gamma_\pm,\tau_\pm)$. Then the operator
$$
D_u\overline{\partial} \oplus D\pi :
C^0(u) \to C^1(u) \oplus
T_{([\gamma_-],\tau_-)}(\C P^{n-1} \times \R) \oplus
T_{([\gamma_+],\tau_+)}(\C P^{n-1} \times \R)
$$
is surjective.
\end{thm}
\begin{proof} We note that the action of $U(n)$ on $\C^n$ preserves
$(S^{2n-1},\lambda)$ and induces an action on $\C P^{n-1}$.
By this $U(n)$-invariance of the equation, it suffices to consider the case when
$\gamma_-$ is the equator given by
$$
\gamma_-(t) = (\cos (2\pi t) + i \sin (2\pi t), 0, \cdots, 0) \in
S^{2n-1} \subset \C^n.
$$
Furthermore $\C P^{n-1}$ is \emph{two-point homogeneous} in that
any two pair of distinct points can be mapped to each other by the
action of $U(n)$ (or more precisely by the action of $PU(n)$).

As a first step, we will establish the following splitting :
\be\label{eq:splitting}
C^0(u) = T_u W^{1,p}_{\delta,(0;2,0)}(\dot\Sigma,\C^n) =
\C^{n}_-(u) \oplus \C^{n}_+(u) \oplus W^{1,p}(u^*T\C^n)
\ee
such that the restriction
\be\label{eq:restDpi}
D\pi:\C^{n}_-(u) \oplus \C^{n}_+(u) \to T_{([\gamma_-],\tau_-)}(\C P^{n-1} \times \R) \oplus
T_{([\gamma_+],\tau_+)}(\C P^{n-1} \times \R)
\ee
is surjective.

First we find a subspace $\C^{n-1} \subset u(n)$ such that
$$
\C^{n-1} \oplus u(n-1) = u(n)
$$
where $U(n-1)$ is the isotropy subgroup of the vector $(1,0,\cdots,0)$.
We identify $\C^{n-1}\cong u(n)/u(n-1)$ with $\{0\}\oplus \C^{n-1} \subset \C^n$.
The action
$$
A \in \C^{n-1} \cong u(n)/u(n-1) \mapsto A \cdot u
$$
defines an embedding
$$
\C^{n-1} \to C^0(u) = T_u W^{1,p}_{\delta,(0;2,0)}(\dot\Sigma,\C^n)
$$
We denote by $\C^{n-1}_-(u)$ the image of this embedding.

Similarly we consider the isotropy group $U(n-1)_{\gamma_+}$
of $[\gamma_+]$ and define an embedding of $\C^{n-1}$
into $C^0(u)$ by the similar way. We denote this by $\C^{n-1}_+(u)$.
Using the fact that $\C P^{n-1}$ is two-point homogeneous, we can choose
this embedding so that $\C^{n-1}_+(u)$ and $\C^{n-1}_-(u)$ are
linearly independent. We note by construction of $\C^{n-1}_\pm(u)$ that
they are transverse to the Hopf action.

Finally we take the generators $X$ of $\C \cong \mbox{aut}(\dot\Sigma)
\cong \operatorname{aut}(\C, \{0\})$ and consider the action
$$
X \mapsto \LL_{X} u.
$$
This action at each end $\gamma_\pm$ of $u$
defines an embedding $\C \to C^0(u)$ which coincides with the
infinitesimal action of the translations along the direction of $\tau$
in the analytic charts chosen at $p_\pm$. The image of this
embedding is linearly independent of $\C^{n-1}_\pm(u)$.
This gives rise to the required splitting by setting
\be\label{eq:Cnu}
\C_\pm^n(u) = \C \dot \gamma_\pm \oplus \C^{n-1}_\pm(u)
\ee
respectively. By construction, it follows that
the map (\ref{eq:restDpi}) is surjective and so
we have established the required splitting (\ref{eq:splitting}).

By the splitting (\ref{eq:splitting}), Theorem \ref{surjectivity}
will follow from the surjectivity of the linear map
$$
D_u\overline \del: C^0(u) \to C^1(u)
$$
which is precisely the content of Proposition \ref{dbaronto}.
This finishes the proof.
\end{proof}

Once we have established the surjectivity of the linearization
operator $D_u\delbar$, the moduli space
$\CM_{(0;2,0)}(\dot\Sigma,\C^n)$ becomes a smooth manifold whose
tangent space can be identified with the kernel of the operator
$D_u\delbar : C^0(u) \to C^1(u)$ for a holomorphic curve $u:
\dot \Sigma \to \C^n$ found in the uniqueness Theorem \ref{uniqueness}.

And the classification theorem, Proposition \ref{degree2curves},
immediately proves the following theorem. This in particular
computes  the index of $D_u \delbar$, when combined with
the surjectivity.

\begin{thm}\label{kernel} Assume $0 < \delta <2\pi$ is sufficiently small.
Let $u$ be the pseudo-holomorphic map constructed in Theorem \ref{uniqueness}
associated to the Reeb orbits $\gamma_\pm$.
Then we have
$$
\dim \operatorname{Ker} D_u\overline\del = 4n+3
$$
and all the elements $u \in \widetilde \CM_{(0;1,1)}(\dot \Sigma,\C^n)$ is
transverse and so is a smooth manifold of dimension $6n$.
Furthermore the quotient space
$$
\widetilde \CM_{(0;1,1)}(\dot \Sigma,\C^n)/ \Aut_{lmd}
$$
is a one-point set.
\end{thm}

We would like to separately state the following obvious corollary
of Theorem \ref{surjectivity} and \ref{kernel}

\begin{cor} We have
$$
\operatorname{Index}D_u\delbar = 4n+3.
$$
\end{cor}

\section{Local models near nodes in Morse background}
\label{sec:ihmodels}
   In this section, we provide the Banach manifold for the
solutions $(u,a)$ of the inhomogeneous Cauchy-Riemann equation $\delbar_{J_0} u=
a$ in $(\mathbb{C}^n,J_0)$, with the same asymptotic condition (\ref{asymp}) for $u$ at
infinity as the homogeneous case in previous section. Here $u:\R\times S^1\to \C^n$ is a map,
$\delbar_{J_0}$ is the standard Cauchy-Riemann operator,
$a$ is a section of $\Lambda^{0,1}(u^*(T\mathbb{C}^n))$
that satisfies
$$
a\left(\frac{\del}{\del\tau}\right) = \nabla f(p)
$$
in $\mathbb{C}^n$ in coordinates $(\tau,t)$, where $z=e^{2\pi(\tau+it)}$. Then the equation $\delbar u=
a$ is nothing but
$$
\left(\frac{\partial}{\partial \t}+J_0\frac{\partial}{\partial t}\right)u(\t,t)=\nabla f(p),
$$
whose inhomogeneous term is a constant vector $\nabla f(p)$ in $(T_pM,J_p)\cong(\C^n,J_0)$.
We remark that the norm of a section of
$\Lambda^{0,1}(u^*(T\mathbb{C}^n))$ is
induced from the asymptotically cylindrical metric $\lambda(x) g_{st}$
rather than the standard metric $g_{st}$ in $\C^n$, so \emph{the norm of
the 1-form $a$ depends on $u(\tau,t)$ pointwise and is not a
constant}.

We need to adapt the local models found in the previous section
because the ambient manifold has the background gradient flow
associated to the given Morse function $f$ which
affects the Hamiltonian perturbation $K$ for the
resolved nodal Floer trajectories. The relevant rescaling procedure
of the Floer equation at the nodes does not yield the homogeneous equation $\delbar u = 0$
but yields the inhomogeneous Cauchy
Riemann equation $\bar\partial u=a$ on $T_pM \cong \C^n$ for the 1-form $a$
with $a(\frac{\del}{\del \tau})=\nabla f(p)$.

Recall the off-shell Banach manifold for homogeneous local models was
$$
W^{1,p}_{\delta,(0;2,0)}(\dot \Sigma,\C^n) =
W^{1,p}_{\delta,(0;2,0)}(\dot \Sigma,\C^n;(1,1)).
$$
For inhomogeneous
local models, we set the Banach manifold to be the set \index{$\mathcal{B}_0$}
$$
\mathcal{B}_0:= W^{1,p}_{\delta,(0;2,0)}(\dot \Sigma,\C ^n)\oplus \C^n\backslash\{0\}
$$
of the pair $(u,a)$ and the Banach bundle over it as
$$\mathcal{L}_0:=\bigcup_{(u,a)\in \mathcal{B}_0} L^p_{\delta}(\dot\Sigma,\Lambda^{0,1}(u^*T\C^n)) $$
 We define the \textit{augmented Cauchy Riemann operator}  $\widehat\partial: \mathcal{B}_0\to
 \mathcal{L}_0$ as
    $$\widehat\partial:
    W^{1,p}_{\delta,(0;2,0)}(\dot \Sigma,\C ^n)\oplus
    \C^n\backslash\{0\} \rightarrow
    \bigcup_{(u,a)\in \mathcal{B}_0} L^p_{\delta}(\dot\Sigma,\Lambda^{0,1}(u^*T\C^n))
    $$
such that
    $$\widehat\partial(u,a)=\bar\partial u- a$$

The following proposition gives the relation between homogeneous
and inhomogeneous Cauchy-Riemann equations:

\begin{prop}\label{prop:decomp}
We equip both $\dot\Sigma$ and $\C^n$ with metrics cylindrical at infinity.
Then the followings hold :
\begin{enumerate}
\item $u_0$ in $W^{1,p}_{\delta,(0;2,0)}(\dot \Sigma,\C^n)$ if and only if
$u:=u_0+a\tau$ is in $W^{1,p}_{\delta,(0;2,0)}(\dot \Sigma,\C ^n)$.
\item $\delbar u_0=0$ if any only if $\bar\partial u-a=0$.
\item Suppose $u_0$ in $W^{1,p}_{\delta,(0;2,0)}(\dot \Sigma,\C^n)$
has the decomposition $u_0 = (s\circ u_0,\Theta\circ u_0)$
that satisfies the asymptotic condition
\beastar
\lim_{\tau\to \pm\infty} \Theta\circ u_0(\tau,t)& = & \gamma_{\pm}(t)\\
\lim_{\tau \to \pm \infty} s\circ u_0 (\tau,t)& = & 2\pi(\tau-\tau_{\pm}).
\eeastar
Then $u = u_0 + a\tau$ is also in $W^{1,p}_{\delta,(0;2,0)}(\dot \Sigma,\C^n)$ and
satisfies the same asymptotics.
\end{enumerate}
\end{prop}
\begin{proof}  Since
$\delbar (a\tau)=a$, it follows that
$\widehat\partial(u,a)=0$ if and only if $\delbar u_0=0$. (1) is a
conclusion of (3), by applying (3) to $u=u_0+a\t$ and $u_0=u-a\t$.
So it remains to prove the statement (3).

Write $u_0(\tau,t)=(\Theta_0(\tau,t),s_0(\tau,t))$ in the cylindrical
end of $\C^n$. Then by the definition of
$W^{1,p}_{\delta,(0;2,0)}(\dot \Sigma,\C ^n)$, there exist $\gamma_\pm \in \CR_1(\lambda)$,
$\tau_\pm \in \R$ such that
$$
e^{ \frac{2\pi\delta|\tau|}{p} }
|\Theta_0(\tau,t)-\gamma_{\pm}(t)|_{S^{2n-1}} \in W^{1,p}(\R\times
S^1,\R)
$$
$$e^{\frac{2\pi\delta|\tau|}{p}} |s_0(\tau,t)-2\pi(\tau-\tau_{\pm})| \in
W^{1,p}(\R\times S^1,\R)
$$
Since $W^{1,p}([k,k+1]\times S^1)\hookrightarrow
C^{0,\alpha}([k,k+1]\times S^1)$ when $p>2$, we have
$$
\max_{\tau\in [k,k+1]} e^{ \frac{2\pi\delta|\tau|}{p} }
|\Theta_0(\tau,t)-\gamma_{\pm}(t)| \le C \| e^{
\frac{2\pi\delta|\tau|}{p} }
|\Theta_0(\tau,t)-\gamma_{\pm}(t)|\|_{W^{1,p} ([k,k+1]\times
S^1,\R)}\to 0
$$
Hence
$$
|\Theta_0(\tau,t)-\gamma_{\pm}(t)| \le C e^{-\frac{2\pi\delta
|\tau|}{p}}
$$
for large enough $\tau$. Similarly
$$
|s_0(\tau,t)-2\pi(\tau-\tau_{\pm})| \le  C e^{-\frac{2\pi\delta
|\tau|}{p}}
$$
for large enough $\tau$. So for large $\tau$,
$$
s_0(\tau,t)\ge 2\pi(\tau-\tau_{\pm})-1
$$
Thus
$$ |u_0(\tau,t)|=|e^{s_0(\tau,t)}|
\ge e^{-1} e^{ 2\pi|\tau-\tau_{\pm}| } \ge C e^{2\pi|\tau|}
$$
for $|\tau|$ sufficiently large.

Let $u(\tau,t)=u_0(\tau,t)+a\tau$, and write
$u(\tau,t)=(\Theta(\tau,t), s(\tau,t))$. Then
\beastar
\Theta(\tau,t)
& = & \frac{u_0(\tau,t)+a\tau}{\|u_0(\tau,t)+a\tau\|}
=\frac{u_0(\tau,t)} {\|u_0(\tau,t)\|}
\cdot \frac{\|u_0(\tau,t)\|} {\|u_0(\tau,t)+a\tau\|}
+\frac{a\tau}{\|u_0(\tau,t)+a\tau\|}\\
& = & \Theta_0(\tau,t) \left(1-\frac{
\|u_0(\tau,t)+a\tau\|-\|u_0(\tau,t)\| }{ \|u_0(\tau,t)+a\tau\|}
\right) +\frac{a\tau}{\|u_0(\tau,t)+a\tau\|} \eeastar
Therefore
\beastar \|\Theta(\tau,t)-\Theta_0(\tau,t) \| &\le&
\frac{\|a\tau\|}{\|u_0(\tau,t)+a\tau\|}+\frac{\|a\tau\|}{\|u_0(\tau,t)+a\tau\|}\\
&\le& \frac{2\|a\|\tau}{Ce^{2\pi|\tau|}-\|a\|\tau }\\
&\le&  C_1 \|a\|e^{-2\pi|\tau|}
\eeastar
for large enough $|\tau|$. Hence
$$
e^{ \frac{2\pi\delta|\tau|}{p} } |\Theta(\tau,t)-\Theta_0(\tau,t)|
\in L^p(\R\times S^1,\R).
$$
Similar straightforward computation also shows
$$
e^{ \frac{2\pi\delta|\tau|}{p} } |\nabla
\Theta(\tau,t)-\nabla\Theta_0(\tau,t)| \in L^p(\R\times S^1,\R).
$$
 Hence
\be \label{Theta} e^{ \frac{2\pi\delta|\tau|}{p} }
|\Theta(\tau,t)-\Theta_0(\tau,t)| \in W^{1,p}(\R\times S^1,\R). \ee
Since
$$
e^{ \frac{2\pi\delta|\tau|}{p} } |\Theta(\tau,t)-\gamma_{\pm}(t)|
\le e^{ \frac{2\pi\delta|\tau|}{p} } \big(
|\Theta(\tau,t)-\Theta_0(\tau,t)|+
|\Theta_0(\tau,t)-\gamma_{\pm}(t)| \big)
$$
we get
$$
e^{ \frac{2\pi\delta|\tau|}{p} } |\Theta(\tau,t)-\gamma_{\pm}(t)|
\in W^{1,p}(\R\times S^1,\R).
$$

Next, we estimate $s(\tau,t)=\log |u_0(\tau,t)+a\tau|$:
$$
|s(\tau,t)-s_0(\tau,t)|
=\log\frac{\|u_0(\tau,t)+a\tau\|}{\|u_0(\tau,t)\|}
=\log\left( 1+
\frac{\|u_0(\tau,t)+a\tau\|-\|u_0(\tau,t)\|} {\|u_0(\tau,t)\|}\right).
$$
Since
$$
\left|\frac{\|u_0(\tau,t)+a\tau\|-\|u_0(\tau,t)\|}
{\|u_0(\tau,t)\|}\right| \le \frac{\|a\tau\|}{\|u_0(\tau,t)\|} \le
\frac{\|a\tau\|}{C_1e^{2\pi|\tau|}} = C_2|\t|e^{-2\pi|\tau|} \to
0,
$$
and $\log(1+h)\sim h$ when $h\to 0$, for large enough $\tau$
$$
|s(\tau,t)-s_0(\tau,t)|\le 2 C_2 |\t|e^{-2\pi|\tau|}.
$$
Now
$$e^{\frac{2\pi\delta|\tau|}{p} }
|s(\tau,t)-2\pi(\tau-\tau_{\pm})| \le e^{
\frac{2\pi\delta|\tau|}{p} } \big( |
s(\tau,t)-s_0(\tau,t)|+|s_0(\tau,t)-2\pi(\tau-\tau_{\pm})|\big),
$$
hence
$$
e^{ \frac{2\pi\delta|\tau|}{p}} | s(\tau,t)-2\pi(\tau-\tau_{\pm})|
\in L^p(\R\times S^1,\R).
$$
We also have
\beastar |\nabla s-\nabla s_0| &=&\left|
\frac{(u+a\tau)\cdot(\nabla(u+a\tau))}{|u+a\tau|^2}
-\frac{u\cdot\nabla u}{|u|^2}\right|\\
&\le& |(u+a\t)\cdot
(\nabla(u+a\t))|\cdot\left|\frac{1}{|u+a\t|^2}-\frac{1}{|u|^2}\right|\\
&+&\frac{1}{|u|^2}\left|(u+a\t)\cdot \nabla(u+a\t)-u\cdot \nabla u
\right| \\
&\le& C_3\left( |u||\nabla u|\frac{|a\t|}{|u|^3} +
\frac{1}{|u|^2}(|a\cdot u|+|a\t\cdot\nabla u|)\right)\\
&\le& C_3 |\t| e^{-2\pi|\t|}
 \eeastar
when $|\t|$ large. So \be e^{ \frac{2\pi\delta|\tau|}{p}}
|s(\tau,t)-s_0(\tau,t)| \in W^{1,p}(\R\times S^1,\R)\label{S}. \ee
Together with
$$ e^{
\frac{2\pi\delta|\tau|}{p}} |s_0(\tau,t)-2\pi(\tau-\tau_{\pm})| \in
W^{1,p}(\R\times S^1,\R),
$$
by triangle inequality
$$ e^{ \frac{2\pi\delta|\tau|}{p}} |
s(\tau,t)-2\pi(\tau-\tau_{\pm})| \in W^{1,p}(\R\times S^1,\R).
$$
This finishes the proof of (3).
\end{proof}

 Denote the moduli space of solutions of $\widehat\partial (u,a)=0$
by $\CM_{(0;2,0)}^+(\dot\Sigma,\C^n;(1,1))$, which is the moduli space of
{\em inhomogeneous local models}. By Proposition \ref{prop:decomp}, we
see
\begin{cor}\label{cor:ihlmd} $\CM_{(0;2,0)}^+(\dot\Sigma,\C^n;(1,1))
\cong \CM_{(0;2,0)}(\dot\Sigma,\C^n;(1,1))\oplus \C^n\backslash \{0\}$.
\end{cor}

By the same argument, it is immediate to check the following lemma
whose proof is omitted.

\begin{lem}\label{oldbundle} $\widehat\partial(u,a)=\bar\partial u-
a$ is indeed in $L^p_{\delta}(\dot\Sigma,\Lambda^{0,1}(u^*T\C^n))$.
\end{lem}

Next we consider the tangent space of any $(u,a)\in\mathcal{B}_0$.
The tangent space consists of elements
$(U,V_{\R^{\pm}},V_{\CR_1^{\pm}(\lambda)},h)$, where $U$ is a section
in $W^{1,p}_{\delta}(\dot\Sigma,u^*T\C^n)$, $ V_{\R^{\pm}}\in
T_{\tau_{\pm}}\R \cong \R$, $V_{\CR_1(\lambda)}\in
T_{\gamma_{\pm}(0)}S^{2n-1}$, and $h\in T_a\C^n$. The linearized
$\widehat \partial$ operator is
$$D_{(u,a)}\widehat \partial:
W^{1,p}_{\delta}(\dot\Sigma,u^*T\C^n)\oplus T_a\C^n \to
L^p_{\delta}(\dot\Sigma,\Lambda^{0,1}(u^*T\C^n)),$$
\bea
D_{(u,a)}\widehat\partial(U,V_{\R_{\pm}},V_{R(\lambda)},h)
=D_u\delbar
 U-h=\delbar U-h.\label{U-h}
\eea
The last identity holds because $D_u \delbar$ becomes the standard
Dolbeault operator in $\C^n$. Recall the projection
$$
\pi:W^{1,p}_{\delta,(0;2,0)}(\dot \Sigma,\C^n;(1,1))\to (S^{2n-1}\times \R)\times(S^{2n-1}\times
\R),$$
$$u\to (\gamma_+(0),\tau_+)\times (\gamma_-(0),\tau_-),$$
$$D\pi:T_u W^{1,p}_{\delta,(0;2,0)}(\dot \Sigma,\C ^n)
\to T_{\gamma_+(0),\tau_+}(S^{2n-1}\times \R) \times
T_{\gamma_-(0),\tau_-}(S^{2n-1}\times \R). $$ We consider the
combined operator \beastar D_u\widehat\partial \oplus D\pi: T_u
W^{1,p}_{\delta,(0;2,0)}(\dot \Sigma,\C ^n) \to
L^p_{\delta}(\dot\Sigma,\Lambda^{0,1}(u^*T\C^n)) & \oplus &
T_{(\gamma_+(0),\tau_+)} (S^{2n-1}\times\R)\\
&\oplus& T_{(\gamma_-(0),\tau_-)} (S^{2n-1}\times\R), \eeastar
$$(U,V_{\R_{\pm}},V_{\CR_1(\lambda)},h)\to (D\delbar U -h,V^+_\R,V^+_{\CR_1(\l)},V^-_\R,V^-_{\CR_1(\l)}) $$

\begin{prop}\label{surj+} There exists a constant $\eta > 0$
depending on $u$ but independent of $a$ such that for $u=u_0+a\tau$
with $|a|<\eta \cdot\min\{e^{\tau_+}, e^{\tau_-}\}$,
$D_u\widehat\partial \oplus D\pi$ is surjective. Here $\tau_{\pm}$
are the asymptotic parameters of $u_0$.
\end{prop}
\begin{proof} In the previous section we have proved the surjectivity of $D_{u_0}\delbar$, where
$u_0$ is in $W^{1,p}_{\delta,(0;2,0)}(\dot \Sigma,\C ^n)$
satisfying $\delbar u_0=0$ with fixed asymptote $\gamma_{\pm}$ and
$\tau_{\pm}$. Here the solution of $\delbar u=a$ is given by
$u=u_0+a\tau$, so for small $a\in \C^n\backslash \{0\}$,
$u=u_0+a\tau$ is a small perturbation from $u_0$ in
$W^{1,p}_{\delta,(0;2,0)}(\dot \Sigma,\C ^n)$ by \eqref{Theta} and
\eqref{S}. Since surjectivity of $D_u\delbar$ is an open
condition, which is preserved under small perturbation from $u$,
for all $a$ with $|a|<\eta$, $D_u\delbar:
W^{1,p}_{\delta}(\dot\Sigma,u^*T\C^n) \to
L^p_{\delta}(\dot\Sigma,\Lambda^{0,1}(u^*T\C^n))$ is surjective,
where $\eta$ is a constant depending on $u_0$.

Especially, this implies $D_{(u,a)}\widehat\del$ is surjective
since we can let $h=0$ in \eqref{U-h}, and the target
$L^p_{\delta}(\dot\Sigma,\Lambda^{0,1}(u^*T\C^n))$, is still the
same.
\end{proof}

Next we describe the kernel of $\widehat\del$:
\begin{prop} For any $(u,a)$ satisfying $\widehat\del(u,a)=0$,
\bea
\operatorname{ker}(D_{(u,a)}\widehat\del)&
=&\{(U,V^{\pm}_{\R},V^{\pm}_{\CR_1(\l)},h) |D_u\delbar U-h=0\}\\
&\cong& \{ (\ker(D_u \delbar)+h\tau,V^{\pm}_{\R},
V^{\pm}_{\CR_1(\l)}, h) \} \label{dim} \eea
\end{prop}

\begin{cor} $\operatorname{Index}D_{(u,a)}\widehat\del
=\operatorname{dim}\operatorname{ker}D_{(u,a)}\widehat\del=6n +3$
\end{cor}
\begin{proof} From \eqref{dim}, we see
$$
\dim \operatorname{ker}D_{(u,a)}=\dim \operatorname{ker} D_u\delbar
+\dim T_a\C^n = 4n+3+2n = 6n+3.
$$
By Proposition \ref{surj+},
$D_{(u,a)}\widehat\partial$ is surjective so we get the index is
equal to the dimension of the kernel.
\end{proof}

On the moduli space $\CM_{(0;2,0)}^+(\dot\Sigma,\C^n;(1,1))$,
using the isomorphism in Corollary \ref{cor:ihlmd} we define the jet evaluation map
$$ev_{jet}: \CM_{(0;2,0)}^+(\dot\Sigma,\C^n;(1,1)) \to \C^n\setminus \{0\}, \;(u,a)\to a. $$
$ev_{jet}$ has the following geometric meaning:
For $u(\t,t) \in \CM_{(0;2,0)}^+(\dot\Sigma,\C^n;(1,1)) $,
let $I(\t)=\int_{S^1} u(\t,t) dt$ be the center of mass flow of $u(\t,t)$,
where the integration is with respect to the standard metric in $\C^n$,
then $I'(\t)=a=ev_{jet}(u)$ for any $\t$. (This is due to the mean value theorem of harmonic functions).
So $ev_{jet}(u)$ gives the direction that the local model $u$  is aligned to.

\begin{defn}[Enhanced nodal Floer trajectory]
We denote \index{$\CM^{nodal}([z_-,w_-],[z_+,w_+];(K,J),(f,J_0),p)$}
$$
\CM^{nodal}([z_-,w_-],[z_+,w_+];(K,J),(f,J_0),p): =
(ev_{jet})^{-1}([\nabla f(p)])
$$
and define
\bea\label{eq:nodalKf}
&{}& \CM^{nodal}([z_-,w_-],[z_+,w_+];(K,J),(f,J_0))\nonumber \\
& = & \bigcup_{p \in M}
\CM^{nodal}([z_-,w_-],[z_+,w_+];(K,J),(f,J_0),p).
\eea
We call
an element $(u_+,u_0,u_+)$ therein an \emph{enhanced nodal Floer trajectory} under
the back ground Morse function $f$.
\end{defn}

\section{Off-shell framework for the gluing}

We first define several function spaces to furnish the Banach
manifolds and bundles needed for the $\overline\partial_{(K,J)}:= (d
+ P_K)^{(0,1)}_J$ operator to become a smooth Fredholm section of an
appropriate infinite dimensional vector bundle. We summarize the
various moduli spaces relevant to this formulation :
\begin{enumerate}
\item For the moduli space $\CM^{SFT}_{(0;2,0)}(\C^n)$, this is the
Morse-Bott setting of the Symplectic Field Theory. We have followed
the description by Fukaya-Oh-Ohta-Ono \cite{fooo07} of the Fredholm
theory where a similar Morse-Bott setting of SFT but with Lagrangian
boundary condition was used. There was also given a Morse-Bott
set-up of the Fredholm theory of SFT by Bourgeois \cite{bourgeois}.

\item For the moduli space $\CM_{(1;1,0)}((K,J);z_-)$ or
$\CM_{(1;0,1)}((K,J);z_+)$, this is standard except the requirement
that the maps are \emph{immersed} at the origin.

\item For the moduli space of nodal Floer trajectories, it is
necessary to match the evaluation maps from (1) and (2). We will
introduce a cylindrical metric on a neighborhood of the nodal point
$p = u_-(o_-) = u_+(o_+)$ in $M$ for any element $u = (u_-,u_+)$
with $u_-\in \CM((K,J);[z_-,w_-])$ and $u_+ \in
\CM((K,J);[z_+,w_+])$, such that the evaluation map of (2) takes the
value in SFT setting. Geometrically this corresponds to blowing up
of the neighborhood of $p$ and then reparameterization by
$S^{2n-1}\times \mathbb{R}$. This setting is very similar to that in
\cite{fooo07} in Lagrangian surgery on a vertex of a holomorphic
triangle. However, instead of blowing up a given vertex in M, here
we need to do this reparametrization for a family of neighborhoods
depending on \emph{varying} $p$, so we need family of cylindrical
metrics fibered over $M$. The precise off-shell formulation is in
the subsequent subsection.

\item For the moduli space of ``disk-flow-disk", we need to
formulate a Fredholm theory for the objects which are allowed to
have dimension jump. When the length of the lines of
``disk-flow-disk" elements shrink to zero, there appears some
subtlety in the Fredholm theory since we encounter a noncompact
family of domains and suitable transversality is needed for such
family. However, this transversality issue can be reduced to a
finite dimensional differential topology lemma and is solved in
Section \ref{sec:smoothing-sls}.
\end{enumerate}

To apply the discussion on $\C^n$ to the 1-jet consideration of
$M$, we need a way of identifying $(T_xM,\omega_x,J_x)$ with the standard almost K\"ahler
structure on $(\C^n,\omega_0,J_0)$. Since this identification
depends on the point $x$, there is no canonical way of doing
this identification, especially when $J$ is non-integrable. One might try to adjust
$J$ so that it becomes integrable if the given point $x$ were a
fixed point. However for our purpose, we will need to provide this
identification at an \emph{unspecified} point and so changing the
given almost complex structure is not appropriate. Because of this,
we need to carry out this identification in a systematic way making
all the choices involved smoothly varying over $x \in M$.

It turns out the notion of \emph{Darboux family} introduced by
Weinstein \cite{alan:fixed} is particularly useful for the above
process.

\subsection{Darboux family and explosion of manifolds}
\label{subsec:Darboux}

In this section, we first recall the notion of \emph{Darboux family}
introduced by Weinstein \cite{alan:fixed} and then carry out the
explosion of manifolds of the Riemannian metric at a point to
produce a smooth family parameterized by $M \times [0, \e_0]$ for a
constant $\e_0$ depending only on the symplectic manifold
$(M,\omega)$.

For a symplectic manifold $M$, each tangent space $T_xM$ inherits
the structure of symplectic vector space with the symplectic
quadratic form $\omega_x$.

\begin{defn}\cite{alan:fixed}\label{darboux}
A \emph{Darboux family} is a family of symplectic diffeomorphisms
$I_x: V_x \to U_x$ such that
\begin{enumerate}
\item $V_x$ and $U_x$ are open neighborhoods of $0 \in T_xM$ and
$x \in M$ respectively
\item $I_x(0) = x$ and $dI_x(0) = id$
\item $I_x^*\omega = \omega_x$
\item $(V_x, U_x ; I_x)$ depends smoothly on $x$.
\end{enumerate}
\end{defn}

To emphasize the readers that $I_x$ plays the same kind of role as
the exponential map at $x$, we denote
$$
I_x =: \exp_x^I.
$$
When $(M,\omega)$ is equipped with a compatible almost complex
structure $J$ so that the triple $(M,\omega,J)$ defines an almost
K\"ahler structure, the above Darboux family automatically assigns
an \emph{almost} complex structure $I_x^*J$ on $V_x$. In addition to
(1), we can require the condition

\medskip
\noindent\hskip0.2in (5) $(I_x^*J)(0) = J_x$ on $T_0(T_xM) \cong
T_xM$.
\medskip

Now we can identify $(T_xM, \omega_x, J_x)$ with $\C^n$ by an
Hermitian isometry and denote by $B^{2n}(\e_0;J_x)$ as the standard
ball of radius $\e_0 > 0$. The ball $B^{2n}(\e_0;J_x) \subset T_xM$
does not depend on this identification but depends on
$(\omega_x,J_x;\e_0)$.

Note that when $M$ is compact, we can choose the family so that
there is $\e_0 > 0$ such that \be\label{eq:numux}
I_x(B^{2n}(\e_0;J_x)) \subset U_x \ee for all $x \in M$. We call any
such $\e_0 > 0$ an \emph{admissible radius} for the Darboux family.
We denote
$$
B(I, J;x,\e_0) = I_x(B^{2n}(\e_0;J_x)) \subset M.
$$
Since we will not change $J$ or $I$, we will simplify and just
denote $B(I, J;x,\e_0)$ by $B_{\e_0}(x)$ whenever there is no danger
of confusion.

Next we recall the \emph{explosion}  constructions of manifolds and
metrics from \cite{alan:explosion} in detail in Appendix \ref{sec:appendix}. We use
this construction in the context of almost K\"ahler structure.

Consider the pointed manifold $(M,p)$ for each $p \in M$ and denote
by $\pi_{E(M,p)}: E(M,p) \to M \times \R$ the explosion of $(M,p)$
at $p$. By construction,  $E(M,p)$ is  defined by beginning with the
product $M \times \R$, removing the ``axis'' $M \times \{0\}$, and
replacing it with the tangent space $T_pM$ at $p$ in $M$. The
differentiable structure on $E(M,p)$ is taken to be the usual
product structure on $M \times (\R \setminus \{0\})$. Charts near $M
\times \{0\}$ is defined with the aid of the above given Darboux
family of coordinates on $M$. We refer to \cite{alan:explosion} or
Appendix 8.4 for more precise details. This enables us to regard
$E(M,p)$ is a family of manifolds `exploding' at $p$ at the time $\e
= 0$. For $\e \neq 0$, the fiber $E_\e = E(M,p)_\e$ is diffeomorphic
to $M$ and for $\e = 0$, $E_0$ is diffeomorphic to the linear space
$T_pM$.

We now consider the explosions $E(M,p)$ as a family parameterized by
$p \in M$. We define
$$
E(M) = \bigcup_{p \in M} E(M,p) \to M
$$
and will provide a fiber bundle structure $E(M) \to M$ : It is
enough to provide compatible local trivializations thereof at each
$p \in M$. Let $U \subset M$ be a neighborhood of $p$ such that
$I^{-1}(U) = B^{2n}(r)$. Without loss of any generality, we will
assume $r = 1$. We will find a trivialization
$$
\Phi: E(M)|_U \cong U \times E(M,p)
$$
by defining diffeomorphisms $\varphi_{p'p}: E(M,p') \to E(M,p)$
depending smoothly on $p' \in U$. For this purpose, we will use
Corollary \ref{projection} in Appendix.

Using the fact that the open ball $\operatorname{Int}B^{2n}(1)$ is
two-point homogeneous under the action of M\"obius transformations
we can find a smooth family of diffeomorphism
$$
\varphi_{p'p}: (B^{2n}(1),\del B^{2n}(1)) \to (B^{2n}(1),\del
B^{2n}(1))
$$
which maps $p' \to p$ and is the identity on $\del B^{2n}(1))$.
Restricting $p' \in B^{2n}(1-\kappa)$ for a fixed small $\kappa >
0$, and suitably modifying the diffeomorphism on $\operatorname{Int
}(B^{2n}(1)) \setminus B^{2n}(1-\kappa)$ once and for all, we can
smoothly extend outside of $B^{2n}(1)$ by setting it to be the
identity. This defines a diffeomorphism $\varphi_{p'p}^I: (M,
\{p'\}) \to (M, \{p\})$ which is the identity outside $U$.
Furthermore one can easily arrange that as $p' \to p$, the map
$\varphi_{pp'} \to id$ in $C^\infty$-topology.

By making the above modification on $\operatorname{Int }(B^{2n}(1))
\setminus B^{2n}(1-\kappa)$ once and for all, the local
trivializations over different $U$ will be compatible and hence we
have shown that $E(M) \to M \times \R$ is locally trivial.

One can see this explosion process more vividly if we consider it in
the point of view of Riemannian manifolds. Let $g$ be a given metric
on $M$ and $0 < inj(g) <\infty$ be the injective radius of $g$. Fix
a constant  $\e_0$ that $0<\e_0<inj(g)$.  In our case, we will
consider the compatible metric $g = \omega(\cdot,J \cdot)$. We will
introduce a family of Riemannian metrics on $M$ for $\e
> 0$, denoted by $g_{\e_0,\e,p}$, in a way that the family satisfies
the following properties :

\begin{prop}\label{DGexplosion} There exists a family of Riemannian
metrics $g_{\delta,p}$ on $M$ for $\delta> 0$ such that
\begin{enumerate}
\item $g_{\delta, p} \equiv g$ for $\delta \geq 2\e_0$, \item The
fiberwise pull-back $\pi_{E(M,p)}^*(g_{\delta,p})|_{E_\delta}$ over
$M \times \R_+ \setminus \{0\}$ extends smoothly to $M \times \R_+$
by defining the metric $g_{0,p}$ on $E_0 \cong T_pM$ to be $g_{0,p}
= g(p)$.
\end{enumerate}
\end{prop}
\begin{proof}
In regard to the expression (\ref{eq:E(Phi)}) of the coordinate
chart applied to $Y = \{p\}$
$$
E(I)(x,\e) = (I(\e x),\e),
$$
we want the family $g_{\delta,p}$ to be defined by
$$
g_{\delta,p} = \frac{1}{\delta^2} I^*g
$$
for $\delta$ near 0. Then the pull-back metric of $g_{\delta,p}$ to
$E_{\delta}$ is given by
$$
\pi_{E(M,p)}^*g_{\delta,p}(x,\delta)|_{TE_\delta} =
\frac{1}{\delta^2} (E(I)\circ R_\delta)^*g(x,\delta)
$$
in the coordinate chart $E(I)$ on $E(U)$. Now a straightforward
calculation shows that this family has the coordinate expression as
$$
\pi_{E(M,p)}^*g_{\delta,p}(x,\delta)(v) = g(I\circ
R_\delta(x))(T_{R_\delta(x)}I(v))
$$
for $v \in TE_\delta$ at $(x,\delta)$ with respect to the canonical
coordinate $E(I)$ associated to the Darboux chart $I$ at $p$. From
this it follows that this family smoothly extends to $E(M,p)$ across
$\delta = 0$ if we set the metric $g_{0,p}$ to be $g(p)$ : Here we
use the condition $T_{0}I = id$ on $T_pM$.
\end{proof}

For any given $0 < \e \leq \e_0$, we can interpolate the scaled metric $g/{\e_0^2}$
and $g/\e^2$ on $B_{\e_0}(p)$ via cylindrical metric. More precise
description of the metrics is in order.

\begin{figure} [ht]
\centering
\includegraphics{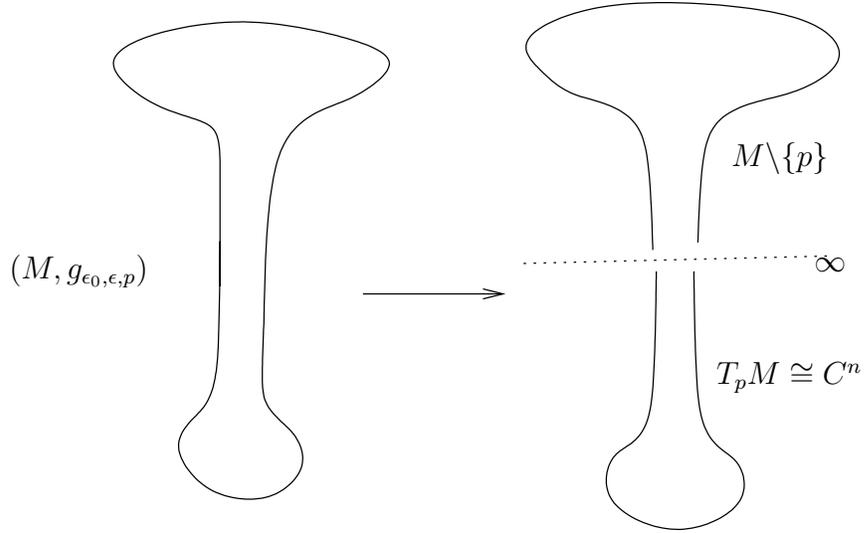}\caption{Explosion of a
manifold at $p$}\end{figure}

Using the Darboux family, for any $p\in M$, we define a 1-parameter family of metrics
$g_{\e_0,\e,p}$  on $M$ as the following:
\be
 g_{\e_0,\e, p}= \left\{
\begin{array}{lcl}
\frac{1}{{\e_0}^2}g  \quad   &x\in& M\backslash B_{\e_0}(p) \\
\r_+(x)\frac{1}{{\e_0}^2}g+ (1-\r_+- \r_-) \frac{I_*(g(p))}{(r\circ
(\exp_p^I)^-1)^2} \\
\hskip0.5in + \r_-(x)\frac{1}{\e^2}I_*(g(p)) \quad & x\in & B_{\e_0}(p)\backslash B_{\e}(p)  \\
\frac{1}{\e^2}I_*(g(p))  \quad & x\in& B_{\e}(p),
\end{array}
\right. \label{splitmetric}
\ee
where the $r$ is the radius function
on $(T_p M, g(p))$, the $\r_+(x)$ is a smooth cut function that
$\r_+(x)=1$ outside $B_{\e_0}(p)$ and $\r(x)=0$ in
$B_{\frac{9}{10}\e_0}(p)$, while $\r_-(x)$ is another smooth cut
function that $\r_-(x)=0$ outside $B_{\frac{11}{10}\e}(p)$ and
$\r(x)=1$ in $B_{\e}(p)$.

We equip $M\backslash \{p\}$ with of metric $g_{\e_0,p}$  making it
to be a manifold with one cylindrical end, where \be g_{\e_0, p}=
\left\{
\begin{array}{lcl}
\frac{1}{{\e_0}^2}g     \quad   &x\in& M\backslash B_{\e_0}(p) \\
\r_+(x)\frac{1}{{\e_0}^2}g+ (1-\r_+(x)) \frac{I_*(g(p))}{(r\circ
(\exp_p^I)^-1)^2} \quad & x\in & B_{\e_0}(p)
\end{array}
\right. \label{cylmetric} \ee

We also equip $T_pM$ with a metric $g_{cyl,\e,p}$  making it to be
a manifold with one cylindrical end as well, where

\be g_{cyl,\e, p}= \left\{
\begin{array}{lcl}
(1- \r_-(x)) \frac{g(p)}{r^2} + \r_-(x)\frac{1}{\e^2}g(p)
\quad & x\in & T_pM \backslash B_{\e}(p)  \\
\frac{1}{\e^2}g(p)  \quad & x\in& B_{\e}(p).
\end{array}
\right. \label{cyllmetric}
\ee
Clearly, the scaling map $\e: (T_pM, g_{cyl,1,p})\to
(T_pM, g_{cyl,\e,p})$, $v \mapsto v/\e$ is an isometry and $(T_pM, g_{cyl,1,p})$ is isometric to
$(\C^n, g_{cyl,1,0})$.

\begin{rem} We have the following observations:
\begin{enumerate}
\item The metric $g_{\e_0,\e,p}$ on $M$ is the interpolation of
the metrics $g_{\e_0,p}$ and $g_{cyl,\e,p}$.
\item $\lim_{\e\to 0}g_{\e_0,\e,p}=g_{\e_0,p}$ on $M\backslash \{p\}$.
\item  The
expression $\frac{I_*(g(p))}{(r\circ (\exp_p^I)^-1)^2}$ is simply
the push-forward of the cylindrical metric
$\frac{g}{r^2}=g_{\R\times S^{2n-1}}$ on $T_pM\backslash\{p\}$ to
$M$ by $\exp_p^I$.
\item The degenerating metric $g_{\e_0,\e,p}$
on $M$ given is non-collapsing as $\e \to 0$, and
$$
g_{\e_0,p}=\lim_{\e \to 0}g_{\e_0,\e,p} \;\qquad\text{on}
\;M\backslash\{p\}.
$$
in the Gromov-Hausdorff topology.
\end{enumerate}
\end{rem}

We note that $B_{\e_0}(p)\backslash B_{\e}(p)$ is identified with
$$
(0, \ln(\e_0)-\ln\e] \times S^{2n-1} \subset \R \times S^{2n-1}
$$
via the map $(r,\Theta) \mapsto (s,\Theta)$ with $s = \ln r$. In the
coordinates $(s,\Theta)$, any point $x \in B_{\e_0}(p)$ is identified to a
pair
$$
(s(x),\Theta(x))\in  (-\infty,0] \times S^{2n-1}.
$$
We call $(s,\Theta)$ the \emph{cylindrical coordinate chart}
near $p$.

Now we equip $E(M) \times M \times \R_+ \to M \times \R_+$ with the structure of Riemannian fibration
with its fibers given by
$$
\begin{cases}
(M, g_{\e_0,\e,p}) \quad & \mbox{for } \, (p,\e) \in M \times \R_+, \, \e \neq 0 \\
(T_pM, g(p)) \quad & \mbox{for } \, (p,\e) = (p,0)
\end{cases}
$$
This fibration over $\e > 0$ will host the off-shell Banach
manifolds for the resolved Floer trajectories arising from nodal
Floer trajectories, while the union
$$
E_0 \# (M \setminus \{p\}) = T_pM \# (M \setminus \{p\})
$$
regarded as the end-connected sum of two symplectic manifolds $E_0$
and $M\setminus \{p\}$: Here $E_0$ has \emph{convex} end and $M \setminus \{p\}$ has
\emph{concave} end both with the unit sphere $S^1(T_pM)$ as their asymptotic boundaries
in the Darboux chart $I_p$. This explosion $E(M) \times M \times \R_+ \to M \times \R_+$
will be implicitly used to define a Banach manifold that host the \emph{enhanced nodal Floer trajectories}.
We turn to the description
of these off-shell Banach manifolds in the next couple of
subsections.

\subsection{Off-shell formulation for perturbed $J$-holomorphic
discs}
\label{subsec:offshellpertdiscs}

The off-shell formulation of the moduli space
$\CM^{SFT}_{(0;2,0)}(\C^n)$ has been given in subsection 5.2. Here
we provide the off-shell setting for the perturbed $J$-holomorphic
disk moduli spaces $\CM((K^-,J^-);[z_-,w_-];A_-)$ and
$\CM((K^+,J^+);[z_+,w_+];A_+)$. Note that if the target is the
compact manifold $M$, then the off-shell Banach manifold hosting
perturbed $J$-holomorphic curves (or discs) is the standard
$W^{1,p}(\dot{\Sigma},M)$. That formulation was used in
``disk-flow-disk" sections.

However, to resolve enhanced nodal Floer trajectories by the scaled
gluing method, we need to blow up every point $p$ in $M$.
Therefore the off-shell Banach manifold should incorporate this
fact for which we need to consider a family of Banach manifold
which forms a (locally trivial) fiber bundle over the target manifold $M$.

We first introduce the Banach manifold $W^{1,p}_{\a}(\dot\Sigma,M;
p,z_+,\gamma_+,\tau_+)$ for each given $(p,,z_+,\gamma_+,\tau_+)$.
Here $\a=\a(\t)$ stands for the weighting function that
\be \a(\t)=
\begin{cases} e^{2\pi\d |\t|} \quad \t\le 0 \\ 1 \quad \qquad \t\ge 0 \end{cases}  \ee
The space $W^{1,p}_{\a}(\dot\Sigma,M; p,z_+,\gamma_+,\tau_+)$
consists of elements $u: (\dot\Sigma,o_+) \to (M,p)$ satisfying the following:
\begin{enumerate}
\item $u\in W^{1,p}_{loc}$
\item In the analytical chart of a positive puncture $e_+\in \Sigma$,
$\lim_{\tau\to +\infty} u(\tau,t)=z_+(t)$ for the periodic orbit
$z_+(t)$ in $M$.
\item For sufficiently large $\tau$, $u(\tau,t)=\exp_{z_+(t)}(\xi(\tau,t))$,
and  $\xi(\tau,t)\in L^p \big( [0,+\infty)\times S^1, z_+^*(TM)
\big)$
\item In the analytical chart $D_+\cong (-\infty,0]\times S^1$ of the marked point $o_+$,
 $u(\tau,t)$ is in a cylindrical coordinate
chart $B_{\e_0}(p)$ of $p$, with $u(\tau,t)=\big(
\Theta_+(\tau,t), s_+(\tau,t)\big)$, satisfying
$$
e^{\frac{2\pi\delta |\tau|}{p}} \| \Theta_+(\tau,t) -
\gamma_+(t)\|_{S^{2n-1}}
 \in W^{1,p}((-\infty,0] \times S^1, \R)
$$
$$
e^{\frac{2\pi\delta |\tau|}{p}}|s_+(\tau,t) - 2\pi(\tau - \tau_+)|
\in W^{1,p}((-\infty,0] \times S^1,\R)
$$
for the \emph{simple} Reeb orbit $\gamma_+(t)$ and $\tau_+\in
\mathbb{R}$.
\end{enumerate}
$W^{1,p}_{\a}(\dot\Sigma,M; p_-, z_-,\gamma_-,\tau_-)$ is defined
similarly, but the weight function $\a$ is replaced by $\a(-\t)$.

\begin{rem}  We only consider simple Reeb orbits because we have
chosen a generic $J$ so that the nodal point $p=u_-(o_-)=u_+(o_+)$
of any nodal Floer trajectory $(u_-,u_+)$ is immersed. Translating
this to the asymptote of $u_{\pm}$ in the cylindrical manifold
$M\backslash\{p\}$, we only get simple Reeb orbits in $S^{2n-1}$.
Therefore, to host such $u_{\pm}$ with immersed at $o_{\pm}$, the
function spaces $W^{1,p}_{\a}(\dot\Sigma,M;
p_{\pm},z_{\pm},\gamma_{\pm},\tau_{\pm})$ with simple Reeb orbits
$\g_{\pm}$ are adequate.

We denote the set of parameterized simple Reeb orbits $\g(t)$ in
$S^{2n-1}$ by $\widetilde{\mathcal R}_1(\l)$. By the Hopf fibration $S^{2n-1}\to
\C P^n$ we see $\widetilde{\mathcal R}_1(\l)\cong S^{2n-1}$, because given any
point in $S^{2n-1}$ to start, the passing $S^1$ fiber is a simple
Reeb orbit.
\end{rem}

\begin{rem}
Since there are only finitely many  nodal Floer trajectories
$(u_-,u_+)$ during gluing, we can assume the $\|\nabla
u_{\pm}(o_{\pm})\|$ ( gradient is with respect to the metric $g$ on
$M$ ) is uniformly bounded. Therefore by possibly shrinking the
cylindrical charts $O_{\pm}$, we can assume they all satisfy
$u_{\pm}(O_{\pm})\subset U_p$. Therefore, for $(\t,t)\in O_{\pm}$,
$u_{\pm}(\t,t)$ is in the cylindrical coordinate of the target $(
M\backslash \{p\}, g_{\d,p} )$ . This said, in cylindrical
coordinates of the domain and target,
$$
u_{\pm}:O_{\pm}\to B_{\e_0}(p)\cong (-\infty,0]\times
S^{2n-1},\;(\tau,t)\to (s_\pm(\tau,t),\Theta_\pm(\tau,t))
$$
has the asymptote
\bea \label{asymppt}
|\nabla^k(\Theta_{\pm}(\tau,t)-\gamma_{\pm}(t))|_{S^{2n-1}} &\le&
C_k e^{\frac{-2\pi c_k|\tau|}{p}}
\qquad\text{and}   \nonumber \\
|\nabla^k(s{\pm}(\tau,t)-2\pi(\tau-\tau_{\pm}))| &\le& C_k
e^{\frac{-2\pi c_k|\tau|}{p}} \label{asymp+}
\eea
for some constant $C_k$ and $c_k$, where $s_\pm = s \circ u_\pm$ and
$\Theta_\pm = \Theta \circ u_\pm$. The $C_k$ and $c_k$ can be made
uniform for all $u_{\pm}$ nearby the (finitely many) nodal Floer
trajectories $(u_-,u_+)$ by the continuity of the $\nabla^{k+1}
u_{\pm}$ translated into cylindrical coordinate. We chose $\d$ in
the definition of $W^{1,p}_{\a}(\dot\Sigma,M;
p_\pm,z_\pm,\gamma_\pm,\tau_\pm)$ to be less than $c_k (k=0,1,2)$.
\end{rem}
\medskip

Then we let \index{$W^{1,p}_{\a}(\dot\Sigma,M;z_\pm)$}
$$
W^{1,p}_{\a}(\dot\Sigma,M;z_+):= \bigcup_{p \in M}
\bigcup_{(\gamma_+,\tau_+) \in \widetilde{\CR}_1(\lambda) \times \R}
W^{1,p}_{\a}(\dot\Sigma,M,p, z_+,\gamma_+,\tau_+),
$$
where $\widetilde{\CR}_1(\l)$  is the set of all parameterized simple Reeb
orbits in $S^{2n-1}$.  So $\mathcal W^{1,p}_{\a}(\dot\Sigma,M;z_+) $ is the
space that hosts all $(K_-,J_-)$-holomorphic discs $u_+$ with
boundary on the periodic orbit $z_+(t)$ in $M$, and immersion at
$o_+$. The moduli space $ W^{1,p}_{\a}(\dot\Sigma,M;z_-)$
is defined similarly.

We will show $ W^{1,p}_{\a}(\dot\Sigma,M;z_+)$ is a Banach
manifold. First we describe the tangent space of a given element $u$
in $ W^{1,p}_{\a}(\dot\Sigma,M;z_+)$. Since
$$
\pi:  W^{1,p}_{\a}(\dot\Sigma,M;z_+)\to M
$$
is a fiber bundle with its fiber at $p \in M$ given by
$$
W^{1,p}_{\a}(\dot\Sigma,M,p, z_+) :=
\bigcup_{(\gamma_+,\tau_+) \in \widetilde{\CR}_1(\lambda) \times \R}
W^{1,p}_{\a}(\dot\Sigma,M,p, z_+,\gamma_+,\tau_+)
$$
we need to consider both the vertical and horizontal variations for $u$.
\medskip
Let  $\chi_+:(-\infty,0]\to [0,1]$ be a smooth function such that
$\chi_+(\tau)=1$ for $\tau\le-2$ and $\chi_+(\tau)=0$ for
$\tau\ge-1$. We consider the quadruple
$(U,V^+_{\widetilde{\CR}_1(\lambda)},V^+_{\mathbb{R}},v_+ )$ satisfying
\begin{enumerate}
\item $ V^+_{\widetilde{\CR}_1(\lambda)}\in T_{\gamma_+}\widetilde{\CR}_1(\l)$,
       $V^+_{\mathbb{R}}\in \mathbb{R}=T_{\tau_+}\mathbb{R}$, and
        $v_+\in T_{p}M$, where $p= u(o_+);$
\item $U \in W^{1,p}_{loc}((\dot\Sigma, u^*TM);$
\item $U \in
W_{\d}^{1,p}([0,+\infty)\times S^1, u^*TM)$, where
$[0,+\infty)\times S^1$ is the analytical chart for the positive
puncture $e_+\in \Sigma$;
\item In the analytical chart $D_+\cong (-\infty,0]\times S^1$ of
the marked point $o_+$, $u(\tau,t)$ is in the cylindrical chart of
$p\in M$. Let %
\beastar \widetilde U(\tau,t) & = & U(\tau,t) - \chi_+(\t)
Pal_{u(\t,t)} U(-\infty,t),
\eeastar %
then $ e^{\frac{2\pi \delta |\tau|}{p}}|\widetilde U(\tau,t)
|\in W^{1,p}((-\infty,0)\times S^1, \R) $. Here
$U(-\infty,t)= V^+_{\widetilde{\CR}_1(\l)}(t)$, and $Pal_{u(\t,t)}
U(-\infty,t)$ is the parallel transport of $U(-\infty,t)$ from
$u(-\infty,t)$ to $u(\t,t)$ along the minimal geodesic in
$(M,g_{\e_0,p})$.
\end{enumerate}

Let $C^0(u)$ be the set of all such quadruples. It becomes a Banach
space with the norm
\be
\Vert(U,V_{\widetilde{\CR}_1(\lambda)}^+,
V_{\R}^+,v_+)\Vert^p_{1,p,\a}  =  \left\Vert e^{\frac{2\pi\delta
|\tau|}{p}}\widetilde U(\tau,t) \right\Vert^p_{W^{1,p}}  {} +
|V_{\R}^+|^p + |V_{\widetilde{\CR}_1(\lambda)}^+|^p + |v_+| ^p.
\ee
Then it is standard to check $ W^{1,p}_{\a}(\dot
\Sigma,M;z_+)$ is a Banach manifold and
$$
C^0(u)=T_u  W^{1,p}_{\a}(\dot\Sigma,M;z_+).
$$
Similarly $W^{1,p}_{\a}(\dot \Sigma, M;z_-)$ is a Banach manifold.
The tangent vector
$$
(U,V_{\widetilde{\CR}_1(\lambda)}^-,
V_{\R}^-,v_-)\in T_u  W^{1,p}_{\a}(\dot \Sigma, M;z_-)
$$
is defined similarly, where $ V_{\widetilde{\CR}_1(\lambda)}^-\in T_{\gamma_-}\widetilde{\CR}_1(\l)$,
$V^-_{\mathbb{R}}\in \mathbb{R}=T_{\tau_-}\mathbb{R}$, and
$v_-\in T_{p}M$, $p= u(o_-)$.
The Banach norm is
\be
\Vert(U,V_{\widetilde{\CR}_1(\lambda)}^-,
V_{\R}^-,v_-)\Vert^p_{1,p,\a}  =  \left\Vert e^{\frac{2\pi\delta
|\tau|}{p}}\widetilde U(\tau,t) \right\Vert^p_{W^{1,p}}  {} -
|V_{\R}^+|^p + |V_{\widetilde{\CR}_1(\lambda)}^-|^p + |v_-| ^p.
\ee

\begin{rem}\label{Uv+} $U$ corresponds to the variation of $u$ within a fixed fiber
$M\backslash\{p\}$, and $v_+\in T_{p}M$ corresponds to the
variation of the fiber $M\backslash\{p\}$ in $\widetilde M$.
\end{rem}
\medskip
Let
\beastar
\mathcal{B}_+ =  W^{1,p}_{\a}(\dot\Sigma,M;z_+), \qquad \mathcal{L}_+ = \bigcup_{u\in \mathcal{B}_+}
L^p_{\a}(\dot\Sigma, \Lambda^{0,1}_{J^+}(u_+^*(TM))),
\quad\text{and} \\
\overline\partial_{(J^+,K^+)}:\mathcal{B}_+ \to \mathcal{L}_+,
\quad (u,p) \to (\delbar_{J^+} u + (P_{K^+})_{J^+}^{(0,1)}(u),p).
\eeastar

Then $\overline\partial_{(J^+,K^+)}$ is a section of the Banach
bundle $\mathcal{B}_+\to \mathcal{L}_+$, and the perturbed
$(J^+,K^+)$-holomorphic disk moduli space $\CM(J^+,K^+;z_+)$ can be
written as the zero set \index{$\CM(J^+,K^+;z_+)$}
$$
\CM_1(J^+,K^+;z_+)=(\overline\partial_{(J^+,K^+)})^{-1}(0).
$$
When we consider the moduli space with more topological restrictions
on $u$, say $\CM_1(K^+,J^+;[z_+,w_+];A_+)$,
we can accordingly restrict to \index{$\CB_+([z_+,w_+];A_+)$}
$$
\delbar_{(J^+,K^+)}: \CB_+([z_+,w_+];A_+)\to \CL_+([z_+,w_+];A_+),
$$
and get $\CM_1(K^+,J^+;[z_+,w_+];A_+)=(\delbar_{(J^+,K^+)})^{-1}(0)$. Here
$$
\CB_+([z_+,w_+];A_+)=\{u\in \CB_+|\; [u\#w_+]=A_+\}
$$
and so $\CB_+ = \CB_+(z_+)$ is decomposed into
$$
\CB_+ = \bigcup_{A_+} \CB_+([z_+,w_+];A_+).
$$
And
$$
\CL_+([z_+,w_+];A_+)= \CL_+|_{\mathcal{B}_+([z_+,w_+];A_+)} \to \CB_+([z_+,w_+];A_+)
$$
is the restriction of the bundle $\CL_+ \to \CB_+$ to $\CB_+([z_+,w_+];A_+)$.
We also note that $\CB_+([z_+,w_+];A_+)$ has the decomposition
$$
\CB_+([z_+,w_+];A_+) = \bigcup_{p \in M} \CB_+([z_+,w_+],p;A_+)
$$
that is a fiber bundle over $M$ with its fiber at $p \in M$ given by
$$
\CB_+([z_+,w_+],p;A_+) = \{(u,o_+) \mid \CB_+([z_+,w_+];A_+), \, u(o_+) = p\}.
$$
\medskip
We now study the linearization of $\delbar_{(J^+,K^+)} $.
First we describe the tangent space $T_{(u,p)} \CB_+([z_+,w_+];A_+)$. It
decomposes
$$
T_{(u,p)} \CB_+([z_+,w_+];A_+) = T_{(u,p)}^v \CB_+([z_+,w_+];A_+)
\oplus T_{(u,p)}^h \CB_+([z_+,w_+];A_+)
$$
into the vertical and horizontal components for the fibration
$\CB_+([z_+,w_+];A_+) \to M ; (u,p) \mapsto u$. Then we have the
canonical identification
$$
T_{(u,p)}^v \CB_+([z_+,w_+];A_+) = T_{(u)}^v \CB_+([z_+,w_+],p;A_+) \cong W^{1,p}_\a(u^*TM ; p,z_+)
$$
where $W^{1,p}_\a(u^*TM ; p,z_+)$ is the set of
$(U,V^+_{\CR_1(\lambda)},V^+_\R)$ satisfying the conditions given
right above Remark \ref{Uv+}.

On the other hand, the horizontal space is not canonically given
and so we will choose them by prescribing their fiber components
in the given trivialization of $(E(M)|_U \cong U \times E(M,p)$.
%These fiber components are described by the family of
%diffeomorphisms $\varphi_{pp'}$ used in our construction of local
%trivialization of $\widetilde M \to M$.
Take a small convex
neighborhood $U$ of $p$, and consider the parameterized line
$\g:[0,1] \to U$ with $\g(0)=p, \,\g(1)$ with constant speed. By
the discussion before Proposition  \ref{DGexplosion} we have
diffeomorphisms $\varphi_{p\g(s)}:(M,\{p\}) \to (M,\g(s))$. We abbreviate
$\varphi_{p\g(s)}$ by $\varphi_s$. Then the fiber component of the
horizontal lifting of $v \in T_{p}M$ at $u$ in this trivialization
is given by
$$
\frac{d}{ds}\Big|_{s=0} \varphi_s \circ u = : X_0\circ u
$$
where $X_s$ is the vector field generating the isotopy $\varphi_s$
to the direction of $v = \vec{pp'}$. Therefore using the local trivialization of
$\CB_+([z_+,w_+];A_+) \to M$ induced by the family of
diffeomorphisms $\varphi_{pp'}$, the horizontal lifting of $v \in
T_{p}M$ is precisely $(X_0\circ u,v)$.

Note that the set of these variations $\{(X_0\circ u,v\}$ defines an a
$2n$-dimensional subspace of $T_u \CB_+([z_+,w_+];A_+)$ isomorphic
to $T_p M$. We denote this subspace by $\widetilde T_p M \subset
\CB_+([z_+,w_+];A_+)$.

Now we are ready to derive the formula for the linearization. When
the variational vector field $U$ is tangent to a fixed target
$M\backslash\{p\}$, the linearization
$D_{u}\delbar_{(J^+,K^+)}(U)$ at $u$ is computed in a standard
way.
\begin{lem}
We have
$$
D_{u}\delbar_{(J^+,K^+)} U \in
L^p_{\a}(\dot\Sigma,\Lambda^{0,1}_{J^+}(u^*(TM))).
$$
\end{lem}
\begin{proof} In the cylindrical end $(-\infty,0] \times
S^{2n-1}$ in $M\backslash\{p\}$, the vector field
$$
U-\widetilde U= \chi_+(\t) Pal_{u(\t,t)} U(-\infty,t)
$$
is asymptotically $J_p$-holomorphic: One way to see this is the
following: identify $(-\infty,0] \times S^{2n-1}$ to
$\C^n\backslash\{0\}$ and regard $u(\t,t)$ in
$\C^n\backslash\{0\}$. then the push forward of $Pal_{u(\t,t)}
U(-\infty,t)$ is very close to the vector field
$$
V_{\R}^+u(z)+ 2\pi e^{2\pi(\t-\t_+ +\sqrt{-1}t)}V_{\CR_1(\l)}(0)
$$
in $\C^n$ when $\t$ is negative enough. Furthermore we also have
$\lim_{\t\to-\infty}u(\t,t)=p$, and $J(u(\t,t))\to J_{p}$.
Therefore we have
$$
|D_{u}\delbar_{(J^+,K^+)} (U-\widetilde U)|\le C e^{-c|\t|}
$$
and so
$$
U=(U-\widetilde U)+\widetilde U \in W^{1,p}_{\a}((-\infty,0]\times S^1, \R).
$$
This finishes the proof.
\end{proof}

When the variational vector field is induced by a change of base
point $p$ in $(M,p)$ in the direction of $v\in T_{p}M$, it is
given by the one whose fiber component of the induced variational
vector field at $u$ is given by $X_0 \circ u$ where $X_s$ is the
vector field generating the isotopy $\varphi_s$. We denote by
$X_v$ the $X_0$ associated to $v \in T_pM$. Of course the
component in $T_{p}M$ is just $v$. Note that the set of these
variations $\{(X_{v} \circ u,v)\}$ defines an a $2n$-dimensional
subspace of $T_{(u,p)} \CB_+([z_+,w_+];A_+)$ isomorphic to $T_p
M$. We denote this subspace by $\widetilde{T_p M} \subset
T_{(u,p)}\CB_+([z_+,w_+];A_+)$ and $(X_{v} \circ u,v): = \widetilde v$.
With this choice, obviously we have the decomposition
\beastar
T_{(u,p)} \CB_+([z_+,w_+];A_+) & = & T_{u} \CB_+([z_+,w_+],p;A_+)
\oplus \widetilde{T_{p}M} \\
& = & W^{1,p}_\alpha(u^*TM;p,z_+) \oplus \widetilde{T_pM}.
\eeastar
Now the linearization of the section $\delbar_{(J^+,K^+)}:
\CB_+([z_+,w_+];A_+)\to \CL_+([z_+,w_+];A_+)$ at $(u,p)$ along
$$
\widetilde v \in \widetilde{T_p M} \subset T_{(u,p)}\CB_+([z_+,w_+];A_+)
$$
is given by $(D_{u}\delbar_{(J^+,K^+)}(X_{v_+}\circ u), v_+)$ in the
above mentioned trivialization of $T_{(u,p)} \CB_+([z_+,w_+];A_+)$.

A straightforward calculation gives rise to the following

\begin{lem}\label{linearization} We have the formula
\beastar
D_u \overline\del_{(J^+,K^+)}(X_{v^+}\circ u)
& = & (u^*\nabla)^{(0,1)}(X_{v^+}\circ u) + T^{(0,1)}
(du, X_{v^+}\circ u) \\
&{}& \quad + DP_{K^+}(u)^{(0,1)}(X_{v^+}\circ u)
\eeastar
where $T$ is the torsion tensor of the almost complex connection
$\nabla$ and the $(0,1)$-parts are taken with respect to $J^+$.
\end{lem}

Combining all these, we have obtained

\begin{prop} For $(u,p) \in \CM_1(K_-,J_-;[z_+,w_+];A_+)$,
$$
D_{(u,p)}\delbar_{(J^+,K^+)}: T_{(u,p)}  W^{1,p}_{\a}(\dot\Sigma,M;z_+)
\to L^p_{\a}(\dot\Sigma, \Lambda^{0,1}_{J^+}(u^*TM))\oplus T_{p}M
$$
is a Fredholm operator with
$$\operatorname
{index}D_u\delbar_{(J^+,K^+)}=n-\mu([z_+,w_+])+2c_1(A_+)
$$
\end{prop}

Next we prove the following transversality result of the section
$\delbar_{(J^+,K^+)}$.

\begin{prop} For generic $J^+\in \CJ_{\omega}M$ and any
$(J^+,K^+)$-holomorphic curve $u$,
$$
D_{(u,p)}\delbar_{(J^+,K^+)}: T_{(u,p)}
 W^{1,p}_{\a}(\dot\Sigma,M;z_+) \to L^p_{\a}(\dot\Sigma,
\Lambda^{0,1}_{J^+}(u^*TM)) \oplus TM
$$
is surjective.
\end{prop}
\begin{proof} We first consider $M\backslash\{p\}$ for a fixed $p$,
and a fixed asymptote $(\g_+,\t_+)$ in the cylindrical metric on
$M \setminus \{p\}$ near the point $p$. For the linearization of
\beastar \delbar_{(\cdot,K^+)} & : &
W^{1,p}_{\a}(\dot\Sigma,M;p,z_+,\g_+,\t_+)\times \CJ_{\omega}
\longrightarrow \\
&{}& \quad \bigcup_{u\in
W^{1,p}_{\a}(\dot\Sigma,M;p,z_+,\g_+,\t_+)} \bigcup_{J^{+}\in
\CJ_{\omega}} L^p_{\a}(\dot\Sigma, \Lambda^{0,1}_{J^+}(u^*TM) ),
\eeastar
standard argument shows that the map
\be
\label{fibsurj}D_u\delbar_{(\cdot,K^+)}: T_{(u,J^+)}(
W^{1,p}_{\a}(\dot\Sigma,M;p,z_+,\g_+,\t_+)\times \CJ_{\omega}) \to
L^p_{\a}(\dot\Sigma, \Lambda^{0,1}_{J^+}(u^*TM) )
\ee
is surjective for any given $(J^+,K^+)$-holomorphic curve $u \in
W^{1,p}_{\a}(\dot\Sigma,M;p,z_+,\g_+,\t_+)$.

Now we enlarge the domain of $\delbar_{(\cdot,K^+)}$ from
$W^{1,p}_{\a}(\dot\Sigma,M;p,z_+,\g_+,\t_+)\times \CJ_{\omega}$ to
$ W^{1,p}_{\a}(\dot\Sigma,M;z_+)\times \CJ_{\omega}$, i.e.
$$
\delbar_{(\cdot,K^+)}: W^{1,p}_{\a}(\dot\Sigma,M;z_+)\times \CJ_{\omega} \to \left(\bigcup_{u\in
 W^{1,p}_{\a}(\dot\Sigma,M;z_+) } \bigcup_{J^+\in
\CJ_{\omega}} L^p_{\a}(\dot\Sigma,
\Lambda^{0,1}_{J^+}(u^*TM))\right) \oplus TM .
$$
Then for any given $(J^+,K^+)$-holomorphic curve $u \in
 W^{1,p}_{\a}(\dot\Sigma,M;z_+)$,
$$
D_u\delbar_{(\cdot,K^+)}:T_{(u_,J^+)}
( W^{1,p}_{\a}(\dot\Sigma,M;z_+)\times \CJ_{\omega} ) \to
L^p_{\a}(\dot\Sigma, \Lambda^{0,1}_{J^+}(u^*TM) )\oplus TM
$$
is surjective because \eqref{fibsurj} is surjective.

Now we consider the projection p:
$ W^{1,p}_{\a}(\dot\Sigma,M;z_+)\times \CJ_{\omega} \to
\CJ_{\omega}$. Then by Sard-Smale theorem, for any generic $J^+\in
\CJ_{\omega}$, specifying to $J_+$ for the above parameterized
family of maps $D_u\delbar_{(\cdot,K^+)}$,
$$
D_u\delbar_{(J^+,K^+)}: T_u  W^{1,p}_{\a}(\dot\Sigma,M;z_+)
\to L^p_{\a}(\dot\Sigma, \Lambda^{0,1}_{J^+}(u^*TM) ) \oplus
TM
$$
is surjective.
\end{proof}

To prepare for the next subsection,  we define the \emph{1-jet evaluation map} for
$(u,o_+) \in \CM_1([z_+,w_+];A_+)$.
Recall that in the cylindrical coordinate chart near $p=u_+(o_+)$,
we use the embedding
$$
\frac{1}{\d}(\exp_x^I)^{-1}: B_{\e_0}(p)\backslash \{p\}\to
(T_pM,J_p)\backslash \{p\} \cong (\C^n,J_0)\backslash\{0\}\cong \R
\times S^{2n-1}
$$
to express $u=(s,\Theta)\subset \R \times S^{2n-1}$ with the
asymptotes $\Theta(\tau,t)\to \gamma_+(t) $ and $s(\tau,t)\to
2\pi(\tau-\tau_+)$. We define the tangential evaluation
map
$$
ev_+^\#: W^{1,p}_{\a}(\dot\Sigma,M;z_+ ) \to \R \times \CR_1(\lambda),
\quad u\to (\t_+,\g_+).
$$
All the discussion above in this section hold for
$\CM_1(K^-,J^-;[z_-,w_-];A_-)$ and $ W^{1,p}_{\a}(\dot \Sigma, M;z_-)$
without change of proofs.

Here the $\a=\a(\t)$ is a similar (but {\em different}) weighting function as
before:
\be \a(\t)=\begin{cases} 1 \qquad \quad \t\le 0 \\
e^{2\pi|\t|} \quad \t\ge 0
\end{cases}
\ee
We abuse the notation $\a$ and the norm $\|\cdot\|_{1,p,\a}$ for
$u_-$ and $u_+$.

The tangential evaluation map of $u\in
W^{1,p}_{\delta}(\dot\Sigma,\widetilde M;z_-)$ is defined similarly
up to the sign of $\t_-$:
$$
ev_-^\#: W^{1,p}_{\d}(\dot\Sigma,M;z_- ) \to \R \times \CR_1(\lambda), \quad u\to
(\t_-,\g_-).
$$
\subsection{Off-shell formulation of enhanced nodal Floer trajectories}
\label{subsec:off-nodal}

Now we are ready to define the Banach manifold hosting the enhanced
nodal Floer trajectories. For notation brevity, we have set
$T_xM\backslash\{0\}=T_x^+M$ and $TM\backslash o_M=T^+M$, where $o_M$ is
the zero section of $TM$. For all
$x$ in $M$, identifying each $(T_xM,J_x)$ to $(\C^n,J_0)$, we get a
family of inhomogeneous local models \be
\CM_{(0;2,0)}^+(\dot\Sigma,T_xM) \cong
\CM_{(0;2,0)}(\dot\Sigma,T_xM)\oplus T_x^+M, \ee and the
corresponding Banach manifolds hosting them
\index{$\mathcal{B}_x$} \be \mathcal{B}_x:=
W^{1,p}_{\delta,(0;2,0)}(\dot \Sigma,T_xM)\oplus T_x^+M, \ee and the
Banach bundles
\be \mathcal{L}_x:=\bigcup_{u\in \mathcal{B}_x}
L^p_{\delta}(\dot\Sigma,\Lambda^{0,1}(u^*T(T_xM)). \ee We emphasize
that in defining $W^{1,p}_{\delta,(0;2,0)}(\dot \Sigma,T_xM)$, the
metric $h$ in the linear space $T_xM$ is cylindrical,  like the one
we defined in $\C^n$ in section \ref{sec:models}: we let
$h(z)=\l(z)g(x)$, where $g(x)$ is the original Riemannian metric on
$T_xM$, $\l:T_pM \to \R_+$ is the same radial function as in
\ref{sec:models}, such that $\l(z)=\frac{1}{|z|_g^2}$ when $|z|_g$
is sufficiently large.

\begin{figure}[ht] \centering \includegraphics{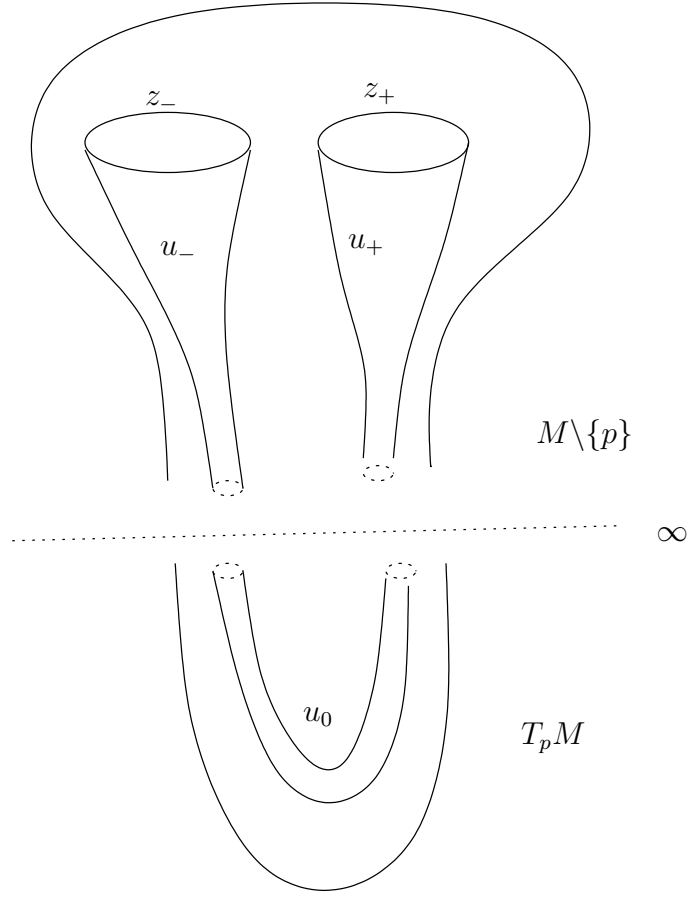}
\caption{The enhanced nodal Floer trajectory}
\end{figure}

The union of these Banach manifolds, which hosts all local models in
different $T_xM$, is \index{$\CB_{lmd}$}
\be \CB_{lmd}=\bigcup_{x\in M}\CB_x:= \bigcup_{x \in M}
W^{1,p}_{\delta,(1;2,0)}(\dot \Sigma,T_xM)\oplus T^+_xM  \ee
and the corresponding Banach bundle is
$$\mathcal
L_{lmd}=\bigcup_{x\in M} \mathcal L_{x}.
$$
Here $T^+_xM$ is the summand encoding the vector $\vec a = \nabla
f(x)$ which is not a zero vector since we assume that the node of
the nodal Floer trajectories occur outside the set of critical
points of the back-ground Morse function $f$. We also note that the
group
$$
\operatorname{Aut}_{lmd}(x): = (T_xM \times (\R \setminus \{0\}))
\times \R
$$
acts on $W^{1,p}_{\delta,(1;2,0)}(\dot \Sigma,T_xM)$ where the
factor $T_xM$ corresponds to the translations on $T_xM$, $\R
\setminus \{0\}$ corresponds to multiplication by non-zero real constant
on $T_xM$ and the last $\R$-factor corresponds to the domain
$\tau$-translations. This action also induces an action on $\CL_{lmd}$.

We let \be \pi: \bigcup_{x \in M} \CM_{(0;2,0)}^+(\dot\Sigma,T_xM)
\to M \ee be the projection to the base $M$, and the symplectic
field theory evaluation map
$$
ev^{SFT}=(ev_-^{SFT}, ev_+^{SFT}):\CM_{(0;2,0)}^+(\dot\Sigma,T_xM)
\to \CR_1(\lambda) \times_{\Delta_M} \CR_1(\lambda),
$$
$$
u_0\to \big(\gamma_{0-}, \gamma_{0+}\big),
$$
where $\g_{0\pm}$ are the asymptotic data of the local model $u_0$. Here we use the
fiber product $\CR_1(\lambda) \times_{\Delta_M} \CR_1(\lambda)$ to emphasize
that $\gamma_-$ and $\gamma_+$ lie in the same $T_pM$, where $p=\pi(u_0)$.
(We recall from subsection \ref{subsec:Darboux} and the paragraph
right therebefore that we have made the identification of $(T_pM,J_p,\omega_p)$ and
$\C^n$ using the Darboux family $I$.)

We form the Banach manifold hosting enhanced nodal Floer
trajectories via the fiber product of
$$
\pi_\Theta \circ ev^\#_-\times \pi_\Theta\circ ev_{SFT} \times
\pi_\Theta ev^\#_+: \CR_1(\lambda) \times (\CR_1(\lambda)\times_{\Delta_M}
\CR_1(\lambda)) \times \CR_1(\lambda)
$$
$$
(u_-,u_0,u_+)\to (\pi_\Theta\circ ev^\#_-(u_-),ev_{SFT}(u_0),
\pi_\Theta\circ ev^\#_+(u_+)):
$$

Let $\Delta_{\CR_1(\lambda)}\subset \CR_1(\lambda)\times \CR_1(\lambda) $ be the
diagonal. Then we define \index{$\CB_{nodal}$}
\bea%
\CB_{nodal}&:=&\left(\pi_\Theta
\circ ev^\#_-\times ev_{SFT} \times \pi_\Theta\circ
ev^\#_+\right)^{-1} (\Delta_{\CR_1(\lambda)}\times_{\Delta_M}
\Delta_{\CR_1(\lambda)}) \nonumber\\
&=&\{(u_-,u_0,u_+)\in \CB_-\times \CB_{lmd}\times \CB_+|\;
u_-(o_-)=u_+(o_+), \nonumber \\
& & \quad \pi_\Theta \circ ev^\#_-(u_-)= ev_{SFT}^-(u_0),
\pi_\Theta\circ ev^\#_+(u_+)=ev_{SFT}^+(u_0) \}
\eea%
 to be the
Banach manifold. Due to the action of $(T_xM \times (\R \setminus
\{0\})) \times \R$ on $\CB_{lmd}$, the same group acts on
$\CB_{nodal}$.

From the matching condition it is clear that for any
$u=(u_-,u_0,u_+)\in \CB_{nodal}$, its tangent space is
\bea \label{BNodal}%
T_u\CB_{nodal}&=&\{(\xi_-,\xi_0,\xi_+)\in T_{u_+}\CB_+\times
T_{u_0}\CB_{lmd} \times T_{u_-}\CB_-|  \nonumber\\
& & \qquad v_-=v_+=v_0=v, V^{\pm}_{\CR_1(\l)}=V^{0\pm}_{\CR_1(\l)}\}.
\eea%
where we have the expressions
\beastar \xi_{\pm} & = &
(U_{\pm}, V^{\pm}_{\CR_1(\l)},V^{\pm}_{\R},v_{\pm}
)\in T_{u_{\pm}}\CB_{\pm} \\
\xi_0 & = & (U_0, V^{0-}_{\CR_1(\l)},
V^{0-}_{\R},V^{0+}_{\CR_1(\l)}, V^{0+}_{\R},v_0)\in
T_{u_0}\CB_{lmd},
\eeastar and the $v_0\in T_pM$  correspond to the
variation of the base point $p$.

\medskip

We have a natural $\operatorname{Aut}_{lmd}(x)$-equivariant section
\beastar
\delbar_{(J,K,f)}&:& \CB_{nodal}\to \mathcal
L_- \times \mathcal L_{lmd} \times \mathcal L_+,\\
& & (u_-,u_0,u_+,p) \to
\left(\delbar_{(J^-,K^-)}u_-,\;\delbar_{(J_p,f)}u_0,\;
\delbar_{(J^+,K^+)}u_+\right),
\eeastar
where $p=u_{\pm}(o_{\pm})$ is the nodal point, and
$\delbar_{(J_p,f)}u_0=\delbar_{J_p}u_0-\nabla f(p)d\t +J_p\nabla
f(p)dt$.

If we put more topological restrictions on $u_-$ and $u_+$, namely
if we let
$$
\CB_{nodal}([z_-,w_-],[z_+,w_+];A_-,A_+)=\{(u_-,u_0,u_+)\in
\CB_{nodal} |\; [u_{\pm}\#w_{\pm}]=A_{\pm} \},
$$
then similarly we have the section
\bea \label{nodaloffshell}
\delbar_{(J,K,f)}:\CB_{nodal}([z_-,w_-],[z_+,w_+];A_-,A_+)\longrightarrow
\nonumber \\
\qquad \mathcal L{([z_-,w_-];A_-)} \times \mathcal L_{lmd} \times \mathcal
L{([z_+,w_+];A_+)},
\eea
and the moduli space of enhanced nodal
Floer trajectories with the background Morse function $f$ defined in
Subsection \ref{subsec:sphere-line-sphere} can be interpreted as
$$
\CM^{nodal}([z_-,w_-],[z_+,w_+];A_-,A_+;(K,J),(f,J_0))
=(\delbar_{(J,K,f)})^{-1}(0)
$$
from \eqref{nodaloffshell}.

\begin{prop}\label{nodalregular} For generic $J^-$ and $J^+$ in $\CJ_{\o}$, any enhanced
nodal Floer trajectory $u=(u_-,u_0,u_+)\in
\CM^{nodal}([z_-,w_-],[z_+,w_+];A_-,A_+;(K,J),(f,J_0))$ is
regular, in the sense that %
\beastar D_{(u,p)}\delbar_{(J,K,f)}& :&
T_{(u,p)}\CB_{nodal}([z_-,w_-],[z_+,w_+];A_-,A_+) \longrightarrow \\
&{}& \quad \mathcal L([z_-,w_-];A_-)\times \mathcal L_{lmd}\times
\mathcal L([z_+,w_+];A_+)\oplus T_pM \eeastar  %
is surjective.
Consequently, there exists a right inverse
$Q^\e|_{\e=0}$ for $D_{(u,p)}\delbar_{(J,K,f)}$:%
\bea \label{Q3}%
Q^\e|_{\e=0}&: & \quad \mathcal L([z_-,w_-];A_-)\times \mathcal
L_{lmd}\times \mathcal L([z_+,w_+];A_+)\oplus T_pM \longrightarrow \nonumber \\
& & T_{(u,p)}\CB_{nodal}([z_-,w_-],[z_+,w_+];A_-,A_+).
\eea %
\end{prop}

Before the proof of the proposition, we state a lemma concerning the
linearization of the inhomogeneous local model equation with respect
to the variation $\widetilde v = (0,v) \in T_{(u_0,p)}\CB_{lmd}$
which is the horizontal lifting of $v \in T_pM$ as constructed
before for the curves $u_\pm$.

\begin{lem} For any $(u_0,v) \in \CM_{(0;2,0)}^+(\dot\Sigma,T_xM)$ satisfying
$$
\delbar_{(J_p,f)}u_0=\delbar_J u_0-\nabla f(p)d\t+J \nabla f(p)dt=0,
$$
and corresponding to the variation $\widetilde v =
(X_v\circ u_0,v)$ in the trivialization constructed by the family
$\varphi_{pp'}$ with $v \in T_pM$, we have
$$
(D_{(u_0,p)}\delbar_{(J_p,f)}(\widetilde v))^v \in \CL_{lmd}.
$$
Consequently, the map
$$
D_{(u_0,p)}\delbar_{(J_p,f)}: T_{u_0}\CB_{lmd} \to \mathcal L_{lmd}
\oplus T_pM
$$
is surjective.
\end{lem}

Now we go back to the proof of the Proposition \ref{nodalregular}:
\begin{proof} Let $u=(u_+,u_0,u_-)$ be an enhanced nodal Floer
trajectory with the nodal point $p$. In subsection
\ref{subsec:offshellpertdiscs}, we have proved that for generic
$J^{\pm}\in\CJ_{\o}$,
$$
D_{(u_{\pm},p)}\delbar_{(J^{\pm},K^{\pm})}:T_{(u_{\pm},p)}\CB([z_{\pm},w_{\pm}];A_{\pm})
\to L^p_\d(\dot\Sigma,\Lambda^{0,1}_{J_{\pm}} (u_{\pm}^*(TM)))
\oplus T_pM
$$
is surjective. By Proposition 5.4, for generic $J^{\pm}\in\CJ_{\o}$
we have that \be \label{tradg}ev_- \times
ev_+:\CM_1(J^-,K^-;[z_-,w_-];A_-)\times \CM_1(J^+,K^+;[z_+,w_+];A_+),
\ee
$$(u_-,u_+)\to (u_-(o_-),u_+(o_+))$$
is transversal to $\Delta_M \subset M\times M$. Therefore, for any
$\eta_{\pm}\in L^p_\d(\dot\Sigma,\Lambda^{0,1}_{J_{\pm}}
(u_{\pm}^*(TM)))$, there exist $
\xi_{\pm}=(U_{\pm},V^{\pm}_{\CR_1(\l)},V^{\pm}_{\R},v_{\pm}) \in
T_{u_{\pm}}\CB([z_{\pm},w_{\pm}];A_{\pm}) $ such that
$$
 D_{u_{\pm}}\delbar_{(J^{\pm},K^{\pm})}\xi_{\pm}=\eta_{\pm},\qquad \text{  and   } \quad v_-=v_+:=v.
 $$

Then for any $\eta_0\in \Lambda^{0,1}_{J_p}(u_0^*(TT_pM))$, by the
transversality of the local models  in $\C^n\cong T_pM$, for $u_0$
in $T_pM$, there exists
$$
\xi_0=(U_0, V^{0-}_{\CR_1(\l)},V^{0+}_{\CR_1(\l)})\in T_{u_0}\CB_p
$$
such that $D_{u_0}\delbar_{(J_p,f)}\xi_0=\eta_0$. Hence for
$(\xi_0,v_0)\in T_{(u_0,v_0)}\CB_{lmd} $ with $v_0=v$,
$$
D_{u_0}\delbar_{(J_p,f)}(\xi_0,v_0)=D_{u_0}\delbar_{J_p,f}\xi_0=\eta_0.
$$
The proposition immediately follows from the surjectivity of
$D_{u_\pm}\delbar_{(J_\pm,K_\pm)}$, $D_{u_0}\delbar_{(J_p,f)}$ and
the transversality of the evaluation map
$$
\pi_\Theta\circ ev_-^\# \times \pi_\Theta \circ ev_+^\# \to
\CR_1(\lambda) \times \CR_1(\lambda)
$$
to the diagonal $\Delta_{\CR_1(\lambda)}$. For any
$\eta=(\eta_-,\eta_0,\eta_+)$, using the above construction we
define $Q^\e|_{\e=0}$ is defined to be
$Q^\e|_{\e=0}(\eta)=(\xi_-,\xi_0,\xi_+,v)$. Then $Q^\e|_{\e=0}$ is
a right inverse of $D_{(u,p)}\delbar_{(J,K,f)}$ and the
proposition follows.
\end{proof}

\begin{rem} For the purpose of constructing approximate right inverse later,
we define the operators
$Q_-$, $Q_0$ and $Q_+$ by
$$ Q_-\eta_- :=\xi_-, \quad Q_0\eta_0 :=\xi_0, \quad Q_+\eta_+ :=\xi_+ $$
for the $(\eta_-,\eta_0,\eta_+)$ and $(\xi_-,\xi_0,\xi_+)$ above.
\end{rem}

\subsection{Appendix : Explosion} \label{sec:appendix}
\medskip

In this appendix, we collect various facts on the so called,
\emph{explosion} construction of manifolds. We verbatim follow the
exposition in the smooth context given by Weinstein in section 4
\cite{alan:explosion}.

Let $Y \subset X$ be a submanifold. The \emph{explosion} of $X$
along $Y$, denoted by $E(X,Y)$, is defined by beginning with the
product $X \times \R$, removing the ``axis'' $X \times \{0\}$, and
replacing it with the normal bundle $N(X,Y) = T_YX/ TY$ to $Y$ in
$X$. The differentiable structure on $E(X,Y)$ is taken to be the
usual product structure on $X \times (\R \times \{0\})$. Charts near
$X \times \{0\}$ are defined with the aid of local coordinates on
$X$. A more precise description of local charts near $X \times
\{0\}$ is in order.

Let $X$ have dimension $n$ and $Y$ have dimension $k$. We abbreviate
the coordinates $(x_1,\cdots, x_n)$ on $\R^n$ by $(y,z)$, where $y =
(x_1,\cdot, x_k)$ and $z = (x_{k+1},\cdots, x_n)$. Suppose that
$\Phi$ is a submanifold fold chart defined on an open subset $U
\subset X$ at $p \in Y \subset X$, i.e., a chart for the pair
$(X,Y)$ defined on an open subset $\UU$ of $(\R^n,\R^k)$ which is
invariant under the retraction $(y,z) \mapsto (y,0)$. The
corresponding chart $E(\Phi)$ for $E(X,Y)$ is defined on the open
subset $\{(y,z',\e) \mid (y,\e z') \in \UU\}$ of $\R^{n+1}$ by

\be\label{eq:E(Phi)} E(\Phi)(y,z',\e) = (\Phi(y,\e z'),\e) \ee

for $\e \neq 0$ with $E(\Phi)(y,z',0)$ defined as the projection of
the tangent vector $T_{(y,0)}\Phi(0,z')$ into the normal bundle
$N(X,Y)$. The following theorem is proved by Weinstein
\cite{alan:explosion}.

\begin{thm}[Lemma 4.3, \cite{alan:explosion}]\label{extend}
Let $f: (X,Y) \to (Z,W)$ be a smooth mapping. Then it uniquely
induces a smooth mapping $E(f): E(X,Y) \to E(Z,W)$ such that
\begin{enumerate}
\item $E(f)$ extends to the restriction $f:X \setminus Y \to
Z \setminus W$,
\item When $\Phi$ and $\Psi$ are local charts of $(X,Y)$ and
$(Z,W)$ respectively and the local representative $\Psi^{-1}f\Phi$
is written as $\Psi^{-1}f\Phi = (g,h): (\R^n,\R^k) \to
(R^m,\R^\ell)$, then the local representation $E(g,h)$ of $E(f)$
with respect to the charts $E(\Phi)$ and $E(\Psi)$ is given for
$\e\neq 0$ by the ``partial difference quotient''
$$
E(g,h)(y,z',e) = (g(y,\e z'),(1/\e)h(y,\e z'),\e)
$$
and for $\e = 0$ by the normal derivative
$$
E(g,h)(y,z',0) = (g(y,0), (\del h/\del z')(y,0),0).
$$
\end{enumerate}
\end{thm}

Two immediate consequences are also derived in
\cite{alan:explosion}.

\begin{cor}[Theorem 4.4, \cite{alan:explosion}]\label{naturality}
Assuming to each pair $(X,Y)$ the exploded manifold $E(X,Y)$ with
the differentiable structure described above defines a covariant
functor from the category of pairs of manifolds to the category of
manifolds over $\R$.
\end{cor}

\begin{cor}\label{projection} The identity map on $X \times (\R \setminus \{0\})$
extends to a unique smooth mapping from $E(X,Y)$ to $X \times \R$.
The restriction of this mapping to $N(X,Y)$ is the bundle projection
onto $Y \times \{0\}$.
\end{cor}

This corollary defines a canonical smooth projection map $E(X,Y) \to
X \times \R$, which we denote by $\pi_{E(X,Y)}$.

\part{Analysis : Scale-dependent gluing and compactification}

\section{Smoothing of nodal Floer trajectories I ; to
`disk-flow-disk'}
\label{sec:smoothing-sls}

The disk-flow-disk moduli spaces
$$
\CM^{para}([z_-,w_-];f;[z_+,w_+];A_{\pm}), \quad
\CM^{\e}([z_-,w_-];f;[z_+,w_+];A_{\pm})
$$
have been defined
in section $4$. Recall that the moduli space of ``disk-flow-disk"
elements of \emph{flow time $\e$} is
\begin{multline} \nonumber
\CM^{\e}([z_-,w_-];f;[z_+,w_+];A_{\pm}):=\{(u_-,\chi,u_+)\mid
u_{\pm}\in \CM(K^{\pm},J^{\pm};\vec z_{\pm};A_{\pm}),\\ \chi: [0,\e]\to
M, \dot{\chi}-\nabla f(\chi)=0, \; u_-(o_-)=\chi(0),\;
u_+(o_+)=\chi(\e) \}
\end{multline}

We give $\CM^{\e}([z_-,w_-];f;[z_+,w_+];A_{\pm}) $ another
interpretation through evaluation maps. This point of view is more
suitable for analyzing the transition from nodal Floer trajectories
to ``disk-flow-disk" elements. Consider the deformed evaluation map
\be
\label{defevt} \phi_f^{\e} ev_- \times ev_+:
\CM_1([z_-,w_-];A_-)\times \CM_1([z_+,w_+];A_+)\to M \times M
\ee
$$
(u_-,u_+)\to (\phi_f^{\e} u_-(o_-), u_+(o_+)),
$$
where $\phi_f^{\e}:M\to M$ is the time-$\e$ flow of the Morse
function $f$. Then it is easy to see
\begin{multline}
\CM^{\e}([z_-,w_-];f;[z_+,w_+];A_{\pm})=\{ (u_-,\chi,u_+) \mid \\
(u_-,u_+)\in (\phi_f^{\e} ev_- \times ev_+)^{-1}(\Delta),\;
\chi(\t)=\phi_f^{\t}u_-(o_-)\text{ for } 0\le \t \le \e \}
\end{multline}
Using the above interpretation, if
\be \label{diagtr} \phi_f^{\e}
ev_- \times ev_+ \text{ is transversal to } \Delta \subset M \times M
\ee and if
$D_{u_{\pm}}\delbar_{(K_{\pm},J_{\pm})}$ are surjective, then by the
inverse function theorem for \eqref{defevt},
$\CM^{\e}([z_-,w_-];f;[z_+,w_+];A_{\pm})$ is a manifold of
dimension
\beastar
& &\operatorname{dim} \CM^{\e}([z_-,w_-];f;[z_+,w_+];A_{\pm})\\
&=& \operatorname{dim}\CM([z_-,w_-];A_-)+
\operatorname{dim}\CM([z_+,w_+];A_+)-2n\\
&=& (n+\m_{CZ}([z_-,w_-])+2c_1(A_-))+ (n-\m_{CZ}([z_+,w_+])+2c_1(A_+))-2n\\
&=& \m_{CZ}([z_-,w_-])-\m_{CZ}([z_+,w_+])+2c_1(A_-)+2c_1(A_+),
\eeastar
if $\e> 0$ is sufficiently small.
\medskip

{\bf Assumption:} For transversality argument, from now on we assume
the critical points of $f$ do not coincide with the nodal points of
nodal Floer trajectories, this can be achieved by a generic $f$.
\medskip

Since we mainly care about the \emph{transition} from disk-flow-disk
elements to resolved nodal Floer trajectories during the PSS
cobordism, only the disk-flow-disk elements with short-time flows will be
considered. So we fix a sufficiently small $\e_0>0$ which is to be determined
later, and consider the ``disk-flow-disk" moduli spaces
$\CM^{\e}([z_-,w_-];f;[z_+,w_+];A_{\pm})$ with $0\le \e \le
\e_0$.  The $\e=0$ case corresponds to the moduli space of nodal
Floer trajectories.

\begin{lem}\label{shortreg} Suppose that the almost complex structures $J^{\pm}$ are
generically chosen so for any nodal Floer trajectory $(u_-,u_+)$,
$D_{u_{\pm}}\delbar_{(K_{\pm},J_{\pm})}$ are surjective, and
$u_-(o_-)$ and $u_+(o_+)$ are immersed points as in Theorem
\ref{intro-immersed}. Then there exists $\e_0>0$, such that for any
$(u_-,\chi,u_+)\in \CM^{\e}([z_-,w_-];f;[z_+,w_+];A_{\pm})$
where $\e\in [0,\e_0]$, the above property is preserved, i.e.
$D_{u_{\pm}}\delbar_{(K_{\pm},J_{\pm})}$ are surjective, and
$u_-(o_-)$ and $u_+(o_+)$ are immersed points.
\end{lem}
\begin{proof} We prove that $u_-(o_-)$ and $u_+(o_+)$ are immersed
points. Otherwise, there exist $\e_i\to 0$, and
$(u_-^i,\chi_i,u_+^i)\in
\CM^{\e_i}(K^{\pm},J^{\pm};[z_-,w_-],f,[z_+,w_+];A_{\pm})$, such that at
least one of $u_-^i(o_-)$ and $u_+^i(o_+)$ is not an immersed point.
Passing to a subsequence we may assume, say $du_-^i(o_-)$, is $0$
for any $i$. Since the energy of any curves $u_{\pm}^i$ is uniformly
bounded due to the boundary condition,  we can take a subsequence
again and get a limiting nodal curve $(u_-^{\infty},u_+^{\infty})$
by Gromov-compactness. The images of $\chi_i$ converge to the nodal
point. No bubbling can occur on $u_-^{\infty}$ or $u_+^{\infty}$,
because if a bubble occurs on $u_{+}^{\infty}$ or $u_{+}^{\infty}$,
then by the semi-positive condition and the genericity of $J^{\pm}$,
we can resolve the bubble to get a at least two dimensional family
of nodal Floer trajectories, contradicting with the rigidity
assumption on nodal Floer trajectories. Therefore, the curves
$(u_-^i,u_+^i)$ converge to $(u_-^{\infty},u_+^{\infty})$ in $C^1$
topology. This implies
$du_-^{\infty}(o_-)=\lim_{i\to\infty}du_-^i(o_-)=0$, contradicting
with the immersion condition at the nodal point.

Since $u_{\pm}(o_{\pm})$ are immersed points, $u_{\pm}$ are
somewhere injective. Then the genericity of $J^{\pm}$ implies that
$D_{u_{\pm}}\delbar_{(K_{\pm},J_{\pm})}$ are surjective.
\end{proof}

To complete the PSS cobordism from ``disk-flow-disk" configurations
to nodal Floer trajectories, we will build a collar neighborhood of
$\CM^{0}([z_-,w_-],f,[z_+,w_+];A_{\pm})$ in
$\CM^{para}([z_-,w_-];f;[z_+,w_+];A_{\pm}) $; More
precisely, for some $\e_0>0$, we will construct a differentiable map
$$
G: \CM^{0}([z_-,w_-],f,[z_+,w_+];A_{\pm})\times [0,\e_0)\to
\CM^{para}([z_-,w_-];f;[z_+,w_+];A_{\pm})
$$
such that for each $\e\in [0,\e_0)$,
$$
G_{\e}:=G(\cdot,\e):\CM^0([z_-,w_-],f,[z_+,w_+];A_{\pm}) \to
\CM^{\e}([z_-,w_-];f;[z_+,w_+];A_{\pm})
$$
is a diffeomorphism. This problem is reduced to the following finite
dimensional differential topology lemma:

\begin{lem} \label{dft} $X, Y, Z$ are differentiable manifolds, and
only $X$ may have boundary. $X$ is compact, $Z$ is a differentiable
submanifold in $Y$, and $I$ is an interval containing $0$. Let
$\Phi: X \times I\to Y$ be a differentiable map. Denote
$\Phi_{\e}:=\Phi(\cdot,\e)$ for $\e\in I$. If $\Phi_0: X\to Y$ is
transversal to $Z$ with nonempty intersection, and $\Phi_0(\partial
X)\cap Z=\emptyset$, then
\begin{enumerate}\item
There exists $\e_0>0$, such that for any $\e \in [0,\e_0]$,
$\Phi_{\e}$ is transversal to $Z$;
\item Furthermore, there exits a differentiable map $G:
\Phi_0^{-1}(Z)\times [0,\e_0] \to \Phi^{-1}(Z)$, such that for any
$\e \in [0,\e_0]$, $G_{\e}:=G(\cdot,\e)$ gives a diffeomorphism from
$\Phi_0^{-1}(Z)$ to $\Phi_{\e}^{-1}(Z)$.
\end{enumerate}
\end{lem}
\begin{proof} Since $X$ is compact, the compact-open topology in the function space
$C^1(X,Y)$ coincides with the strong topology $C^1_S(X,Y)$ as in
\cite{Hirsch}. Since the set of maps transversal to $Z$ in
$C^1_S(X,Y)$ is open, by the condition of $\Phi$ and $\Phi_0$ we
conclude that for $\e_0$ sufficiently small,
\be \label{all-transversal}\text{ for all } \e\in [0,\e_0],
\Phi_{\e} \text{ is transversal to } Z.\ee
Since $\Phi_{\e}(X)\cap Y \neq\emptyset$, and $\Phi_0(\partial
X)\cap Y=\emptyset$, the pre-image of all intersections lies in
$\text{int}(X)$, where $ \text{int}(X)$ is the interior of $X$.
Therefore for $\e_0$ sufficiently small, for all  $\e\in [0,\e_0]$
\be\label{nonept}\Phi_{\e}(X)\cap Y \neq\emptyset \ee
by analyzing the local behavior of intersections. We also have
$\Phi_{\e}(\partial X)\cap Y=\emptyset$, using the compactness of
$\partial X$ and the continuous dependence of $\Phi_{\e}$ on $\e$.

Clearly \eqref{all-transversal} implies $\Phi: X\times [0,\e_0]\to
Y$ is transversal to $Z$. So $\Phi^{-1}(Z):=W$ is a  differentiable
submanifold in $X\times [0,\e_0]$. Actually it is in $
\text{int}(X)\times [0,\e_0]$, because $\Phi(\partial X\times
[0,\e_0])\cap Z=\emptyset$. So $\partial W\subset X\times
\{0,\e_0\}$.

Note the following elementary fact during the proof of parameterized
transversality in \cite{Hirsch}:
\be\label{fact} \Phi_{\e} \text{ transversal to } Y
\Longleftrightarrow \e \text{ is a regular value of } \pi: W\to
I,\ee
where $\pi: X\times I \to I$ is the natural projection. Then we have
a submersion $\pi: W\to [0,\e_0]$ by translating
\eqref{all-transversal} via \eqref{fact}. By \eqref{nonept} $\pi$ is
surjective. Picking any metric on $W$, then the gradient vector
field $\nabla \pi$ never vanishes on $W$. Let the time-$\t$ flow of
the gradient vector field to be $\varphi^{\t}_{\pi}$. By Morse
theory we have the diffeomorphism $\varphi^{\e}_{\pi}:
\pi^{-1}(0)\to \pi^{-1}(\e)$ for all $\e\in [0,\e_0]$, using that
$\partial W\subset X\times \{0,\e_0\}$. Noting that
$\pi^{-1}(\e)=\Phi_{\e}^{-1}(Z)$, the map
$G:=\varphi^{(\cdot)}_{\pi}:\Phi_{0}^{-1}(Z)\times [0,\e_0]\to
\Phi^{-1}(Z)$ is desired.
\end{proof}

First we derive
\begin{cor} \label{tcor}For given generic $f$ and $J$, there exists a constant $\e_0>0$,
such that for all $(u_-,\chi,u_+)\in
\CM^{\e}([z_-,w_-];f;[z_+,w_+];A_{\pm})$ where $\e\in
(0,\e_0]$, the linearized operator $E(u)$ in Section
\ref{subsec:sphere-line-sphere} is surjective.
\end{cor}
\begin{proof} In the above lemma, take $X=\overline\CM_1([z_-,w_-];A_-)\times
\overline\CM_1([z_+,w_+];A_+)$, $Y=M\times M$, $Z=\Delta$ and
$\Phi:X\times I \to Y $ to be
$$\phi_f^{\e} ev_- \times ev_+:
\CM([z_-,w_-];A_-)\times \CM([z_+,w_+];A_+)\to M \times M
$$
$$(u_-,u_+)\to (\phi_f^{\e} u_-(o_-), u_+(o_+)),$$
which smoothly extends to $\overline\CM_1([z_-,w_-];A_-)\times
\overline\CM_1([z_+,w_+];A_+)$. Then $X$ is a compact manifold,
$\Phi$ is a differentiable map, and  $\Phi_0=ev_-\times ev_+$ is
transversal to $Z=\Delta$ by our assumption on $J^{\pm}$. Then by
the above Lemma \ref{dft}, the condition \eqref{diagtr} can be
achieved for all $\e\in [0,\e_0]$.

For given generic $J^{\pm}$, Lemma \ref{shortreg} says
$D_{u_-}\delbar_{(K_-,J_-)}$ and $D_{u_+}\delbar_{(K_-,J_-)}$ are
surjective for $(u_-,\chi,u_+)\in
\CM^{\e}([z_-,w_-];f;[z_+,w_+];A_{\pm})$ where $\e\in
[0,\e_0]$. Combining the condition \eqref{diagtr}, by Proposition
\ref{prop:dfdindex}, the corollary follows.
\end{proof}

Then we prove the central result of this section
\begin{prop} \label{mprop} For given generic $f$ and $J$, there exists a constant
$\e_0>0$ and a differentiable map
$$
G: \CM^0([z_-,w_-],f,[z_+,w_+];A_{\pm})\times [0,\e_0]\to
\CM([z_-,w_-],f,[z_+,w_+];A_{\pm}),
$$
such that for any $\e\in [0,\e_0]$,
$$
G_{\e}:=G(\cdot,\e): \CM^0([z_-,w_-],f,[z_+,w_+];A_{\pm})\to
\CM^{\e}([z_-,w_-];f;[z_+,w_+];A_{\pm}),
$$
is a diffeomorphism.
\end{prop}
\begin{proof} We take $X, Y, Z$ and  $\Phi$ the same as the above
corollary. Then all conditions in lemma \ref{dft} hold except the
condition $\Phi_0(\partial X)\cap Z=\emptyset$. We show this
condition also holds. Otherwise, we can find $(u_-,u_+) \in
\partial(\overline\CM_1([z_-,w_-];A_-)\times
\overline\CM_1([z_+,w_+];A_+)) $, such that
$\Phi_0((u_-,u_+))=(ev_-\times ev_+) (u_-,u_+) \in \Delta$. In
other words, $u_-(o_-)=u_+(o_+)$, and at least one of $u_-$ and
$u_+$ is in the compactified space
$\overline\CM_1(\vec z_{\pm};A_{\pm})$, say $u_+\in
\overline\CM_1([z_+,w_+];A_+))$. Then $u_+$ must contain some
bubble. This is impossible because it contradicts with the rigidity
assumption of nodal Floer trajectories, as explained in the proof of
Lemma \ref{shortreg}.

Then we apply part (2) of Lemma \ref{dft} and get the desired map
$G$.
\end{proof}

From Proposition \ref{mprop}, we see the moduli space
$$
\CM^{para} = \bigcup_{\e\in [0,\e_0)}
\CM^{\e}([z_-,w_-];f;[z_+,w_+];A_\pm)
$$
is a one dimensional manifold with boundary
$\CM^{0}([z_-,w_-];f;[z_+,w_+];A_\pm)$.

\begin{rem} There is a slight cheating in the proof of Corollary
\ref{tcor} and Proposition \ref{mprop}: The
$X:=\overline\CM_1([z_+,w_+];A_+))\times
\overline\CM_1([z_-,w_-];A_-)$ is not really a compact manifold
with boundary. However, for small $\e_0$, we can show for all $\e\in
[0,\e_0]$, $\Phi_{\e}(X-\text{int}(X))\cap \Delta=\emptyset$,
this is by the same argument as in Lemma \ref{shortreg}. Then we can
shrink $X$ a bit to $X^{shr}$, where $X^{shr}$ is a compact manifold
with boundary, and $\Phi_{\e}(X-X^{shr})\cap\Delta=\emptyset$ for
all $\e\in[0,\e_0]$. Then we can replace $X$ by $X^{shr}$ and apply
Proposition \ref{mprop}.
\end{rem}

\section{Smoothing of nodal Floer trajectories II ; to Floer trajectories}
\label{sec:gluing}

In this section, we will carry out the gluing of the perturbed
$J$-holomorphic curves $u_{\pm}$ and the local model curve and
produce $\e$-dependent one-parameter family of resolved Floer
trajectories.

Let's recall the domains of these curves. The domain of $u_+$ is a
punctured Riemann surfaces $\dot{\Sigma}_{+}$ with a puncture $e_+$
and a marked point $o_+$, where $\Sigma_{+}\cong S^2$, and
$\dot{\Sigma}_+\cong \C$. Similarly for the domain $\dot{\Sigma}_-$
of $u_-$. Let \beastar
E_{-}  = \{(\tau,t)|(\tau,t)\in (-\infty,0]\times S^1\} && O_{-}  = \{(\tau,t)|(\tau,t)\in [0, +\infty)\times S^1 \} \\
E_{+}  = \{(\tau,t)|(\tau,t)\in [0,+\infty)\times S^1\} && O_{+}  =
\{(\tau,t)|(\tau,t)\in (-\infty, 0]\times S^1 \} \eeastar be the
analytic charts on $\Sigma_{\pm}$ around the punctures $e_{\pm}$ and
marked points $o_{\pm}$ respectively. Then $z = e^{2\pi(\tau + it)}$
is the given analytic coordinates near $e_+$ and $o_-$, and $z =
e^{-2\pi(\tau + it)}$ is the analytic coordinates near $e_-$ and
$o_+$. Note that the $(\t,t)$ in different charts are different
local coordinates, but to keep the notation simple we still denote
them by the same variables $(\t,t)$. We have
$$
u_{\pm}:\dot{\Sigma}_{\pm}\to (M,\omega,J), \quad
u_{\pm}(o_{\pm})=p.
$$
We remark that the analytic charts $E_{\pm},O_{\pm}$ are unique up
to  $\t$-translation and $t$-rotation.

On the  $E_{\pm}$ and $O_{\pm}$ we put the metric
$g_{(-\infty,0]\times S^1}$ or $g_{[0,+\infty)\times S^1}$ in the
obvious way. We extend the metric to the remaining  part of
$\dot{\Sigma}_{\pm}$ in any way, and then fix it.

For the gluing purpose, we need to consider $3$ metrics on the
manifold $M$: the original metric $g$, the Darboux-cylindrical
metric $g_{\delta,p}$ and the degenerating metric $g_{\delta,\e,p}$.
The definitions  of these metrics are in order:

Let $\delta>0$ be a fixed number less than the injective radius of
$(M,g)$. Assume $\delta$ is so small that for every $p$ in $M$,
$B_{\delta}(p)$ is contained in a Darboux neighborhood $U_p$ of $p$.
Then
$$
\frac{1}{\delta}(\exp_p^I)^{-1}: B_{\delta}(p) \to B_1(0) \subset
(T_pM,g_p)\cong (\C^n,g_{st}).
$$
Via the diffeomorphism, $\R \times S^{2n-1}  \cong
\C^n\backslash\{0\}; (s,\Theta) \mapsto (e^s\Theta)$, we can pull
back the standard metric on $\R \times S^{2n-1}$ to define the
metric on $B_{\delta}(p)\backslash \{p\}$ such that it is isometric
to $S^{2n-1}\times (-\infty,0]$.

\subsection{Construction of approximate solutions }
\label{subsec:approxsol} Given any nodal Floer trajectory
$(u_-,u_+)$, from \eqref{asymppt} $u_{\pm}$ has the asymptote
\bea
  |\nabla^k(\Theta_{\pm}(\tau,t)-\gamma_{\pm}(t))|_{S^{2n-1}} &\le&
C_k e^{\frac{-2\pi c_k|\tau|}{p}}
\qquad\text{and}   \nonumber  \\
|\nabla^k(s{\pm}(\tau,t)-2\pi(\tau-\tau_{\pm}))| &\le& C_k
e^{\frac{-2\pi c_k|\tau|}{p}}, \label{asymp+} \eea
where $s_\pm = s \circ u_\pm$ and $\Theta_\pm
= \Theta \circ u_\pm$. Here the number $p>2$ shouldn't be confused
with the point $p$ on $M$.

Let $f : M \to \R$ be a given Morse function. We choose $f$ so that
$\|f\|_{C^2}$ is sufficiently small. In particular, we assume
\be\label{eq:|nablap|} |\nabla f| \leq 1 \ee which can be always
achieved by rescaling $f$.

Given the nodal Floer trajectory $(u_-,u_+)$ with the nodal point
$p$, we construct a \emph{normalized local model curves} $u_0$ in
the following way :
\begin{lem}\label{normalizedmodel}
$u_0$ defines a proper map and satisfies \beastar u_0 : \R\times S^1
&\to& (T_pM,J_p)\cong
(\C^n,J_{st})\\
(t,\tau) &\to& (\Theta_0(\tau,t),s_0(\tau,t))\in \R \times
S^{2n-1} \qquad \eeastar when $|\tau|$ large, and satisfies the
followings :
\begin{itemize}
 \item $\frac{\partial u_0}{\partial {\bar z}}= \nabla f(p)$;
 \item In the cylindrical end of $\C^n$, it has the same asymptote as
$u_{\pm}$ in each of its ends, in the sense that
\bea e^\frac{{2\pi\delta|\tau|}}{p}
|\Theta_{0}(\tau,t)-\gamma_{\pm}(t)|_{S^{2n-1}} &\in&
W^{1,p}(O_{\pm})
\qquad\text{and}     \nonumber \\
e^\frac{{2\pi\delta|\tau|}}{p}
|s_{0}(\tau,t)-2\pi(\pm\tau-\tau_{\pm})| &\in& W^{1,p}(O_{\pm}),
\label{asymp} \eea
where
$$
O_+\cong  (-\infty,0]\times S^1, \quad O_- \cong [0,+\infty)\times
S^1
$$
are the cylindrical charts of the ends $\{\pm\infty\}\times S^1$ in
$\R\times S^1$ respectively,
\item $\int_{S^1}u_0(0,t)dt=0$.
\end{itemize}
Such $u_0$ is unique.
\end{lem}
\begin{proof} All the properties are immediate consequences of
the expression of the  model curves $u(z)=\vec A z+\vec B/z +\vec C
+\vec a \t $ given in (\ref{eq:uz}). Here in \eqref{eq:uz} we take
$\vec A=e^{-2\pi\t_+}\g_+(0),\; \vec B=e^{-2\pi\t_-}\g_-(0)$ and
$\vec a=\nabla f(p)$. We only comment on the last two properties.
For the last one, we have only to choose $\vec C = 0$ in
(\ref{eq:uz}). On the other hand,  for the second property, we use
the fact ${\t}/{e^z} \to 0$ as $\t \to \infty$ and so the
contribution of $\nabla f(p) \tau $ is negligible compared to $A
e^{2\pi (\tau + it)} + B e^{-2\pi (\tau + it)}$. (Detailed
calculation was carried out in section \ref{sec:ihmodels} ). For the
uniqueness of $u_0$, notice that $w_0:=u_0-\nabla f(p)\t$ is a
holomorphic function from $S^1\times \R$ to $\C^n$, and on
$o_{\pm}\in D_{\pm}$, $w_0$ can only have simple pole because the
$\Theta$ component of $w_0$ converges to simple Reeb orbits
$\g_{\pm}(t) \subset S^{2n-1}$. Therefore, the Laurent series of
$w_0(z)$ must be $w_0(z)=\vec A z+\vec B/z +\vec C$ for some
constant vectors $\vec A, \vec B$ and $\vec C$. Since $w_0$ has the
same asymptote as $u_0$, the  $\vec A, \vec B$ and $\vec C$
coincides with the ones given in the beginning of the proof.
\end{proof}

\begin{rem} From the above lemma we see $u_0$ can be explicitly given as
\be \label{normal0} u_0(\tau, t)=\vec A z+ \vec B/ z + \vec{a}\tau,
\ee where $\vec A=e^{-2\pi\t_+}\g_+(0), \vec B=e^{-2\pi\t_-}\g_-(0),
\vec{a}=\nabla f(p)$ and $z=e^{2\pi(\tau+it)}$. We call $u_0$ the
\emph{normalized local model}, because $\vec C$ has been normalized
to zero.

From the expression of $u_0$, and the definition of $\vec A$ and
$\vec B$,  we get
 \be
 |\nabla^k(u_0(\tau,t)-e^{2\pi(\tau-\tau_+
)}\gamma_+(t))|\le C_k e^{\frac{-2\pi c_k(\tau- \tau_+ )}{p}}, \;
\tau>0. \label{asymp0} \ee in the \emph{cylindrical metric}
$|\cdot|$ in $\C^n$ for some constants $C_k$ and $c_k$. Similar
result holds for another end of $u_0$ when $\tau<0$. Note the
convergence \eqref{asymp0} is stronger than our original requirement
\eqref{asymp}, because $\d$ is chosen to be smaller than the least
$c_k$.

In Theorem \ref{immersed}, we have proved that for generic $J$, for
any nodal Floer trajectory $(u_-,u_+)$, the $[du_-(o_-)]$ and
$[du_+(o_+)]$ are linearly independent. Consequently, $\vec A$ and
$\vec B$ are linearly independent in $\C^n$ for the normalized local
model $u_0$ sitting in $T_pM$, where $p=u_+(o_+)=u_-(o_-)$ is the
node. From the linear independence of $\vec A$ and $\vec B$ we get
\be\label{eq:normalized} \min_{t \in S^1} |u_0(0,t)| \ge b>0 \ee for
some constant $b$.
\end{rem}

We consider the \emph{scaled local model curve}  \be
\label{scaledmodel} u^{\e}_0:=\e u_0=\e(\vec A z+ \vec B/ z +
\vec{a}\tau ) \ee
From the asymptote \eqref{asymp0} of $u_0$, we derive \be
|\nabla^k(u_0^{\e}(\tau,t)-e^{2\pi(\tau-\tau_+ -
2R(\e))}\gamma_+(t))|\le C_k e^{-2\pi c_k(\tau- \tau_+ - 2R(\e))},
\; \tau>0. \label{asympe} \ee Here  $$R(\e)=-\frac{1}{4\pi}\ln \e.$$
Similar result holds for another end of $u_0^{\e}$ when $\tau<0$.

\begin{lem} \label{shrink} Consider the
scaled local model curve $u_0^\e$ chosen in \eqref{scaledmodel}. For
any given $0<\a<2$, there exists $\delta_\e > \e$ such that
\be\label{eq:deltae} \delta_\e \to 0, \quad \delta_\e/\e \to \infty
\ee and
$$
u_0^\e([-\a R(\e),\a R(\e) ] \times S^1) \subset B_p(\delta_\e)
\subset M.
$$
\end{lem}
\begin{proof}
We have when $\t\to +\infty$,
$$
|u_0^\e(\a R(\e),t)| \sim \e\cdot |\vec A| \cdot
e^{2\pi\cdot(-\a\frac{1}{4\pi}\ln\e)} = |\vec A|
\e^{1-\frac{\a}{2}}.
$$
Similar result holds for the other end when $\t\to -\infty$. So the
choice
$$
\delta_\e = \e^{1-\frac{\a}{2}}
$$
will do our purpose.
\end{proof}

We choose different cylindrical coordinates near the marked point
$o_\pm$ of $u_{\pm}$ and get the re-parametrization of the outer
curves there
$$
u_{+}^{\e}=u_{+}(\tau - 2R(\e),t) \quad\text{and}\quad
u_{-}^{\e}=u_{-}(\tau + 2R(\e),t).
$$
We compare the asymptote of $u_0^{\e}$ and $u_{+}^{\e}$ for $\tau$
in the range of $[R(\e) -1, R(\e) + 1]$; It turns out that they get
close exponentially as $\e \to 0$: In $[R(\e) -1, R(\e) + 1]$, by
\eqref{asymp+}
\bea
&{}& |\nabla^k(u_+^{\e}(\tau,t)-e^{2\pi(\tau-\tau_+ -2R(\e))}\gamma_+(t))| \nonumber \\
&=&|\nabla^k(u_+(\tau-2R(\e),t)-e^{2\pi(\tau-\tau_+ -2R(\e))}\gamma_+(t))|  \nonumber \\
&<& C_k e^{\frac{-2\pi c_k|\tau-\tau_{+} - 2R(\e)|}{p}}
\label{shift} \eea
Combining \eqref{asympe} and \eqref{shift}, we see for $\tau\in
[R(\e) -1, R(\e) + 1]$,
\bea \label{coincidence}
|\nabla^k(u_0^{\e}(\tau,t)-u_+^{\e}(\tau,t))|&<& 2\max_{\tau \in
[R(\e) - 1, R(\e) + 1]} C_k e^{\frac{-2\pi c_k|\tau - 2R(\e) -\tau_{+}|}{p}} \nonumber \\
&\le& C_k e^{\frac{-2\pi c_k|R(\e)+ 1 -2R(\e)-\tau_+|}{p}} \nonumber \\
&=& C_k e^{\frac{-2\pi c_k|R(\e)-1+\t_+|}{p} }  \to 0\eea
as $\e \to
0$ . Similarly we can prove the closeness of $u_0^{\e}$ and
$u_-^{\e}$ when $\tau$ is in $[-R(\e) -1, -R(\e) + 1]$.

\begin{figure}[ht] \centering \includegraphics{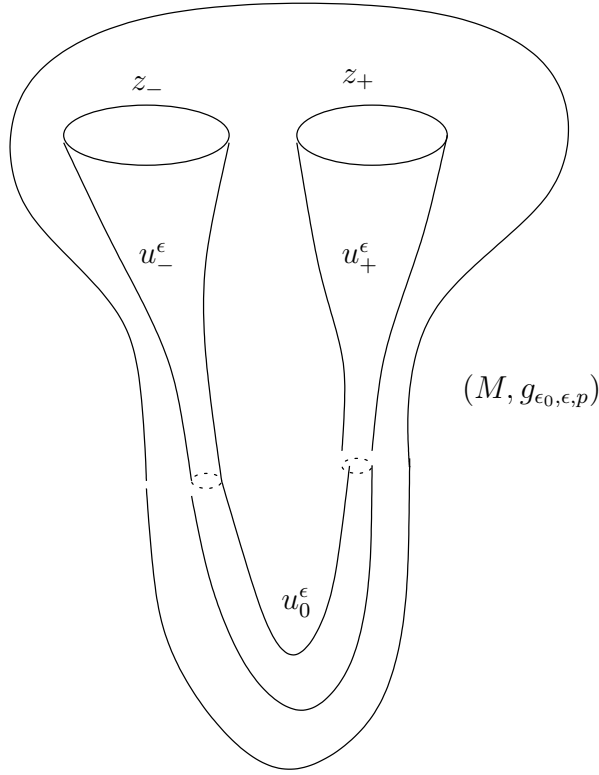}
\caption{Preglued solution}\end{figure}

Next we construct the approximate solution  \be u^{\e}_{app}(\tau,t)
= \left\{
\begin{array}{lcl}
u_-                &\tau\in&  \Sigma_-\backslash O_-\\
u_-^{\e}(\tau,t)  &\tau\in&  [-2R(\e),-R(\e) -1]\\
\chi_\e(\tau)u_0^{\e}(\tau,t) +(1-\chi_\e(\tau))u_-^{\e}(\tau,t)
&\tau\in&
[-R(\e) -1, -R(\e)+1]\\
u_0^{\e}(\tau,t) & \tau\in& [-R(\e)+1,R(\e)-1]\\
\chi_\e(\tau)u_0^{\e}(\tau,t) +(1-\chi_\e(\tau))u_+^{\e}(\tau,t)
&\tau\in&
[R(\e)-1,R(\e) + 1]\\
u_+^{\e}(\tau,t) &\tau\in&  [R(\e)+1,2R(\e)] \\
u_+              &\tau\in & \Sigma_+\backslash O_+
\end{array}
\right. \ee where the cut-off function  $\chi_\e:\R\to [0,1]$
satisfies \bea \chi_\e(\tau) & = & \begin{cases} 1 \quad & \mbox{for
}|\tau| \leq
R(\e) - 1 \\
0 \quad & \mbox{for }|\tau| \geq R(\e) +1
\end{cases} \label{eq:chie}\\
|\chi_\e'(\tau)| & \leq & 1.\label{eq:chie'} \eea In the above
formula, the summation $"+"$ is with respect to the linear space
structure of $T_p M$ ( By the Darboux cylindrical chart,  we can
think the local model lies in $T_pM$ ).

Now by applying a version of the implicit function theorem or the
Newton's iteration method, we want to perturb $u_{app}^{\e}$ to a
genuine solution $u^{\e}$ of the resolved Floer trajectory equation
\be \delbar_J u^{\e}+ (P_{\e \chi_\e(\tau)f}(u^{\e}))^{(0,1)}_J=0,
\ee where \beastar P_{\e \chi_\e(\tau)f}(u^{\e}) & = &
\e \chi_\e(\tau)(J X_f(u^{\e})d\tau - X_f(u^{\e}) dt) \\
& = & \e \chi_\e(\tau)(\nabla f(u^{\e})d\tau - J\nabla f(u^{\e})
dt). \eeastar For the simplicity of notations, we write
$$
a^\e = P_{\e \chi_\e(\tau)f}(u^{\e})
$$
and then
$$
a^\e\left(\frac{\del}{\del\tau}\right) = \e \chi_\e(\tau)\nabla
f(u^{\e}).
$$
In the conformal coordinates $(\tau,t)$ and the cylindrical metric,
we have the identity
$$
|a^\e|^2 = 2\left|a^\e\left(\frac{\del}{\del\tau}\right)\right|^2
$$
and so it will be enough to estimate the latter norm. Therefore we
will carry out estimation of this latter norm below.

\subsection{Error estimates of approximate solutions
}\label{dbar-error} With the choice of metric $g_{\e_0,\e,p}$  in
the beginning of this section, we carry out the error estimates,
i.e.,  the point estimate and $L^p$ estimate for the norm
$$
|\delbar_J u^\e_{app} -
(P_{K_{R(\e)}}(u^\e_{app}))^{(0,1)}_J|_{g_{\e_0,\e,p}}.
$$

{\bf Convention:} In many estimates of this subsection there are
different constants $C$'s. The exact values are not important; The
importance is that  all of them are independent on $\e$. For this
reason, we just denote them by the \emph{same} symbol $C$ and
shouldn't cause problems.

We split this estimation into three regions :
%-----------modification6:
\begin{enumerate}
\item the region for $|\tau|\le \frac{2}{3}R(\e)$,
\item the region for $ \frac{2}{3} R(\e)\leq |\tau| \leq
R(\e)+1$,
\item the region for $ R(\e) + 1 \le |\tau| \le 2R(\e)$.
\end{enumerate}

{\bf Case 1}: For $|\tau|\le \frac{2}{3}R(\e)$,
$u_{app}^{\e}=u_0^{\e}$. Recall $R(\e) = -\frac{1}{4\pi}{\ln \e}$.
By taking $\a=\frac{2}{3}$ in Lemma \ref{shrink} we have
$$
u_0^{\e}\left(\left[-\frac{2}{3}R(\e),\frac{2}{3}R(\e)\right]\times
S^1\right) \subset B_{\d_\e}(p),
$$
where $\d_\e=\e^{\frac{2}{3}}$. The local model $u_0^{\e}\subset
(T_pM,J_p)\cong \C^n$ satisfies $\overline\partial_{J_p}
u_0^{\e}-\e\nabla f(p) = 0$. Therefore
\bea%
\label{dJ} \delbar_J
u_{app}^{\e}-a^{\e} &=& (\delbar_J
u_0^{\e}-a^{\e})-(\delbar_{J_p}u_0^{\e}-\e \nabla
f(p)) \nonumber \\
&=& (\delbar_Ju_0^{\e}-\delbar_{J_p}u_0^{\e})-
\e(\chi_\e(\tau)\nabla f(u_0^{\e})-\nabla f(p))\nonumber \\
&=&\frac{1}{2}(J-J_p)du_0^{\e}\circ i-\e (\nabla f(u_0^{\e})-\nabla
f(p)).
\eea%
We have
\bea%
\|J(x)-J_p\| & \le& C\|DJ(p)\|_{B_{\d_\e}(p)}\cdot |x|_g \label{JJ}\\
|\nabla f(x)-\nabla f(p)| &\leq& C \|D^2f\|_{B_{\d_\e}(p)}\cdot
|x|_g \label{eq:fue-fp}
\eea%
where $|x|_g$ is the Euclidean norm $g(p)$ in the Darboux chart at
$p$.

On the other hand for the normalized local model $u_0^\e$ with
$z=e^{2\pi(\tau+it)}$, we have
\be%
 \left|\frac{\partial
u_0^{\e}}{\partial \tau}\right|_g, \, \left|\frac{\partial
u_0^{\e}}{\partial t}\right|_g \sim |u_0^\e|_g \leq C \d_\e.
\ee%
Therefore
\be%
 \label{due} |du_0^{\e}|_g\le C \d_\e
\ee%
Since $u_0^{\e}\left([-\frac{2}{3}R(\e),\frac{2}{3}R(\e)]\times S^1
\right)\subset B_{\d_\e}(p)$, on the image of $u_0^{\e}$, the almost
complex structure deviates from the standard complex structure $J_p$
on $T_pM$ by
\be%
\label{JJ} \|J(u_0^\e)-J_p\|\le C\|DJ(u_0^\e)\| \cdot |u_0^\e|_g
\ee%
 where $\|\cdot\|$ is the operator norm of linear
maps $L:V\to V$. We emphasize that the norm $\|L\|$ is independent
on the conformal class of constant metrics on $V$. Therefore
\eqref{JJ} holds regardless of our choice of metrics $g$ or
$g_{\delta,\e,p}$.

On the other hand, we obtain
\be%
\label{eq:fue-fp} \e |\nabla f(u^{\e}_0)-\nabla f(p)|_g \leq C \e
|u_0^\e|_g.
\ee%

Now we are ready to estimate $|\delbar_J
u_{app}^{\e}-a^{\e}|_{g_{\e_0,\e,p}}$. By \eqref{dJ}, \eqref{JJ} and
\eqref{due},
\begin{eqnarray}%
| \delbar_J u_{app}^{\e}-a^{\e}|_g
&\le& \frac{1}{2}\|J-J_0\||du_0^{\e}|_g+\e |\nabla f(u_0^\e) - \nabla f(p)|_g \nonumber \\
&\le& C(|u_0^\e|_g|du_0^{\e}|_g+\e |u_0^\e|_g)  \nonumber \\
&\le& C( \d_\e^2 + \e \d_\e) \label{dug}
\end{eqnarray}%

Since $u_0^{\e}(\tau,t)\in B_{\d_\e}(p)$,   and noting in
$B_{\d_\e}(p)$ the metric $g_{\e_0,\e,p}\le \frac{1}{\e^2}g$, by
\eqref{dug}, we have \be
 \label{comb}|\delbar_J
u_{app}^{\e}-a^{\e}|_{g_{\e_0,\e,p}}\le \frac{1}{\e}|\delbar_J
u_{app}^{\e}-a^{\e}|_{g}\le C(\d_\e^2/\e + \d_\e). \ee
%-----------modification7:
 Since
$\d_\e=\e^{\frac{2}{3}}$, we obtain \be
 \label{ds} |\delbar_J
u_{app}^{\e}-a^{\e}|_{g_{\e_0,\e,p}} \le C
(\e^{\frac{4}{3}-1}+\e^{\frac{2}{3}})\le C\e^{\frac{1}{3}}. \ee
 This error converges to 0 as $\e \to 0$.\\

{\bf Case 2:} For $\frac{2}{3}R(\e)\le|\tau| \le R(\e) +1 $, by
Lemma \ref{shrink} again we have $u_0^{\e}(\t,t)\in B_{\d_\e}(p)$,
where $\d_\e=C\e^\frac{1}{2}$. On the other hand,
\be \label{away}|u_0^{\e}(\t,t)|_g \ge \e\cdot \min\{|\vec A|,|\vec
B|\}\cdot e^{2\pi\cdot\frac{2}{3}R(\e)} =\b\e^{\frac{2}{3}} \ee
when $\e$ is small. So the image of $u_0^{\e}(\t,t)$ is contained in
$$
B_{\d_\e}(p)\setminus B_{\b\e^{ \frac{2}{3}} }(p),
$$
where the metric $g_{\e_0,\e,p}$ is cylindrical and so
$g_{\e_0,\e,p}(x)= \frac{1}{|x|_g^2}g(x)$. Therefore
\be \label{deway}|\delbar_J u_0^{\e}-a^{\e}|_{g_{\e_0,\e,p}}
=\frac{1}{|u_0^{\e}(\t,t)|_g } |\delbar_J u_0^{\e}-a^{\e}|_{g}. \ee
Similar to the second inequality in \eqref{dug}, we have
\begin{eqnarray}
| \delbar_J u_{app}^{\e}-a^{\e}|_g
&\le& \frac{1}{2}\|J-J_0\||du_0^{\e}|_g+\e |\chi_\e \nabla f(u_0^\e) -\nabla f(p)|_g \nonumber\\
&\le& C(|u_0^\e|_g|du_0^{\e}|_g+\e)  \label{diffway}
\end{eqnarray}

Combining \eqref{away},\eqref{deway} and \eqref{diffway} we get
\be%
\label{dun}|\delbar_J u_0^{\e}-a^{\e}|_{g_{\e_0,\e,p}} \le%
C (|du^{\e}_0|_g+\frac{\e}{\b\e^{\frac{2}{3}}}) \le%
C(\d_\e+\e^{\frac{1}{3}})\le C'{\e}^{\frac{1}{3}} .
\ee%

For $|u_0^{\e}(\tau,t)-u_{\pm}^{\e}(\tau,t)|_{g_{\e_0,\e,p}}$,
since the metric $g_{\e_0,\e,p}$ is cylindrical in this part, we
also have from \eqref{coincidence} (which is in cylindrical
metric) \be \label{coin1}
|u_0^{\e}(\tau,t)-u_{\pm}^{\e}(\tau,t)|_{g_{\e_0,\e,p}}\le C
e^{\frac{ -2\pi c_0 R(\e)}{p}}=C \e^{\frac{c_0}{2p}}. \ee

%-----------modification8:
Combining the above \eqref{dun} and \eqref{coin1}, with respect to
the metric $g_{\e_0,\e,p}$,
\begin{eqnarray} \label{coin} %
& &|\delbar_J
u_{app}^{\e}-a^{\e}|_{g_{\e_0,\e,p}} \nonumber \\
&=& \left|\delbar_J\left(
\chi_\e(\tau)u_0^{\e}+(1-\chi_\e(\tau))u_{\pm} \right)
-a^{\e}\right|_{g_{\e_0,\e,p}}  \nonumber\\
&\le& \chi_\e(\tau)|\delbar_J
u_0^{\e}-a^{\e}|_{g_{\e_0,\e,p}}+(1-\chi_\e(\tau))(|\delbar_J
u_{\pm}^{\e}|_{g_{\e_0,\e,p}}+|a^{\e}|_{g_{\e_0,\e,p}}) \nonumber\\
&{}& \quad +\chi'_{\e}(\tau)|u_0^{\e}-u_{\pm}^{\e}|_{g_{\e_0,\e,p}}  \nonumber\\
&\le& 1\cdot C'\e^{1\over 3 }+
1\cdot(0+\frac{\e}{\b\e^{\frac{2}{3}}}\cdot |\nabla f|_g)+
1\cdot C \e^{\frac{c_0}{2p}}  \nonumber\\
&\le& C \e^{\min\{\frac{1}{3}, \frac{c_0}{2p} \}}.
\end{eqnarray}%
 In the second of the above inequalities, we have used that
$u_{\pm}^{\e}=u_+(\tau-2R(\e),t)$ is $J$-holomorphic,
$|a^{\e}|_{g_{\e_0,\e,p}}\le
\frac{1}{\b\e^{\frac{2}{3}}}|a^{\e}|_g$, and $|\chi'(\tau)|
\leq 1$.\\

{\bf Case 3:} For $ R(\e) + 1 \le |\tau|\le 2R(\e)$,
$u_{app}^{\e}=u_{\pm}^{\e}$ are $J$-holomorphic, and $a^{\e}=0$, so
$$\delbar_J u_{app}^{\e}-a^{\e}\equiv 0.$$
%-----------modification9:

In all, we have obtained the point estimate for any $(\tau,t)\in
[-2R(\e),2R(\e)]\times S^1$:
$$
Err(\e): =|\delbar_J u_{app}^{\e}-a^{\e}|_{g_{\e_0,\e,p}}\le  C
\e^{\frac{1}{3}}
$$

\medskip
{\bf The $L^p_{\a_{\d,\e}}$ estimate (The weight $\a_{\d,\e}$ is
defined in the next section )}: Note that on
$\dot{\Sigma}_{\pm}\backslash O_{\pm}$ the $ u^\e_{app}$ coincides
with the two original solutions $u_+$ and $u_-$, so we only need
to integrate $|\delbar_J u_{app}^{\e}-a^{\e}|_{g_{\e_0,\e,p}}^p$
over $[-2R(\e), 2R(\e) ]\times S^1$. Note in such region the
weight function
$$
|\r_{\e}(\tau)|\le e^{2\pi\d\cdot (2R(\e)) }=\e^{-\d}.
$$
Therefore we get
\be \label{dufinal}%
 \|\delbar_J
u_{app}^{\e}-a^{\e}\|^p_{\a_{\d,\e}} \le (Err(\e))^p \cdot
{\e}^{-\d} \cdot 4R(\e)= -C\e^
{\min\{\frac{p}{3},\frac{c_0}{2}\}-\d}\ln\e,
\ee%
\be \label{dbarErrorEstmt} %
i.e. \quad \|\delbar_{(J_\e,K_\e)}u^\e_{app}\|_{p,\a_{\d,\e}}\le
L\cdot(R(\e))^{\frac{1}{p}} \cdot e^{-\frac{4\pi a R(\e)}{p}} ,
\ee%
where $L$ and $a$ are constant independent on $\e$, and
$a=\min\{\frac{1}{3},\frac{c_0}{2p}\}-\frac{\d}{p}$.

If we choose $0<\d<\min\{\frac{p}{3},\frac{c_0}{2}\}$ in the
beginning, then $a>0$ and so
$$
\|\delbar_{(J_\e,K_\e)}u^\e_{app}\|_{p,\a_{\d,\e}}\to 0
$$
as $\e\to 0$, the error estimate is established. For gluing
purpose in later sections, we further assume $\d$ is small in the
beginning such that $0<\d<a$.

\subsection{The off-shell setting of resolved nodal Floer trajectories}
\label{subsec:off-resolve-nodal}
 We define the Banach manifold  to
host all resolved nodal Floer trajectories near the enhanced nodal
Floer trajectories $u=(u_-,u_0,u_+)$. The construction is some
``smoothing" of the Banach manifold for enhanced nodal Floer
trajectories in section \ref{subsec:off-nodal}. Roughly speaking, We
smooth the target $(M\backslash\{p\})\sqcup T_pM$ to
$(M,g_{\e_0,\e,p})$, smooth the exponential weight for the outer
curves and local model to the exponential weight for the approximate
solution, and smooth the Morse-Bott movements for the outer curves
and local model near their ends. The precise description is in
order:

 First we define the Banach manifold $\CB_{res}^{\e}(z_-,z_+;p)$
 \index{$\CB_{res}^{\e}(z_-,z_+;p)$} for
any $\e\in (0,\d_0)$ and $p\in M$, where $\d_0>0$ is a small
constant to be determine later. $\CB_{res}^{\e}(z_-,z_+;p)$
consists of maps $u$ from $\dot\Sigma$ to the Riemannian manifold
$(M,g_{\e_0,\e,p})$ satisfying:
\begin{enumerate}
\item  $u\in W^{1,p}_{loc}(\dot\Sigma,M)$
\item  $\lim_{\tau\to +\infty}u(\tau,t)=z_{+}(t)$ and
$\lim_{\tau\to -\infty}u(\tau,t)=z_{-}(t)$ for all $t\in
S^1$.
\item  When $\tau>0$
is large enough, $u(\tau,t)=\exp_{z_{+}(t)}\xi(\tau,t)$ for $
\xi(\tau,t)\in W^{1,p}_{\delta}([0,\infty) \times S^1,z_+^*(TM))$.
Similarly for the other end converging to $z_-(t)$.
%\item $\text{diam}(u(0,t))<C\e$, where $C$ is some constant independent on $\e$,
%and the center of mass of the curve $u(0,t) \; (t\in S^1)$, denoted
%by $\overline u(0)$, is $p$.
\item For each $u\in \CB_{res}^{\e}(z_-,z_+;p)$,  its tangent space $T_u \CB_{res}^{\e}(z_-,z_+;p)$
is identified as $W^{1,p}_{\a_{\d,\e}}(u^*(TM))$, defined as the
following: For $V$  be a section of $u^*(TM)$,
%with \be \label{B1} \overline V(0)=\int_{S^1}V(0,t)dt=0, \ee

%-----------modification10:
we define the \index{$W^{1,p}_{\a_{\d,\e}}$} $W^{1,p}_{\a_{\d,\e}}$ norm of $V$ to be
\bea\label{eq:Vnorm} %
\|V\|^p_{1,p,\a_{\d,\e}}&=& |V(-R(\e),0)|^p +
|V(R(\e),0)|^p \nonumber \\
&+& \int_{[-2R(\e),2R(\e)] \times S^1} \a_{\d,\e}(\tau) (|V-V_0|^p+
|\nabla(V- V_0)|^p)d\tau dt \nonumber \\
&+& \int_{|\t|>2R(\e)} (|V-V_0|^p+ |\nabla(V- V_0)|^p)d\tau dt,
\eea %
where in the above identity, all metric $|\cdot|$ are the metric
$g_{\e_0,\e,p}$, and
\bea %
V_0(\t,t) &=&
\beta_\e^-(\tau)Pal_{u(\t,t)}Pal_{u(-R(\e),t)}(V(-R(\e),0)) \nonumber\\
&+&
\beta_\e^+(\tau)Pal_{u(\t,t)}Pal_{u(R(\e),t)}(V(R(\e),0)),
\eea %
where in the above expression $Pal_{u(\t,t)}(V(\t',t'))$ is the
parallel transport of $V(\t',t')$ along the minimal geodesic of
the metric $g_{\e_0,\e,p}$ from $u(\t',t')$ to $u(\t,t)$. The cut
off function  $\b_{\e}^{\pm}:\R\to [0,1]$ is smooth, $0 \le
|\frac{d}{d\t}{\b^{\pm}_{\e}}|\le 1$,
%-----------modification11:
\beastar \beta_{\e}^-(\tau)& = &\left\{
\begin{array}{ccl} &1 &\mbox{for }\, 2 \le \t \le 2R(\e)-2\\
&0 &\mbox{for } \t\le 1 \mbox{ or } \t\ge 2R(\e)-1
\end{array} \right.\\
 \beta_{\e}^+(\tau)& = &\left\{
\begin{array}{ccl} &1 &\mbox{for }\, -2R(\e)+2\le \t \le -2\\
&0 &\mbox{for } \t\ge -1 \mbox{ or } \t\le -2R(\e)+1.
\end{array}
\right. \eeastar
\noindent The weight function $\a_{\d,\e}$  is
smooth, \be \a_{\d,\e}(\tau)=
\begin{cases}
e^{2\pi\d|\tau|}
&\mbox{for }\, |\tau|\le R(\e)-1 \\
\sim e^{\pi\d R(\e)} &\mbox{for }\,|\tau| \in
[R(\e)-1,R(\e)+1]\\
e^{2\pi\d|\tau-2R(\e)|} &\mbox{for }\,|\tau| \in [R(\e)+1,2R(\e)]\\
1  &\mbox{for }\, |\t|\ge 2R(\e).
\end{cases}
\ee
\end{enumerate}
In the above the ``$\sim$" means that the ratio of $\a_{\d,\e}(\t)$
and $e^{\pi\d R(\e)}$ is between $\frac{1}{2}$ and $\frac{3}{2}$ for
$|\tau| \in [R(\e)-1,R(\e)+1]$.

%\begin{rem} [About the center of mass $\overline u(0)$] The center of
%mass of a curve in a Riemannian manifold
%is well defined whenever the diameter of the curve is sufficient
%small. By the condition $\text{diam}(u(0,t))<C\e$, we can define the
%center of mass of the curve $u(0,t) \; (t\in S^1)$ when $\e$ small.
%We can also use Darboux charts, regarding $u(0,t)$ in $\C^n$ to
%define the center of mass $\overline u(0)$ as $\int_{S^1}u(0,t)dt$.
%\end{rem}

%\begin{rem} In the above definition, to make sense of $V_0$, we
%must have $u([-2R(\e),2R(\e)]\times S^1)$ contained in the Darboux
%cylindrical neighborhood of $p$. This is true for such $u$ in
%$\CB^\e_{res}(z_-,z_+,p)$ close enough to a nodal Floer trajectory
%$(u_-,u_+)$.
%\end{rem}

\begin{rem}[About the ``Morse-Bott" variation] The vector field
$V_0$ is induced from the ``Morse-Bott" variation $V(\pm
R(\e),0)$, which is the approximation of the true Morse-Bott
variation of the asymptotes $(\t_{\pm},\g_{\pm})$ at infinity of
the enhanced nodal Floer trajectories $(u_-,u_0,u_+)$. Given
$(\t',t')$, there may be different minimal geodesics connecting
the points $u(\t,t)$ and $u(\t',t')$ so the symbol $Pal_{u(\t,t)}$
is ambiguous, but such $(\t,t)$ form at most 1 dimensional subset
in $\R\times S^1$ so won't affect the $\|\cdot\|_{1,p,\a_{\d,\e}}$
norm.
\end{rem}

\begin{rem} [About the cut-off function $\beta_\e^{\pm}$] Recall for the vector field on
$u_+$, we take out the $J$-holomorphic vector field induced from the
Morse-Bott move when $\t<-2$,  and then measure the remaining part
by $W^{1,p}_\d$ norm. Similarly for the vector field on any local
model,  we take out the $J$-holomorphic vector field induced from
the Morse-Bott move when $|\t|>2$ and then measure the remaining
part by $W^{1,p}_\d$ norm. Since $u^{\e}_{\pm}=u_{\pm}(\cdot-\pm
2R(\e))$, the $J$-holomorphic vector field induced from the
Morse-Bott move $V(\pm R(\e))$ on the approximate solution is taken
out when $2\le |\t|\le 2R(\e)-2$. That is why we design the above
cut-off function $\beta_\e^{\pm}$.
\end{rem}

\begin{rem} [About the exponential weight function $\a_{\d,\e}(\t)$]
Since we will use the bound of the right inverses of
$D_{u_{\pm}^\e}\delbar_{(J^{\pm},K^{\pm})}$ and
$D_{u_0^\e}\delbar_{(J_p,a^\e)}$ to estimate the bound of the right
inverse of $D_{u_{app}^\e}\delbar_{(J^{\pm},K^{\pm},\e f)}$, the
weight function $\a_{\d,\e}(\tau)$ has to be the concatenation of
the weight functions for $u_{\pm}^\e$ and $u_0^\e$. Also the ratio
between $\a_{\d,\e}(\t)$ and them must be uniformly bounded up and
below. That is why we give the above expression for
$\a_{\d,\e}(\tau)$.
\end{rem}

Therefore, we have an $\e$-family of Banach manifolds
$\CB_{res}^\e(z_-,z_+;p)$, and an $\e$-family of equations
$\delbar_{(J_\e,K_\e,\e f)}u^{\e}=0$ defined on each Banach bundle
$$\pi:\CL^{\e}_{res}(z_-,z_+;p)\to \CB^{\e}_{res}(z_-,z_+;p),$$ where
$$\CL^{\e}_{res}(z_-,z_+;p)=\bigcup_{u\in
\CB^{\e}_{res}(z_-,z_+;p)}L^p_{\a_{\d,\e}}(\Lambda^{0,1}(u^*TM\otimes)),$$
and each fiber $L^p_{\a_{\d,\e}}(\Lambda^{0,1}(u^*TM\otimes))$
consists of sections $V$ of $\Lambda^{0,1}(u^*TM\otimes)$ such that
\be \|V\|_{p,\a_{\d,\e}}=\int_{|\tau|\ge 2R(\e)} |V|^p d\tau dt
 + \int_{|\tau|\le 2R(\e)}\a_{\d,\e}(\t)|V|^p d\tau dt
\ee where the norm $|\cdot|$ is in terms of the metric
$g_{\e_0,\e,p} $.

We define \index{$\CB_{res}^{\e}(z_-,z_+)$}
$$\CB_{res}^{\e}(z_-,z_+):=\bigcup_{p\in M}\CB_{res}^{\e}(z_-,z_+;p).$$
For $u\in \CB_{res}^{\e}(z_-,z_+)$, its tangent space consists of
elements $U=(V,v)$ where $V\in T_u\CB_{res}^{\e}(z_-,z_+;p)$ and
$v\in T_pM$, with the norm
   $$\|U\|_{1,p,\a_{\d,\e}}=\|V\|_{1,p,\a_{\d,\e}}+|v|.$$
Here $v$ represents the variation of the target Riemannian
manifolds $(M,g_{\e_0,\e,p})$, which are parameterized by $M$.
\begin{rem} The trivialization of the family of Riemannian manifolds
$$\bigcup_{p\in M} (M,g_{\e_0,\e,p})$$ is to regard them as pointed
manifolds $(M,p)$ and use the trivialization given in Subsection
\ref{subsec:Darboux}.
\end{rem}
Let
$$\CL_{res}^{\e}(z_-,z_+):=\bigcup_{p\in M}\CL_{res}^{\e}(z_-,z_+;p),$$
then we have a natural section
$$\delbar_{(J_\e,K_\e,\e
f)}:\CB_{res}^{\e}(z_-,z_+)\to \CL_{res}^{\e}(z_-,z_+).$$ The
linearization of $\delbar_{(J_\e,K_\e)}$ at $u^\e_{app}$
$$
D\delbar_{(J_\e,K_\e, \e f)}(u^\e_{app}) :
W^{1,p}_{\a_{\d,\e}}(u^*TM)\to
L^p_{\a_{\d,\e}}(\Lambda^{0,1}(u^*TM\otimes))
$$
is given by
$$
D\delbar_{(J_\e,K_\e, \e f)}(u^\e_{app}) = D_u\delbar_J(u^\e_{app})
+ DP_{K_\e}(u^\e_{app}).
$$
The linearization of $D\delbar_{(J_\e,K_\e, \e f)}(u^\e_{app})$ with
respect to $v$ is similar to that given in Subsection
\ref{subsec:offshellpertdiscs} and Subsection
\ref{subsec:off-nodal}.

\subsection{Construction and estimates of the right inverse }

Given the approximate solution $u^{\e}_{app}$, we will construct the
approximate right inverse
$$
Q^{\e}:L^p_{\a_{\d,\e}}(\Lambda^{0,1}(u^{\e }_{app})^*TM) \to
W^{1,p}_{\a_{\d,\e}}( (u^{\e }_{app})^*TM)\oplus T_pM
$$
of the differential operator
$$
D_u\delbar_{(J_\e,K_\e)}(u^\e_{app}) : W^{1,p}_{\a_{\d,\e}}( (u^{\e
}_{app})^*TM)\oplus T_pM \to L^p_{\a_{\d,\e}}(\Lambda^{0,1}(u^{\e
}_{app})^*TM),
$$
and show $Q^{\e}$ is uniformly bounded in operator norm. For
notation brevity, we write $D_u\delbar_{(J_\e,K_\e)}$ as
$D_u\delbar$ if there is no danger of confusion.

The method is similar to that of gluing two $J$-holomorphic discs in
$M$ using cylindrical domains, as
in section 29 in \cite{fooo07}. Indeed, we glue each of the two
ends of our local model $u_0$ with the outer curves $u_-$ and
$u_+$ respectively, and the gluing at two ends are somewhat
independent, so locally our construction looks like gluing two
curves $u_0$ with $u_-$ or $u_0$ with $u_+$ respectively.

We introduce various cut-off functions to patch the approximate
right inverse.
\be \chi_S^+(\tau)= \left\{
\begin{array}{lcl}0 &\tau& \le S-1   \\
1 &\tau& \ge S+1,
\end{array}
\right. \ee
with $|\nabla{\chi_S^+}|\le 1$, and put
$$\chi_S^-(\t)=1-\chi_S^+(\t), \quad \chi_S^0(\t)=1-\chi_S^+(\t)
-\chi_{-S}^-(\t).$$ In the following of this section, it is
important to let $S$ have the same order as $R(\e)$) but smaller
than $R(\e)$. For convenience we set
$$S=\frac{1}{4} R(\e).$$

\medskip
We also need some ``transporting" and ``combining" operators to
define the approximate right inverse. Recall that when $\t\in
[\frac{1}{2}R(\e),\frac{3}{2}R(\e)]$, the shifted outer curve
$u_+^\e(\t,t)=u_+(\t-2R(\e),t)$, the scaled local model
$u^\e_0(\t,t)=\e u_0(\t,t)$ and the approximate solution
$u^\e_{app}(\t,t)$ are exponentially close to each other in the
cylindrical metric $g_{\e_0,\e,p}$. Therefore, for any $|\tau|<
\frac{1}{2}R(\e)$, we can define the transform
$$J^S_{0,\e}:\G (\Lambda^{0,1}( (u_0)^*T\C^n ))\to  \G (\Lambda^{0,1}( (u^\e_{app})^*TM)) $$
in the following way : given any $\eta$ in $\G (\Lambda^{0,1}(
(u_0)^*T\C^n ))$, we push forward it to $\G (\Lambda^{0,1}(
(u_0^\e)^*T\C^n ))$ using the scaling $\e:\C^n\to \C^n, \, u_0\to
u_0^\e$. Then we cut it by $\chi_{R(\e)+S}^0$ and use parallel
transport $Pal_{0,\e}$  along minimal geodesics connecting
$u^\e_0(\t,t)$ and $u^\e_{app}(\t,t)$ to get a section on $\G
(\Lambda^{0,1}( (u^\e_{app})^*TM))$. In short, we denote
\be
J^S_{0,\e}\eta=Pal_{0,\e} (\chi^0_{R(\e)+S}((\e)_*\eta))).
\ee
Similarly we can define the transform
$$J^S_{+,\e}:\G (\Lambda^{0,1}( (u_+)^*TM ))\to  \G (\Lambda^{0,1}( (u^\e_{app})^*TM))$$
as the following : for any $\eta(\t,t) \in \G (\Lambda^{0,1}(
(u_+)^*TM )) $, we shift it to  $\eta(\t-2R(\e))$ and regard it as a
1-form on $u^\e_+(\t,t)=u_+(\t-2R(\e),t))$. We cut it by
$\chi_{R(\e)-S}^+$ and use parallel transport $Pal_{+,\e}$ along
minimal geodesics connecting $u^\e_+(\t,t)$ and $u^\e_{app}(\t,t)$
to get a section on $\G (\Lambda^{0,1}( (u^\e_{app})^*TM))$. In
short, \be J^S_{+,\e}\eta=Pal_{+,\e}
(\chi^+_{R(\e)-S}(\eta(\t-2R(\e),t))).\ee Similarly we define
$$J^S_{-,\e}:\G (\Lambda^{0,1}( (u^\e_-)^*T\C^n ))\to  \G (\Lambda^{0,1}( (u^\e_{app})^*TM)) $$
to be \be J^S_{-,\e}\eta=Pal_{-,\e}
(\chi^-_{-R(\e)+S}(\eta(\t+2R(\e),t))).
\ee
For the reversed ones
$$J^S_{\e,0}:\G (\Lambda^{0,1}( (u^\e_{app})^*TM ))\to  \G (\Lambda^{0,1}( (u_{0})^*\C^n)) $$
$$J^S_{\e,+}:\G (\Lambda^{0,1}( (u^\e_{app})^*TM ))\to  \G (\Lambda^{0,1}( (u_{+})^*TM)) $$
$$J^S_{\e,-}:\G (\Lambda^{0,1}( (u^\e_{app})^*TM ))\to  \G (\Lambda^{0,1}( (u_{-})^*TM)), $$
the definitions are similar. For example for $J^S_{\e,+}$, for any
$\eta$ in $\G (\Lambda^{0,1}( (u^\e_{app})^*TM ))$, we cut it by
$\chi_{R(\e)-S}^+$ then use parallel transport $Pal_{\e,+}$
from
$u^\e_{app}$ to $u^\e_+$ to get an element in $\G (\Lambda^{0,1}(
(u^\e_{+})^*TM)) $, and then shift it  to $\G (\Lambda^{0,1}(
(u_{+})^*TM)) $. In short, \be J^S_{\e,+}\eta=Pal_{\e,+}
(\chi^+_{R(\e)-S}(\t+2R(\e))(\eta(\t+2R(\e),t))).\ee

It is easy to check the following identities:
 \bea \label{transID}
J^S_{0,\e}\circ J^S_{\e,0}( \chi^0_{R(\e)}\eta)=
\chi^0_{R(\e)}\eta \nonumber\\
J^S_{+,\e}\circ J^S_{\e,+}(
\chi^+_{R(\e)}\eta)= \chi^+_{R(\e)}\eta  \nonumber \\
J^S_{-,\e}\circ J^S_{\e,-}( \chi^-_{R(\e)}\eta)= \chi^-_{R(\e)}\eta
\eea
\medskip

For an enhanced nodal Floer trajectory $u=(u_-,u_0,u_+)$, the
``combining" operator  $$I^S: T_u W^{1,p}_{\a} \to
T_{u^\e_{app}}W^{1,p}_{\a_{\d,\e}}
$$
is defined as the following: for $\xi=(\xi_-,\xi_0,\xi_+)\in
T_u\CB_{noal}$ (defined in \eqref{BNodal}), i.e.
\be %
\xi_-\in T_{u_-}W^{1,p}_{\a}(\dot{\S},\widetilde{M};z_-),\;
\xi_0\in T_{u_0}\CB_{lmd}, \;\text{and} \; \xi_+\in
W^{1,p}_{\a}(\dot{\S},\widetilde{M};z_+)\label{3xi}
\ee %
with
the matching condition   \be
\xi_-(+\infty,t)=\xi_0(-\infty,t)=V^-(t), \quad
\xi_+(-\infty,t)=\xi_0(+\infty,t)=V^+(t), \label{matchV}\ee
 then by the above $\e$-scaling
and $\pm 2R(\e)$ shifting we can regard $\xi_-,\xi_0$ and $\xi_+$
as the elements in $
T_{u^\e_-}W^{1,p}_{\a_\e}(\dot{\S},\widetilde{M};z_-),
T_{u^\e_0}\CB_{lmd}$ and
$T_{u^\e_+}W^{1,p}_{\a_\e}(\dot{\S},\widetilde{M};z_+)$
respectively, with the same matching condition (\ref{matchV})
(Here the weighting function $\a_\e(\t)$  is the shifting of the
weighting function $\a$, namely $\a_\e(\t)=\a(\t-\pm 2R(\e))$ for
$u_{\pm}^\e$ respectively). For convenience we still denote them
by $\xi_-,\xi_0$ and $\xi_+$. Then we define

\begin{multline}
I^S(\xi_-,V^-,\xi_0,V^+,\xi_+)(\t,t)\\=\left\{
\begin{array}{lcl}
\xi_+(\t,t) &\t\in& [\frac{5}{4}R(\e),+\infty) \\
V^+ + \chi^-_{R(\e)+S}(Pal_{0,\e}(\xi_0)-V^+) +
\chi^+_{R(\e)-S}(Pal_{+,\e}(\xi_+) -V^+)
&\t\in& [\frac{3}{4}R(\e),\frac{5}{4}R(\e)]\\
\xi_0(\t,t)  &\t\in& [-\frac{3}{4}R(\e),\frac{3}{4}R(\e)]\\
V^- + \chi^+_{-R(\e)-S}(Pal_{0,\e}(\xi_0)-V^-) +
\chi^-_{-R(\e)+S}(Pal_{-,\e}(\xi_-) -V^-)
&\t\in& [-\frac{3}{4}R(\e),-\frac{1}{4}R(\e)]\\
\xi_-(\t,t)  &\t\in& (-\infty,-\frac{5}{4}R(\e)]
\end{array} \right.
\end{multline}
In the above expression the $V^+$ should be regarded as a {\em
vector field} obtained by the parallel transport of $V^+(t)$ to
$u^\e_{app}(\t,t)$ in the following way: In the cylindrical
coordinate $(s,\Theta)$ of $M\backslash\{p\}$, write
$V^+(t)=(V^+_{\R},V^+_{\CR_1(\l)}(t))\in
T_{(\t_+,\g_+(t))}(\R\times S^{2n-1})$. For $u^\e_{app}$ close
enough to $(u_-,u_0,u_+)$, and $|\t|\in
[\frac{3}{4}R(\e),\frac{5}{4}R(\e)]$, $u^\e_{app}(\t,t)$ is in the
the cylindrical metric part of $(M,g_{\e_0,\e,p})$, which can be
isometrically identified as a part of $\R\times S^{2n-1}$. Then we
use the connection of $\R\times S^{2n-1}$ to do the parallel
transport, namely, transport $V^+_\R$ trivially , and transport
$V^+_{\CR_1(\l)}(t)$ along the minimal geodesic connecting
$\g_+(t)$ and $\Theta(u^\e_{app}(\t,t))$ in $S^{2n-1}$. We remark
that the above intropolation $I^S$ also performs on the common
base variation $v\in T_pM$ shared by $\xi_-,\xi_0$ and $\xi_+$,
resulting in $v$ again.

\begin{rem}
The somewhat complicated interpolation among $V^{\pm}, \xi_0$ and
$\xi_{\pm}$ using $\chi_{\pm R(\e)\pm S}^{\pm}$ instead of the
simple interpolation between $\xi_0$ and $\xi_{\pm}$ using
$\chi^{\pm}_{\pm R(\e)}$ is responsible for the better accuracy of
our approximate right inverse with respect to the exponential weight
$\a_{\d,\e}$, because it makes the interpolation happen at the
places $\t=\pm R(\e)\pm S$ avoiding the ``peaks" of the weight
function $\a_{\d,\e}$ at $\t=\pm R(\e)$. We will see the advantage
of this while doing the estimates of the approximate right inverse later.
\end{rem}

Now we define the approximate right inverse $Q^\e$  for $\e>0$
using the right inverse $Q^\e|_{\e=0}$ on $(u_-,u_0,u_+)$ defined
\eqref{Q3}. For $\eta \in L^p_{\a_{\d,\e}}(\Lambda^{0,1}(u^{\e
}_{app})^*TM)$, we let \bea
  Q^\e(\eta)&=& I^S \big( \;  Q^\e|_{\e=0} (\, J^S_{\e,-} (\chi^-_{R(\e)}\eta),
  J^S_{\e,0} (\chi^0_{R(\e)}\eta),  J^S_{\e,+} (\chi^+_{R(\e)}\eta)\,)   \;    \big) \nonumber \\
  &=& I^S \big( Q_-\circ J^S_{\e,-} (\chi^-_{R(\e)}\eta)
,V^-_\e, Q_0\circ J^S_{\e,0} (\chi^0_{R(\e)}\eta),V^+_\e, Q_+\circ
J^S_{\e,+} (\chi^+_{R(\e)}\eta)
  \big)  \nonumber \\
  &=& I^S (\xi_{\e,-}, V^-_\e, \xi_{\e,0}, V^+_\e, \xi_{\e,+},v)
\eea where
\be \xi_{\e,-}=Q_-\circ J^S_{\e,-} (\chi^-_{R(\e)}\eta),\;
\xi_{\e,0}=Q_0\circ J^S_{\e,0} (\chi^0_{R(\e)}\eta),\; \xi_{\e,+}=
Q_+\circ J^S_{\e,+} (\chi^+_{R(\e)}\eta) \label{defxi} \ee and
$V_\e^{+},V_\e^-$ are their matching asymptotes at infinity, namely
$$ V_\e^{-}(t)=\xi_{\e,-}(+\infty,t)=\xi_{\e,0}(-\infty,t), \quad V_\e^{+}(t)=\xi_{\e,+}(-\infty,t)=\xi_{\e,0}(+\infty,t), $$
and $v\in T_pM$ is the common base variation shared by
$\xi_-,\xi_0$ and $\xi_+$.
\\

{\bf Convention of uniform constants:} In the remaining part of
this subsection, in many estimates there are different constants
$C$'s. The exact values are not important; The importance is that
 all of them are independent on $\e$. For this reason, we
just denote them by the \emph{same} symbol $C$ and shouldn't cause
problems.
\\

We show the norm of $Q^\e$ is uniformly bounded:
\begin{prop} \label{unibdInv} There exists a constant $C=C(\d)>0$
independent on $\e$ such that
$$\|Q^\e (\eta)\|_{1,p,\a_{\d,\e}} \le C\|\eta\|_{p,\a_{\d,\e}}.$$
for all $\eta \in L^p_{\a_{\d,\e}}(\Lambda^{0,1}(u^{\e
}_{app})^*TM)$.
\end{prop}
\begin{proof} We first show that for $(\xi_-,\xi_0,\xi_+)$ in (\ref{3xi})
with matching condition (\ref{matchV}), and $S=\frac{1}{4}R(\e)$,
\be \| I^S(\xi_-,V^-,\xi_0,V^+,\xi_+)\|_{1,p,\a_{\d,\e}}\le
C(\|\xi_-\|_{1,p,\a} + \|\xi_0\|_{1,p,\d}+ \|\xi_+\|_{1,p,\a}),
\label{3norms} \ee where $C$ is independent on $\e$.

By the definition of $I^S$, to prove (\ref{3norms}), it is enough to
estimate the norm of the right hand side on $[-2R(\e),2R(\e)]\times
S^1$. Let $\xi^\e_{app}=I^S(\xi_-,V^-,\xi_0,V^+,\xi_+))$. Then
\beastar \xi^\e_{app}(\t,t)&=&Pal_{0,\e}(\xi_0) + Pal_{+,\e}(\xi_+)
-V^+ \quad\text{ for } \t\in [R(\e)-1,R(\e)+1] \\
\xi^\e_{app}(\t,t)&=&Pal_{0,\e}(\xi_0) + Pal_{-,\e}(\xi_-) -V^-
\quad\text{ for } \t\in [-R(\e)-1,-R(\e)+1] \eeastar By Sobolev
inequality and the definitions of the norms $\|\cdot\|_{1,p,\d}$
and $\|\cdot\|_{1,p,\a}$, we have the point estimate \beastar
|\xi^\e_{app}(\pm R(\e))-V^{\pm}|\le Ce^{\frac{-2\pi\d R(\e)}{p}}(
\|\xi_0\|_{1,p,\d}+ \|\xi_{\pm}\|_{1,p,\a} ). \eeastar We also
have the energy estimate
\begin{multline} \Big[\int_{[0, 2R(\e)]\times S^1 } \a_{\d,\e} \cdot
\Big(|\nabla(Pal_{u^\e_{app}(\t,t)}(\xi^\e_{app}( R(\e)))-V^+)|^p
\\ +|Pal_{u^\e_{app}(\t,t)}(\xi^\e_{app}(R(\e)))-V^+|^p \Big) d\t
dt\Big]^{1\over p} \le C(\|\xi_-\|_{1,p,\a} + \|\xi_0\|_{1,p,\d}+
\|\xi_+\|_{1,p,\a}) \label{Xinew} \end{multline}
and a similar inequality for $V^-$ in $[-2R(\e),0]$.

Postponing the proof of \eqref{Xinew} to the next lemma, it is
enough to estimate \beastar \int_{[0,2R(\e)]\times S^1 } \a_{\d,\e}
\cdot
(|\nabla(\xi^\e_{app}-V^+)|^p&+&|\xi^\e_{app}-V^+|^p) d\t dt \qquad\text{ and } \\
\int_{[-2R(\e),0]\times S^1 } \a_{\d,\e} \cdot
(|\nabla(\xi^\e_{app}-V^-)|^p&+&|\xi^\e_{app}-V^-|^p) d\t dt.
\eeastar
Since the above two terms are similar, we only estimate the first
term. The first term is estimated by
\beastar &{}& C\int_{[0,2R(\e)]\times S^1} e^{2\pi\d|\t|}\cdot
\Big(|\chi^-_{R(\e)+S}(\t) (\xi_0-V^+) |^p + |\nabla
(\chi^-_{R(\e)+S}(\t) (\xi_0-V^+)) |^p
 \Big) d\t dt \\
&+& C\int_{[0,2R(\e)]\times S^1} \a(\t)\cdot
\Big(|\chi^+_{R(\e)-S}(\t) (\xi_+ -V^+) |^p + |\nabla
(\chi^+_{R(\e)-S}(\t) (\xi_+-V^+)) |^p\Big)d\t dt.
 \eeastar
This is because the weight $\a_{\d,\e}(\t,t)$ is estimated by
$e^{2\pi\d|\t|}$ and $\a(\t)$ respectively on the support of
$\chi^-_{R(\e)+S}(\t)$ and $\chi^+_{R(\e)-S}(\t)$ . Then it is easy
to estimate the above expression by the right hand side of
(\ref{3norms}).

To compete the proof of Proposition \ref{unibdInv} it is enough to
show \be \label{Jbd}
  \|J^S_{\e,\pm } (\chi^\pm_{R(\e)}\eta)\|_{p,\a}\le
  B\|\eta\|_{p,\a_{\d,\e}}, \qquad
  \|J^S_{\e,0} (\chi^0_{R(\e)}\eta)\|_{p,\d}\le
  B\|\eta\|_{p,\a_{\d,\e}}.
 \ee
for some fixed constant $B$. Note that
\beastar
J^S_{\e,\pm }
(\chi^\pm_{R(\e)}\eta)& = & Pal_{\e,\pm} (\chi^\pm_{R(\e)}\eta)\\
J^S_{\e,0} (\chi^0_{R(\e)}\eta)& = & Pal_{\e,0}(\chi^0_{R(\e)}\eta),
\eeastar
and the weight $\a_{\d,\e}(\t)$ restricting on the support of
$\chi^-_{R(\e)}$, $\chi^0_{R(\e)}$, $\chi^+_{R(\e)}$ agrees with the
weights of $u^\e_-, u^\e_0$ and $u^\e_+$ respectively (More
precisely, they agree in the sense that their ratio remains in the
finite interval $[\frac{1}{2},\frac{3}{2}]$), (\ref{Jbd}) follows by
taking $B=2$.
\end{proof}

\begin{lem} There exists a constant $C$ independent on $\e$ (but may
be dependent on $\d$) such that
\begin{multline} \Big[\int_{[0, 2R(\e)]\times S^1 }
\a_{\d,\e} \cdot \Big(|\nabla(Pal_{u^\e_{app}(\t,t)}(\xi^\e_{app}(
R(\e)))-V^+)|^p
\\+|Pal_{u^\e_{app}(\t,t)}(\xi^\e_{app}(R(\e)))-V^+|^p \Big) d\t
dt\Big]^{1\over p} \le C(\|\xi_-\|_{1,p,\a} + \|\xi_0\|_{1,p,\d}+
\|\xi_+\|_{1,p,\a})  \end{multline} and
\begin{multline} \Big[\int_{[ -2R(\e),0]\times S^1 }
\a_{\d,\e} \cdot \Big(|\nabla(Pal_{u^\e_{app}(\t,t)}(\xi^\e_{app}(
-R(\e)))-V^-)|^p
\\+|Pal_{u^\e_{app}(\t,t)}(\xi^\e_{app}(-R(\e)))-V^-|^p \Big) d\t
dt\Big]^{1\over p} \le C(\|\xi_-\|_{1,p,\a} + \|\xi_0\|_{1,p,\d}+
\|\xi_+\|_{1,p,\a}).
\end{multline}
\end{lem}
\begin{proof} We only prove the first inequality; the second one is
similar. Put
$$V(\t,t)=Pal_{u^\e_{app}(\t,t)}(\xi^\e_{app}(R(\e)))-Pal_{u^\e_{app}(\t,t)}V^+. $$
Then for $\t\in [R(\e)-S, R(\e)+S]$,
\bea
|V(\t,t)|&=&|\chi^-_{R(\e)+S}(\t)\big(Pal_{u^\e_{app}(\t,t)}(\xi_0(R(\e)))
-Pal_{u^\e_{app}(\t,t)}V^+\big)\nonumber\\
&+& \chi^+_{R(\e)-S}(\t)\big(Pal_{u^\e_{app}(\t,t)}(\xi_+(-R(\e)))
-Pal_{u^\e_{app}(\t,t)}V^+\big)| \nonumber\\
&\le&
|Pal_{u^\e_{app}(\t,t)}(\xi_0(R(\e)))-Pal_{u^\e_{app}(\t,t)}V^+|\nonumber \\
&+&|Pal_{u^\e_{app}(\t,t)}(\xi_+(-R(\e)))-Pal_{u^\e_{app}(\t,t)}V^+|.\label{TwoMiddle}
 \eea
\eqref{TwoMiddle} still holds outside $[R(\e)-S,R(\e)+S]$ by the
definition of $V(\t,t)$.

For the third row in \eqref{TwoMiddle}, using the invariance of
vector norm under parallel transport from the tangent space at
$u^\e_{app}(\t,t)$ in $M$ to the tangent space at $u^\e_0(R(\e))$
in $M$, we get
\bea %
&&|Pal_{u^\e_{app}(\t,t)}(\xi_0(R(\e)))-Pal_{u^\e_{app}(\t,t)}V^+|
\nonumber \\
&=& |\xi_0(R(\e))-Pal_{u^\e_0(R(\e))} Pal_{u^\e_{app}(\t,t)}V^+|\nonumber\\
&\le& |\xi_0(R(\e))-V^+|+ |V^+- Pal_{u^\e_0(R(\e))}
Pal_{u^\e_{app}(\t,t)}V^+| \label{GaussTriangle}.
\eea %
To estimate the second term of the last inequality, we only need
to consider the parallel transport of $V^+$ in the $S^{2n-1}$
component, since the $\R$ component has trivial connection. We
need to compare the difference of the parallel transport of
$V^+_{\CR_{\l}}$ along two different geodesic paths in $S^{2n-1}$:
one is from $\g_+(t)$ to $\Theta\circ u^\e_0(R(\e),t)$, the other
is from $\g_+(t)$ to $\Theta\circ u^\e_{app}(\t,t)$ and then to
$\Theta\circ u^\e_0(R(\e),t)$. Since parallel transport is
governed by a first order linear ODE, from the exponential
convergence of $\Theta\circ u^\e_{app}(\t,t)$ and $\Theta\circ
u^\e_0(R(\t,t)$ to $\g_+(t)$, and the $C^0$ continuous dependence
of solutions of ODE on its coefficients, we get
\be
|Pal_{u^\e_{app}(\t,t)}(\xi_0(R(\e)))-Pal_{u^\e_{app}(\t,t)}V^+|
\le  |\xi_0(R(\e))-V^+| + Ce^{\frac{-2\pi
c(R(\e)-|\t-R(\e)|)}{p}}|V^+| \label{GaussTri1}
\ee
Similar argument yields
\bea
|Pal_{u^\e_{app}(\t,t)}(\xi_+(-R(\e)))-Pal_{u^\e_{app}(\t,t)}V^+| &\le &
C|\xi_+(-R(\e))-V^+|\nonumber \\
&{}& \quad +Ce^{\frac{-2\pi c(R(\e)-|\t-R(\e)|)}{p}}|V^+|
. \label{GaussTri2}
\eea
It is obvious that \eqref{GaussTri1} and \eqref{GaussTri2} hold
when $\t$ is outside $[R(\e)-S,R(\e)+S]$, by definitions of the
cut functions. Plugging these in \eqref{TwoMiddle} we have the
point estimate
\bea |V(\t,t)|\le C(|\xi_0(R(\e))-V^+| & + &
|\xi_+(-R(\e))-V^+|)\nonumber \\
& + & Ce^{\frac{-2\pi c(R(\e)-|\t-R(\e)|)}{p}}|V^+|
\label{PV}\eea
 for all $\t$.

\medskip
Similarly we can estimate $|\nabla V(\t,t)|$, using the $C^1$
continuous dependence of solutions of ODE on its coefficients, and
the $C^1$ exponential convergence of $\Theta\circ
u^\e_{app}(\t,t)$ and $\Theta\circ u^\e_0(R(\t,t)$ to $\g_+(t)$.
We get
 \bea |\nabla V(\t,t)|&\le&
C(|\xi_0(R(\e))-V^+|
+|(\nabla\xi_0)(R(\e))-\nabla V^+|) \nonumber \\
&+& C(|\xi_+(-R(\e))-V^+| +|(\nabla\xi_+)(-R(\e))-\nabla V^+|)\nonumber  \\
&+& Ce^{\frac{-2\pi c(R(\e)-|\t-R(\e)|)}{p}}|V^+|. \label{PV1}\eea

Now we integrate \eqref{PV} and \eqref{PV1} on $[0,2R(\e)]\times
S^1$ to get the the $W^{1,p}_{\a_{\d,\e}}$ estimate of $V(\t,t)$. We
have
\bea\label{Sob}
&{ }& \int_{[0,2R(\e)]\times S^1} (|V(\t,t)|^p+|\nabla
V(\t,t)|^p)
\a_{\d,\e}(\t)d\t dt \nonumber \\
&\le& C\Big[|\xi_0(R(\e))-V^+|^p+
|\xi_+(-R(\e))-V^+|^p \nonumber\\
&{} & + |(\nabla \xi_0)(R(\e))-\nabla V^+|^p
+|(\nabla\xi_+)(-R(\e))-\nabla V^+|^p
\Big] \frac{e^{2\pi\d R(\e)}}{\d} \nonumber\\
&{} & + \frac{C}{c-\d}e^{-2\pi(c-\d)R(\e)}|V^+|^p \nonumber\\
&=& C  \Big[
\Big(|\xi_0(R(\e))-V^+|^p +
|(\nabla \xi_0)(R(\e))-\nabla V^+|^p\Big) \nonumber\\
&{} & + \Big(|\xi_+(-R(\e))-V^+|^p  + |(\nabla \xi_+)(-R(\e))-\nabla
V^+|^p\Big) \Big] \frac{e^{2\pi\d R(\e)}}{\d}\nonumber \\
&{}& + \frac{C}{c-\d}e^{-2\pi(c-\d)R(\e)}|V^+|^p \nonumber\\
&\le&
C\Big[ \int_{[R(\e)-1,R(\e)+1]\times S^1} \Big(|\xi_0(\t,t)-V^+|^p +
|(\nabla \xi_0)(\t,t)-\nabla V^+|^p \Big) e^{-2\pi\d (R(\e)-\t)}
d\t dt  \nonumber\\
&{}& + \int_{[-R(\e)-1,-R(\e)+1]\times S^1}  \Big(|\xi_+(\t,t)-V^+|^p
+ |(\nabla \xi_+)(\t,t)-\nabla V^+|^p \Big) e^{-2\pi\d (R(\e)+\t)}
d\t dt \Big] \nonumber\\
&{} & \times  \frac{e^{2\pi\d R(\e)}}{\d} \qquad+\qquad  \frac{C}{c-\d}e^{-2\pi(c-\d)R(\e)}|V^+|^p
\nonumber\\
&\le& C \Big[\int_{\R\times S^1}\Big(|\xi_0(\t,t)-V^+|^p + |(\nabla
\xi_0)(\t,t)-\nabla V^+|^p \Big) e^{-2\pi\d |\t|} d\t dt  \nonumber\\
&{}& +\int_{\R\times S^1}\Big(|\xi_+(\t,t)-V^+|^p + |(\nabla
\xi_+)(\t,t)-\nabla V^+|^p \Big) \a(\t) d\t dt\Big]\nonumber\\
&{}& + \frac{C}{c-\d}e^{-2\pi(c-\d)R(\e)}|V^+|^p
\eea
where in \eqref{Sob} we have used Sobolev embedding
$W^{1,p}\hookrightarrow C^0$, and that $e^{2\pi\d(R(\e)- \t)}$
restricted on $[ R(\e)-1, R(\e)+1]$ (or $e^{2\pi\d(R(\e)+ \t)}$
restricted on $[ -R(\e)-1, -R(\e)+1]$) is bounded between
constants $e^{-2\pi \d}$ and $e^{2\pi \d}$ independent on $\e$.

Hence
\begin{multline} \Big[\int_{[0, 2R(\e)]\times S^1 }
\a_{\d,\e} \cdot \Big(|\nabla(Pal_{u^\e_{app}(\t,t)}(\xi^\e_{app}(
R(\e)))-V^+)|^p
\\+|Pal_{u^\e_{app}(\t,t)}(\xi^\e_{app}(R(\e)))-V^+|^p \Big) d\t
dt\Big]^{1\over p} \le C(\|\xi_-\|_{1,p,\a} + \|\xi_0\|_{1,p,\d}+
\|\xi_+\|_{1,p,\a}). \nonumber \end{multline} The lemma follows.
\end{proof}

To show $Q^\e$ is an approximate right inverse, we start with the
following lemma concerning the ``uniform stabilization" property of
the action of $Q_0$ and $Q_{\pm}$ on compactly supported (or one
side compact supported) 1-forms $\eta$ :
\begin{lem}\label{uniStb} There exist a constant $C$ independent on
$\d$ and $\e$ such that
\begin{enumerate} \item  If  $\eta_0 \in
L^p_\d(\Lambda^{0,1}(u_{0})^*\C^n)$,  then for any
$\xi_0=Q_0(\eta_0)$ with asymptote $V^{\pm}$ on its two ends, we
have \bea |\xi_0(\t,t)-V^+|\le C e^{\frac{-2\pi
\d|\t|}{p}}\|\eta_0\|_{p,\d} &\text{ for } & \t>1
,\\
|\xi_0(\t,t)-V^-|\le C e^{\frac{-2\pi \d|\t|}{p}} \|\eta_0\|_{p,\d}
 &\text{ for } & \t<-1 , \eea
\item  If $\eta_{\pm} \in  L^p_{\a}(\Lambda^{0,1}(u_{\pm})^*TM)$
, then for any $\xi_{\pm}=Q_{\pm}(\eta_{\pm})$ with the asymptote
$V^{\pm}$ as $\t\to \pm\infty$, we have
\bea |\xi_+(\t,t)-V^+|\le C
e^{\frac{- 2\pi \d|\t|}{p}}\|\eta_+\|_{p,\a}  &\text{ for } & \t<-1 , \label{stb3}  \\
|\xi_-(\t,t)-V^-|\le C e^{\frac{-2\pi \d|\t|}{p}}\|\eta_-\|_{p,\a}
&\text{ for } & \t>1 , \eea
where $\a$ is the weighting function we introduced in subsection
\ref{subsec:offshellpertdiscs}.
\end{enumerate}
\end{lem}
\begin{proof} The proofs of the inequalities are similar, so we just
prove \eqref{stb3}. By definition $|\xi_+ -V^+| \in
W^{1,p}_\a(u_+^*(TM))$. Suppose $\t\in (L-1,L+1) \subset (-\infty,0]
$ for some $L$, then $|\xi_+(\t,t)-V^+|e^{\frac{2\pi\d |\t|}{p}} \in
W^{1,p}([L-1,L+1]\times S^1, u_+^*(TM)) $. By Sobolev embedding
$C^0([L-1,L+1]\times S^1, u_+^*(TM)) \hookrightarrow
W^{1,p}([L-1,L+1]\times S^1, u_+^*(TM)) $, there is a constant $C$
independent of $L$ and
depending only on the metric $g$ on the compact $M$, such that
\bea |\xi_+(\t,t)-V^+| e^{\frac{2\pi\d |\t|}{p}} &\le &  C\cdot
\|(\xi_+(\t,t) -V^+) e^{\frac{2\pi\d
|\t|}{p}}\|_{W^{1,p}([L-1,L+1]\times S^1, u_+^*(TM))  } \nonumber  \\
 { }&\le & C\|\xi_+\|_{1,p,\a} \quad (\text{See (\ref{eq:Vnorm})} \nonumber \\
{ }&\le & C \|Q_+\|\|\eta_+\|_{p,\a} \quad (\text{By Proposition }
\ref{unibdInv}  ).\nonumber \eea
Hence $$|\xi_+(\t,t)-V^+|\le C e^{\frac{- 2\pi
\d|\t|}{p}}\|\eta_+\|_{p,\a}.  $$
\end{proof}

The following lemma concerns the commutativity of the operator
$D\delbar$ with the operators $I^S$ and $J^S_{*,\e}$:
\begin{lem} \label{commuteD} For any $\eta\in L^p_{\a_{\d,\e}}(\Lambda^{0,1}(u^\e_{app})^*TM)$ and
the corresponding $\xi_{\e,-}, \xi_{\e,0}$ and $\xi_{\e,+}$ defined
in (\ref{defxi}), and $S=\frac{1}{4}R(\e)$, we have
\begin{multline}\label{commuteDe}
\|(D_{u^\e_{app}}\delbar\circ I^S) (\xi_{\e,-}, \xi_{\e,0},
\xi_{\e,+})\\
- ( (J^S_{-,\e}+J^S_{0,\e} +J^S_{+,\e})\circ
D_{(u_-,u_0,u_+)}\delbar (\xi_{\e,-}, \xi_{\e,0}, \xi_{\e,+})
)\|_{p,\a_{\d,\e}}\\
\le C\big(e^{\frac{-4\pi \d S}{p}}+ dist(u^\e_{app},u^\e_-)+
dist(u^\e_{app},u^\e_+)+ dist(u^\e_{app},u^\e_0)\big)
\|\eta\|_{p,\a_{\d,\e}},
\end{multline}
where $C$ is a constant independent on $\d$ and $\e$.
\end{lem}
\begin{proof} There are three reasons why $I$ and $J$ do not commute
with $D_{u^\e_{app}}\delbar$ and $D_{(u_-,u_0,u_+)}\delbar$. One is
that we use the parallel transport along the minimal geodesics from
$u^\e_{app}(\t,t)$ to $u^\e_0(\t,t)$ and $u^\e_{\pm}(\t,t)$ and vice
versa. The terms caused by parallel transport are estimated by the
second, third and the forth terms of the right hand side of the
inequality.

The second reason is that on $u_0$ we use $D_{u_0}\delbar_{J_p}$
while on the $\e u_0$ portion of $u^\e_{app}$ we use $D_{\e
u_0}\delbar_{J}$. The deviation of $J$ from $J_p$ is controlled by
dist($\e u_0(\t,t)$, $p$). The estimate of \eqref{commuteDe} on this
part is similar to the $\delbar u^\e_{app}$ error estimate we have
carried out in Subsection \ref{dbar-error}, which takes care of the
deviation of $J$ from $J_p$, which we do not repeat. This contribution
is of the order $Ce^{\frac{-4\pi\d S}{p}}\|\eta\|_{p,\a_{\d,\e}}$,
if we choose  $\d <p \min\{\frac{1}{3},\frac{c_0}{2p}\}$ in the
beginning (see \eqref{coin} for the relevant estimate).

The third and more essential point is that we have used the cut-off
functions. We need to control the terms caused by ${\chi^{\pm}_{\pm
R(\e)\pm S}}'(\t)(Pal_{\pm,\e}(\xi_{\e,\pm})-V^{\pm})$ and
${\chi^{\pm}_{\pm R(\e)\pm
S}}'(\t)(Pal_{0,\e}(\xi_{\e,0})-V^{\pm})$. By the definition of the
cut-off functions, these terms are supported in
$$\pm[R(\e)\pm
S-1,R(\e)\pm S+1] \times S^1 \subset \R \times S^1.$$ For $(\t,t)$
in these regions, by Lemma \ref{uniStb}, these terms are controlled
by $Ce^{-4\pi \d S}\|\eta\|_{p,\a_{\d,\e}}$. For example,
\beastar
&{  }& |{\chi^+_{R(\e)-S}}'(\t)(Pal_{+,\e}(\xi_{\e,+}(\t,t))-V^+)| \nonumber \\
&\le& C e^{\frac{-2\pi \d|\t-2R(\e)|}{p}}\|Q_+\circ
J^S_{\e,+}(\chi^+_{R(\e)}\eta)\|_{p,\a} \qquad (\text{Lemma}
\ref{uniStb})
\nonumber \\
&\le& C e^{\frac{-2\pi \d \cdot (R(\e)+S)}{p}}\|Q_+\| \|
J^S_{\e,+}(\chi^+_{R(\e)}\eta)\|_{p,\a}
\nonumber \\
&\le& C e^{\frac{-2\pi \d \cdot (R(\e)+S)}{p}}\|Q_+\|
\|\eta\|_{p,\a_{\d,\e}}
 \eeastar
where in the last inequality we have used \eqref{Jbd}.

On the other hand, the weight $\a_{\d,\e}$ on the support of $
{\chi^+_{R(\e)-S}}'(\t)$ is $e^{2\pi\d |\t| }\le  e^{2\pi \d \cdot
(R(\e)-S+1)}$. Therefore the $L^p_{\a_{\d,\e}}$ contribution from
these terms is no more than
$$C e^{-2\pi \d(R(\e)+S) } e^{2\pi\d(R(\e)-S+1)}
\|\eta\|_{p,\a_{\d,\e}}\le Ce^{-4\pi \d S}\|\eta\|_{p,\a_{\d,\e}}$$
The proposition follows.

\end{proof}

With the above lemmas we can prove that $Q^\e$ is an approximate
right inverse:
\begin{prop} For sufficiently small $\e>0$,
\be \|(D_{u^{\e}_{app}}\delbar\circ Q^{\e})\eta-\eta
\|_{p,{\a_{\d,\e}}} \le \frac{1}{2} \|\eta\|_{p,{\a_{\d,\e}}} \ee
for all $\eta\in L^p_{\a_{\d,\e}}(\Lambda^{0,1}(u^\e_{app})^*TM)$.
\end{prop}
\begin{proof} From the definition of $Q^\e$, and Lemma
\ref{commuteD} we have
 \beastar & {} & |(D_{u^{\e}_{app}}\delbar\circ
Q^{\e})\eta-\eta
\|_{p,{\a_{\d,\e}}}\\
&=& \|(D_{u^\e_{app}}\delbar\circ I^S)
(\xi_{\e,-}, \xi_{\e,0}, \xi_{\e,+})-\eta\|_{p,{\a_{\d,\e}}}\\
&\le& \| (J^S_{-,\e}+J^S_{0,\e} +J^S_{+,\e})\circ
D_{(u_-,u_0,u_+)}\delbar \circ(Q_- J^S_{\e,-}\chi^-_{R(\e)}\eta, Q_0
J^S_{\e,0}\chi^0_{R(\e)}\eta,Q_+ J^S_{\e,+} \chi^+_{R(\e)}\eta)\\
&&-\eta \|_{p,{\a_{\d,\e}}}
+ o(S)\|\eta\|_{p,{\a_{\d,\e}}}\\
&=&\| (J^S_{-,\e}+J^S_{0,\e} +J^S_{+,\e})\circ(
J^S_{\e,-}\chi^-_{R(\e)}\eta,J^S_{\e,0}\chi^0_{R(\e)}\eta,
J^S_{\e,+}\chi^+_{R(\e)}\eta)-\eta\|_{p,{\a_{\d,\e}}}+o(S)\|\eta\|_{p,{\a_{\d,\e}}}\\
&=&\|\chi^-_{R(\e)}\eta + \chi^0_{R(\e)}\eta +\chi^+_{R(\e)}\eta
-\eta\|_{p,{\a_{\d,\e}}}+o(S)\|\eta\|_{p,{\a_{\d,\e}}} \\ &=&
o(S)\|\eta\|_{p,{\a_{\d,\e}}},
 \eeastar
where $o(S)$ is a term going to 0  when $S\to \infty$, and the
second to last identity is due to (\ref{transID}) and
$\chi^-_{R(\e)}+ \chi^0_{R(\e)}+ \chi^+_{R(\e)}=1$.

When $\e$ is sufficiently small, $S=\frac{1}{4}R(\e)$ is very
large and we get
\be%
\|(D_{u^{\e}_{app}}\delbar\circ
Q^{\e})\eta-\eta \|_{p,{\a_{\d,\e}}} \le \frac{1}{2}
\|\eta\|_{p,{\a_{\d,\e}}}.
\label{apprQ}
\ee%
\end{proof}

By \eqref{apprQ}, $D_{u^{\e}_{app}}\delbar\circ Q^{\e}$ is
invertible, and
$$
\|(D_{u^{\e}_{app}}\delbar\circ
Q^{\e})^{-1}\|=\|\Sigma_{k=0}^{\infty}(D_{u^{\e}_{app}}\delbar\circ
Q^{\e}-id)^k\|\le \Sigma_{k=0}^{\infty}\frac{1}{2^k} = 1.
$$
So we can construct the true right inverse for
$D_{u^{\e}_{app}}\delbar$ to be
$Q^{\e}\circ(D_{u^{\e}_{app}}\delbar\circ Q^{\e})^{-1}$. For
convenience we still  denote it by $Q^\e$. From its construction
and proposition \ref{unibdInv} we see $\|Q^\e\|$ is bounded by a
uniform constant $C$ for all $\e>0$.

\begin{lem} For all $\e>0$, and $\xi \in
T_{u^\e_{app}}\CB^\e_{res}(z_-,z_+)$, we have the uniform Sobolev
inequality
\be%
|\xi|_{\infty}\le C_p \|\xi\|_{1,p,\a_{\d,\e}} \label{GSob},
\ee%
where the constant $C_p$  is independent on $\e$.
\end{lem}
\begin{proof} Since the base variation term $|v|$  appears on both sides of the
above inequality, we can assume $v=0$ for $\xi$.
%namely, we only need to prove the above inequality in $(M,g_{\e_0,\e,p})$.%
For the maps $u^\e_{app}: \R\times S^1\to (M,g_{\e_0,\e,p})$,  the
domain $\R\times S^1$ with standard metric is noncollapsing; The
targets $(M,g_{\e_0,\e,p})$ for all $0<\e<\e_0$ and $p\in M$ is a
family of noncollasping Riemannian manifolds. Therefore, the
Sobolev constant
$$
c_p(u^\e_{app}):=\sup_{0\neq\xi\in\G((u^\e_{app})^*TM)}
\frac{|\xi|_{L^{\infty}}}{\|\xi\|_{W^{1,p}} }
$$
is uniformly
bounded above for all $\e$, where $\G((u^\e_{app})^*TM)$ is the
set of all $C_0^{\infty}$ sections of $(u^\e_{app})^*TM$. So we
have a uniform constant $c_p$ such that $|\cdot|_{L^\infty}\le c_p
\|\cdot\|_{1,p}$. Certainly this is still true if we change the
$W^{1,p}$ norm to positive exponential weighted norm:
$|\cdot|_{L^\infty}\le c_p \|\cdot\|_{1,p,\d}$, where the weight
$\d(\t,t)$ is $e^{2\pi\d |\t|}$ on $\R\times S^1$.

However, our norm $\|\cdot\|_{1,p,\a_{\d,\e}}$ is not the usual
weighted Sobolev norm, because we first take out the ``Morse-Bott
variation" of $\xi$ and then measure the remaining part by the
weighted Sobolev norm. To get the Sobolev inequality, notice that
$$
|\xi(\t,t)|\le
|\xi-\b_{\e}^{\pm}(\t)Pal_{u^\e_{app}(\t,t)}Pal_{u^\e_{app}(\pm
R(\e),t)}\xi(\pm R(\e),0)|+ |\xi(\pm R(\e),0)|.
$$
Apply $|\cdot|_{L^\infty}\le c_p \|\cdot\|_{1,p,\d}$ to the first
term on the right side of the above inequality, and recall the
definition of the norm $\|\cdot\|_{1,p,\a_{\d,\e}}$, then we get
$$|\xi|_{\infty}\le \max\{c_p,1\} \|\xi\|_{1,p,\a_{\d,\e}}.$$
Letting $C_p=\max\{c_p,1\}$, the lemma follows.
\end{proof}

\begin{prop}\label{contDu} For every $u^\e_{app}\in \CB^\e_{res}(z_-,z_+)$,
there exist constants $h_\e=Ke^{-2\pi\d R(\e)}$, where the
constant $K$ is independent on $\e$, such that for every $0<\e\le
\e_0$, and every $\xi\in T_{u^\e_{app}}(\CB^\e_{res}(z_-,z_+))$
with $\|\xi\|_{L^{\infty}}\le h_\e $, we have
\be \label{quadraticestimate}%
\|d\CF_{u^\e_{app}}(\xi)\xi'-(D_{u^\e_{app}}\delbar_{(J_\e,K_\e)})\xi'
\|_{p,\a_{\d,\e}}\le \frac{1}{2C} \|\xi'\|_{1,p,\a_{\d,\e}}
\ee%
for all $\xi'\in T_{u^\e_{app}}(\CB^\e_{res}(z_-,z_+))$. Here $C$
is the uniform bound for $\|Q^\e\|$.
%Moreover,
%when $\|\xi\|_{L^{\infty}}\le c_0 $ and $\|\xi'\|_{L^{\infty}}\le
%c_0 $, we have
%\be%
%\|\CF_{u^\e_{app}}(\xi+\xi') -\CF_{u^\e_{app}}(\xi)
%-d\CF_{u^\e_{app}}(\xi)\xi'\| \le
%C\|\xi\|_{L^{\infty}}\|\xi'\|_{1,p,\a_{\d,\e}}
%\ee%
\end{prop}
\begin{proof} The proof is a variation of the proposition
3.5.3 in \cite{mcd-sal04}. The point estimate is the same as
\cite{mcd-sal04}. The main differences are that our norm
$\|\cdot\|_{1,p,\a_{\d,\e}}$ is not the usual $W^{1,p}$ norm in
\cite{mcd-sal04}, and our target manifold $(M,g_{\e_0,\e,p})$ is
stretching when $\e\to 0$.

Let $\CF_{u^\e_{app}}: T_{u^\e_{app}}\CB^\e_{res}(z_-,z_+)\to
L^p_{\a_{\d,\e}}(\Lambda^{0,1}(u^\e_{app})^*TM ) )$ be
$$
\CF_{u^\e_{app}}(\xi)=(\Phi_{u^\e_{app}}(\xi))^{-1}(\delbar_{(J_{\e,K_\e})}(\exp_{u^\e_{app}}\xi
)),
$$
where $\Phi_{u^\e_{app}}(\xi):(u^\e_{app})^*TM\to
(\exp_{u^\e_{app}}(\xi))^*TM \;$ is the parallel transport in
$(M,g_{\e_0,\e,p})$ along the geodesics $s\to
exp_{u^\e_{app}(z)}(s\xi(z))$. Then the differential of
$\CF_{u^\e_{app}}$ satisfies
$\CF_{u^\e_{app}}(0)=D_{u^\e_{app}}\delbar_{(J_\e,K_\e)}.$

For each $\e>0$, $(M,g_{\e_0,\e,p})$ is a compact Riemannian
manifold. The point estimate in the proof of proposition 3.5.3 in
\cite{mcd-sal04} yields
\be%
|d\CF_{u^\e_{app}}(\xi)\xi'-D_{u^\e_{app}}\delbar_{(J_\e,K_\e)}
\xi'|\le A (|du||\xi||\xi'|+|\nabla \xi||\xi'|+ |\xi||\nabla
\xi'|). \label{pointDu}
\ee%
Here the constant $A>0$ is determined by the Sobolev constant
$C_p$, hence $A$ is \emph{uniform} for all $\e>0$.

By our construction of $u^\e_{app}$, there exists a uniform
constant   $B$ for all $\e>0$, such that
$|du^\e_{app}(\t,t)|_{g_{\e_0,\e,p}} \le B$ for all $(\t,t)\in
\R\times S^1$. We consider three cases for $\xi$:

{\bf Case 1}: $\xi(\pm R(\e),0)=0$, i.e. there is no Morse-Bott
variation. In this case,  $\|\xi\|_{1,p,\a_{\d,\e}}$ is the usual
weighted Sobolev norm. Multiplying $e^{\frac{2\pi\d|\t|}{p}}$ to
both sides of \eqref{pointDu} and taking the $L^p$ integration
over $\R\times S^1$, we get
\beastar%
&{ }
&\|d\CF_{u^\e_{app}}(\xi)\xi'-D_{u^\e_{app}}\delbar_{(J_\e,K_\e)}
\xi'\|_{p,\a_{\d,\e}}\\
&\le & A \left(B \|\xi\|_{1,p,\a_{\d,\e}}|\xi'|_{\infty}+
\|\xi\|_{1,p,\a_{\d,\e}}|\xi'|_{\infty}+|\xi|_{\infty}\|\xi'\|_{p,\a_{\d,\e}}\right)\\
&\le & A(B+2C_p)\|\xi\|_{1,p,\a_{\d,\e}}\cdot C_p
\|\xi'\|_{p,\a_{\d,\e}},
\eeastar%
where in the last inequality we have used the Sobolev inequality
\eqref{GSob}. The proposition is proved by taking
$h_\e=(AC_p(B+2C_p))^{-1}\frac{1}{2C}$.

{\bf Case 2}:
$\xi(\t,t)=\b^{\pm}_\e(\t)Pal_{u^\e_{app}(\t,t)}Pal_{u^\e_{app}(\pm
R(\e),t)}\xi(\pm R(\e),0)$, i.e. $\xi$ is purely induced from the
Morse-Bott variation. In this case,
$\|\xi\|_{1,p,\a_{\d,\e}}=|\xi(R(\e),0)|$, and by construction
$|\xi|_\infty \le \|\xi\|_{1,p,\a_{\d,\e}}$.  We also have
\beastar%
|\nabla\xi|_\infty&=&|\nabla\left(
\b^{\pm}_\e(\t)Pal_{u^\e_{app}(\t,t)}Pal_{u^\e_{app}(\pm
R(\e),t)}\xi(\pm R(\e),0)  \right)|_\infty\\
&\le& |\xi(\pm R(\e),0)|+ |du^\e_{app}|_{\infty}\cdot |\xi(\pm
R(\e),0)|\\
&\le& (1+B)|\xi(\pm R(\e),0)|\\
&=&(1+B)\|\xi\|_{1,p,\a_{\d,\e}}.
\eeastar%
Therefore at any $(\t,t)$,
\beastar%
& & |du||\xi||\xi'|+|\nabla \xi||\xi'|+ |\xi||\nabla \xi'|\\
&\le&
B\|\xi\|_{1,p,\a_{\d,\e}}|\xi'|_\infty+(1+B)\|\xi\|_{1,p,\a_{\d,\e}}|\xi'|_\infty
+\|\xi\|_{1,p,\a_{\d,\e}}|\nabla \xi'|
\eeastar%
Multiplying $e^{\frac{2\pi\d|\t|}{p}}$ to both sides of the above
inequality and taking the $L^p$ integration over $\R\times S^1$,
and noticing that $\xi$ is supported in $\t \in [0,\pm 2R(\e)]$,
we get
\beastar%
&{ }
&\|d\CF_{u^\e_{app}}(\xi)\xi'-D_{u^\e_{app}}\delbar_{(J_\e,K_\e)}
\xi'\|_{p,\a_{\d,\e}}\\
&\le& A(1+2B)\|\xi\|_{1,p,\a_{\d,\e}} |\xi'|_\infty\cdot
\left(\int_{[0,\pm 2R(\e)]\times S^1}e^{2\pi\d|\t|}d\t
dt\right)^{\frac{1}{p}}\\
&+& A\|\xi\|_{1,p,\a_{\d,\e}}\left(\int_{\R \times
S^1}e^{2\pi\d|\t|} |\nabla \xi'| d\t dt\right)^{\frac{1}{p}}\\
&\le& A(1+2B) \|\xi\|_{1,p,\a_{\d,\e}} \cdot C_p
\|\xi'\|_{1,p,\a_{\d,\e}}\cdot (\frac{1}{2\pi\d})^{1\over p }\cdot
e^{\frac{4\pi\d R(\e)}{p}} +A\|\xi\|_{1,p,\a_{\d,\e}}
\|\xi'\|_{1,p,\a_{\d,\e}}\\
&=&A\left(\frac{(1+2B)C_p}{(2\pi\d)^{1\over p}}e^{\frac{4\pi\d
R(\e)}{p}}+1\right)\|\xi\|_{1,p,\a_{\d,\e}}
\|\xi'\|_{1,p,\a_{\d,\e}}\\
&\le& K(A,B,\d,C_p) e^{\frac{4\pi\d
R(\e)}{p}}\|\xi\|_{1,p,\a_{\d,\e}} \|\xi'\|_{1,p,\a_{\d,\e}},
\eeastar%
where $K(A,B,\d,C_p)$ is a constant independent on $\e$.
Therefore, the proposition is proved by taking
$h_\e=(K(A,B,\d,C_p))^{-1}\frac{1}{2C}\cdot e^{-\frac{4\pi\d
R(\e)}{p}}$.

{\bf Case 3}: For general $\xi$. We can write $\xi=\xi_1+\xi_2$,
where $\xi_1$ is in case 1 and $\xi_2$ is in case 2. Then we can
apply triangle inequality on the terms involving $\xi$ in the
point estimate \eqref{pointDu}, and then $L^p$ integrate the point
estimate. The proof reduces to case 1 and case 2.

Combining the three cases, there exists a constant $K$ independent
on $\e$, such that for $h_\e:=K e^{-\frac{4\pi\d R(\e)}{p}}$ and
$\|\xi\|_{1,p,\a_{\d,\e}}\le h_\e$,
$$
\|d\CF_{u^\e_{app}}(\xi)\xi'-(D_{u^\e_{app}}\delbar_{(J_\e,K_\e)})\xi'
\|_{p,\a_{\d,\e}}\le \frac{1}{2C} \|\xi'\|_{1,p,\a_{\d,\e}}.
$$
\end{proof}

\begin{rem} In our setting, for each $\e$, the almost complex structure
$J_\e$ is $(\t,t)$-dependent while in \cite{mcd-sal04} it is not.
But since $J(\t,t)\equiv J_0$ for $\t\in[-R(\e),R(\e)]$, and
$J(\t,t)\equiv J(t)$ for $|\t|>R(\e)+1$, our $J_{\e}(\t,t)$ is
actually a compact family of almost complex structures (smoothly
parameterized by $\pm [R(\e),R(\e)+1]\times S^1$). Therefore, the
proof in the compact family of $J_\e(\t,t)$ case is the same as
the fixed $J$ case (see the remark 3.5.4 in \cite{mcd-sal04}).
\end{rem}

\begin{rem}
Unlike \cite{mcd-sal04}, in our case
$\|du^\e_{app}\|_{p,\a_{\d,\e}}\to \infty$, and we only have
$|du^\e_{app}|_\infty\le B$; In the case 2,
$|\xi|_{p,\a_{\d,\e}}\to \infty$, and we only have
$|\xi|_\infty\le |\xi(R(\e),0)|$. The loss of the exponential
decay of $du^\e_{app}$ and $\xi$ is caused by the stretching of
the target manifold $(M,g_{\e_0,\e,p})$ when $\e\to 0$. This is
the reason that our estimate \eqref{quadraticestimate} is weaker
than that in \cite{mcd-sal04}, where the latter is on a
\emph{fixed} compact Riemannian manifold.
\end{rem}

For gluing we need the following abstract implicit function
theorem in \cite{mcd-sal04}:
\begin{prop} Let $X, Y$ be Banach spaces and $U$ be an
open set in $X$. The map $f: X\to Y$ is continuous differentiable.
For $\x_0\in U$, $D:=df(x_0): X\to Y$ is surjective and has a
bounded linear right inverse $Q:Y\to X$, with $\|Q\|\le C$.
Suppose that there exists $h>0$ such that for all $x\in
B_{h}(x_0)\subset U$,
$$\|df(x)-D\|\le \frac{1}{2C}.$$
Then for $\|f(x_0)\|\le \frac{h}{4C}$, there exists a unique $x\in
B_{h}(x_0)$ such that
$$f(x)=0, \qquad x-x_0\in \text{Image} Q, \qquad \|x-x_0\|\le 2C\|f(x_0)\|. $$
\end{prop}

\medskip

Now we apply the above implicit function theorem in the following
setting:
 $$
 X=T_{u^\e_{app}}\CB^\e_{res}(z_-,z_+),\quad Y=L^p_{\a_{\d,\e}}(\Lambda^{0,1}(u^\e_{app})^*TM), \quad
 f=\CF_{u^\e_{app}}, \quad x_0=0.
 $$
Then from proposition \ref{contDu} we have
$$
  \|d\CF_{u^\e_{app}}(\xi)-d\CF_{u^\e_{app}}(0)\|\le \frac{1}{2C}
$$
for $\xi$ in $X$ with $\|\xi\|_{1,p,\a_{\d,\e}}\le h_\e:= K
e^{-\frac{4\pi\d R(\e)}{p}}$.  From \eqref{dbarErrorEstmt}, we
have
$$
\|\CF_{u^\e_{app}}(0)\|\le L\cdot (R(\e))^{\frac{1}{p}}\cdot
e^{-\frac{4\pi a R(\e)}{p}},
$$
hence for $\e$ small,
$$
\|\CF_{u^\e_{app}}(0)\|\le \frac{1}{4C} \cdot K e^{-\frac{4\pi\d
R(\e)}{p}} =\frac{h_\e}{4C}.
$$
Here we have used our choice of $0<\d<a$ in the beginning, so
$(R(\e))^{\frac{1}{p}}\cdot e^{-\frac{4\pi a R(\e)}{p}} $ decays
faster than $e^{-\frac{4\pi\d R(\e)}{p}}$.  By the above abstract
implicit function theorem we have finished the gluing and prove
the following theorem, which is a half of Theorem
\ref{compactify}.
\begin{thm}\label{embedding} Let $(K_{\e},J_{\e})$ be the family of
Floer data defined in (\ref{eq:KR}). Then
\begin{enumerate}
\item there exists a topology on
$\CM_{(0;1,1)}^{para}([z_-,w_-],[z_+,w_+]);\{(K_{\e},J_{\e})\}$ with
respect to which the gluing construction defines a proper embedding
\beastar
Glue& : &(0,\e_0) \times \CM^{nodal}_{(0;1,1)}([z_-,w_-],[z_+,w_+];(H,J),(f,J_0))\\
& \to & \CM_{(0;1,1)}^{para}([z_-,w_-],[z_+,w_+]);\{(K,J)\})
\eeastar for sufficiently small $\e_0$.
\item the above mentioned topology can be compactified into
$$
\overline \CM_{(0;1,1)}^{para}([z_-,w_-],[z_+,w_+]);\{(K,J)\})
$$
where $\overline
\CM_{(0;1,1)}^{para}([z_-,w_-],[z_+,w_+]);\{(K,J)\})$ is given by
\beastar
&{}& \overline \CM_{(0;1,1)}^{para}([z_-,w_-],[z_+,w_+]);\{(K,J)\})\\
& =& \bigcup_{0 < \e \leq \e_0}\CM_{(0;1,1)}([z_-,w_-],[z_+,w_+]);\{(K_{\e},J_{\e})\}\\
&{}& \quad \cup
\CM^{nodal}_{(0;1,1)}([z_-,w_-],[z_+,w_+];(H,J),(f,J_0)) \eeastar as
a set,
\item
the embedding $Glue$ smoothly extends to the embedding \beastar
\overline{Glue} & : & [0,\e_0) \times \CM^{nodal}_{(0;1,1)}([z_-,w_-],[z_+,w_+];(H,J),(f,J_0))\\
& \to & \overline
\CM^{para}_{(0;1,1)}([z_-,w_-],[z_+,w_+]);\{(K,J)\}) \eeastar that
satisfies
$$
\overline{Glue}(u_+,u_-,u_0;0) = Glue(u_+,u_-,u_0).
$$
\end{enumerate}
\end{thm}
\medskip

For $0 < \e \leq \e_0$, we denote by
$$
Glue(u_+,u_-,u_0;\e) \in
\CM_{(0;1,1)}^{para}([z_-,w_-],[z_+,w_+]);\{(K_{\e},J_{\e})
$$
the image of
$$
(u_+,u_-,u_0;\e) \in [0,\e_0) \times
\CM^{nodal}_{(0;1,1)}([z_-,w_-],[z_+,w_+];(H,J),(f,J_0))
$$
under the embedding $Glue$.

In the rest of the paper, we will prove that $\overline{Glue}$ is
surjective onto an open neighborhood of
$\CM^{nodal}_{(0;1,1)}([z_-,w_-],[z_+,w_+];(H,J),(f,J_0))$ in
$$
\overline \CM_{(0;1,1)}^{para}([z_-,w_-],[z_+,w_+]);\{(K,J)\}).
$$

\section{Adiabatic degeneration : analysis of the thin part}

In this section, we consider a one-parameter family $(K_\e,J_\e)$ as
provided in subsection \ref{subsec:resolved} with $R = R(\e) \to
\infty$, \be\label{eq:length} \e R(\e) \to \ell \ee with $\ell \geq
0$ as $\e \to 0$. Motivated by the gluing construction in the
previous section, we will be particularly interested in the case
where $\ell = 0$, e.g.,
$$
R(\e) = - \frac{\log \e}{4\pi}.
$$
We recall $\delta_\e$ satisfies
\be\label{eq:delta/e}
\delta_\e/\e \to \infty \quad \mbox{as } \, \e \to 0.
\ee
We use the Hamiltonian defined by
\be\label{eq:KRe}
K_\e(\tau,t) = \begin{cases}\kappa^+_\e(\t) \cdot H_t \quad  & \tau \geq R(\e)\\
\rho_\e(\tau)\cdot \e f \quad  &  |\tau| \leq R(\e) \\
\kappa^-_\e(\tau)\cdot H_t \quad  & \tau \leq - R(\e).
\end{cases}
\ee The $K_\e(\tau,t)$ was defined before in (\ref{eq:KR}),
where $\kappa^+_\e(\t)$ was defined in (\ref{eq:betaR})and
$\rho_\e(\tau)$ was defined in (\ref{eq:chi}). We then study the family of equation \be\label{eq:duPK} (du +
P_{K_{R(\e)}}(u))^{(0,1)}_J = 0 \ee as $\e \to 0$. For the
simplicity of notation, we denote $K_\e(\t,t,x) = K_\e$. By
definition of $K_\e$ and $J_\e$, as $\e \to 0$, on the domain
$$
[-R(\e)+1,R(\e)-1] \times S^1
$$
we have $K_\e(\tau,t) \equiv \e f$ and $J_R(\tau,t) \equiv J_0$,
and so (\ref{eq:duPK}) becomes
$$
\frac{\del u}{\del \tau} + J_0\left(\frac{\del u}{\del t} - \e
X_f(u) \right) = 0.
$$
Furthermore $K_\e(\tau,t) \equiv H_t\, dt$, $J_R(\tau,t) \equiv J_t$
on
$$
\R \times S^1 \setminus [-R(\e) +1, R(\e)-1] \times S^1
$$
(\ref{eq:duPK}) is cylindrical at infinity, i.e., invariant under the
translation in $\tau$-direction at infinity.

Note that on any fixed compact set $B \subset  \R \times S^1$, we will have
$$
B \subset [-R(\e),R(\e)]\times S^1
$$
for all sufficiently small $\e$. And as $\e \to 0$,
$K_\e \to 0$ on $B$ in $C^\infty$-topology, and hence the equation
(\ref{eq:duPK}) converges to $\delbar_{J_0} u = 0$ on $B$ in that $J
\to J_0$ and $K_\e \to 0$ in $C^\infty$-topology. On the other hand,
after translating the region $(-\infty, -(R - \frac{1}{3}]$ to the
right (resp. $[R -\frac{1}{3},\infty)$ to the left)  by $2R -
\frac{2}{3}$ in $\tau$-direction, (\ref{eq:duPK}) converges to
$$
\frac{\del u}{\del \tau} + J_0\left(\frac{\del u}{\del t} -
X_H(u) \right) = 0
$$
on $(-\infty, 0] \times S^1$ (resp. on $[0, \infty) \times S^1$)
and $\delbar_{J_0}u = 0$ on $[0,R - \frac{1}{3}]
\times S^1$ (resp. on $[-R + \frac{1}{3},0] \times S^1$).

Now we are ready to state the meaning of the \emph{level-0 convergence} for
a sequence $u_n$ of solutions $(du + P_{K_{\e_n}})^{(0,1)}_{J_{\e_n}} = 0$
as $n \to \infty$. After taking away bubbles, we assume that we have
the derivative bound
\be\label{eq:|du|<C}
|du| < C < \infty
\ee
where we take the norm $|du|$ with respect to the given metric $g$ on $M$.
We denote
$$
\Theta_\e = \left[-R(\e) + \frac{1}{3}, R(\e) - \frac{1}{3}\right] \times S^1
$$
and consider the local energy
$$
E_{J,\Theta_\e}(u) = \int_{-R(\e) + 1/3}^{R(\e)-1/3
}\int_{S^1} |du|_J^2 dt \, d\tau.
$$
There are two cases to consider :
\begin{enumerate}
\item there exists $c > 0$ such that $E_{J,\Theta_{\e_n}}(u_n) > c > 0$ for all
sufficiently large $n$,
\item $\lim_{n \to \infty} E_{J,\Theta_{\e_n}}(u_n) = 0$.
\end{enumerate}

For the case (1), standard argument produces a non-constant bubble and
so we will mainly consider the case (2). Therefore from now on, we
will assume
\be\label{eq:energy0}
\lim_{j \to \infty} E_{J,\Theta_{\e_j}}(u_j) = 0.
\ee
Now we consider the reparameterization
$$
\overline u_j(\tau,t) = u_j\left(\frac{\tau}{\e_j}, \frac{t}{\e_j}\right)
$$
on the domain $[-\e_jR(\e_j), \e_j R(\e_j)] \times \R/2\pi \e_j \Z$.
A straightforward calculation shows that $\overline u_j$ satisfies
$$
\frac{\del \overline u}{\del \tau} + J_0\left(\frac{\del \overline u}{\del t} -
X_f(u) \right) = 0
$$
or equivalently
$$
\frac{\del \overline u}{\del \tau} + J_0\frac{\del \overline u}{\del t} +
\operatorname{grad}_{J_0} f(u) = 0
$$
on $[-\e_jR(\e_j), \e_j R(\e_j)]
\times \R/2\pi \e_j \Z$. For the
simplicity of notation, we will sometimes denote
$$
R_j = R(\e_j).
$$
The following result was proved in Part II of \cite{oh:dmj}. A similar
result was also obtained by Mundet i Riera and Tian.
(See Theorem 1.3 \cite{mundet-tian}.)

\begin{thm}[\cite{oh:dmj}, \cite{mundet-tian}]\label{centrallimit}
Suppose
$$
\ell = \lim_{j \to \infty}\e_jR(\e_j), \quad
\lim_{j \to \infty}E_{J,\Theta_{R_j}}(u_j) = 0.
$$
Then there exists a
subsequence, again denoted by $u_j$, such that the reparameterized
map $\overline u_j$ converges to a $\chi:[-\ell,\ell] \to M$ satisfying
$\dot \chi + \operatorname{grad}_J f(\chi) = 0$ in
$C^\infty$-topology. In particular, when $\ell = 0$, the original map
$u_j|_{\Theta_{\e_j}}$ converges to a point $p \in M$.
\end{thm}

Under this assumption $\lim_{\e \to 0} E_{J,\Theta_R}(u_j) = 0$,
after taking away bubbles, the translated sequences
$u_j(\cdot \pm R_j+1),\cdot) : (-\infty, R_j] \times S^1 \to M$ of solutions
$u_j$ of (\ref{eq:duPK}) as above converge to $u_-, \, u_+: \R \times S^1 \to M$
that satisfies the equation
$$
\frac{\del u}{\del \tau} + J_0\left(\frac{\del u}{\del t} -
X_{H_\pm}(u) \right) = 0
$$
in compact $C^\infty$-topology
where $H_\pm$ are the Hamiltonians
$$
H_+(\tau,t,x) = \kappa^+(\tau) H(t,x), \quad H_-(\tau,t,x) =
\kappa^-(\t) H(t,x).
$$
We recall \eqref{eq:beta} for the definition of $\kappa^\pm$.
We phrase this convergence \emph{$u_j$ converges to the nodal
Floer trajectory $(u_-,u_+)$}.

In the next section, we will carry out a
detailed study of microscopic picture of this convergence near the node $p$.

\section{Controlled nodal degeneration of Floer trajectories}
\label{sec:controlled}

In this section, we will give a precise description of the
degeneration of the solutions
\be\label{eq:dbarKRe} \frac{\del
u}{\del \tau} + J_0\left(\frac{\del u}{\del t} - X_{K_\e}(u)
\right) = 0
\ee
to a nodal Floer trajectories as $\e \to 0$, where $K_\e$ is the
Hamiltonian as defined in (\ref{eq:KRe}).

We choose a sequence $\e_j \to 0$ and let $R_j$ be any sequence such
that $\e_j R_j \to 0$ as $j \to \infty$, e.g., $R_j = - \log \e_j/2\pi$.
We start with the convergence in the sense of stable maps.

\subsection{Convergence in level 0 : stable map convergence}

We partition $\R \times S^1$ into the union
$$
\R \times S^1 = (-\infty, -R_j] \cup (-R_j,R_j) \cup [R_j,\infty).
$$

Let $u_j$ be a sequence of solutions of (\ref{eq:dbarKRe}) for $\e = \e_j$.
Then we note that $u_j$ satisfies
$$
\frac{\del u}{\del \tau} + J_0\left(\frac{\del u}{\del t} -
\chi(\tau - R_j+1)) X_H(u) \right) = 0
$$
on $(-\infty, -R_j +1] \times S^1$,
$$
\frac{\del u}{\del \tau} + J_0\left(\frac{\del u}{\del t} -
\chi(-\tau + (R_j-1)) X_H(u) \right) = 0
$$
on $[R_j-1,\infty) \times S^1$, and
$$
\frac{\del u}{\del \tau} + J_0\left(\frac{\del u}{\del t} -
\rho_{R_j}(\tau - R_j) X_{\e_j f}(u) \right) = 0
$$
on $[-R_j, R_j] \times S^1$.

If we consider the translated sequence $u_j(\tau - (R_j -1), t)$,
then it satisfies the equation
$$
\frac{\del u}{\del \tau} + J_0\left(\frac{\del u}{\del t} -
\chi(\tau) X_H(u) \right) = 0
$$
and
$u_i(\cdot + (R_j-1),\cdot)$ satisfies
$$
\frac{\del u}{\del \tau} + J_0\left(\frac{\del u}{\del t} -
\chi(-\tau) X_H(u) \right) = 0.
$$

It is important to note that the last two equations do \emph{not}
depend on the parameters $\e$ (and $R$) and so carries the well-defined
moduli space of solutions.
In the similar vein, we note that as $\e_j \to 0$, the last
equation `converges' to the equation
$$
\frac{\del u}{\del \tau} + J_0\frac{\del u}{\del t} = 0
$$
which is again independent of the parameters $\e_j$.

Now we recall our basic hypothesis
$$
0 = \mu([z_-,w_-];H) - \mu([z_+,w_+];H) = 0 \quad \mbox{or } -1.
$$
We will also require all the relevant moduli spaces entering in
the gluing constructions are transversal
and the almost complex structure $J$ is generic in that all
the nodes in this dimension are immersed as proven in
Theorem \ref{immersed}. This can be always achieved if $(M,\omega)$ is
semi-positive. In general, we will apply the machinery of
Kuranishi structure \cite{fukaya-ono} : Since Theorem \ref{immersed}
holds for a generic choice of $J_0$ when both $u_\pm$ are smooth,
and the corresponding smooth moduli space of Floer trajectories
are transversal for a generic choice of $J$, we can always put
the trivial obstruction bundle on the Floer moduli spaces. Non-trivial
obstruction bundles will appear only in the sphere bubble components.
Therefore we may safely assume that for a generic choice of $J$,
the nodes of all the relevant nodal Floer trajectories are
immersed.

We will further assume that
$u_i$ does not split-off at $\pm \infty$. More precisely, we assume that
both $u_i(\cdot - (R_j-1),\cdot)$ and $u_i(\cdot + (R_j+1),\cdot)$
uniformly converge respectively  as $i \to \infty$. This will follow from
the dimensional restriction by a generic choice of $J$.

Under these hypotheses, a straightforward
dimension counting argument, Gromov-Floer compactness and
Theorem \ref{centrallimit} imply
\begin{enumerate}
\item $|du_j|_{C^0} < C$ for all $j$ and
$u_j$ converges uniformly in fine $C^\infty$ topology and
\item $u_j(\cdot - (R_j+1),\cdot) \to u_-$ as $j \to \infty$ where
$u_-$ satisfies
$$
\frac{\del u}{\del \tau} + J_0\left(\frac{\del u}{\del t} -
\chi(\tau) X_H(u) \right) = 0
$$
and $u_j(\cdot + (R_j+K_j+1),\cdot) \to u_+$ satisfies
$$
\frac{\del u}{\del \tau} + J_0\left(\frac{\del u}{\del t} -
\chi(-\tau) X_H(u) \right) = 0.
$$
\end{enumerate}

We denote by $Glue(u_-,u_+,u_0;\e)$ the gluing solution
constructed in the previous sections out of $u_-,\, u_+, \,u_0$
and the parameter $\e$ with $R = - \log \e/2\pi$. Denote by
$Glue(\e)$ the set of the gluing solutions constructed in section
\ref{sec:gluing}. In the next section, we will prove that provided
$\e$ is sufficiently small, any solution $u$ of (\ref{eq:dbarKRe})
`sufficiently $C^0$-close to $Glue(\e)$' will become indeed
$Glue(u_-,u_+,u_0;\e)$ for some choice of $(u_-,u_+)$ and $u_0$.
We now make this statement precise in the rest of this section.

We \emph{fix} conformal identifications
\beastar
\varphi_- & : & \Sigma_- \to S^2 \setminus \{N\}, \quad \varphi_-(-\infty) = N \\
\varphi_+ & : & \Sigma_+ \to S^2 \setminus \{S\}, \quad
\varphi_+(+\infty) = S \eeastar so that they are compatible to the
analytic coordinates prescribed near $p_+ \in \Sigma_+$ and $q_-
\in \Sigma_-$ in subsection \ref{subsec:punctures}. As was shown
in subsection \ref{subsec:node}, this will determine the unique
points $o_+ \in \Sigma_+, \, o_- \in \Sigma_-$ respectively such
that
$$
\varphi_-(o_-) = N, \quad \varphi_+(o_+) = S.
$$
This will in turn determine a unique conformal identification modulo $\tau$-translations,
which we also denote by $\varphi_\pm$
\beastar
\varphi_- & : & (\Sigma_-,q_-,o_-) \to \R \times S^1 \\
\varphi_+ & : & (\Sigma_+,p_+,o_+) \to \R \times S^1.
\eeastar
We then form a disjoint union
$$
\dot \Sigma = \dot \Sigma_- \cup \dot \Sigma_+
$$
with $o_-$ and $o_+$ identified.

We now consider the Floer trajectories $u_\pm : \dot \Sigma_\pm \to M$
with the node $p = u_-(o_-) = u_+(o_+)$.
Since we assume that $u_-$ and $u_+$ are immersed at the node $p$
and $J_0$-holomorphic, there exists a sufficiently small $\e_0 > 0$ such that
both $u_-^{-1}(B^{2n}_p(\e_0))$ and $u_+^{-1}(B^{2n}_p(\e_0))$ are
conformally isomorphic to $D^2 \setminus \{0\}$. Denote
\bea\label{eq:S+S-}
S_+ & = & u_+^{-1}(B^{2n}_p(\e_0)) \subset \dot \Sigma_+,\nonumber\\
S_- & = & u_-^{-1}(B^{2n}_p(\e_0)) \subset \dot \Sigma_-
\eea
and $S = S_- \cup S_+ \subset \dot \Sigma_- \cup \dot \Sigma_+ =\dot \Sigma$

For further discussion,
we will need the following proposition. This is a standard result whose proof
can be derived from \cite{fukaya-ono}, \cite{mundet-tian} and so omitted.

\begin{prop}\label{expdecay} We denote by $\operatorname{mod}(\Sigma)$ the
conformal modulus of the annulus $\Sigma$. Let $(M,\omega,J)$ be an almost K\"ahler
manifold and $\Sigma$ be a Riemann surface of annulus type with
$\operatorname{mod}(\Sigma) = L < \infty$. Suppose that
$h: \Sigma \to M$ is a smooth map satisfying
$$
h(\Sigma) \subset B_p^{2n}(\e)
$$
and
$$
\delbar_{J_{R(\e)}}h + (P_{K_{R(\e),\e}})_{J_{R(\e)}}^{(0,1)}(h) = 0.
$$
Identify $\Sigma \cong [-L,L] \times S^1 \to M$ conformally. Then
there exist $\e'_0 > 0$ and $C, \, k> 0$ depending only on
$(M,\omega,J)$ but independent of $h, \, L$ such that whenever $0
\leq \e < \e'_0$, \be\label{eq:expdecay} |du|_{g_{J_0}}(\tau',t')
\leq Ce^{-\operatorname{dist}(\tau', \del[-L,L])} \ee for all $\tau'
\in [-L + 1, L -1]$, and \be\label{eq:length}
\operatorname{leng}(u(\tau',\cdot)) \leq C
e^{-k\operatorname{dist}(\tau', \del[-L,L])}. \ee
\end{prop}

We now derive the following lemma from this proposition.

\begin{lem} Let $k> 0$ be the constant
given in Proposition \ref{expdecay}.
There exists $\e_0 > 0$ such that
$$
u_j^{-1}(B^{2n}_p(\e_0)) =: \Sigma_j(\e_0)
$$
has a topological type of annulus and decompose $\R \times S^1$
$$
\Sigma_{j,-} \cup \Sigma_j(\e_0) \cup \Sigma_{j,+}
$$
such that $\R \times S^1 \setminus \Sigma_j(\e_0) = \Sigma_{j,-}
\coprod \cup \Sigma_{j,+}$.
\end{lem}
\begin{proof} Let $\e_0'$ be the constant given in Proposition
\ref{expdecay}. Theorem \ref{centrallimit} implies that
\be\label{eq:B2ne'0}
u_j\left(\left[-R_j+\frac{1}{3},R_j-\frac{1}{3}\right] \times S^1\right)
\subset B^{2n}(\e'_0)
\ee
for all sufficiently large $j$ and so the exponential decay (away from the boundary)
(\ref{eq:expdecay}) holds. It follows from this that
$\Sigma_j(\e_0')$ is of annulus type.
\end{proof}

Now we are ready to give the meaning of the stable map convergence of
$u_j$ to the nodal Floer trajectories $(u_-,u_+)$. This
is a variation of those given in \cite{fukaya-ono},
\cite{HWZ:smallenergy} applied to the current circumstance.

Following \cite{HWZ:smallenergy}, we introduce a definition

\begin{defn}[Definition 4.1, \cite{HWZ:smallenergy}]\label{defn:deform} A \emph{deformation}
of a compact Riemann surface $(A,j)$ of annulus type is a continuous
surjection map $f: A \to S$ onto the nodal surface, so that $f^{-1}(o)$ is
a smooth embedded circle, and
$$
f: A \setminus f^{-1}(o) \to S\setminus \{o\}
$$
is an orientation preserving diffeomorphism. On $S\setminus \{o\}$
we have the pushed forward complex structure $f_*j$.
\end{defn}

For each given nodal surface $S$, we recall a construction of a family of deformations
in the following way (See \cite{fukaya-ono}) parameterized by $\alpha \in \C$
with $|\alpha|$ sufficiently small.

\begin{exm}[Fukaya-Ono, \cite{fukaya-ono}]\label{exm:fukaya-ono}
We choose the unique biholomorphic map
$$
\Phi_\alpha : T_{o_-} S_- \backslash \{ o_- \} \to  T_{o_+} S_+ \backslash \{ o_+ \}.
$$
such that $u \otimes \Phi_\alpha(u) = \alpha$. In terms of analytic
coordinates at $o_- \in S_-$ and $o_+ \in S_+$, the coordinate
expression of $\Phi_\alpha$ is given by the map $\Phi_\alpha(z) =
\frac{\alpha}{z}$.

We denote $|\alpha|= R_\alpha^{-2}$ for $|\alpha|$ sufficiently small
and so $R_\alpha$ sufficiently large so that the composition
$$
{exp}_{S_-}^{-1} \circ \Phi_\alpha \circ {exp}_{S_+} :
D_{o_+}(R_\alpha^{-\frac12}) \setminus D_{o_+}(R_\alpha^{-\frac32})
\rightarrow D_{o_-} (R_\alpha^{-\frac12}) \setminus D_{o_-}
(R_\alpha^{-\frac32})
$$
is a diffeomorphism. By composing with the biholomorphism
$$
[-\ln R_\alpha^{-1/2}, \ln R_\alpha^{1/2}] \times S^1 \to
D_{o_+}(R_\alpha^{-\frac12}) \setminus D_{o_-}(R_\alpha^{-\frac32}) ;
(\tau,t) \mapsto e^{2\pi((\tau-R) + it)} = e^{-2\pi R}z
$$
with $z = e^{2\pi(\tau +it)}$ the standard coordinate on $\C$, this diffeomorphism becomes
nothing but
\beastar
[-\ln R_\alpha^{-\frac12}, \ln R_\alpha^{\frac12}] \times S^1 & \to &
[-\ln R_\alpha^{-\frac12}, \ln R_\alpha^{\frac12}] \times S^1\\
(\tau,t) &\mapsto& (-\tau,-t)=(\tau',t')
\eeastar

We glue the metrics on
$$
D_{o_+} (R_\alpha^{-\frac12}) - D_{o_-}(R_\alpha^{-\frac32})
$$
without changing the metric outside $D_{o_+} (R_\alpha^{-\frac12})$ on $\Sigma_0$.
Identify $D_{o_+} (R_\alpha^{-\frac12})$ with an open set in $\C \owns z$
with the standard metric. Consider the biholomorphism
$\Phi_\alpha : z \rightarrow \frac{\alpha}{z}$, for which we have
$$
(\Phi_\alpha)^* {|dz|}^2 = {\left|\frac{\alpha}{z^2}\right|}^2 {|dz|}^2.
$$
Note that on $|z|=R^{-1}$, we have
\beastar
\Phi_\alpha (\{ z | |z| = \sqrt{\alpha} \})
& = & \{ z | |z|= \sqrt{\alpha} \}\\
(\Phi_\alpha)^* {|dz|}^2 & = & {|dz|}^2.
\eeastar
We choose a function and fix it once and for all
$$
\chi_{R_\alpha} : (0, \infty) \rightarrow (0, \infty)
$$
such that
\begin{enumerate}
\item $(\Phi_\alpha)^* (\chi_{R_\alpha} {|dz|}^2) = \chi_{R_\alpha} {|dz|}^2$
\item $\chi_{R_\alpha} (r) \equiv 1$ if $r > |\alpha|^{3/8} = R_{\alpha}^{-3/4}$.
\end{enumerate}

By the definition of $\chi_{R_\alpha}$, we can replace the given metric
$g_{o_+}=|dz|^2$ by $\chi_{R_\alpha}(|z|) |dz|^2$ inside the disc $D^2(|\alpha|^{1/4})$,
and denote the resulting metric by $g_v'$. We would like to emphasize
that this modification process is canonical depending only on the
fixed complex charts at the singular points and on the choice of $\chi_{R_\alpha}$.
As a result, this modification process does not add more parameters in
the description of deformation of stable curves.
Hence we have constructed a family of stable
curves parameterized by a neighborhood of the origin in
$ T_{o_+} S_+ \otimes T_{o_-} S_-$. We denote the constructed Riemann
surface with the conformal structure constructed in this way by
$$
(S_\alpha, j_\alpha).
$$
We set $S_0$ to be the given nodal Riemann surface $S$. We can define a surjective continuous
map $f_\alpha : S_\alpha \to S$ by the projection from the graph of
$w = \frac{\alpha}{z}$ to the union of the $z$-axis and $w$-axis
that is invariant under the diagonal reflection.

This finishes construction of one-parameter family of deformations of
the given nodal Riemann surface. We call this explicit deformation
\emph{Fukaya-Ono's} deformation and will always consider this
deformation in the following discussion.
\end{exm}

\begin{defn}[Real deformation] We call the deformation $(S_\alpha,j_\alpha;f_\alpha)$
a \emph{real deformation} if $\alpha \in \R_+$.
\end{defn}

We go back to the study of convergence $u_j: \R \times S^1 \cong
\dot\Sigma \to M$.

For a given $\mu >0$ and a collection of sufficiently large $R_\alpha$,
we denote
$$
W_{o,\alpha}(\mu) := \left(D_{o_+}(\mu) -
D_{o_+}(R_\alpha^{-1})\right)\bigcup \left(D_{o_-}(\mu) -
D_{o_-}(R_\alpha^{-1} )\right)
$$
the prescribed neck region in $S_\alpha$. The following definition
is essentially the same one as the stable map convergence given in
Definition 10.2 \cite{fukaya-ono}.

\begin{defn}[Level 0 convergence] We say that $u_n$ converges to $Glue(u_-,u_+)$
in level 0 if
\begin{enumerate}
\item for any $\mu > 0$, $u_n|_{S_n \setminus
W_o(\mu)} \to Glue(u_-,u_+)$ in $C^\infty$ on compact sets,
\item  there exists a sequence of real deformations
$f_n: (S_n,j_n) \to (S,j)$ such that $(f_n)_*j_n \to j_\pm$ in compact $C^\infty$-topology
on $S \setminus \{o\}$,
\item
$\lim_{\mu \to 0}(\limsup_{n \to 0}  Diam(u_n (W_{o,n}(\mu))))=0$.
\end{enumerate}
\end{defn}

In terms of this definition, the standard definition of stable map convergence of
Floer trajectories to a nodal Floer trajectory as given in \cite{fukaya-ono},
\cite{liu-tian} can be translated into

\begin{prop}
Consider the partitions $(-L_j,L_j)$ associated to the
surface $\Sigma_{0,j} = u_j^{-1}(B(\e_0,p))$.
Then the maps $u_j : \Sigma_{0,j} \to M$ converge to the nodal
Floer trajectory
$$
Glue(u_-,u_+)|_{S(\e_0)}
$$
where $S(\e_0) = S_-(\e_0) \cup S_+(\e_0)$.
\end{prop}

Note that the level 0 convergence does not reflect the immersion
property of the nodes. It turns out that the level 0 convergence to
nodal trajectories with \emph{immersed nodes} has a finer convergence property
which we now explain.

\subsection{One-jet convergence to nodal curves with
immersed nodes}

Now we are ready to give the precise meaning of the convergence
$u_j$ to $(u_-,u_+,u_0)$, where $u_0$ is a local model obtained in section
\ref{sec:models}.

We start with the description of the sequence of Floer trajectories
$u_j$ over the central region $\Sigma_j(\e_0)$.
Fix a sufficiently small $\e_0 > 0$ for which Proposition \ref{expdecay}
holds. We choose a conformal diffeomorphism
$$
\psi_{j,\operatorname{int}} : [-L_j,L_j] \to \Sigma_j(\e_0)
$$
with $2L_j = \operatorname{mod}(\Sigma_j(\e_0))$. We denote the
corresponding conformal coordinates by $(\tau',t')$. We would like
to emphasize that this coordinates $(\tau',t')$ may not be the same
one as the original coordinates $(\tau,t)$ in $\R \times S^1$.

Applying Proposition \ref{expdecay} to the maps
$$
h_j = u_j \circ \psi_{j,int},
$$
we obtain

\begin{cor}\label{uiembedded} There exist $\e_0 > 0$, a sequence $\e_i' \to 0$
and a subsequence $j_i$ of $j$'s in turn so that
\begin{enumerate}
\item $u_i$ is embedded on $u_i^{-1}({B_p^{2n}(\e_0)\setminus B_p^{2n}(\e_i')})$
and $u_i^{-1}({B_p^{2n}(\e_0)\setminus B_p^{2n}(\e_i')})$ is a disjoint union of
two components $\Sigma_{i, \e_i' \leq r \leq \e_0}^\pm$ of cylindrical type.
\item
$$
\e_i'/\e_{j_i} \to \infty, \quad
\operatorname{mod}(\Sigma_{i, \e_i' \leq r \leq \e_0}^\pm) \to \infty.
$$
\end{enumerate}
\end{cor}
\begin{proof} The first statement is an immediate consequence of a
diagonal sequence argument from the stable map convergence $u_i$ and
the immersion property of $u_\pm$.

For the second statement, we pick any sequence $\e'_i \to 0$
and consider the modulus $\operatorname{mod}(u_j^{-1}(B_p^{2n}(\e'_i))$.
By the first statement, we have
$$
\lim_{j \to \infty} \operatorname{mod}(u_j^{-1}(B_p^{2n}(\e'_i)) =
\infty
$$
for each fixed $i$. Take the subsequence $j_i$ of $j$ so that
$$
\e'_i/ \e_{j_i}, \quad \operatorname{mod}(u_j^{-1}(B_p^{2n}(\e'_i))) \geq i
$$
for each $i$ : this is possible since $\e_j \to 0$.
This finishes the proof.
\end{proof}

By renumbering $j_i$, we will just denote $j_i$ by $i$ and so we are given
two sequences
$$
\e_i, \, \e_i' \to 0, \quad  \e_i'/\e_i \to \infty
$$
and $\operatorname{mod}(u_i^{-1}(B_p^{2n}(\e_i'))=: 2L_i' \to
\infty$ as $i \to \infty$. We will assume this for the rest of this
section.

Now we define a rescaled map
$$
\widetilde{u_{i,\operatorname{int}}} : [-L_i',L_i'] \times S^1  \to \C^n
$$
by
$$
\widetilde{u_{i,\operatorname{int}}}(z) = \frac{1}{\epsilon_i} (u_i
\circ \psi_{i,\operatorname{int}})(z)
$$
and study its convergence behavior.

We consider the decomposition of the Riemann surface
$$
\R \times S^1 \cong \dot\Sigma = \Sigma_{i,-} \cup \Sigma_{i,0} \cup \Sigma_{i,+}
$$
where $\Sigma_{i,0} = u_i^{-1}(B^{2n}(\e'_i))$ and
$$
\R \times S^1 \setminus \Sigma_{j,0} = \Sigma_{j,-}\cup \Sigma_{j,+}
$$
We denote the translated sequences
\beastar
u'_{i,-} & = & u_i(\cdot - (R_i+1),\cdot) : (-\infty, R_i] \times S^1 \to M\\
u'_{i,+} & = & u_i(\cdot + (R_i+1),\cdot) : [-R_i, \infty) \times
S^1 \to M \eeastar and their conformal reparameterizations by
\beastar v_{i,-} & = & u'_i \circ \varphi_- :
\varphi_-^{-1}((-\infty,R_i] \times S^1)
\to M\\
v_{i,+} & = & u'_i \circ \varphi_+ :
\varphi_+^{-1}([-R_i,\infty)\times S^1) \to M. \eeastar It is easy
to see from the definitions that we can choose $R_i = R_i(\e_i)$ so
that
$$
\varphi_-^{-1}((-\infty,R_i] \times S^1) \supset S^2 \setminus
D^2_S(C\e_i), \quad \varphi_-^{-1}([-R_i,\infty) \times S^1) \supset
S^2 \setminus D^2_N(C\e_i)
$$
for some constant $C > 0$ independent of $i$.

Now we are ready to give the main definition of the refined convergence.
\emph{As before $u_0$ stands for a local model obtained in section \ref{sec:models}}.

\begin{defn}[$\{\e_i\}$-controlled one-jet convergence]\label{econtrolled}
We say that a sequence $u_{\e_i}$ of solutions for (\ref{eq:dbarKRe})
\emph{converges to $(u_-,u_+;u_0)$ in the $\{\e_i\}$-controlled way}
if the following holds :
\begin{enumerate}
\item $u_{\e_i}$ converges to $Glue(u_-,u_+)$ in level 0,
\item we have $\operatorname{mod}(\Sigma_{i,0}') = 2L_i' \to \infty$,
\item there exists a sequence of automorphisms $g_{v,\lambda}$
given by $g_{v,\lambda}(u) = \lambda u + v$ for
some vectors $v_i \in \C^n $ and $\lambda_i \in \R$
such that we have the inequality
\be\label{eq:econtrolled} \left|\nabla^k \left(g_{v_i,\lambda_i}^{-1}
\left(\frac{1}{\e_i} u_i
\circ \psi_{i,\operatorname{int}} + \tau \vec a \right)  - u_0\right) (\tau,t)\right|
\le \min\left( \delta_{k,i}, C_k e^{-c_k
\vert \tau - L_i' \vert}\right)
\ee
on $[-L_i',L_i'] \times S^1$ in
the given Darboux chart at $p$ with respect to the {\it cylindrical}
metrics on $\R \times S^1$ and $g'_{\C^n}$.
\end{enumerate}
\end{defn}

Surjectivity proof will be finished by the following convergence theorem.

\begin{thm}\label{1-jetconvergence} Suppose that $u_-, \,
u_+$ are immersed at the node
$$
p = u_-(o_-) = u_+(o_+).
$$
Let $Glue(u_-,u_+)$ be the nodal Floer trajectory formed by
$u_-$ and $u_+$ with nodal points $p = u_-(o-) = u_+(o_+)$.
Suppose that $u_n$ converges to $Glue(u_-,u_+)$ in level 0.
Then there exists a subsequence $u_{n_i}$ and a sequence $\e_i \to 0$
such that $u_{n_i}$ converges to $(u_-,u_+;u_0)$ in the $\{\e_i\}$-controlled way.
\end{thm}

We will give the proof of this theorem in the next section.

Once we prove this theorem, the well-known argument by Donaldson
\cite{donald} proves the following which will finish the proof of
surjectivity. We omit the details of this last step but refer to
section 62.7 of chapter 10 of \cite{fooo07} for relevant details of this last step
in a similar context.

\begin{thm}\label{surjectivity} Let $R(\e) = - \frac{1}{2\pi}\log \e$ and $K_\e$ be the Hamiltonian
as defined in (\ref{eq:KRe}). There exists small constants $\e_1, \, \e_2$ with $\e_1 <
\e_2^{100}$ such that for any $0 < \e < \e_1$ and any solution $u:\R \times S^1 \to M$ of
$$
\delbar_{J_\e} u  + P_{K_\e,J_\e}^{(1,0)}(u) = 0
$$
satisfying
$$
\max_{z \in \R \times S^1} \operatorname{dist}(u, Glue(\e)) < \e_2
$$
indeed has the form $u = Glue(u_-,u_+,u_0;\e)$ for some $u_-, \, u_+$
and $u_0$.
\end{thm}

Here the choice of exponent `100' is not significant which
is made imitating the statement of Theorem 62.2 \cite{fooo07}.

The following proposition will be important in the energy estimates
needed to prove the above theorem. Again this is the analog to Proposition
62.79 \cite{fooo07} in the current context. Here since we consider the case
where $\gamma_{+,j}$ converges to $\gamma_{a^+}$ in $C^0$-topology,
we can write
$$
\gamma_{+,j}(t) - \gamma_{a^+}(t) :=
\exp_{\gamma_+(t)}^{-1}(\gamma_{+,j}(t))
$$
for the unit vector
$$
a^+ : = \lim_{\tau \to \infty}
\frac{du^+\left(\frac{\del}{\del\tau}\right)}
{\left|du^+\left(\frac{\del}{\del\tau}\right)\right|}.
$$
Similar remark applies to $\gamma_{-,j}$.

We identify a Darboux neighborhood
of $p$ with an open neighborhood of $0 \in \C^n \cong T_pM$.

\begin{prop}\label{chord-converge} For each given $k$,
there exist $I_0, \, R_0$ and constant $o(i,R_0|k)$ with
$$
\lim_{i\to \infty}\lim_{R_0 \to \infty}o(i,R_0|k) = \infty
$$
such that for all $-\frac{1}{2\pi}\log \e_i' +R_0\leq s \leq
-\frac{1}{2\pi}\log \e_0-R_0$ the followings hold :
\begin{enumerate}
\item $s$ is a regular value of $s\circ u_i$ and the curve $u_i(\Sigma_i(\e_0)) \cap
(\{s\} \times S^{2n-1})$ is parameterized by the union of two disjoint
circles $\gamma_{i,s}^\pm: S^1 \times S^{2n-1}$ for which we have
\be\label{eq:nablakgammai}
|\nabla^k(\gamma_{i,s}^\pm - \gamma_{a^\pm})| \leq o(i,R_0|k).
\ee
\item For $s_1 \in [-\frac{1}{2\pi}\log \e_i' + R_0, -\frac{1}{2\pi}\log \e_0 - R_0]$,
the set
$$
\Sigma_{i,s_1-1\leq s \leq s_1 +1} = u_i(\Sigma_i) \cap ([s_1-1,s_1+1] \times S^{2n-1})
$$
is a disjoint union of two components $\Sigma_{i,s_1-1\leq s \leq s_1 +1}^\pm$ such that
each of $\Sigma_{i,s_1-1\leq s \leq s_1 +1}^\pm$ has a parameterization
$$
u_{i,s_1-1 \leq s\leq s_1+1}^\pm : [-1/2\pi,1/2\pi] \times S^1 \to
\Sigma_{i,s_1-1\leq s \leq s_1 +1}^\pm
$$
for which we have \be\label{eq:nablakuis1} |\nabla^k(u_{i,s_1-1 \leq
s\leq s_1+1}^\pm - u^{\text{\rm flat}}_{a^\pm,s_1})| < o(i,R_0|k)
\ee where
$$
u^{\text{\rm flat}}_{a^\pm,s_1}(\tau,t) = (2\pi \tau + s_1,
\gamma_{a^\pm}(t))
$$
\end{enumerate}
\end{prop}
\begin{proof} We first note that regularity of $s$ follows from Corollary
\ref{uiembedded}.

Consider the composition
$$
v_-: = u_-\circ \varphi_-^{-1} : S^2 \setminus \{S\} \to M
$$
This map $v_-$ extends smoothly to $S^2$ and its derivative $dv_-(S)
\neq 0$ by the immersion assumption on the node. Now we consider the
translated sequence $u'_j: = u_j(\cdot -(R_j +1),\cdot)$ which
converges to $u_-$ in compact $C^\infty$ topology on $(-\infty, R_j]
\times S^1$ as $j \to \infty$ and define
$$
v_j: = u'_j \circ \varphi_-^{-1} : \varphi_-((-\infty, R_j] \times
S^1) \to M
$$
as before. Then by the hypothesis of level 0 convergence,
$S^2 \setminus \varphi_-((-\infty, R_j + 1] \times S^1)$
shrinks to the point $\{S\}$ as $j \to \infty$ and
$v_j \to v_-$ in compact $C^\infty$ topology and $v_-$ is immersed
at $S$.

For the rest of the statements, we will prove them by contradiction.
Suppose to the contrary. Then we can choose a sequence $s_i$ with
$$
s_i - \frac{1}{2\pi} \log \e_i', \quad \frac{1}{2\pi}\log \e_0 - s_i \to \infty
$$
such that one of the following alternatives must hold :
\begin{enumerate}
\item There exist $k$ and $c>0$ such that
$$
\vert \nabla^k(\gamma_{a^+} - \gamma_{i,s_i}^+) \vert > c
$$
or
$$
\vert \nabla^k(\gamma_{a^-} - \gamma_{i,s_i}^-) \vert > c
$$
for any parametrization $\gamma_{i,s_i}^\pm$ of $u_i(\Sigma_i^\pm)
\cap (\{ s_i\} \times S^{2n-1})$.
\item There exist $k$ and $c>0$ such that
$$
\vert \nabla^k(u_{i,s_i-1\le s \le s_i+1} - u^{\text{\rm
flat}}_{a^\pm,s_i}) \vert > c
$$
for any parametrization $u_{i,s_i-1\le s \le s_i+1}^\pm$ of
$\Sigma_{i,s_1-1\leq s \leq s_1 +1}^\pm$.
\end{enumerate}

In terms of $r$ coordinates, we have
$$
[s_i-1,s_i+1] \times S^{2n-1} \leftrightarrow
[e^{s_i-1},e^{s_i+1}] \times S^{2n-1} = [\e_0^{1/2\pi}e^{-K_i-1},\e_0^{1/2\pi}e^{-K_i+1}]
\times S^{2n-1}
$$
for $K_i:= \frac{1}{2\pi}\log \e_0 - s_i \to \infty$. Since $v_-$ is
immersed at $\{S\}$ and $v_-(S) = p$, the subset
$$
v_-([\e_0^{1/2\pi}e^{-K_j-1},\e_0^{1/2\pi}e^{-K_i+1}]
\times S^{2n-1}): = A_i \subset S^2 \setminus \{S\}
$$
is of annulus type and shrinks to the point $S$ as $j \to \infty$.
Therefore by taking a diagonal sequence argument and using the
immersion property of $v_-$ at $S$, if necessary, we may assume
$$
|\nabla^k(v_i - v_-)|_{A_i} \to 0.
$$
This is then translated into \be\label{eq:nablakuisi-1} \vert
\nabla^k(u_{i,s_i-1\le s \le s_i+1} - u^{\text{\rm
flat}}_{a^\pm,s_i}) \vert \to 0 \ee as $i \to \infty$. This in
particular rules out the second possibility.

On the other hand, the immersion property of $u_-$ at
$o_-$ implies that
$$
|u_+(z) - z \cdot a^+| = O(|z|^2).
$$
Therefore we have
$$
||u_+(z)|-|z||a^-|| = |r(u_+(z)) - r(z\cdot a^+)| \leq O(|z|^2)
$$
Since $\Theta(u_+(z)) = \frac{u_+(z)}{|u_+(z)|}$, we obtain
\beastar
|\Theta(u_+(z)) - \Theta(z \cdot a^+)| & = & \left|\frac{u_+(z)}{|u_+(z)|}
-\frac{z \cdot a^+}{|z||a^-|}\right| \\
& \leq & \frac{|u_+(z) - z \cdot a^+|}{|z||a^-|} +
\frac{||u_+(z)|-|z||a^-||}{|z||a^-|} \leq O(|z|). \eeastar Similarly
we have \beastar
|\gamma_{j,s_i}^+-\gamma_{a^+}| & = & |\Theta(u_{j,+}(z)) -\Theta(z \cdot a^+) |\\
& = & |\Theta(u_+(z)) - \Theta(z \cdot a^+)| + |\Theta(u_{j,+}(z)) - \Theta(z \cdot a^+)|\\
\eeastar
For the first term, we have
$$
|\Theta(u_+(z)) - \Theta(z \cdot a^+)| \leq O(|z|)
$$
and for the second term, we have
$$
\lim_{i \to \infty}|\Theta(u_{i,+}(z)) - \Theta(z \cdot a^+)| = 0.
$$
Therefore we have obtained
\be \label{eq:gammaa+}
|\Theta(u_+(z)) - \Theta(u_{i,+}(z))| \to 0
\ee
since $z$ satisfies $\e_0^{1/2\pi}e^{-K_i-1} \leq |z| \leq
\e_0^{1/2\pi}e^{-K_i+1}$ and $K_i \to \infty$ as $j \to \infty$.

Combining (\ref{eq:nablakuisi-1}) and (\ref{eq:gammaa+}), we can prove
$$
\vert \nabla^k(\gamma_{a^-} - \gamma_{i,s_i}^-) \vert \to 0
$$
inductively over $k = 0, \cdots, $ as $\gamma_{i,s_i}^-(t) = \Theta
(u_{i,+}^{-1}(B_p(e^{s_i})))$. This contradicts to the hypothesis and
so the proposition is proved.
\end{proof}

\section{Surjectivity of the scale-dependent gluing family}
\label{sec:surjectivity}

The main goal of this section is to prove Theorem \ref{1-jetconvergence},
which will imply that the enhanced resolutions of
Floer nodal trajectories $(u_-,u_+,u_0)$ exhaust all the solutions of
(\ref{eq:duPK}) which are close to those of the enhanced nodal trajectories
$(u_-,u_+)$ in a suitable sense.

To prepare the proof, we consider the map
$$
\widetilde u_i : \Sigma_i(\e_0) \to T_pM
$$
defined by
$$
\widetilde u_i(\tau,t) = \frac{1}{\e_i} (\exp^I_p)^{-1}\circ u_i
$$
where $\Sigma_i(\e_0) = u^{-1}_i(B^{2n}_p(\e_0))$.

We denote the pull-back almost complex structure on $(\exp^I_p)^{-1}(B_p^{2n}(\e_0))
\subset T_p M$ by $\widetilde J_i$ which is defined by
$$
\widetilde J_i = (\exp^I_p \circ R_{\e_i})^*J_0
$$
on $B_p^{2n}(\e_0/\e_i) \subset T_pM$. Then $\widetilde u_i$
satisfies the equation \be\label{eq:tildeduJPK} \delbar_{\widetilde
J_i} \widetilde u_i + (P_{K_{R_i,\e_i}}^{(1,0)})_{\widetilde
J_i}(\widetilde u_i) = 0. \ee We first describe the metrics on the
domain $\C$ and the target $\C^{n}$, with which we evaluate the
$C^k$ norms of $\xi_i$'s.

For $\e_i'$ chosen before it follows, by choosing $\e_0$ smaller
if necessary, that
$\Sigma_i(\e_0) \setminus \Sigma_i(\e_i')$ is a disjoint union of
two domains of cylindrical type. We denote
$$
\widetilde u_i^{-1}(B_p^{2n}(\e_0)) \setminus \widetilde
u_i^{-1}(B_p^{2n}(\e_i'))= C_{i,1}(\e_i',\e_0) \cup
C_{i,+}(\e_i',\e_0).
$$
Whenever there is no danger of ambiguity, we will just denote
$C_{i,\pm}$ for $C_{i,\pm}(\e_i',\e_0)$
respectively.

We recall that we have used the metrics as follows :
For the target, we use the metric, denoted by $g'_{\C^n}$, to
satisfy the following properties :
\begin{enumerate}
\item $g'_{\C^n}$ is a flat Euclidean metric on
the Euclidean ball $B^{2n}(2)$ of radius 2.
\item Outside the (Euclidean) ball
$B^{2n}(4)$, it is the standard product
metric on $[\log 4,\infty) \times S^{2n-1}(3)$. (Here $
S^{2n-1}(3)$ is the round sphere of radius $3$.
\item $g_{\C^n}$ is of nonnegative curvature.
\end{enumerate}

For the domain, we require the metric, denoted by $g'_{\C}$, to
have totally geodesic boundary and to satisfy the following
properties :
\begin{enumerate}
\item $g'_{\C}$ is a flat Euclidean metric on the Euclidean
ball $B^2(1)$ of radius 1.
\item Outside the (Euclidean) ball
$B^2(2)$ of radius $2$, $g'_{\C}$ is the standard product
metric $[0,\infty) \times [0,3\pi/2]$.
\item $g'_{\C}$ is of nonnegative curvature.
\end{enumerate}

We now recall that for any contact hypersurface $(N,\xi)$  of a
symplectic manifold $(M,\omega)$  has the canonical \emph{
co-orientation}. If a smooth map $u: \Sigma \to M$ from an
oriented surface $\Sigma$ is transversal to a contact hypersurface
$N \subset M$, then the preimage $u^{-1}(N)$ has a natural
orientation induced by the co-orientation of $N \subset M$. Call
this the induced orientation on $u^{-1}(N)$ and denote $o_{ind}$.

When $\Sigma$ is given a complex structure $j$, it carries the
complex orientation on it and its boundary $\del\Sigma$ has the
boundary orientation $o_{bdy}$ defined by the convention
$$
\vec{n} \oplus o_{bdy}  = o_{\Sigma}
$$
where $\vec n$ is the unit normal outward to $\Sigma$ on the
boundary.

Now assume that $\Sigma$ is oriented and $\del \Sigma = \coprod_j
\del_j \Sigma$ where each $\del_i\Sigma$ denotes a connected
component of $\del \Sigma$. If $u: \Sigma \to M$ is transversal to
a contact hypersurfaces $N_j \subset M$ and $u^{-1}(N_j) = \del_j
\Sigma$, then $\del_j \Sigma$ carries two orientations $o_{ind}$
and $o_{bdy}$.

\begin{defn}\label{orientation}
Let $(u,\Sigma)$ as above. We say that a component $\del_i \Sigma$
is an \emph{outside boundary} if $o_{ind} = o_{bdy}$, and an \emph{
inside boundary} if $o_{ind} = - o_{bdy}$. We denote by
$\del_{out}\Sigma$ the union of outside boundaries and by
$\del_{in}\Sigma$ the union of inside boundaries.
\end{defn}

\begin{thm}\label{mainconvergence} Let $\widetilde u_i$ satisfy (\ref{eq:tildeduJPK}).
There exist $a_i$ and $\delta_{k,i}> 0$, such that $\lim_{i\to\infty} a_{i,\pm}
=a_\pm$, $\lim_{i\to\infty}\delta_{k,i} = 0$ and $u_i$ satisfies the
following properties :
\begin{enumerate}
\item There exists an open subset $\UU_{i,\text{\rm out}}$ of $\Sigma_i(\e_0)$
containing $C_{i,+} \cup C_{i,-}$, a sequence $L_{i,\pm} \to \infty$
such that there exists a biholomorphic embedding
\beastar
\psi_{i,\text{\rm neck},+} & : & [0,2L_{i,+}] \times S^1 \to C_{i,+}\\
\psi_{i,\text{\rm neck},-} & : & [-2L_{i,-},0] \times S^1 \to C_{i,-}
\eeastar
such that $\psi_{i,\text{\rm neck},\pm}$ satisfies
$$
\vert \nabla^k((\widetilde u_i\circ \psi_{i,\text{\rm neck},\pm})
 - u^{\text{\rm flat}}_{a_i}))\vert (\tau,t)
< C_k e^{-c_k \min\{|\tau |, |L_{i,\pm}\pm \tau|\}}.
$$
on $[0,2L_{i,+}] \times S^1$ (or on $[-2L_{i,-},0] \times S^1$ respectively).
Here $c_k$, $C_k$ are independent of $i$ and we put
$$
u^{\text{\rm flat}}_{a_{i,\pm}}(\tau,t) = (2\pi \tau,\gamma_{a_{i,\pm}}(t))
$$
and use the cylindrical metrics for both the domain and the target.
\item There exist a sequence $L_{i,0} \to \infty$,
open sets $\UU_{i,\text{\rm int}} \subset \Sigma_i(\e_0)$ and a biholomorphic map
$$
\psi_{i,\text{\rm int}} : [-L_{i,0},L_{i,0}]
\times S^1 \to \UU_{i,\text{\rm int}}
$$
with the following properties :
\begin{enumerate}
\item $\UU_{i,\text{\rm int}} \cap \UU_{i,\text{\rm out}}
= \operatorname{Im}(\psi_{i,\text{\rm int}}) \cap \operatorname{
Im}(\psi_{i,\text{\rm neck}})$.
\item $u_i\circ \psi_{i,\text{\rm int}}$ satisfies Definition \ref{econtrolled}, i.e.,
there exists a sequence of automorphisms $g_{v_i,\lambda_i}$
of $\C^n$ such that
$$
\left|\nabla^kg_{v_i,\lambda_i}^{-1}\left(\frac{1}{\e_i} u_i \circ \psi_{i,\text{\rm
int}}(\tau,t) + \tau \vec a)\right) -u_0(\tau,t) \right| \le
\min\left( \delta_{k,i}, C_k e^{-c_k \min\{|\tau \pm
L_{i,0}|\}}\right)
$$
on $[-L_{i,0},L_{i,0}] \times S^1$ with
in the {\it cylindrical} metrics on $\R \times S^1$ and $g'_{\C^n}$.
\end{enumerate}
\end{enumerate}
\end{thm}

Here we use conformal parameterizations on $C_{i,+}$ and $C_{i,-}$ given
as above because the boundary orientation $o_{bdy}$ on $C_{i,+} \cap \del B_p(\e_{i,out})$
of the complex orientation on $C_{i,+}$ coincide with the above
induced orientation $o_{ind}$ while that of $C_{i,-}$ is opposite.

The remaining section will be occupied by the proof of this theorem.

We start with the following characterization of small energy
cylinders, which can be proved by the same method as in
\cite{hofer93}, \cite{HWZ:smallenergy}, \cite{fooo07}. We denote by
$u^{\operatorname{\rm flat}}_{a,s_1}$ the cylindrical strip defined
by
$$
u^{\operatorname{\rm flat}}_{a,s_1}(\tau,t) = (s_1 + \tau,\gamma_a(t))
$$
as before.

\begin{thm}[Theorem 1.3 \cite{HWZ:smallenergy}, Theorem 62.85 \cite{fooo07}]
\label{smallenergy}
Let $R>0$ be given and let  $u :
[-R,R] \times S^1 \to \R \times S^{2n-1}$ be a $J_0$-holomorphic map.
For each $E_0 > 0$ and $k$ there exist positive constants $e_0$,
$R_0$, $c_k$, and $C_k$ as follows :
Whenever $u$ satisfies
\begin{enumerate}
\item $E(u) \le E_0$,
\item $R \ge R_0$,
\item $E_{d\lambda}(u) \le e_0$,
\item the loop $u_0(t) : = u(0,t)$ satisfies
$$
\int u_0^*\lambda \le 3\pi,
$$
\end{enumerate}
we can find $a \in S^{2n-1}$ and $s_1 \in \R$ for which we have
$$
\vert \nabla^k(u - u^{\operatorname{\rm flat}}_{a,s_1})\vert(\tau,t)
\le C_ke^{- c_k  (R - \vert\tau \vert)}
$$
on $(\tau,t) \in [-R+10,R-10] \times [0,1]$.
\end{thm}

\subsection{Convergence in the neck regions}

We define
\be \UU_{i,out} = \Sigma_i(\e_0) \setminus
u_i^{-1}(B_p^{2n}(\e_{i,in})).
\label{eq:UUiout}
\ee
In this subsection, we will study
convergence of $\widetilde u_i$ on the neck regions
$$
u_i^{-1}(B_p^{2n}(\e_{i,out}))
\setminus u_i^{-1}(B_p^{2n}(\e_{i,in}))
$$
for a choice of two sequences $\e_{i,out} > \e_{i,in} >> \e_i$ such that
\be
\e_{i,out}, \, \e_{i,in} \to 0, \quad |\log \e_{i,out} - \log \e_{i,in}| \to \infty
\label{eq:logeioutiin}
\ee
as $i \to \infty$.
It follows from Proposition \ref{chord-converge}
that for a suitable choice of $\e_{i,out}$, $u_i^{-1}(B_p^{2n}(\e_{i,out}))$
define a sequence of open Riemann surfaces which converges to
a conformal cylinder $\R \times S^1$. Furthermore the standard symplectic area
of the rescaled maps $\widetilde u_i$ converges to infinity, but its end behavior is
controlled by the hypotheses that $u_i$ converges to the nodal
curve $(u_-,u_+)$ whose node is \emph{immersed}.

While the standard argument for the closed Riemann surface does not
apply to the sequence of the rescaled maps $\widetilde u_i$ defined
on open Riemann surfaces $\Sigma_i' =
u_i^{-1}(B_p^{2n}(\e_{i,out}))$, the imposed end behavior enables us
to apply the strategy employed by Hofer \cite{hofer93} estimating
the horizontal and vertical energies separately. However in our
current circumstance, we need to apply Hofer's strategy to the case
where the target manifold is neither complete nor cylindrical but
only \emph{approximately cylindrical}.

We need to concern parameterization of maps $\widetilde u_i$ on the annular regions
$$
u_i^{-1}(B_p^{2n}(\e_{i,out})) \setminus u_i^{-1}(B_p^{2n}(\e_{i,in})) = C_{i,-} \cup C_{i,+}
$$
where both $C_{i\pm}$ are of cylindrical type. We prove the following result
whose proof duplicates the one from \cite{fooo07} used in a similar
context.

\begin{prop}\label{modCipm} Let $(u_-,u_+)$ be a Floer trajectory with
immersed nodes as before and suppose $u_i$ converges to $(u_-,u_+)$ in level 0
and let $C_{i,\pm}$ be as above. Suppose that $\e_{i,out}, \, \e_{i,in}$
are chosen so that (\ref{eq:logeioutiin}) holds.
Then we have
$$
\operatorname{mod}(C_{i,\pm}) \to \infty.
$$
\end{prop}
\begin{proof} Since both cases are essentially the same, we will just
treat the case of $C_{i,+}= C_i$.

Since $u_i$ are immersed on $C_i$, the image $u_i(C_i)$ carries
the metric $g_{ind}$ induced from the compatible metric $g_J$ on
$M$. We denote by $g_0$ on $[\log \e_{i,in}, \log \e_{i,out}]
\times [0,2\pi]$ the standard product metric. Using Proposition
\ref{chord-converge}, we can find a diffeomorphism
$$
\Phi_i: u_i(C_{i,+}) \to [\log \e_{i,in}, \log \e_{i,out}] \times
[0,2\pi]
$$
so that
\begin{enumerate}
\item $\Phi_i(u_i(C_i) \cap \{s\} \times S^{2n-1}) = \{s\} \times
[0,2\pi]$.
\item For each sufficiently small $\e > 0$, we have
$$
|(\Phi_i)_*(g_{ind}) - g_0|_{C^1} < \e
$$
on $[\log \e_{i,in},  \log \e_{i,out}] \times S^1$
for all sufficiently large $i$.
\end{enumerate}

Let $\psi_{j,+}:[-L_{i+},L_{i+}] \times S^1 \to C_{i+}$ be
the orientation preserving conformal diffeomorphism such that
$$
\psi_{i+}(\{\pm L_{i+}\} \times S^1) \subset \del_{\pm} C_i.
$$
Denote by $g_1$ the standard metric on $[-L_{i+},L_{i+}] \times S^1$
and $g_2 = (u_i\circ \psi_{i+})^*g_{ind}$. Since $u_i\circ \psi_{i+}$
is pseudo-holomorphic and so
$$
u_i \circ \psi_{i+}:([-L_{i+},L_{i+}] \times S^1,g_1) \to (u_i(C_i), g_{ind})
$$
is conformal, we have $f^2 g_1 = (u_i \circ \psi_{i+})^*g_{ind}$ and so
$$
g_1 = f^2 g_2
$$
where $f: [-L_{i+},L_{i+}] \times S^1 \times S^1 \to \R$ is a
positive smooth function.

We compute
\bea\label{eq:intLj+}
\left(\int_{[-L_{i+},L_{i+}] \times S^1} f \Omega_{g_2}\right)^2
& \leq & \left(\int_{[-L_{i+},L_{i+}] \times S^1} f^2 \Omega_{g_2}\right)
\left(\int_{[-L_{i+},L_{i+}] \times S^1} \Omega_{g_2}\right) \nonumber\\
& \leq & \operatorname{Area}([-L_{i+},L_{i+}] \times S^1;g_1) \nonumber \\
&{}& \quad \times (1 + \e)\left(\int_{[\log \e_{i,in}, \log \e_{i,out}] \times
S^1} \Omega_{g_0}\right)\nonumber \\
& \leq & ((2\pi)\cdot (2L_{i+}))\cdot (1+\e) \cdot
((2\pi)(\log \e_{i,out} - \log \e_{i,in}))\nonumber \\
& = & ((2\pi)\cdot (2L_{i+}))\cdot (1+\e)(2\pi) \cdot (\log \e_{i,out} -\log \e_{i,in}).
\nonumber\\
\eea
On the other hand, we derive
$$
\int_{[-L_{i+},L_{i+}] \times S^1} f \Omega_{g_2} \geq (1+\e)^{-1}\int
_{\log \e_{i,in}}^{\log \e_{i,out}} \operatorname{leng}_{g_0}
(u_i^{-1}\circ \gamma_{i,s})\, ds.
$$
Since the winding number of the curve $u_i^{-1}\circ \gamma_{i,s}$ is one
and $\gamma_{i,s} \to \gamma_{a^+}$, we have
$$
\operatorname{leng}_{g_0} (u_i^{-1}\circ \gamma_{i,s}) \to 2\pi
$$
as $i \to \infty$. Hence we have proved
$$
\int_{[-L_{i+},L_{i+}] \times S^1} f \Omega_{g_2} \geq
(1+\e)^{-1} (\log \e_{i,out} -\log \e_{i,in}) \times 2\pi.
$$
Substituting this into (\ref{eq:intLj+}), we obtain
\beastar
&{}& \left((1+\e)^{-1} (\log \e_{i,out} -\log \e_{i,in}) \times 2\pi\right)^2\\
& \leq & ((2\pi)\cdot (2L_{i+}))\times (1+\e)((2\pi) \cdot (\log \e_{i,out}
 -\log \e_{i,in})
\eeastar
and so
$$
\log \e_{i,out} -\log \e_{j,in} \leq (1+\e) 2L_{i+}.
$$
This proves $L_{i+} \to \infty$ as $i \to \infty$ since we have chosen $\e_{i,out},
\, \e_{i,in}$ so that
$$
|\log \e_{i,out} -\log \e_{i,in}| \to \infty.
$$
\end{proof}

Next we recall the definition of the energy that Hofer introduced in \cite{hofer93},
which we denote by $E_\Sigma$ and $E_{d\lambda}$ restricted to the case of
$\C^n\setminus \{0\}$ with the standard symplectic form $\omega_0$.
If we denote by $\lambda$ the standard contact form on
$S^{2n-1}(1) \subset \C^n$ and by $(r,\Theta)$ the polar coordinates
of $\C^n \cong \R_+ \times S^{2n-1}(1)$, then we have
$$
\omega_0 = d(r \Theta^*\lambda).
$$

Using the diffeomorphism $\R \to \R_+ ; s \mapsto  e^s$,
we identify $\C^n \setminus \{0\}$ with $\R \times S^{2n-1}$.
Then pull-back of the standard complex structure $J_0$ on $\C^n \setminus \{0\}$
is invariant under the translation of $\R$-direction on $\R \times S^{2n-1}$
as well as $\Theta^*d\lambda$ and $\Theta^*\lambda$.

We will pull-back the symplectic form $\omega$ on $M$ by the map $\exp_p^I : = I^{-1}$
to $\omega_0$ by a Darboux chart $I$ near the nodal point $p \in M$ such that
$I^*J(0) = J(p)$. The following lemma is immediate whose proof is omitted.

\begin{lem} Let $\widetilde J_\e$ be the almost complex structure
on $(T_pM, \omega_p) \cong (\C^n,\omega_0)$ defined by
$\widetilde J_\e = (\exp_p^I\circ R_\e)^*J$.
Then there exists $\e_0 > 0$
such that we have
$$
|\widetilde J_{\e}(x) - J_p| \leq C \e |x|
$$
for all $|x| \leq \e_0$ where $|\cdot |$ is the norm induced by the standard metric on
$\C^n \cong T_pM$. In particular, we have
$$
|\widetilde J_{\e_i}(x) - J_p| \leq C \delta_i
$$
for all $x \in \frac{1}{\e_i}B^{2n}_p(\delta_i)\cong B^{2n}_p(\delta_i /\e_i)\subset T_pM$
for any $0 < \delta_i \leq \e_0$.
\end{lem}

Now we introduce the following

\begin{defn}[$d\lambda$-energy] Let $\Sigma$ be a compact surface with
boundary and let $u: \Sigma \to \C^n \setminus \{0\}$.
We define the {\it $d\lambda$-energy},
denoted by $E_{d\lambda}$ by
$$
E_{d\lambda}(u) = \int_{\Sigma}(\Theta\circ u)^*d\lambda.
$$
\end{defn}

We also use another energy denoted by $E_\Sigma$ \cite{hofer93}.
Consider the interval $[a,b] \subset \R$ and
let $\CC= \CC_{[a,b]}$ be the set of smooth functions
$$
\rho : (a,b) \to [0,1]
$$
such that
\begin{enumerate}
\item
$\supp \rho$ is of compact support,
\item
$\int_a^b \rho(u) \,du = 1$.
\end{enumerate}
Then we consider its integral, denoted by $\widetilde \rho$,
$$
\widetilde \rho (s) = \int_{a}^s \rho(u) \, du.
$$
Composing $\widetilde \rho$ with the projection to the $\R$-direction, we
regard $\widetilde \rho$ as a function on $\R \times S^{2n-1}
\cong \C^n \setminus B^{2n}(1)$. Note that $\widetilde \rho \equiv 0$ near
the lower limit $s = a$ and $\widetilde \rho \equiv 1$ near the upper
limit $s =b$.

\begin{defn}\label{Elogei} Let $\rho \in \CC$ and $\widetilde \rho$ as above.
We define $E_\Sigma(u;(a,b))$ by
$$
E_\Sigma(u;(a,b)) = \sup_{\rho \in \CC}
\int_{\Sigma}\, u^* d(\widetilde \rho\widetilde \Theta^*\lambda).
$$
\end{defn}

We now prove the following

\begin{lem}\label{dlam-energybound} Denote $C_i = u_i^{-1}(B_p(\e_{i,out}) \setminus
B_p(\e_{i,in}))$ and consider the restriction of $u_i$ on $C_i$. We have
$$
\lim_{i \to \infty} E_{d\lambda;C_i}(\widetilde u_i) = 0
$$
and in particular $E_{d\lambda;C_i}(\widetilde u_i)$ is uniformly bounded.
\end{lem}
\begin{proof} Note that $C_i$ has decomposition
$$
C_i = C_{i_,1} \cup C_{i,+} :
$$
$C_{i,-},\, C_{i,+} $ are surfaces of annular type such that
\beastar
\del C_{i,-} & = &\del_+ C_{i,-} \cup \del_- C_{i,-} \\
\del C_{i,+} & = &\del_+ C_{i,+} \cup \del_- C_{i,+}
\eeastar
where $\widetilde u_i(\del_+ C_i) \subset \del B^{2n}(\e_{i,out}/\e_i)$
and $\widetilde u_i(\del_- C_i) \subset \del B^{2n}(\e_{i,out}/\e_i)$.

Since both cases can be treated the same, we will focus on $C_{i,-}$.
By Stokes' formula, we obtain
$$
E_{d\lambda}(\widetilde u_i) = \int_{C_{i,-}} (\Theta \circ\widetilde u_i)^*d\lambda
= \int_{\del_+ C_{i,-}} (\Theta \circ \widetilde u_i)^*\lambda
- \int_{\del_- C_{i,-}} (\Theta \circ \widetilde u_i)^*\lambda.
$$
Proposition \ref{chord-converge} implies
\beastar
\lim_{i \to \infty} \Theta\circ \widetilde u_i|_{\del_+ C_{i,-}}
& = & \gamma_- \\
\lim_{i \to \infty} \Theta\circ \widetilde u_i|_{\del_- C_{i,-}}
& = & \gamma_-
\eeastar
where $\gamma_-$ is the Reeb orbit of $S^{2n-1}$ comes from the tangent
cone of $u_-$ at the node. This implies $
\lim_{i \to \infty} E_{d\lambda;C_{i,-}}(\widetilde u_i) = 0$.

Since the same argument applies to $C_{i,+}$
if we replace $\gamma_-$ by $\gamma_+$, we have proved
\be\label{eq:dlambda=0}
\lim_{i \to \infty} E_{d\lambda;C_i}(\widetilde u_i) = 0.
\ee

\end{proof}

Next we study $E_{C_i}(\widetilde u_i)
= E_{C_i}(\widetilde u_i;(\log (\e_{i,out}/\e_i), \log (\e_{i,in}/\e_i))$.

\begin{lem}\label{lam-energybound} For any given $\delta > 0$, there exists
$N = N(\delta)$ such that
$$
E_{C_i}(\widetilde u_i) < 2\pi + \delta
$$
for all $i \geq N$.
\end{lem}
\begin{proof}
Let $\rho \in \CC$ and $\widetilde \rho$ be the associated integral
$$
\widetilde \rho(s) = \int_{0}^s \rho(u)\, du
$$
Noting that $\widetilde \rho \equiv 0$ near $\del_-C_i$, we
use Stokes' theorem to show
$$
\int_{C_i} \widetilde u^*_i d(\widetilde \rho \Theta^*\lambda) = \int
\gamma_{i,-,\operatorname{out}}^*\lambda +
\int \gamma_{i,+,\operatorname{out}}^*\lambda
$$
where $\gamma_{i,\pm,out}=\widetilde u_i|_{\del_+ C_{i,\pm}}$.
Proposition \ref{chord-converge} implies that
$\gamma_{i,\pm,out} \to \gamma_\pm$ respectively and so each term converges
to $2\pi$ as $i \to \infty$. Hence there exists $N \in \Z_+$ such that
if $i \geq N$,
$$
\int_{C_{i,\pm}} \widetilde u^*_i d(\widetilde \rho \Theta^*\lambda) \leq 2\pi + \delta
$$
for any $\rho \in \CC$. Fixing any such $N$ and taking the supremum over $\rho \in \CC$,
we have proved $E_\Sigma(\widetilde u_i) \leq 2 \pi +\delta$ for all $i$.
\end{proof}

Now we take a conformal parameterization $\varphi_{i,\pm}:
[-L_{i,\pm}, L_{i,\pm}] \times S^1 \cong C_{i,\pm}$
and consider the composition $\widetilde u_i \circ \varphi_{i,\pm}
=: v_{i,\pm} $. Proposition
\ref{modCipm} implies $L_{i,+} \to \infty$ as $i \to \infty$
and Lemma \ref{dlam-energybound} implies
$$
\lim_{i \to \infty} \int_{-L_{i,\pm}}^{L_{i,\pm}}
\int_{S^1} v_{i,\pm}^* \Theta^* d\lambda = 0,
\quad E_\Sigma(v_{i,\pm}) \leq 2\pi + \delta.
$$
Once we have these energy bounds and Theorem \ref{smallenergy}, the
argument from \cite{hofer93}, \cite{HWZ:smallenergy} imply the
following proposition when applied to $C_{i,-}$ and $C_{i,+}$. (See
also chapter 10 \cite{fooo07}.)

\begin{prop} Let $\Sigma_i' = C_{i,-} \cup C_{i,+}$ be the decomposition
mentioned before, and let $v_{i,\pm}$ be the above map restricted to one of
the two components respectively. Then
the sequence $v_{i,\pm}$ converge to holomorphic cylinders
$\widetilde u_{\infty,\pm}: \R \times S^1 \to \R \times S^{2n-1}_p \cong \C^n
\setminus \{0\}$ with
$$
\widetilde u_{\infty,\pm} = (s\circ \widetilde u_{\infty,\pm},
\Theta\circ \widetilde u_{\infty,\pm})
$$
given by
$$
s \circ \widetilde u_{\infty,\pm} (\tau',t') = (2\pi \tau
+ s_{\pm}, \gamma(2\pi t + \theta_{\pm}))
=: u^{\operatorname{flat}}_{a_\pm,s_\pm}
$$
for the real numbers $s_\pm$ and $\theta_\pm$, where $\gamma_\pm$ are the Reeb
orbit associated to the tangent cone of $u_-$ or $u_+$ respectively
on $C_{i,-}$ on $C_{i,+}$.
\end{prop}

\subsection{Convergence in the central region}

Now we focus our attention on the central region
\be
\UU_{i,int}: = u_i^{-1}(B_p^{2n}(\delta_i)).
\label{eq:UUiint}
\ee
By the convergence proved in Theorem \ref{centrallimit}, there exists
$\delta_i \to 0$ such that
$$
u_i(\UU_{i,int}) \subset B_p^{2n}(\delta_i)
$$
and $u_i$ satisfies $\delbar_J u_i = 0$ near
$u_i^{-1}(\del B_p^{2n}(\delta_i))$.
We may choose $\e_{i,in}$ and $\e_{i,out}$ so that
$$
\e_{i,in} < \delta_i < \e_{i,out} < \e_0.
$$
We denote $\Sigma_i'' = u_i^{-1}(B_p^{2n}(\delta_i))$. Then
we have the maps $\widetilde u_i$ that satisfies
$$
\widetilde u_i(\Sigma_i'') \subset B_p^{2n}(\delta_i/\e_i), \quad
\widetilde u_i(\del \Sigma_i'') \subset \del B_p^{2n}(\delta_i/\e_i).
$$
In terms of the orientation convention provided in Definition \ref{orientation},
both boundaries of $\Sigma_i''$ are outside boundaries.

We again consider the rescaled maps $\widetilde u_i : \Sigma_i'' \to T_pM
\cong \C^n$ given by
$$
\widetilde u_i(z) = \frac{1}{\e_i} (\exp_p^I)^{-1} \circ u_i(z).
$$
By definition of $\widetilde J_\e$, this map satisfies
\be\label{eq:dwidetildeui} (d\widetilde u_i + R_{\e_i}^*P_{\e_i f}
(\widetilde u_i))^{(0,1)}_{\widetilde J_{\e_i}} = 0 \ee where
$R_{\e}: \C^n \to \C^n$ is the rescaling map $x \mapsto \e x$ on
$\C^n$.

The following lemma is immediate check whose proof is omitted.

\begin{lem}
We can rewrite (\ref{eq:dwidetildeui}) as
\be\label{eq:BiCi}
\delbar_{\widetilde J_{\e_i}} \widetilde u_i + P_{\vec a}(\widetilde u_i)_{J_p}^{(0,1)} =
C_{\e_i}(\widetilde u_i) \cdot \widetilde u_i
\ee
where $\vec a = \nabla f(p)$ and we have
\be\label{eq:leqBCeiw}
|\widetilde J_{\e_i}(\widetilde u_i) - J_p|\leq C \e_i |u_i|,
\quad |C_{\e_i}(\widetilde u_i)\widetilde u_i|\leq C \delta_i
\ee
as long as $|u| \leq \delta_i/\e_i$.
\end{lem}

We now examine the left hand side of (\ref{eq:BiCi}).
We conformally parameterize $\Sigma_i'' \cong [-L_i, L_i] \times S^1$
with conformal coordinates denoted by $(\tau',t')$.
Then we prove the following lemma by the same way as
Proposition \ref{modCipm}.

\begin{lem}
Let $\operatorname{mod}(\Sigma_i'')$ be the conformal modulus of
$\Sigma_i''$ as defined above. Then
$\operatorname{mod}(\Sigma_i'') \to \infty$.
\end{lem}

We can  write
$$
\widetilde u_i(\tau',t') = - \tau' \vec a + \xi_i(\tau',t')
$$
at least as long as $|\tau \vec a| < \delta_i/\e_i$, or equivalently
for $\tau$ satisfying
$$
|\tau| \leq \frac{\delta_i}{\e_i |\vec a|}.
$$
With this conformal coordinate, we can write
$$
(\delbar_{\widetilde J_{\e_i}} \widetilde u_i
+ P_{\vec a}(\widetilde u_i)_{J_p}^{(0,1)})\left(\frac{\del}{\del \tau'}\right)
= \frac{\del \widetilde u_i}{\del \tau'} + \widetilde J_{\e_i} \frac{\del \widetilde u_i}{\del t'}
+ \vec a = \frac{\del \xi}{\del \tau'} +  \widetilde J_{\e_i} \frac{\del \xi}{\del t'}.
$$
Therefore (\ref{eq:BiCi}) is equivalent to
\be
\frac{\del \xi_i}{\del \tau'} + \widetilde J_{\e_i} \frac{\del \xi_i}{\del t'} =
C_{\e_i}(\widetilde u_i)\left(\frac{\del}{\del \tau'}\right) \cdot \widetilde u_i.
\ee
In particular we have
$$
\left|\frac{\del \xi_i}{\del \tau'} + \widetilde J_{\e_i} \frac{\del \xi_i}{\del t'}\right|
\leq C \d_i
$$
on $B^{2n}(\delta_i/\e_i)$.
Therefore \emph{if we prove that $\xi_i$ (or equivalently $\widetilde u_i$) converges
locally in $C^1$-topology},
then the limit of $\xi_i$ must be holomorphic and hence the local limit of
$\widetilde u_i$ will have the form
$$
-\tau' \vec a + \xi_\infty(\tau',t'), \quad \mbox{with }\, \delbar \xi_\infty =0
$$
as we are expecting. We will now prove this convergence.

Consider the energies of $\widetilde u_i$ given by
\be\label{eq:Eint}
E_{\operatorname{int}}(\widetilde u_i) = \int_{\{z \in \Sigma_i''
\mid \, |\widetilde u_i(z)|_{\C^n}  \le 4 \}}
\widetilde u_i^*d(e^{2s}\lambda)
\ee
and
\be\label{eq:Edlambda}
E_{d\lambda}(\widetilde u_i;S) := \int_{\{z \in \Sigma_i'' \mid \,
|\widetilde u_i(z)|_{\C^n}  \ge  e^S)\}} \widetilde
u_i^*d\lambda.
\ee
Next let $\CC$ be the set of all nonnegative smooth function
$\rho : \R \to \R$ whose support is compact and is contained in
$[2,\infty)$ and  such that $\int \rho(s) = 1$, and
$\widetilde \rho$ be the function defined by
$$
\widetilde \rho(s) = \int_{2}^s \rho(u)\, du.
$$
Then we define
\be\label{eq:neck}
E_{neck}(\widetilde u_i) = \sup_{\rho \in \CC}\int
\widetilde u_i^* d(\widetilde \rho \lambda).
\ee
\begin{lem}\label{energybounded} $E_{\operatorname{neck}}(\widetilde u_i)$ and
$E_{\operatorname{int}}(\widetilde u_i)$ are uniformly bounded above over
$i$.
\end{lem}
\begin{proof} We recall the energy $E_\Sigma(u_i;[\log \e_i, \log \delta_i])$ from
Definition \ref{Elogei} over those $\rho$ defined on
$[\log \e_i, \log \delta_i]$. Then by the same proof as
Lemma \ref{lam-energybound}, we have the uniform upper bound
$$
E_\Sigma(u_i;[\log \e_i,\log \delta_i]) < C
$$
for some $C$ independent of $i$. It is easy to see from the scaling property that
$$
E_{\operatorname{neck}}(\widetilde u_i) \le E(u_i)
$$
and hence $E_{\operatorname{neck}}(\widetilde u_i)$  is uniformly bounded.

On the other hand, we have
$$
E_{\operatorname{int}}(\widetilde u_i) \le
\e_i^{-2}\int_{u_i^{-1}(B_p^{2n}(2 \e_i))}u_i^*\omega_0
$$
by definition. But Stokes' formula gives rise to
\beastar
\int_{u_i^{-1}(B_p^{2n}(2 \e_i))}u_i^*\omega_0 & = & \int_{u_i^{-1}(B_p^{2n}(2 \e_i))}u_i^*d(r^2
\Theta^*\lambda) \\
& = & \int_{u_i^{-1}(\del B_p^{2n}(2 \e_i))} (2 \e_i)^2 \Theta^*\lambda \cong 4 \e_i^2
(2\pi + 2\pi)
\eeastar
by the immersion property of the node and the $\e_i$-controlled convergence of $u_i$ to
$(u_-,u_+,u_0)$ mentioned in the previous section.
This finishes the proof.
\end{proof}

We also prove the following lemma in the same way as Lemma \ref{dlam-energybound}

\begin{lem} We have
$$
\lim_{S\to\infty}\limsup_{i\to\infty} E_{d\lambda}(\widetilde u_i;S)
= 0.
$$
\end{lem}

We can obtain the same kind of estimates for $\xi_i = \widetilde u_i + \tau a$
from the identity
$$
d\xi_i = d\widetilde u_i + a \, d\tau.
$$
\begin{lem} \label{energyxi}
\beastar
\lim_{i \to 0}E_{int}(\widetilde u_i) = \lim_{i\to 0} E_{int}(\xi_i) & = &
4^3 \pi \\
\lim_{i \to 0}|E_{d\lambda}(\widetilde u_i;S) - E_{d\lambda}(\xi_i;S)| & = & 0\\
\lim_{i \to \infty} |E_{neck}(\widetilde u_i) - E_{neck}(\xi_i)|& = & 0.
\eeastar
\end{lem}
\begin{proof} The proofs for $E_{d\lambda}$ and $E_{int}$ are
similar. We will just prove the identity for $E_{int}$.
By definition, we have
\beastar
E_{int}(\widetilde u_i) & = & \int_{\Sigma_i''\cap \widetilde u_i^{-1}(B^{2n}(4))}
\widetilde u_i^*d(r^2\lambda)\\
& = & 4^2 \int_{\del (\Sigma_i''\cap \widetilde u_i^{-1}(B^{2n}(4)))}
(\del\widetilde u_i)^*\lambda \to 4^2
\left(\int \gamma_+^*\lambda + \int \gamma_-^*\lambda\right) \\
& = & 4^2 4 \pi= 4^3\pi.
\eeastar
Here we again used the immersion property of nodes and the fact that both
ends of the cylinder are positive.
The same applies to $\xi_i$ because $\lim_{i\to \infty}|\vec a|/
\left|\frac{\del \widetilde u_i}{\del \tau}\right| \to 0$.

Next we examine $E_{neck}$. For each $\rho \in \CC$, we evaluate
$$
\int \widetilde u_i^* d(\widetilde \rho \lambda)
= \int_{\del B^{2n}(\delta_i)} (\del_+ \widetilde u_i)^*\lambda
$$
where $\del_+\widetilde u_i: = \widetilde u_i|_{\del_+}$ and $\del_+$
is the outside boundary of $\widetilde u_i^{-1}(\del B^{2n}(\delta_i))$.
Therefore we have obtained
$$
E_{neck}(\widetilde u_i) - E_{neck}(\xi_i)
= \int_{\del B^{2n}(r_{out})}
\left((\del_+ \widetilde u_i)^*\lambda -
(\del_+ \xi_i)^*\lambda\right)
$$
for all $\rho \in \CC$. As $i \to \infty$, the conformal coordinates $(\tau',t')$ of
the domain $C_{i,int}:= \widetilde u_i^{-1}(B^{2n}(\delta_i)
\cong [-L_i',L_i'] \times S^1 $
converges to the given coordinates $(\tau,t)$ near $\tau' = L_i'$,
it follows that we have
$$
\left|\frac{\del \tau}{\del t'}\right| \leq C
$$
near $L_i'$ as $i \to \infty$ and so $|a\tau|_{C^1;\del C_i} \to 0$
in the cylindrical metrics of the domain and the target.
Therefore it follows
$$
\left|\widetilde u_i - \xi_i\right|_{C^1;\del C_{i,int}} \to 0
$$
in the cylindrical metric.
We note that this convergence is uniform over $\rho \in \CC$
as long as $\supp \rho$ is contained in a ball $B^{2n}(r)$ of
common radius $r > 0$. Furthermore the convergence of
$\widetilde u_i(\pm L_i',t) \to \gamma_\pm$ as $L_i' \to \infty$.
Combining all these, we obtain
$$
\lim_{i \to \infty}|E_{neck}(\widetilde u_i) - E_{neck}(\xi_i)| = 0.
$$
This finishes the proof.
\end{proof}

We note that both $E_{d\lambda}$ and $E_{neck}$ are invariant under
the automorphisms of $\C^n$, i.e., under homothety and translations.
By applying a suitable sequence of automorphisms
$g_{v_i,\lambda_i}$ to $\xi_i$ we can achieve
\be\label{eq:vixii}
\min_{t\in S^1}|g_{v_i,\lambda_i}\circ \xi_i(0,t)| = 1
\ee
for all $i$.

We now prove the following derivative bound.

\begin{prop}\label{sup<CL} Denote
$\overline \xi_i = g_{v_i,\lambda_i}\circ \xi_i$.
For each $L$, there exists a constant $C = C(L)$ such that
$$
\sup_{-L \leq |\tau'| \leq L} |d\xi_i(\tau',t')| < C(L)
$$
\end{prop}
\begin{proof} The proof will be given by a bubbling-off analysis
which is a variation of the proof of Proposition 27 \cite{hofer93}.
Suppose to the contrary that there exists a sequence
$z_k  \in [-R_0,R_0] \times S^1 \subset \Sigma_k'' \cong [-L_k,L_k] \times S^1$
with $L_k \to \infty$ such that
$$
|d\overline \xi_k(z_k)| \to \infty.
$$
The following is from \cite{hofer-viterbo}, \cite{fooo07}.

\begin{lem}[Lemma 62.149, \cite{fooo07}]
There exists another sequence $z_i' \in [-R_0-1,R_0+1] \times S^1$
satisfying the following properties :
\begin{enumerate}
\item $|d\overline \xi_i(z_i')|: = C_i \to \infty$
\item If $d_{g_\C'}(z',z_i') \leq C_i^{-1/2}$ for $z' \in \C$, then
$|d\overline \xi_i|_{g_\C',g_{\C^n}'} \leq 2 C_i$.
\end{enumerate}
\end{lem}

The following is a verbatim translation of Lemma 62.151 \cite{fooo07}
in our context. For readers' convenience, we duplicate
it therefrom with minor modifications.

\begin{lem}[Lemma 62.151, \cite{fooo07}]
\label{tildeuibdd} The sequence $\overline \xi_i(z'_i) \in \C^n$
is bounded.
\end{lem}
\begin{proof} The proof is by contradiction.
Suppose to the contrary that
$$
R_{3,i} = \vert \overline \xi_i(z'_i)\vert_{\C^n} \to \infty.
$$
We put
$$
D_i = \{ u \in \C \mid {\operatorname{dist}}_{g'_{\H}}(C_i
^{-1}u+z'_i,z'_i) < \min\{C_i^{-1}\sqrt{R_{3,i}}/2,C_i^{
-1/2}\}, \, \,\, C_i^{-1}u+z'_i \in \H\}.
$$
We note that $D_i$ is a convex domain of its diameter with the
order of
$$
\min\{\sqrt{R_{3,i}}/2, C_i^{1/2}\}
$$
which goes to $\infty$ as $i \to \infty$ by the hypotheses.

We define $ \widetilde{\xi}_i : D_i \to \C^n $ by
$$
\widetilde{\xi}_i(u) = \overline \xi_i(C_i^{-1}u+z'_i).
$$
Then we have
\be\label{eq:|dxi|geq1}
|d\widetilde \xi_i(z_i)| \geq 1.
\ee
We now prove
\be\label{eq:>2S0}
\inf_{u \in D_i}\vert \widetilde{\xi}_i(u)\vert \ge
\sqrt{R_{3,i}}\left(\sqrt{R_{3,i}} -1\right) > 2S_0
\ee
if $i$ is sufficiently large. We note
\bea\label{eq:tildetildeui}
\vert \widetilde{\xi}_i(u)\vert & \geq & \vert
\widetilde{\xi}_i(0)\vert - \vert \widetilde{\xi}_i(u)
- \widetilde{\xi}_i(0)\vert \nonumber \\
& = & \vert {\xi}_i(z_i')\vert - \vert
\widetilde{\xi}_i(u) - \widetilde{\xi}_i(0)\vert.
\eea
We have $\vert \overline \xi_i(z_i')\vert = R_{3,i}$ and
\beastar
\vert \widetilde{\xi}_i(u) - \widetilde{\xi}_i(0)\vert & \leq & \int_0^1 |u\cdot \nabla
\widetilde{\xi}_i(su)|\, ds \\
& = & \int_0^1 |u\cdot C_i^{\prime -1}\nabla {\overline \xi}_i(
C_i^{\prime -1}(su) + z_i') \vert\, ds \\
& \le & \int_0^1 |C_i^{\prime -1}u||\nabla {\overline \xi}_i(
C_i^{\prime -1}(su) + z_i') \vert\, ds.
\eeastar
But since $su \in D_i$ for all $s \in [0,1]$, we have
$$
\operatorname{dist}(C_i^{\prime -1}(su) + z'_i,z'_i) \le
C_i^{\prime -1/2}.
$$
Then (\ref{eq:tildetildeui}) implies
$$
|\nabla {\overline\xi}_i( C_i^{\prime -1}(su) + z_i')| \leq 2C_i'.
$$
Therefore we have
$$
\vert \widetilde{\xi}_i(u) - \widetilde{\xi}_i(0)\vert \leq 2\vert u\vert \le \sqrt{R_{3,i}}.
$$
Substituting these into (\ref{eq:tildetildeui}), we derive
$$
\vert \widetilde{\xi}_i(u)\vert \geq R_{3,i} -
\sqrt{R_{3,i}} = \sqrt{R_{3,i}}(\sqrt{R_{3,i}}
-1).
$$
This finishes the proof of (\ref{eq:>2S0}).

Since $(H_{-1}^\alpha)' \cap (\C^n\setminus B^{2n}(2S_0)) \subset
\R^n \cup \Lambda$, (\ref{eq:>2S0}) allows us
to regard $\widetilde{\xi}_i$ as a sequence of maps
$$
\widetilde{\xi}_i : D_i \to \R \times S^{2n-1} \cong \C^n \setminus \subset
\C^n.
$$
We derive from Lemma \ref{energyxi}
$$
E(\widetilde{\xi}_i) \le E_0, \quad
E_{d\lambda}(\widetilde{\xi}_i) \to 0.
$$
Then we can find $s'_i \to
\infty$ and a subsequence such that $\frak T_{s'_i} \circ
\widetilde{\xi}_i$ converges to a map
$$
\widetilde{\xi}_{\infty} : D_{\infty}
\to \R \times S^{2n-1}
$$
in compact $C^{\infty}$ topology. Therefore we derive
 $|d\widetilde \xi_\infty(z_\infty)| \geq 1$ from (\ref{eq:|dxi|geq1}).
But this gives rise to a contradiction,
which finishes the proof.
\end{proof}

Now we go back to the proof of Proposition \ref{sup<CL}.

Define a new map $\widetilde v_k : D_k \to \C^{n}$ by
$$
\widetilde v_k (u) = \overline \xi_k\left(z_k +\frac{u}{C_k}\right)
$$
where $D_k \subset \C$ is defined by
$$
D_k = \{u \in \C \mid d_{cyl}(z_k' + u/C_k,z_k') < C_k^{-1/2}, \,
z_k' + u/C_k \in [-L_k,L_k] \times S^1 \}
$$
Since $z_k' \in [-R_0,R_0] \times S^1$, it follows
\beastar
z_k' + u/C_k &\in & [-R_0 + C_k^{-1/2}, R_0 + C_k^{-1/2}] \times S^1\\
&\subset & [-R_0 -1, R_0 +1] \times S^1 \subset [-L_k,L_k] \times S^1
\eeastar
and so the map $\widetilde v_k$ is well-defined on
$D_k \cong B^2(C_k^{1/2})$. Then $\widetilde v_k$ satisfies the following
properties :
\begin{enumerate}
\item $\widetilde v_k(0) = \overline \xi_k(z_k)$ is bounded,
\item $E(\widetilde v_k) < C$,
\item $\int_{D_k} \widetilde v_k^*\Theta^*d\lambda \to 0$ as $k \to \infty$
\item $|d\widetilde v_k(u)|\leq 2$ on $D_k$ and $|d\widetilde v_k(0)| = 1$
\item $|\delbar \widetilde v_k| \to 0$ as $k \to \infty$.
\end{enumerate}
Therefore by taking a diagonal subsequence of $\widetilde v_k$ converges to
a holomorphic map $\widetilde v_\infty : \C \to \C^n$ that satisfies
\be\label{eq:const}
\int_\C \widetilde v_\infty^* \Theta^*d\lambda = 0,
\quad E(\widetilde v_\infty) < \infty
\ee
and
\be\label{eq:nonconst}
|d\widetilde v_\infty(0)| = 1,\quad |d\widetilde v_\infty(u)| \leq 2.
\ee
But (\ref{eq:const}) implies $\widetilde v_\infty$ must be constant
while (\ref{eq:nonconst}) implies it cannot, a contradiction.
This finishes the proof of Proposition \ref{sup<CL}.
\end{proof}

By the elliptic regularity, we derive from (\ref{eq:vixii})
and Proposition \ref{sup<CL} that the $C^{k}$ norm of $\overline \xi_i$
for all $k \geq 0$ is uniformly bounded on any bounded
subset of $\R \times S^1$.
Therefore, by Ascoli-Arzela's theorem, we can find a subsequence
of $\overline \xi_i$ that converges to a holomorphic map
$$
\xi_{\infty} : \R \times S^1 \to \C^{n}
$$
in compact $C^{\infty}$ topology. By (\ref{eq:vixii}), $\xi_\infty$
cannot be a constant map.

The following energy bound is an
immediate consequences of Lemma \ref{energybounded}.

\begin{lem}\label{uinftybound} $E_{\operatorname{int}}(\xi_{\infty})$
and $E_{\operatorname{neck}}(\xi_{\infty})$ are finite.
\end{lem}

Next we prove the following theorem.

\begin{thm}\label{scaledconvergence}
There exists a sequence of vectors $v_i$ and a
subsequence of $\overline \xi_i = \xi_i - v_i$ that converges to
a holomorphic map
$$
\xi_\infty: \R \times S^1 \to \C^n
$$
in compact $C^\infty$-topology satisfying the following
properties :
\begin{enumerate}
\item $E_{int}(\xi_\infty)$ and $E_{neck}(\xi_\infty)$ are
finite.
\item In the decomposition $\xi_\infty = (s\circ \xi_\infty,
\Theta \circ \xi_\infty)$ outside $B^{2n}(1)$, we have
\beastar
\lim_{\tau' \to \infty} \Theta\circ \xi(\tau',t) & = & \gamma_+(t),\\
\lim_{\tau' \to -\infty} \Theta\circ \xi(\tau',t) & = & \gamma_-(t)
\eeastar
where $\gamma_\pm$ are the Reeb orbits of $S^{2n-1}(1) \subset \C^n
\cong (T_pM,\omega_p,J_p)$ associated to the tangent cones of $u_+, \,
u_-$ at the node $p = u_+(\infty) = u_-(-\infty)$ respectively.
\end{enumerate}
\end{thm}

\begin{proof} We start with the following result
proved by Hofer \cite{hofer93}.

\begin{lem}[Theorem 31, \cite{hofer93}]\label{hofer93}
Suppose that $\xi_\infty$ is a proper non-constant pseudo-holomorphic
with finite $E_\Sigma$-energy. There exists a closed Reeb
orbit $\gamma: S^1 \to S^{2n-1}$ and a sequence $\tau_k \to \infty$
such that $\gamma_k = \xi_\infty(\tau_k, \cdot)$
converges in $C^\infty$ to $\gamma$.
Similar statement holds also for $\tau_k \to \infty$.
\end{lem}

We will now improve this convergence to

\begin{prop}\label{limgammapm} $\xi_{\infty}$ is a proper holomorphic cylinder
such that
$$
\lim_{\tau \to \pm\infty}\xi_\infty(\tau,\cdot) = \gamma_\pm
$$
in $C^\infty$ where $\gamma_\pm$ are the Reeb orbits associated to
the tangent cones of the node of $(u_-,u_+)$.
\end{prop}
\begin{proof}
The main tool for such a convergence result is Theorem \ref{smallenergy} the
characterization of the asymptotics of $J_0$-holomorphic maps with
small $d\lambda$-energy $E_{d\lambda}$.

Let $\gamma_a$ be the Reeb orbit provided in Theorem \ref{smallenergy}
for $u = \xi_\infty$. We will treat only the case as $\tau \to + \infty$
since the case $\tau \to -\infty$ will be the same. In our situation,
we have the vector $a = a\pm = \frac{du_\pm(o_\pm)}{|du_\pm(o_\pm)|}$.

We will show that there exists a constant
$s_1 \in \R$ such that $\xi_\infty$ satisfies
$$
|\xi_{\infty}(z) - u^{\text{\rm flat}}_{a,s_{1}}(z)|_{C^k} \to 0
$$
in exponential order as $|z| \to \infty$.

Let $E_0 = E_{\operatorname{neck}}(\xi_{\infty})$. We take $e_0$ as
in Theorem \ref{smallenergy}. Since $E(\xi_{\infty}) < \infty$, we can
choose $S$ such that
$$
E(\xi_{\infty};S) \le e_0.
$$
Then, we can apply Theorem \ref{smallenergy} to the restriction of $\xi_{\infty}$ to $[S,S+2R] \times [0,1]$.

Note $\xi_{\infty}([S,S+2R] \times [0,1]) \subset [\log 4,\infty) \times S^{2n-1}$.
Put
$$
\gamma(t) = \xi_{\infty}(S+R,t), \quad \gamma_i(t) =
\xi_i(S+R,t).
$$
Then, by Lemma \ref{lam-energybound}, we have :
$$
\int_{0}^1 \gamma^*\lambda
= \lim_{i\to\infty}\int_{0}^1 \gamma_i^*\lambda
\le 3\pi.
$$
Therefore we have
constants $R_{2,j}$ and  $s_{1,j}$ such that $R_{2,j} \to
\infty$ and
\be\label{eq:S+R2j}
|\nabla^k(\xi_{\infty} - u^{\text{\rm
flat}}_{a_j,s_{1,j}})|(\tau,t) \le C_ke^{-
c_k  |\tau - S - R_{2,j}|}
\ee
on $(\tau,t) \in [S+10,S-10+2R_{2,j}] \times S^1$.
\par
Since the intervals $[S+10,S-10+2R_{2,j}]$ are nested as $R_{2,j} \nearrow
\infty$, we should also have $s_{1,j} \to s_1$ as $j \to \infty$ for
$s_1$ appearing in Theorem \ref{smallenergy}. Then (\ref{eq:S+R2j}) implies
$$
\vert \nabla^k(\xi_{\infty} - u^{\text{\rm
flat}}_{a_{\infty},s_{1}})\vert(\tau,t) \le
C'_ke^{- c'_k  \vert\tau \vert},
$$
on $(\tau,t) \in [S+10,\infty) \times [0,1]$. Therefore we have
finished the proof.
\end{proof}

Theorem \ref{scaledconvergence} follows from Proposition \ref{limgammapm}.
\end{proof}

Since every Reeb orbits of $S^{2n-1}$ with the action
$\int \gamma^*\lambda \leq 3 \pi$ is one of $\gamma_a$, we have
$\gamma_\pm = \gamma_{a_\pm}$ for some $a_\pm \in S^{2n-1}$.
To finish the proof of Theorem \ref{1-jetconvergence},
it remains to prove the $\e$-controlled convergence (\ref{econtrolled}).

We take an isomorphism $\psi : \R \times S^1 \to \R \times S^1$ such that $\psi(\pm\infty) =
\pm\infty$ and
$$
d_{g_{\C^n}'}(\xi_{\infty} \circ \psi(0,0),0) = \min_{(\tau,t) \in \R \times S^1}
d_{g_{\C^n}'}(\xi_\infty(\tau,t),0).
$$
We now define the map
$$
\psi_{i,\operatorname{int}} : [-\infty,R_i) \times S^1 \to \H
$$
for some $R_i$ to be determined later in the proof. Since we have
$$
\frac{1}{\e_i}( (\exp_p^I)^{-1}\circ u_i \circ \psi_{i,\operatorname{int}})
+ \tau' a = \overline \xi_i \circ \psi_{i,int}
= g_{v_i,\lambda_i} \circ \xi_i \circ \psi_{i,int}
$$
by the definitions of $g_{v_i,\lambda_i}$ and $\xi_i$ and
$g_{v_i,\lambda_i} \xi_i \circ \psi_{i,int}$ converges
to $\xi_\infty$, it follows that
$g_{v_i,\lambda_i}^{-1}\left(\frac{1}{\e_i} ((\exp_p^I)^{-1}\circ u_i \circ \psi_{i,\operatorname{int}})
- \tau' a\right)$ converges to $u_a^{\text{\rm flat},s_1}$
on compact $C^{\infty}$ topology. For the notational convenience, we will
drop $(\exp_p^I)^{-1}$ from $(\exp_p^I)^{-1}\circ u_i \circ \psi_{i,\operatorname{int}}$
and just denote it by $u_i \circ \psi_{i,\operatorname{int}}$.
\par

By the diagonal sequence argument, we can choose a sequence $R_i\to
\infty$ so that
\be\label{eq:limsup}
\lim_{i\to\infty}\sup_{R_0 \leq \tau' \leq 2R_i}
\left| \nabla^k \left(g_{v_i,\lambda_i}^{-1}\left(\frac{1}{\e_{1,i}}u_i \circ \psi_{i,int}
(\tau',t') + \tau' a\right)- u_{a,s_1}^{\text{\rm flat}}\right) \right| = 0.
\ee
It follows
that there exist $S_3$, $I_0$ such that the following holds for $i
\ge I_0$ :
\begin{enumerate}
\item
$$
\int_{[S_3,2R_i)\times [0,1]} \left(\frac{1}{\e_{1,i}}
\left(g_{v_i,\lambda_i}^{-1} \circ u_i \circ
\psi_{i,\operatorname{int}}\right)\right)^* d\lambda < e_0
$$
\item $2R_i - S_3 \ge R_0$.
\end{enumerate}
We can apply Theorem \ref{smallenergy} to obtain  $s'_i$ such that
\be\label{eq:2Ri-tau}
\left|\nabla^k \left(g_{v_i,\lambda_i}^{-1}\left(\frac{1}{\e_i} (u_i \circ \psi_{i,\text{int}}(\tau',t')
+ \tau' a\right)
- u^{\text{\rm flat}}_{a'_i,s'_i}(\tau',t')\right)\right| \le C_k e^{-c_k
\min\{\vert 2R_i-\tau\vert, \vert \tau - S_3|\}}.
\ee
Comparing (\ref{eq:limsup}) with (\ref{eq:2Ri-tau}) we have $s'_i \to
s_1$. Perturbing $\psi_{i,\operatorname{int}}$ slightly and re-choosing $s_i$,
we may assume $s'_i = 0$.

Therefore we obtain
$$
\left|\nabla^k \left(g_{v_i,\lambda_i}^{-1}\left(\frac{1}{\e_i} (u_i \circ \psi_{i,\text{int}}(\tau',t')
+ \tau' a \right)- u^{\text{\rm flat}}_{a'_i,s'_i}(\tau',t')\right)\right|
\le C_k e^{-c_k \min\{\vert 2R_i-\tau\vert, \vert \tau - S_3|\}}.
$$
Now the proof of Theorem \ref{1-jetconvergence} is finished.

\part{Application : a proof of PSS isomorphism}
\label{part:proof}

In this part, we combine the analysis carried out in the previous
sections with the standard cobordism argument to give the proof of
$\Psi\circ \Phi = id$ in homology. For completeness's sake, we also
give an explanation of the proof $\Phi\circ \Psi = id$ whose proof can be
given by a more or less standard argument in Floer theory.
The isomorphism proof in this part is complete as it is for
the semi-positive $(M,\omega)$. However we have been careful
to provide our compactification of the relevant moduli spaces
so that one can easily put Kuranishi structure \cite{fukaya-ono} on them
to generalize the isomorphism property to arbitrary
compact $(M,\omega)$. Since this is not our main purpose of the
paper, we do not pursue complete details and leave them for
interested readers.

\begin{rem} For example, observing that Proposition \ref{trans-Upsilon}
holds for a generic choice of almost complex structures on any symplectic
manifold, whether it is semi-positive or not, one can repeat
the construction carried out in \cite{LuG} in our setting instead
of in the setting of \cite{PSS} that \cite{LuG} uses.
\end{rem}

\section{Review of Floer complex and operators}

In this section, we give a brief summary of basic operators in the
standard Floer homology theory. Details of construction of these
operators are important for the argument in our proof
of isomorphism property of the PSS map. While these constructions
are standard, we closely follow the exposition presented in
\cite{oh:alan,oh:montreal}.

For each nondegenerate $H:S^1 \times M \to \R $ with $\phi_H^1 = \phi$,
we know that the
cardinality of $\mbox{Per}(H)$ is finite. We consider the
free $\Q$ vector space generated by the critical set of $\CA_H$
$$
\mbox{Crit}\CA_H = \{[z,w]\in \widetilde\Omega_0(M) ~|~ z \in
\mbox{Per}(H)\}.
$$
\begin{defn}\label{novikovchain} Consider the formal sum
\be\label{eq:beta}
\beta = \sum _{[z, w] \in \mbox{\rm Crit}\CA_H} a_{[z, w]}
[z,w], \, a_{[z,w]} \in \Q
\ee
\begin{enumerate}
\item We call those $[z,w]$ with $a_{[z,w]}
\neq 0$ {\it generators} of the sum $\beta$ and write
$$
[z,w] \in \beta.
$$
We also say that $[z,w]$ {\it contributes} to
$\beta$ in that case.
\item We define the {\it support} of $\beta$ by
$$
\mbox{supp}(\beta): = \{ [z,w] \in \mbox{Crit}\CA_H \mid
a_{[z,w]} \neq 0 \, \mbox{ in the sum (\ref{eq:beta})}\}.
$$
\item We call the formal sum $\beta$
a {\it Novikov Floer chain} (or simply a {\it Floer chain}) if
\be\label{eq:Novikov}
\#\Big(\mbox{supp}(\beta) \cap \{[z,w] \mid \CA_H([z,w])
\geq \lambda \}\Big) < \infty
\ee
for any $\lambda \in \R$. We denote by $CF_*(H)$
the set of Floer chains.
\end{enumerate}
\end{defn}

We now explain the description of $CF(H)$ as a module over
the {\it Novikov ring} as in \cite{floer:fixed}, \cite{hofer-sal}.
Consider the abelian group
$$
\Gamma = \frac{\pi_2(M)}{\ker c_1 \cap \ker \omega}
$$
and the formal sum
$$
R = \sum_{A \in \Gamma} r_Aq^{A}, \quad r_A \in \Q.
$$
We define
$$
\supp(R) = \{ A \in \Gamma \mid r_A \neq 0 \}.
$$
The (upward) Novikov ring defined by
$$
\Lambda_\omega = \Lambda_\omega^\uparrow = \left\{
\sum_{A \in \Gamma} r_Aq^{A} \mid
\forall \lambda \in \R, \#\{A \in \Gamma \mid r_A \neq 0,
\omega(A) < \lambda\} < \infty \right\}.
$$
Then we have the valuation on $\Lambda_\omega$ given by
\be\label{eq:downv}
v(R)
= \min\{\omega(A) \mid A \in \mbox{supp }R\}.
\ee
We recall that $\Gamma$ acts on
$\Crit\CA_H$ by `gluing a sphere'
$$
[z,w] \mapsto [z, w\# (-A)]
$$
which in turn induces the multiplication of $\Lambda_\omega$ on
$CF(H)$ by the convolution product. This enables one to regard
$CF(H)$ as a $\Lambda_\omega$-module. We will try to consistently
denote by $CF(H)$ as a $\Lambda_\omega$-module, and by $CF_*(H)$
as a graded $\Q$ vector space.

Suppose $H$ is a nondegenerate one-periodic Hamiltonian function
and $J$ a one-periodic  family of compatible
almost complex structures.  We first
recall Floer's construction of the Floer boundary map, and the
transversality conditions needed to define the Floer homology
$HF_*(H,J)$ of the pair.

The following definition is useful for the later discussion.

\begin{defn}\label{eq:pi2zz'} Let $z, \, z' \in \text{Per}(H)$. We
denote by $\pi_2(z,z')$ the set of homotopy classes of smooth maps
$$
u: [0,1] \times S^1  \to M
$$
relative to the boundary
$$
u(0,t) = z(t), \quad u(1,t) = z'(t).
$$
We denote by $[u] \in \pi_2(z,z')$ its homotopy class and by $C$ a
general element in $\pi_2(z,z')$.
\end{defn}

We define by $\pi_2(z)$ the set of relative homotopy classes
of the maps
$$
w: D^2 \to M; \quad w|_{\del D^2} = z.
$$
We denote by $\pi_2(M)$ the \emph{free} homotopy class of maps $u: S^2 \to M$
which forms a \emph{groupoid}. We note that $\pi_2(M)$ is not the usual 2-nd homotopy
group, i.e., not the set homotopy classes of \emph{based} maps.
We note that there is a natural action of the groupoid $\pi_2(M)$ on $\pi_2(z)$
and $\pi_2(z,z')$ by the obvious operation of a `gluing a sphere'.
Furthermore there is a natural map of $C \in \pi_2(z,z')$
$$
(\cdot) \# C: \pi_2(z) \to \pi_2(z')
$$
induced by the gluing map
$$
w \mapsto w \# u.
$$
More specifically we will define the map $w \# u: D^2 \to M$ in
the polar coordinates $(r,\theta)$ of $D^2$ by the formula
\be\label{eq:wsharpu}
w \# u:(r,\theta) = \begin{cases} w(2r,\theta) & \quad  \text{for }\, 0
\leq r \leq \frac{1}{2} \\
w(2r-1,\theta)  & \quad  \text{for } \, \frac{1}{2} \leq r \leq 1
\end{cases}
\ee
once and for all. There is also the natural gluing map
$$
\pi_2(z_0,z_1) \times \pi_2(z_1,z_2) \to \pi_2(z_0,z_2)
$$
$$
(u_1, u_2) \mapsto u_1\# u_2.
$$
We also explicitly represent the map $u_1\# u_2: [0,1] \times S^1 \to M$ in the
standard way once and for all similarly to (\ref{eq:wsharpu}).

\begin{defn} We define the {\it relative
Conley-Zehnder index} of $C \in \pi_2(z,z')$ by
$$
\mu_H(z,z';C) = \mu_H([z,w]) - \mu_H([z',w\# C])
$$
for a (and so any) representative $u:[0,1] \times S^1 \times M$ of
the class $C$. We will also write $\mu_H(C)$, when there is no
danger of confusion on the boundary condition.
\end{defn}
It is easy to see that this definition does not
depend on the choice of bounding disc $w$ of $z$, and so the
function
$$
\mu_H: \pi_2(z,z') \to \Z
$$
is well-defined.

We now denote by
$$
\CM(H,J;z,z';C)
$$
the set of finite energy solutions of
\be\label{eq:HJCR}
\dudtau + J\Big(\dudt - X_H(u)\Big) = 0
\ee
with the asymptotic condition and the homotopy condition
\be\label{eq:asymp[u]=C}
u(-\infty) = z, \quad u(\infty) = z'; \quad [u] = C.
\ee
(See \cite{floer:fixed}, \cite{hofer-sal}.)
Here we remark that although $u$ is a priori defined on $\R \times
S^1$, it can be compactified into a continuous map $\overline u:
[0,1] \times S^1 \to M$ with the corresponding boundary condition
$$
\overline u(0) = z, \quad \overline u(1) = z'
$$
due to the exponential decay property of finite energy
solutions $u$ of (4.2),
recalling we assume $H$ is nondegenerate. We will call $\overline
u$ the {\it compactified map} of $u$. By some abuse of notation,
we will also denote by $[u]$ the class $[\overline u]\in
\pi_2(z,z')$ of the compactified map $\overline u$.

The Floer boundary map
$$
\del_{(H,J)}; CF_{k+1}(H) \to CF_k(H)
$$
is defined under the following conditions by studying the equation
(\ref{eq:HJCR}) for a Floer-regular pair $(H,J)$
and satisfies $\del\del = 0$, which enables us to take its homology.
The Floer homology is defined by
$$
HF_*(H,J): =\ker \del /\operatorname{im}\del.
$$
One may regard this either as a graded $\Q$-vector space or as
a $\Lambda_\omega$-module.

Next we describe the Floer chain map.
When we are given a family $(\CH,j)$ with $\CH = \{H^s\}_{0\leq s
\leq 1}$ and $j = \{J^s\}_{0\leq s \leq 1}$ and a cut-off function
$\rho:\R \to [0,1]$, the chain homomorphism
$$
h_\CH=h_{(\CH,j)}: CF_*(H_\alpha) \to CF_*(H_\beta)
$$
is defined by considering the non-autonomous form of (\ref{eq:HJCR}).

Consider the pair $(\CH_\R,j_\R)$
that are {\it asymptotically constant}, i.e., there exists $R > 0$
such that
$$
J(\tau) \equiv J(\infty), \quad H(\tau) \equiv H(\infty)
$$
for all $\tau$ with $|\tau| \geq R$. We will always consider the form
$$
(\CH_\R,j_\R) = \{(H^{\rho(\tau)},J^{\rho(\tau)})\}
$$
where $(H^s,J^s)$ is a homotopy over $s \in [0,1]$ and $\rho:\R \to [0,1]$
is a function as defined before. We study the following equation
(\ref{eq:HJCR})
\be\label{eq:HHjrho} \dudtau
+ J^{\rho(\tau)}\Big(\dudt- X_{H^{\rho(\tau)}}(u)\Big) = 0. \ee We
denote by
$$
\MM(\CH,j;\rho)
$$
the set of finite energy solutions of (\ref{eq:HHjrho}).

For a Floer-regular pair $(\CH,j)$, we can define a
continuous map of degree zero
$$
h_{(\CH,j;\rho)}: CF(H_\alpha) \to CF(H_\beta)
$$
by the matrix element
$n_{(\CH,j;\rho)}([z_\alpha,w_\alpha],[z_\beta,w_\beta])$ similarly as
for the boundary map. Then $h_{(\CH,j)}$ has degree 0 and satisfies
the identity
$$
h_{(\CH,j;\rho)}\circ \del_{(H_\alpha,J_\alpha)}
=\del_{(H_\beta,J_\beta)}\circ h_{(\CH,j;\rho)}.
$$
Two such chain maps $h_{(j^1,\CH^1)}, \, h_{(j^2,\CH^2)}$ are also
chain homotopic \cite{floer:fixed}.

Now we examine Floer chain homotopy maps and the composition law
$$
h_{\alpha\gamma} = h_{\beta\gamma} \circ h_{\alpha\beta}
$$
of the Floer isomorphism
\be\label{eq:halphabeta}
h_{\alpha\beta}: HF_*(H_\alpha) \to HF_*(H_\beta).
\ee
Although the above isomorphism {\it in homology} depends
only on the end Hamiltonians $H_\alpha$ and $H_\beta$, the
corresponding chain map depends on the homotopy $\CH =
\{H(\eta)\}_{0 \leq \eta \leq 1}$ between $H_\alpha$ and
$H_\beta$, and also on the homotopy $j = \{J(\eta)\}_{0\leq \eta
\leq 1}$. Let us fix nondegenerate Hamiltonians $H_\alpha, \,
H_\beta$ and a homotopy $\CH$ between them. We then fix a homotopy
$j = \{J(\eta)\}_{0 \leq \eta \leq 1}$ of compatible almost
complex structures and a cut-off function $\rho:\R \to [0,1]$.

We recall that we have imposed the homotopy condition
\be\label{eq:w^+sharpu}
[w^+]=[w^-\# u] ; \quad [u] = C \quad \text{ in } \quad
\pi_2(z^-,z^+)
\ee
in the definition of $\CM(H,J;[z^-,w^-],[z^+,w^+])$
and of $\CM((\CH,j;\rho);[z_\alpha,w_\alpha],
[z_\beta,w_\beta])$. One consequence of (\ref{eq:w^+sharpu}) is
$$
[z^+,w^+] = [z^+, w^-\# u] \quad \text{ in } \quad \Gamma
$$
but the latter is a weaker condition than the former.
In other words, there could be more than one distinct elements
$C_1, \, C_2 \in \pi_2(z^-,z^+)$ such that
$$
\mu(z^-,z^+;C_1) = \mu(z^-,z^+;C_2), \quad \omega(C_1) =
\omega(C_2).
$$
When we are given a homotopy $(\overline j, \overline \CH)$ of
homotopies with $\overline j = \{j_\kappa\}$, $\overline\CH =
\{\CH_\kappa\}$, we also define the elongations
$\CH^{\overline\rho}$ of $\CH_\kappa$ by a homotopy of cut-off
functions $\overline \rho=\{\rho_\kappa\}$: we have
$$
\CH^{\overline\rho} = \{ \CH_\kappa^{\rho_\kappa} \}_{0 \leq
\kappa \leq 1}.
$$
Consideration of the parameterized version of (\ref{eq:HHjrho})
for $ 0 \leq
\kappa \leq 1$ defines the chain homotopy map
$$
H_{\overline\CH} :CF_*(H_\alpha) \to CF_*(H_\beta)
$$
which has degree $+1$ and satisfies
\be\label{eq:chn-homotopy}
h_{(j_1, \CH_1;\rho_1)} - h_{(j_0,\CH_0:\rho_0)} =
\del_{(J^1,H^1)} \circ H_{\overline\CH} + H_{\overline\CH} \circ
\del_{(J^0,H^0)}.
\ee
Again the map $H_{\overline\CH}$ depends on the choice of a
homotopy $\overline j$ and $\overline\rho = \{\rho_\kappa\}_{0
\leq \kappa \leq 1}$ connecting the two functions $\rho_0, \,
\rho_1$. Therefore we will denote
$$
H_{\overline \CH} = H_{(\overline \CH,\overline j; \overline
\rho)}
$$
as well. Equation (\ref{eq:chn-homotopy}) in particular proves that two chain maps  for
different homotopies $(j_0,\CH_0;\rho_0)$ and $(j_1,
\CH_1;\rho_1)$ connecting the same end points are chain homotopic
 and so proves that the isomorphism (\ref{eq:halphabeta})
in homology is independent of the homotopies
$(\overline\CH,\overline j)$ or of $\overline \rho$.

Next, we consider the triple
$$
(H_\alpha, \, H_\beta, \, H_\gamma)
$$
of Hamiltonians and homotopies $\CH_1, \, \CH_2$ connecting from
$H_\alpha$ to $H_\beta$ and $H_\beta$ to $H_\gamma$ respectively.
We define their concatenation $\CH_1 \# \CH_2 = \{H_3(s)\}_{1 \leq
s \leq 1}$ by
$$
H_3(s) = \begin{cases} H_1(2s) &\quad 0 \leq s \leq \frac{1}{2} \\
H_2(2s-1) & \quad \frac{1}{2} \leq s \leq 1.
\end{cases}
$$
We note that due to the choice of the cut-off function
$\rho$, the continuity equation (\ref{eq:HHjrho}) is {\it autonomous}
for the region $|\tau| > R$ i.e., is invariant under the translation
by $\tau$. When we
are given a triple $(H_\alpha, \, H_\beta, \, H_\gamma)$, this
fact enables us to glue solutions of two such equations
corresponding to the pairs $(H_\alpha,H_\beta)$ and
$(H_\beta,H_\gamma)$ respectively.

Now a more precise explanation is in order. For a given pair of
cut-off functions
$$
\rho = (\rho_1, \rho_2)
$$
and a positive number $R > 0$, we define an elongated homotopy of
$\CH_1 \# \CH_2$
$$
\CH_1 \#_{(\rho;R)} \CH_2 =\{ H_{(\rho;R)}(\tau) \}_{-\infty <
\tau < \infty}
$$
by
$$
H_{(\rho;R)}(\tau,t,x)
= \begin{cases} H_1(\rho_1(\tau + 2R),t,x) & \quad \tau \leq 0 \\
H_2(\rho_2(\tau - 2R), t, x) & \quad \tau \geq 0.
\end{cases}
$$
Note that
$$
H_{(\rho;R)} \equiv
\begin{cases} H_\alpha  \quad & \text{for } \, \tau
\leq - (R_1+2R)\\
H_\beta \quad  & \text{for } \, -R \leq  \tau \leq R\\
H_\gamma \quad & \text{for } \, \tau \geq R_2 + 2R
\end{cases}
$$
for some sufficiently large $R_1, \, R_2 > 0$ depending on the
cut-off functions $\rho_1, \, \rho_2$ and the homotopies $\CH_1,
\, \CH_2$ respectively. {\it In particular this elongated homotopy
is always smooth, even when the usual glued homotopy $\CH_1\#
\CH_2$ may not be so.} We define the elongated homotopy
$j_1\#_{(\rho;R)} j_2$ of $j_1\# j_2$ in a similar way.

For an elongated homotopy $(j_1\#_{(\rho;R)} j_2, \CH_1 \#_{(\rho,
R)} \CH_2)$, we consider the associated perturbed Cauchy-Riemann
equation
$$
\begin{cases} \frac{\del u}{\del \tau} +
J_3^{\rho(\tau)}\Big(\frac{\del u}{\del t}
- X_{H_3^{\rho(\tau)}}(u)\Big) = 0\\
\lim_{\tau \to -\infty}u(\tau) = z^-,  \, \lim_{\tau \to
\infty}u(\tau) = z^+
\end{cases}
$$
with the condition (\ref{eq:w^+sharpu}).

Now let $u_1$ and $u_2$ be given solutions of (\ref{eq:HHjrho})
associated to $\rho_1$ and $\rho_2$ respectively. If we define the
pre-gluing map $u_1 \#_R u_2$ by the formula
$$
u_1\#_R u_2(\tau,t) =
\begin{cases} u_1(\tau + 2R,t)  & \quad\text{for }\, \tau \leq -R \\
u_2(\tau - 2R, t) & \quad\text{for }\, \tau \geq R
\end{cases}
$$
and a suitable interpolation between them by a partition of unity
on the region $ -R \leq \tau \leq R$, the assignment defines a
diffeomorphism
$$
(u_1, u_2, R) \to u_1 \#_R u_2
$$
from
$$
\CM\Big(j_1,\CH_1;[z_1,w_1],[z_2,w_2]\Big) \times
\CM\Big(j_2,\CH_2;[z_2,w_2],[z_3,w_3]\Big) \times (R_0, \infty)
$$
onto its image, provided $R_0$ is sufficiently large. Denote by
$\overline \del_{(\CH,j;\rho)}$ the corresponding perturbed
Cauchy-Riemann operator
$$
u \mapsto \frac{\del u}{\del \tau} +
J_3^{\rho(\tau)}\Big(\frac{\del u}{\del t} -
X_{H_3^{\rho(\tau)}}(u)\Big)
$$
acting on the maps $u$ satisfying the asymptotic condition
$u(\pm\infty) = z^\pm$ and fixed homotopy condition $[u] = C \in
\pi_2(z^-,z^+)$. By perturbing $u_1\#_R u_2$ by the amount that is
smaller than the error for $u_1 \# _R u_2$ to be a genuine
solution, i.e., less than a weighted $L^p$-norm, for $p > 2$,
$$
\|\overline \del_{(\CH,j;\rho)}(u_1\#_{(\rho;R)} u_2)\|_p
$$
in a suitable $W^{1,p}$ space of $u$'s, one
can construct a unique genuine solution near $u_1 \#_R u_2$. By an
abuse of notation, we will denote this genuine solution also by
$u_1 \#_R u_2$. Then the corresponding map defines an embedding
\beastar
\CM\Big(j_1,\CH_1;[z_1,w_1],[z_2,w_2]\Big) \times
\CM\Big(j_2,\CH_2;[z_2,w_2],[z_3,w_3]\Big)
\times (R_0, \infty) \to \\
 \to  \CM\Big(j_1\#_{(\rho;R)} j_2,\CH_1\#_{(\rho;R)}
\CH_2;[z_1,w_1],[z_3,w_3]\Big).
\eeastar
Especially when we have
$$
\mu_{H_\beta}([z_2,w_2]) - \mu_{H_\alpha}([z_1,w_1]) =
\mu_{H_\gamma}([z_3,w_3]) - \mu_{H_\beta}([z_2,w_2]) = 0
$$
 both $\CM(j_1,\CH_1;[z_1,w_1],[z_2,w_2])$ and
$\CM(j_2,\CH_2;[z_2,w_2],[z_3,w_3])$ are compact, and so consist
of a finite number of points. Furthermore the image of the above
mentioned embedding exhausts the `end' of the moduli space
$$
\CM\Big(j_1\#_{(\rho;R)} j_2,\CH_1\#_{(\rho;R)}
\CH_2;[z_1,w_1],[z_3,w_3]\Big)
$$
and the boundary of its compactification consists of the broken
trajectories
$$
u_1\#_{(\rho; \infty)} u_2 = u_1 \#_\infty u_2.
$$
This then proves the following gluing identity
\begin{prop}\label{R0} There exists $R_0> 0$ such that for
any $R \geq R_0$ we have
$$
h_{(\CH_1,j_1)\#_{(\rho;R)}(\CH_2,j_2)} = h_{(\CH_1,j_1;\rho_1)} \circ
h_{(\CH_2,j_2;\rho_2)}
$$
as a chain map from $CF_*(H_\alpha)$ to $CF_*(H_\gamma)$.
\end{prop}
Here we remind the readers that the homotopy $\CH_1\#_{(\rho;R)}
\CH_2$ itself is an elongated homotopy of the glued homotopy
$\CH_1 \# \CH_2$. This proposition then gives rise to the composition law
$h_{\alpha\gamma} = h_{\beta\gamma}\circ h_{\alpha\beta}$
in homology.

This finishes the summary of construction of Floer complex and basic operations in
Floer theory. In particular, the chain homotopy map is defined whenever the
family $(\overline \CH, \overline J)$ where
$\overline \CH=\{\CH_\kappa\}, \, \overline J = \{J^\kappa\}$ are
smooth families over $0 \leq \kappa \leq 1$.
However the chain homotopy map used in PSS map that we have been considering
in the present paper is not this kind but induced by the concatenation of
two non-compact homotopies over $-\infty \leq \ell < 0$ and $0 < \e \leq 1$.

\section{$\Psi\circ \Phi = id$ ; Floer via Morse back to Floer}
\label{sec:fmf}

Consider the PSS deformation defined over $\kappa \in [-\infty,1]$.
We fix a homotopy $(K^\kappa,J^\kappa)$ as any generic homotopy from
$(K^{\e_0},J^{\e_0})$ to $(K^1,J^1) = (H(t,x), J)$.

Fix a sufficiently small $\e_0> 0$ and a sufficiently large $\ell_0 > 0$.
We divide the deformation into the following 5 pieces
$$
(K^{\kappa}, J^{\kappa}) \quad \mbox{ for $[\e_0\leq \kappa \leq 1]$},
$$
$$
(K_{R(\e)}, J_{R(\e)}) \quad \mbox{ for $0< \kappa \leq \e_0$},
$$
$$
(H^{\rho_-},J^{\rho_-}) {}_{o_-}*(f,J_0;[-\ell,\ell])*_{o_+}
(H^{\rho_+},J^{\rho_+}) \quad \mbox{ for $ -\ell_0 \leq \ell < 0$}:
$$
and
$$
(H^{\rho_-},J^{\rho_-}) {}_{o_-}*(f,J_0;[-\ell,\ell])*_{o_+}
(H^{\rho_+},J^{\rho_+}) \quad \mbox{ for $ -\infty \leq \ell < -\ell_0$}:
$$
Here $(f,J_0;[-\ell,\ell])$ stands for the deformation
$$
\ell \in (-\infty,0) \mapsto (f,J_0;[-\ell,\ell])
$$
where $f$ is a Morse function with respect to the metric $g_{J_0}$
and we consider its gradient trajectories over the interval $[-\ell,\ell]$.

We denote by $\CM_\kappa^{\Psi\Phi}([z_-,w_-]),[z_+,w_+])$
the moduli space of configuration corresponding
to $\kappa$ and form the parameterized moduli space
$$
\overline \CM^{para}_{\Psi\Phi}([z_-,w_-],[z_+,w_+];f) = \bigcup_{\kappa \in [-1,\infty]}
\CM_\kappa^{\Psi\Phi}([z_-,w_-],[z_+,w_+]).
$$
By the nondegeneracy hypothesis and the index condition,
$\CM_\kappa^{\Psi\Phi}$ is empty except at a finite number of points
$$
\kappa \in (-\ell_0,-\ell_1) \cup (\e_0, 1)
$$
but a priori those $\kappa$ could be accumulated in $[-\ell_1, \e_0]$.
The one-jet transversality of the enhanced nodal Floer trajectory
moduli space, which corresponds to $\kappa = 0$ and the
main gluing result of the present paper, proves that this accumulation
cannot be possible. As a result,
$$
\CM_\kappa^{\Psi\Phi}([z_-,w_-]),[z_+,w_+]) = \emptyset
$$
for all $\kappa \in [-\ell_1,\e_0]$ if we choose $\ell_1, \, \e_0$ sufficiently
small. Together with the main gluing compactness result of the present paper,
this discussion proves the following proposition

\begin{prop}
There exist constants $\ell_0, \, \ell_1, \, \e_0$ and
$\e_1$ such that the followings hold :
\begin{enumerate}
\item Suppose $\mu_H([z_-,w_-]) - \mu_H([z_+,w_+]) = -1$.
Then $\overline \CM^{para}_{\Psi\Phi}([z_-,w_-]),[z_+,w_+])$ is a compact
zero dimensional manifold such that
$$
\CM_{\kappa}^{\Psi\Phi}([z_-,w_-],[z_+,w_+]) = \emptyset
$$
for $\kappa \in [-\infty,-\ell_0] \cup [-\ell_1,\e_0] \cup [1 - \e_1,1]$.

\item Suppose $\mu_H([z_-,w_-]) - \mu_H([z_+,w_+]) = 0$.
Then $\overline \CM^{para}_{\Psi\Phi}([z_-,w_-]),[z_+,w_+];f)$ is a compact
one dimensional manifold with boundary
$\del \overline \CM^{para}_{\Psi\Phi}([z_-,w_-]),[z_+,w_+];f)$ consisting of
\beastar
&{}& \del \overline \CM^{para}_{\Psi\Phi}([z_-,w_-]),[z_+,w_+];f)  = \\
&{}& \CM_{1}([z_-,w_-],[z_+,w_+])
\cup \CM_{-\infty}([z_-,w_-],[z_+,w_+]) \\
&{}& \quad \cup \left(\bigcup_{[z,w]}\overline \CM^{para}_{\Psi\Phi}([z_-,w_-],[z,w])
\# \CM_{\kappa =1}([z,w],[z_+,w_+])\right) \\
&{}& \quad \cup \left(\bigcup_{[z,w]}\CM_{\kappa =1}([z_-,w_-],[z,w])
\# \overline \CM^{para}_{\Psi\Phi}([z,w],[z_+,w_+])\right)
\eeastar
where the union is taken over all $[z,w]$ with $\mu_H([z_-,w_-]) - \mu_H([z,w]) = -1$
for the first and $\mu_H([z,w]) - \mu_H([z_+,w_+]) = -1$ for the second.
\end{enumerate}
\end{prop}
Statement (1) in this proposition allows one to define the matrix coefficients
the order
$$
\# \overline \CM^{para}_{\Psi\Phi}([z_-,w_-]),[z_+,w_+];f).
$$
We then define the map
$$
h_{pss}^{\Psi\Phi}: CF_*(H) \to CF_{* +1}(H)
$$
by the matrix coefficients
$$
\langle h_{pss}^{\Psi\Phi}([z_-,w_-],[z_+,w_+] \rangle
: = \# \overline \CM^{para}_{\Psi\Phi}([z_-,w_-]),[z_+,w_+];f).
$$
Then Statement (2) concerning the description of the boundary of the
one dimensional moduli space
$\overline \CM^{para}_{\Psi\Phi}([z_-,w_-]),[z_+,w_+];f)$ is translated into the equation
$$
\Psi \circ \Phi - id = \del \circ h_{pss}^{\Psi\Phi} + h_{pss}^{\Psi\Phi} \circ \del.
$$
This finishes the proof $\Psi\circ \Phi = id$ in homology.

\section{$\Phi\circ \Psi = id$ ; Morse via Floer back to Morse}

In this section, for each given pair $p, \, q\in \operatorname{Crit}
f$, we consider the parameterized moduli space
$$
\overline\CM^{para}_{\Phi\Psi}(p,q) = \bigcup_{0 \leq R \leq \infty}
\CM^{\Phi\Psi}_R(p,q) :
$$
We define $\CM^{\Phi\Psi}_R(p,q)$ in the following way.

First for each $0 < R < \infty$, we introduce the moduli space
$\CM_{(2;0,0))}((K^R,J^R))$ of finite energy solutions of
\be\label{eq:K^RJ^R} \delbar_{(K^R,J^R)} u = 0 \ee on $\Sigma$ which
is a Riemann surface with two marked points $\{o_-,o_+\}$ so that
$\Sigma \setminus \{o_\pm\} \cong \R \times S^1$ conformally. We
first define a family of Riemann surface $(\Sigma,j_R)$ by the
connected sum
$$
(D^-,o_-) \cup C_R \cup (D^+,o_+), \quad j_R = j_{D^-}\# j_{C_R} \#
j_{D^+}
$$
where $C_R$ is the cylinder $[-R,R] \times S^1$, $j_{C_R}$ the
standard conformal structure and $j_R$ is the obvious glued
conformal structure on $D^- \cup C_R \cup D^+$. We denote $(\tau,t)$
the conformal coordinates on $D^- \cup C_R \cup D^+ \setminus
\{o_-,o_+\}$ extending the standard coordinates on $\C_R$.

In this conformal coordinates, we fix a family of cut-off functions
$\chi^R$ by
$$
\chi^R(\tau) = \begin{cases} 1 - \kappa^+(\tau-R) \quad & \mbox{for }\, \tau \geq 0 \\
1- \kappa^-(\tau + R) \quad & \mbox{for }\,\tau \leq 0
\end{cases}
$$
for $ 1 \leq R < \infty$, and $\chi^R = R \chi^1$ for $ 0 \leq R
\leq 1$. We note that $\chi^0 \equiv 0$ and $\chi^R$ has compact
support and $\chi^R \equiv 1$ on any given compact subset if $R$ is
sufficiently large. Therefore the equation (\ref{eq:K^RJ^R}) is
reduced to $\delbar_{J_0}u =0$ near the marked points $o_\pm$.
Then we define $(K^R,J^R)$ as in subsection \ref{subsec:resolved}.

We have two evaluations
$$
ev_{o_\pm}: \CM_{(2;0,0)}(K^R,J^R) \to M ; \quad ev_{o_\pm}(u) =
u(o_\pm).
$$
We denote \beastar \widetilde \CM^-(p;f) & = & \{ \chi : \R \times M
\mid \dot\chi + \nabla f(\chi) = 0, \,
\chi(-\infty) = p \}\\
\widetilde \CM^+(q;f) & = & \{ \chi : \R \times M \mid \dot\chi +
\nabla f(\chi) = 0, \, \chi(+\infty) = q \} \eeastar and define
\beastar
\widetilde \CM^-_1(p;f) & = & \widetilde \CM^-(p;f) \times \R,\\
\widetilde \CM^+_1(q;f) & = & \widetilde \CM^+(q;f) \times \R.
\eeastar $\tau_0 \in \R$ acts on both by the action
$$
(\tau_0, (\chi,\tau)) \mapsto (\chi(* - \tau_0), \tau + \tau_0).
$$
This action is free and so their quotients
$$
\CM^-_1(p;f) = \widetilde \CM^-_1(p;f)/\R, \quad \CM^+_1(q;f) =
\widetilde \CM^+_1(q;f)/\R
$$
become smooth manifold of dimension $\mu_{Morse}(p;f)$ and $2n -
\mu_{Morse}(q;f)$ respectively. We have the evaluation maps
$$
ev_+: \CM^+_1(q;f) \to M, \quad ev_-:\CM^-_1(p;f) \to M
$$
whose image has one-one correspondence with the unstable manifold
$W^u(p;f)$ and the stable manifold $W^s(q;f)$ respectively.

Now we define the moduli space $\CM^{\Phi\Psi}_R(p,q;A)$ to be the
fiber product \beastar \CM^{\Phi\Psi}_R(p,q;A) & = &
\CM^-_1(p;f){}_{ev_-}\times_{ev_{o_-}}\CM_{(2;0,0)}(K^R,J^R;A)
{}_{ev_{o_+}}\times_{ev_+} \CM^+_1(q;f)\\
& = & \{((\chi_-,\tau_-),u,(\chi_+,\tau_+))\mid \chi_-(\tau_-) =
u(o_-), \, \chi_+(\tau_+) = u(o_+) \} \eeastar and
$$
\overline \CM^{\Phi\Psi,para}(p,q;A) = \bigcup_{0 \leq R \leq
\infty} \CM^{\Phi\Psi}_R(p,q;A).
$$
A straightforward calculation shows that
$$
\dim^{virt} \CM^{\Phi\Psi}_R(p,q;A) = \mu_{Morse}(p) -
\mu_{Morse}(q) + 2c_1(A).
$$

\begin{prop}\label{bdyMMPhiPsi} Choose a generic pair $(f,J_0)$.
\begin{enumerate}
\item Suppose that $\mu_{Morse}(p) - \mu_{Morse}(q) + 2c_1(A) = -1$.
Then there exist some $\e_1 > 0$ and $R_1 > 0$ such that Then
$\overline \CM^{\Phi\Psi,para}_R(p,q;A)$ is a compact 0 dimensional
manifold such that
$$
\CM^{\Phi\Psi}_R(p,q;A) = \emptyset
$$
if $0 \leq R \leq \e_1$ or $R \geq R_1$.
\item Suppose that $\mu_{Morse}(p) - \mu_{Morse}(q) + 2c_1(A) = 0$.
Then $\overline \CM^{\Phi\Psi,para}(p,q;A)$ is a compact
one-manifold with boundary given by
$$
\del \overline \CM^{\Phi\Psi,para}(p,q;A) = \CM^{\Phi\Psi}_0(p,q;A)
\cup \CM^{\Phi\Psi}_\infty(p,q;A) \cup \bigcup_{r}\overline
\CM^{\Phi\Psi}(p,r;A)
$$
where the union $\bigcup_r$ is taken over $r \in
\operatorname{Crit}f$ such that
$$
\mu_{Morse}(p) - \mu_{Morse}(r) + 2c_1(A) = -1.
$$
\end{enumerate}
\end{prop}
\begin{proof} We recall that when $R = 0$, the equation
(\ref{eq:K^RJ^R}) is reduced to $\delbar_{J_0} u = 0$. Since
$\mu_{Morse}(p) - \mu_{Morse}(q) + 2c_1(A)=-1$ represents the
virtual dimension of $\CM^{\Phi\Psi}_0(p,q;A)$,
$\CM^{\Phi\Psi}_0(p,q;A)$ must be empty for a generic choice of
$(f,J_0)$. Here we emphasize the fact that this moduli space depends
only on $(f,J_0)$ for which the genericity argument can be applied
independent of the parameter $R$. Therefore the same must be the
case when $R_1 \leq \e_1$ for a sufficiently small $\e_1> 0$. This
finishes the proof.

We leave the proof of Statement (2) to the readers.
\end{proof}

Using Statement (1), we define the chain homotopy map
$$
h_{pss}^{\Phi\Psi} : CM_*(f,J_0;\Lambda_\omega) \to
CM_{*+1}(f,J_0;\Lambda_\omega)
$$
by the matrix element
$$
\langle h_{pss}^{\Phi\Psi}(p), q \# (-A) \rangle = \sum_{(r,A)}
\#\left(\bigcup_{r}\overline \CM^{\Phi\Psi,para}(p,r;A)\right).
$$
Next we prove the following lemma

\begin{lem} Suppose $\mu_{Morse}(p) - \mu_{Morse}(q) + 2c_1(A)= 0$.
Then if $A \neq 0$,
$$
\dim \CM^{\Phi\Psi}_0(p,q;A) \geq 2
$$
unless $\CM^{\Phi\Psi}_0(p,q;A)= \emptyset$. And when $A = 0$, we
have
$$
\dim \CM^{\Phi\Psi}_0(p,q;A) \geq 1
$$
unless $p = q$.
\end{lem}
\begin{proof} If $A \neq 0$,
$u$ is non-constant in $(u;o_-,o_+) \in \CM_{(2;0,0))}((K^R,J^R)) $.
Then the conformal automorphism on the domain $(\Sigma;o_-,o_+)$ produces at
least a real 2-dimensional family which contradicts the index hypothesis.
(See \cite{floer:fixed}, \cite{FHS} for the semi-positive case and
\cite{fukaya-ono}, \cite{liu-tian} in general.)

On the other hand, if $A = 0$, any $J_0$-holomorphic sphere must be
constant and so the corresponding configuration $(\chi_-,const,
\chi_+)$ becomes a full gradient trajectory $\chi =\chi_- \#
\chi_+$. Unless $\chi$ is constant, i.e., unless $p = q$,
$\R$-translation produces at least one-dimensional family which
again contradicts to the index hypothesis. This finishes the proof.
\end{proof}

Now we are ready to finish the proof of the identity
\be\label{eq:PhiPsi-id} \Phi\circ\Psi - id = h_{pss}^{\Phi\Psi}
\del_{(f,J_0)}^{Morse} + \del_{(f,J_0)}^{Morse} h_{pss}^{\Phi\Psi}.
\ee A priori, Proposition \ref{bdyMMPhiPsi} only implies
$$
\sum_{q,A} \langle (\Phi\circ\Psi - id)(p), q \#(-A) \rangle =
\sum_{q,A} \langle h_{pss}^{\Phi\Psi} \del_{(f,J_0)}^{Morse}(p) +
\del_{(f,J_0)}^{Morse} h_{pss}^{\Phi\Psi}(p), q\# (-A) \rangle.
$$
But the above lemma implies
$$
\langle (p), q \#(-A) \rangle = 0
$$
unless $A = 0$ and $p = q$. This finishes the proof of
(\ref{eq:PhiPsi-id}).

%\printindex

\end{document}